\let\chooseClass1   
\let\chooseClass3   

\ifx\chooseClass1
\documentclass[reqno]{memo-l}
\fi

\ifx\chooseClass2
\IfFileExists{extarticle.cls}{
\documentclass[14pt]{extarticle}
\emergencystretch 7 pt
\setlength{\oddsidemargin}{-17 mm} 
\setlength{\textwidth}{192 mm}     
\setlength{\textheight}{246 mm}    
\setlength{\topmargin}{-31 mm}     
}{
\documentclass[12pt]{article}
\textheight = 8.5in
\textwidth 6.3in
\setlength{\oddsidemargin}{0mm}
\setlength{\topmargin}{-15 mm}
}{}
\fi

\ifx\chooseClass3
\documentclass[12pt]{book}
\fi

\ifx\chooseClass4
\documentclass[a4paper,14pt]{extbook}
\fi

\ifx\chooseClass1
\makeatletter
\@addtoreset{subsection}{chapter}
\makeatother
	\else
\makeatletter

    \def\chaptermark#1{%
      \markboth {{%
        \ifnum \c@secnumdepth >\m@ne
          \if@mainmatter
            \thechapter. \ %
          \fi
        \fi
        #1}}{}}
    \def\sectionmark#1{%
      \markright {{%
        \ifnum \c@secnumdepth >\z@
          \thesection. \ %
        \fi
        #1}}}

\def\@makechapterhead#1{%
  \vspace*{0\p@}%
  {\parindent \z@ \raggedright \normalfont
    \ifnum \c@secnumdepth >\m@ne
      \if@mainmatter
        \large\bfseries \@chapapp\space \thechapter
        \par\nobreak
        \vskip 20\p@
      \fi
    \fi
    \interlinepenalty\@M
    \large \bfseries #1\par\nobreak
    \vskip 40\p@
  }}

\def\@makeschapterhead#1{%
  \vspace*{0\p@}%
  {\parindent \z@ \raggedright
    \normalfont
    \interlinepenalty\@M
    \large \bfseries  #1\par\nobreak
    \vskip 40\p@
  }}

\def\@seccntformat#1{\csname the#1\endcsname.\quad}
\renewcommand\section{\@startsection {section}{1}{\z@}%
                                    {3.25ex \@plus1ex \@minus.2ex}%
                                   {-2.3ex \@plus -.2ex}%
                                   {\normalfont\normalsize\bfseries}}
\renewcommand\subsection{\@startsection{subsection}{2}{\z@}%
                                    {3.25ex \@plus1ex \@minus.2ex}%
                                     {-1.5ex \@plus -.2ex}%
                                     {\normalfont\normalsize\bfseries}}
\makeatother
\fi

\ifx\chooseClass1
\swapnumbers
\newtheorem{theorem}[subsection]{Theorem}
\newtheorem{corollary}[subsection]{Corollary}
\newtheorem{lemma}[subsection]{Lemma}
\newtheorem{proposition}[subsection]{Proposition}

\theoremstyle{definition}
\newtheorem{definition}[subsection]{Definition}
\newtheorem{example}[subsection]{Example}
\newtheorem{exercise}[subsection]{Exercise}

\theoremstyle{remark}
\newtheorem{remark}[subsection]{Remark}
\newtheorem{warning}[subsection]{Warning}
	\else
\usepackage{amsthm}
\newtheoremstyle{boldhead}
{\topsep}
{\topsep}
{\slshape}
{}
{\bfseries}
{.}
{ }
{\thmname{#1}\thmnumber{ #2}\thmnote{ (#3)}}

\newtheoremstyle{boldremark}
{\topsep}
{\topsep}
{\upshape}
{}
{\bfseries}
{.}
{ }
{\thmname{#1}\thmnumber{ #2}\thmnote{ (#3)}}

\swapnumbers

\theoremstyle{boldhead}
\newtheorem{theorem}[subsection]{Theorem}
\newtheorem{corollary}[subsection]{Corollary}
\newtheorem{lemma}[subsection]{Lemma}
\newtheorem{proposition}[subsection]{Proposition}

\theoremstyle{boldremark}
\newtheorem{definition}[subsection]{Definition}
\newtheorem{example}[subsection]{Example}
\newtheorem{exercise}[subsection]{Exercise}
\newtheorem{remark}[subsection]{Remark}
\newtheorem{warning}[subsection]{Warning}
\fi

\usepackage{    amsmath,
                amsfonts,
                amssymb,
                amsxtra,
                stmaryrd,
                verbatim,
}

\usepackage[matha,mathx]{mathabx}

\numberwithin{section}{chapter}
\numberwithin{equation}{section}

\usepackage{t-angles} 
\IfFileExists{url.sty}{\usepackage{url}}{}
\providecommand{\url}[1]{{\tt #1}}

\usepackage{ifpdf}
\ifpdf
    \IfFileExists{hyperref.sty}{\usepackage[pdftex]{hyperref}}{}
\else
    \IfFileExists{hyperref.sty}{\usepackage[hypertex]{hyperref}}{}
\fi

    \ifpdf
        \usepackage[pdftex]{graphicx}
    \else
        \usepackage[dvips]{graphicx}
    \fi

\IfFileExists{euscript.sty}{\usepackage[mathcal]{euscript}}{}
\IfFileExists{mathrsfs.sty}{\usepackage{mathrsfs}}{\let\mathscr\mathfrak}

\usepackage{diagrams}
\diagramstyle[height=2em,balance,righteqno,PostScript=dvips,nohug]

\IfFileExists{dsfont.sty}{\usepackage[sans]{dsfont}%
\newcommand\1{{\mathds 1}}}{\newcommand\1{{1\mkern-5mu {\mathrm I}}}}

\newlength{\mylabelwidths}
\newenvironment{myitemize}{\begin{list}{}{%
\setlength{\labelwidth}{\mylabelwidths}%
\setlength{\leftmargin}{\mylabelwidths}\addtolength{\leftmargin}{0.5em}%
\setlength{\itemsep}{-0.0\baselineskip}}}
{\end{list}}

\newarrowtail{sscurlyvee}{\raise.18ex\hbox{$\sss\succ$}}%
{\mkern-.5mu\raise.18ex\hbox{$\sss\prec$}\mkern.4mu}%
{\sss\curlyvee}{\sss\curlywedge}

\newarrow{Cof}{sscurlyvee}---{->}
\newarrow{DashTo}{}{dash}{}{dash}{->}
\newarrow{Epi}----{triangle}
\newarrow{MapsTo}{mapsto}---{->}
\newarrow{Mono}{boldhook}---{->}
\newarrow{TTo}----{->}
\newarrow{Twoar}===={=>}

\newcommand\DD{{\mathbb D}}
\newcommand\NN{{\mathbb N}}
\newcommand{\QQ}{{\mathbb Q}}
\newcommand\ZZ{{\mathbb Z}}

\newcommand{\cb}{{\mathcal B}}
\newcommand{\cc}{{\mathcal C}}
\newcommand{\ce}{{\mathcal E}}

\newcommand{\co}{{\mathcal O}}
\newcommand{\cp}{{\mathcal P}}
\newcommand{\cq}{{\mathcal Q}}
\newcommand{\cs}{{\mathcal S}}
\newcommand{\ct}{{\mathcal T}}
\newcommand{\cv}{{\mathcal V}}
\newcommand{\cw}{{\mathcal W}}
\newcommand{\cz}{{\mathcal Z}}

\newcommand{\mm}{{\mathsf M}}
\newcommand{\one}{{\mathsf1}}
\newcommand{\sfb}{{\mathsf b}}
\newcommand{\sfi}{{\mathsf i}}
\newcommand{\sfj}{{\mathsf j}}
\newcommand{\sfh}{{\mathsf h}}

\newcommand{\sfv}{{\mathsf v}}
\newcommand{\sfw}{{\mathsf w}}

\newcommand{\bb}{{\mathbf b}}
\newcommand{\bd}{{\mathbf d}}
\newcommand{\bdelt}{{\boldsymbol\delta}}
\newcommand{\beps}{{\boldsymbol\eps}}
\newcommand{\bfeta}{{\boldsymbol\eta}}
\newcommand{\bi}{{\mathbf i}}
\newcommand{\bj}{{\mathbf j}}
\newcommand{\bn}{{\mathbf n}}
\newcommand{\bone}{{\mathbf1}}
\newcommand{\bv}{{\mathbf v}}
\newcommand{\bull}{{\scriptscriptstyle\bullet}}

\newcommand{\bw}{{\mathbf w}}
\newcommand{\hu}{{\mathsf{hu}}}
\newcommand{\Sim}{{\mkern-4.5mu\sim\mkern1.5mu}}
\newcommand{\sk}{{\mathsf{sk}}}
\newcommand{\su}{{\mathsf{su}}}
\newcommand{\tdt}{\otimes\dots\otimes}

\newcommand{\lar}[1]{\hbox{\,\large #1\,}}
\newcommand{\sS}[2]{\vphantom{#2}#1 #2}
\newcommand{\n}[1]{\nobreakdash-\hspace{0pt}}
\newcommand{\ainf}[1]{$A_\infty$\nobreakdash-\hspace{0pt}}
\newcommand{\ainfm}[1]{$\mathrm{A}_\infty$\nobreakdash-\hspace{0pt}}
\newcommand{\mainf}{\mathrm{A}_\infty}
\newcommand{\Cat}{{\mathcal C}at}

\newcommand{\bbodot}{\bigodot^{\vrule height-.1em width.7em depth.14em}}

\let\con\triangleright
\let\emptyset\varnothing
\let\eps\varepsilon
\let\ge\geqslant
\let\kk\Bbbk
\let\le\leqslant

\let\mb\mathbf
\let\msf\mathsf
\let\rto\xrightarrow
\let\sss\scriptstyle
\let\tens\otimes
\let\ttt\textstyle
\let\und\underline
\let\wh\widehat

\newcommand\bott{{\bot\mkern-11mu\bot}}
\newcommand\botth{\hat{\bot\mkern-11mu\bot}}
\newcommand\bottho{\hat{\bot\mkern-11mu\bot}_\circ}
\newcommand\botto{{\bot\mkern-11mu\bot}_\circ}
\newcommand\TT{{\top\mkern-11mu\top}}

\newcommand\cO{{\mathcal O}}
\newcommand\mcC{{\mathsf C}}
\newcommand\sff{{\mathsf f}}
\newcommand\sfg{{\mathsf g}}

\newcommand{\tensu}[1]{\underset{#1}{\otimes}}
\newcommand{\Tu}[1]{T_{#1}}

\newcommand{\alg}{\textup{-alg}}
\newcommand{\bimod}{\textup{-bimod}}
\newcommand{\cAlg}{\textup{-cAlg}}
\newcommand{\cCoalg}{\textup{-cCoalg}}
\newcommand{\coalg}{\textup{-coalg}}
\newcommand{\comm}{\textup{comm-}}
\newcommand{\comodul}{\textup{-comod}}
\newcommand{\comodur}{\textup{comod-}}

\newcommand{\comodplur}{\textup{comod}_+\text-}
\newcommand{\comodplurnd}{\textup{comod}^{nd}_+\text-}
\newcommand{\comodcoopl}{\textup{comod}_+\textup{Coop}_{++}}
\newcommand{\Fc}{\textup{Fc}}
\newcommand{\Fo}{\textup{Fo}}
\newcommand{\modul}{\textup{-mod}}
\newcommand{\modur}{\textup{mod-}}
\newcommand{\modplur}{\textup{mod}_+\text-}
\newcommand{\modplurnd}{\textup{mod}^{nd}_+\text-}

\newcommand{\Pa}{P}

\DeclareMathOperator\Ab{Ab}
\DeclareMathOperator\Ass{{\it As1}}
\DeclareMathOperator\augcoop{augCoop}
\DeclareMathOperator{\Bbar}{Bar}
\DeclareMathOperator{\sfBar}{{\sf Bar}}
\DeclareMathOperator{\Card}{Card}
\DeclareMathOperator{\Cobar}{Cobar}
\DeclareMathOperator{\sfCobar}{{\sf Cobar}}
\DeclareMathOperator{\augcoopl}{augCoop_{++}}
\DeclareMathOperator{\Coop}{Coop}
\DeclareMathOperator{\acCoop}{acCoop}
\DeclareMathOperator{\caCoop}{caCoop}
\DeclareMathOperator{\CACoop}{CACoop}
\DeclareMathOperator{\cCoop}{cCoop}
\DeclareMathOperator{\ucCoop}{ucCoop}
\DeclareMathOperator{\cOp}{cOp}

\DeclareMathOperator\dg{\mathbf{dg}}
\DeclareMathOperator\End{End}
\DeclareMathOperator\END{{\mathcal E}{\it nd}}
\DeclareMathOperator\ev{ev}
\DeclareMathOperator\gr{\mathbf{gr}}
\DeclareMathOperator\height{height}
\DeclareMathOperator\Hom{Hom}
\DeclareMathOperator\id{id}
\DeclareMathOperator\Id{Id}
\DeclareMathOperator\im{Im}
\DeclareMathOperator\inj{in}
\DeclareMathOperator\Inp{Inp}

\DeclareMathOperator\IV{v}
\DeclareMathOperator\uv{u}
\DeclareMathOperator\Ker{Ker}
\DeclareMathOperator{\ModOp}{ModOp}
\DeclareMathOperator\Mor{Mor}
\DeclareMathOperator\nucoop{nuCoop}

\DeclareMathOperator\Ob{Ob}
\DeclareMathOperator\oin{\overline{\textup{in}}}
\newcommand{\op}{{\operatorname{op}}}
\DeclareMathOperator\Op{Op}
\DeclareMathOperator\opr{\overline{\textup{pr}}}
\DeclareMathOperator\pr{pr}
\DeclareMathOperator\troot{root}
\DeclareMathOperator\Set{\mathcal Set}
\DeclareMathOperator\str{\mathbf{str}}
\DeclareMathOperator\tr{\mathbf{tr}}
\DeclareMathOperator{\Tw}{Tw}
\DeclareMathOperator{\ucOp}{ucOp}
\DeclareMathOperator{\uccOp}{uccOp}
\DeclareMathOperator{\ucdgOp}{uc\mathbf{dg}Op}
\DeclareMathOperator{\UCCOp}{UCCOp}

\newcommand{\corref}[1]{Corollary~\ref{#1}}
\newcommand{\defref}[1]{Definition~\ref{#1}}
\newcommand{\exaref}[1]{Example~\ref{#1}}
\newcommand{\exeref}[1]{Exercise~\ref{#1}}
\newcommand{\lemref}[1]{Lemma~\ref{#1}}
\newcommand{\propref}[1]{Proposition~\ref{#1}}
\newcommand{\remref}[1]{Remark~\ref{#1}}
\newcommand{\secref}[1]{Section~\ref{#1}}

\ifx\chooseClass1
\else
\usepackage{makeidx}

\fi

\usepackage{amsmidx}
\makeindex{TTindex}
\makeindex{TTsymb}

\begin{document}
\frontmatter
\title{Curved cooperads and homotopy unital $A_\infty$-algebras}
\author{Volodymyr Lyubashenko}
\ifx\chooseClass1
\address{Institute of Mathematics NASU \\
3 Tereshchenkivska st. \\
Kyiv-4, 01601 MSP \\
Ukraine}
\email{lub@imath.kiev.ua}
\date{February 2014}

\subjclass[2010]{Primary 18D50; Secondary 18C15 }

\keywords{Curved operad, curved cooperad, bar construction, cobar construction, twisting cochain, homotopy unital $A_\infty$-algebra}

\begin{abstract}
We provide bar and cobar constructions as functors acting between various categories of curved operads and curved cooperads.
Cobar and bar constructions are adjoint to each other.
Given a twisting cochain between a curved augmented cooperad $C$ with an extra grading and a curved operad $\co$ we construct a couple of adjoint functors between the category of curved $\co$\n-modules and the category of curved $C$\n-comodules.
The important feature is that the curved operad $\co$ is not necessarily augmented.
\end{abstract}
\fi

\maketitle

\tableofcontents

\allowdisplaybreaks[1]

\chapter*{Introduction}
Accordingly to Positselski \cite{0905.2621,1202.2697} a curved (co)algebra is a graded (co)unital associative (co)algebra equipped with a degree 1 (co)derivation which nearly turns it into a $\dg$-(co)algebra.
Namely, the square of this (co)derivation does not vanish, but it is an inner (co)derivation.
Curved algebras arise naturally in geometry, when a connection and its curvature are involved \cite[Section~0.6]{0905.2621}, \cite{Polishchuk:curved-dg}, \cite[Example~1.5]{Lyu-curved-coalgebras}.

In this article curved notions extend to operads and cooperads.
In distinction from the algebra case, there is a naturally arising example of a curved cooperad.
Namely, there is an augmented curved cooperad $C$ whose cobar construction $\Cobar C$ is isomorphic to the $\dg$\n-operad $A_\infty^\hu$ of homotopy unital \ainf-algebras. 
The latter are studied in \cite{FukayaOhOhtaOno:Anomaly,LyuMan-unitalAinf,Lyu-Ainf-Operad,1110.1959}.
The peculiar isomorphism \(\Cobar C\simeq A_\infty^\hu\) was already observed by Hirsh and Mill{\`e}s \cite[Theorem~6.1.8, Proposition~6.1.9]{MR2993002}.
We extend the bar and cobar constructions to wider categories of operads and cooperads than \cite{MR2993002}.
In particular, the operad can be really curved instead of being differential graded.
Besides, Hirsh and Mill{\`e}s work also with (co)properads.
The latter notions were introduced by Vallette \cite{MR2320654}.
He studied bar and cobar constructions for augmented (co)properads \cite[Section~4.2]{math/0609002}, see also Merkulov and Vallette \cite{0707.0889}.

As we shall see, the augmentation is not necessary for curved operads, as it was for ordinary $\dg$\n-operads in the work of Getzler and Jones \cite[Section~2]{Getzler:hep-th/9403055}.
Instead it is assumed that the unit \(\eta:\1\to\co\) of the operad $\co$ has a direct complement.
That is, there is a functional \(\sfv:\co\to\1\) splitting the embedding $\eta$.
Unlike an augmentation this functional is not supposed to agree with multiplication.
On the other hand, for curved cooperads the augmentation is more significant.
Currently, the picture looks asymmetrically: unit-complemented $\dg$\n-operads are related by bar-cobar adjunction to augmented curved cooperads.
This asymmetry persists when the considered categories are enlarged.

In notation style we follow \cite{Lyu-curved-coalgebras} where the case of curved algebras and coalgebras is considered.
However, the ground is taken from \cite{1402.0408}.
While in the former article the studied (co)algebras are graded $\kk$\n-modules over a graded commutative ring $\kk$, in the latter one (co)algebras are objects of a category $\cv$ with suitable properties.
The category $\gr=\gr\text-\kk\modul$ of $\ZZ$\n-graded $\kk$\n-modules is one of the instances of $\cv$.

In the first chapter we list in axiomatic form properties of the ground category $\cv$ which resemble the category of graded $\kk$\n-modules.
These properties allow to consider morphisms of arbitrary (integer) degree.

In the second chapter we explain what are operads and cooperads in $\cv$.
For us these terms mean non-symmetric operads and cooperads.
We do not work with symmetric ones.
Operads can be defined as algebras for certain monad which is a sum over trees.
We describe this kind of trees in detail.
The same kind of trees is used in defining a comonad, whose coalgebras are called conilpotent non-counital cooperads.
General cooperads can be defined not only in traditional way, but also as coalgebras in a colax Monoidal category of collections in $\cv$.

Chapter 3 deals with cofree conilpotent non-counital cooperads which are cofree coalgebras over the mentioned comonad.
Categories of non-counital cooperads and of augmented cooperads are equivalent.
This gives a large supply of augmented cooperads and we study this notion from the point of view of comonads.

In chapter 4 we consider derivations (resp. coderivations) of certain degree of operads (resp. cooperads).
These are identified with certain infinitesimal morphisms.
It's easy to describe derivations from a free operad to an arbitrary operad.
Similarly using relevant comonad one describes coderivations from a conilpotent non-counital cooperad to a cofree conilpotent non-counital cooperad.
This allows to describe also coderivations between associated augmented cooperads.
We discuss also infinitesimal deformations of comodules over cooperads.
An important example of a cooperad is described.

In chapter 5 we define curved operads and curved cooperads as well as their versions: unit-complemented curved operads and counit-complemented curved cooperads.
Their special cases such as unit-complemented $\dg$\n-operads, curved augmented cooperads, augmented curved cooperads are also defined.
Cobar construction $\Cobar$ and bar construction $\Bbar$ are described as functors.

These functors are adjoint to each other as shown in chapter~6.
Moreover, the relevant sets of morphisms are in bijection with the set of twisting cochains.
There is a functor sending curved cofree coalgebras over a cooperad $C$ to curved $\Cobar C$-algebras.
Similarly, there is a functor sending curved free algebras over an operad $\co$ to curved $\Bbar\co$-coalgebras.
Furthermore, a twisting cochain determines a functor taking a curved module over the curved operad to a curved comodule over the curved cooperad.
With some extra conditions and restrictions a twisting cochain determines a functor taking a curved comodule over the curved augmented cooperad to a curved module over the curved operad.
The latter two functors are adjoint to each other.

\ifx\chooseClass4
\subsection*{Summary.}
We provide bar and cobar constructions as functors acting between various categories of curved operads and curved cooperads.
Cobar and bar constructions are adjoint to each other.
Given a twisting cochain between a curved augmented cooperad $C$ with an extra grading and a curved operad $\co$ we construct a couple of adjoint functors between the category of curved $\co$\n-modules and the category of curved $C$\n-comodules.
The important feature is that the curved operad $\co$ is not necessarily augmented.
\fi

\mainmatter

\chapter{The ground category}
\setcounter{section}1
In this article \((\cv,\tens,\1)\) denotes\index{TTsymb}{V@$\cv$} a closed symmetric Monoidal additive category with countable colimits and countable products.
Countable direct sum in $\cv$ means the coproduct.
We assume that $\cv$ is idempotent-\hspace{0pt}split (Karoubian) and the tensor product preserves countable colimits.
Often we present it as a colax Monoidal category \((\cv,\tens,\lambda)\), defined as an opposite to a lax Monoidal category \((\cv^\op,\tens,\lambda^\op)\), \textit{cf.} \cite[Definition~2.5]{BesLyuMan-book}.
For the sake of convenience we use the structure morphisms \(\lambda^\phi:\tens^{j\in J}\tens^{i\in\phi^{-1}j}X_i\to\tens^{i\in I}X_i\), \(\phi:I\to J\in\co_\sk\), for colax Monoidal categories, assuming them invertible in Monoidal case.

Define a symmetric strict monoidal category $\cz$ with \(\Ob\cz=\ZZ$.
The sets of morphisms $\cz(m,n)$ are empty if $m\ne n$ and \(\cz(n,n)=\mu_2\overset{\text{def}}=\{1,-1\}\) is the group of two elements for $n\in\ZZ$.
The tensor product on objects is the addition, the tensor product on morphisms \(\tens:\cz(a,a)\times\cz(b,b)\to\cz(a+b,a+b)\) is the multiplication \(\mu_2\times\mu_2\to\mu_2\) in the group $\mu_2$.
Finally, the symmetry is chosen as
\[ c_{a,b} =(-1)^{ab} \in \cz(a+b,b+a), \qquad a,b\in\ZZ.
\]

We assume given a strong colax symmetric monoidal functor \((\1[-],\nu):\cz\to\cv\), \(n\mapsto\1[n]\), \(\nu_{a,b}:\1[a+b]\rto\sim \1[a]\tens\1[b]\).
For the main exposition we shall suppose that this monoidal functor is strict.
In particular \(\1[0]=\1\) is the unit object of $\cv$.
We require that \(-1\in\cz(n,n)\) were represented by $-\id_{\1[n]}$.
Let us derive some consequences of this structure.

First of all, there are functors \([n]=\_\tens\1[n]:\cv\to\cv\), \(X\mapsto X[n]=X\tens\1[n]\), $n\in\ZZ$.
Together they form a translation structure \cite[Definition~13.2]{BesLyuMan-book}, that is, a monoidal functor
\[ [-]: \ZZ \rMono \cz \rTTo^{\1[-]} \cv \rTTo^{R^\tens} \End\cv,
\]
where $\ZZ$ is the discrete monoidal subcategory of $\cz$ with the set of objects $\ZZ$, \(X(YR^\tens)=X\tens Y\).
Almost the same notion is called a weak action of the group $\ZZ$ on a category $\cv$ \cite[D\'efinition~1.2.2]{MR1453167}.
For the sake of simplicity in the next chapters we act as if the monoidal functor $[-]$ were strict.

Let us use the closedness of $\cv$.
The identity map $\id_{X[n]}$ admits a presentation
\[ \id_{X[n]} =\bigl\langle X\tens\1[n] \rTTo^{1\tens\sigma^n} X\tens\und\cv(X,X[n]) \rTTo^\ev X[n] \bigr\rangle
\]
for a unique morphism
\[ \sigma^n =\sigma_X^n: \1[n] \to \und\cv(X,X[n]) \in \cv.
\]

\begin{proposition}
For all $a,b\in\ZZ$, \(X\in\Ob\cv\)
\begin{multline}
\bigl\langle \1[a+b] \rTTo_\sim^{\nu_{a,b}} \1[a]\tens\1[b] \rTTo^{\sigma_X^a\tens\sigma_{X[a]}^b} \und\cv(X,X[a])\tens\und\cv(X[a],X[a][b])
\\
\rTTo^{1\tens\und\cv(1,X\tens\nu_{a,b}^{-1})} \und\cv(X,X[a])\tens\und\cv(X[a],X[a+b]) \rto m \und\cv(X,X[a+b]) \bigr\rangle =\sigma_X^{a+b}.
\label{eq-X1[a+b]1nu-X1[a]1[b]}
\end{multline}
\end{proposition}

\begin{proof}
This equation follows from the computation in notation from \cite{BesLyuMan-book}:
\begin{multline*}
\bigl\langle X\tens\1[a+b] \rTTo_\sim^{1\tens\nu} X\tens\1[a]\tens\1[b]
\rTTo^{1\tens\sigma^a\tens\sigma^b} X\tens\und\cv(X,X[a])\tens\und\cv(X[a],X[a][b])
\\
\rTTo^{1\tens1\tens\und\cv(1,X\tens\nu^{-1})} X\tens\und\cv(X,X[a])\tens\und\cv(X[a],X[a+b])
\\
\hfill \rTTo^{1\tens m} X\tens\und\cv(X,X[a+b]) \rto\ev X[a+b] \bigr\rangle \quad
\\
=\bigl\langle X\tens\1[a+b] \rTTo_\sim^{1\tens\nu} X\tens\1[a]\tens\1[b] \rTTo^{1\tens\sigma^a\tens\sigma^b} X\tens\und\cv(X,X[a])\tens\und\cv(X[a],X[a][b])
\\
\hfill \rTTo^{\ev\tens1} X[a]\tens\und\cv(X[a],X[a][b]) \rto\ev X[a][b] \rTTo^{X\tens\nu^{-1}} X[a+b] \bigr\rangle \quad
\\
\hskip\multlinegap =\bigl\langle X\tens\1[a+b] \rTTo_\sim^{1\tens\nu} X\tens\1[a]\tens\1[b] \rTTo^{1\tens\nu^{-1}} X\tens\1[a+b] \bigr\rangle =\id. \hfill
\end{multline*}
Thus composition \eqref{eq-X1[a+b]1nu-X1[a]1[b]} equals $\sigma_X^{a+b}$.
\end{proof}

\begin{proposition}
The degree $n$ map $\sigma_X^n$ is dinatural in $X$: for all \(f:X\to Y\in\cv\)
\begin{diagram}[LaTeXeqno]
\1[n] &\rTTo^{\sigma_X^n} &\und\cv(X,X[n])
\\
\dTTo<{\sigma_Y^n} &= &\dTTo>{\und\cv(X,f[n])}
\\
\und\cv(Y,Y[n]) &\rTTo^{\und\cv(f,Y[n])} &\und\cv(X,Y[n])
\label{dia-sigma-dinatural}
\end{diagram}
\end{proposition}

\begin{proof}
In fact,
\begin{align*}
&\bigl\langle X\tens\1[n] \rTTo^{1\tens\sigma^n} X\tens\und\cv(X,X[n]) \rTTo^{1\tens\und\cv(1,f[n])} X\tens\und\cv(X,Y[n]) \rto\ev Y[n] \bigr\rangle
\\
&=\bigl\langle X\tens\1[n] \rTTo^{1\tens\sigma^n} X\tens\und\cv(X,X[n]) \rto\ev X[n] \rTTo^{f[n]} Y[n] \bigr\rangle =f[n] 
\\
&=\bigl\langle X\tens\1[n] \rTTo^{f\tens1} Y\tens\1[n] \bigr\rangle =\bigl\langle X\tens\1[n] \rTTo^{f\tens\sigma^n} Y\tens\und\cv(Y,Y[n]) \rto\ev Y[n] \bigr\rangle
\\
&=\bigl\langle X\tens\1[n] \rTTo^{1\tens\sigma^n} X\tens\und\cv(Y,Y[n]) \rTTo^{1\tens\und\cv(f,1)} X\tens\und\cv(X,Y[n]) \rto\ev Y[n] \bigr\rangle,
\end{align*}
which implies the desired statement.
\end{proof}

Equip the category $\Ab^\ZZ$ of graded abelian groups $A=(A^n)_{n\in\ZZ}$ with the usual monoidal product \((A\tens B)^n=\coprod_{l+m=n}A^l\tens_\ZZ B^m\) and the signed symmetry \(c:a\tens b\mapsto(-1)^{ab}b\tens a\).
Consider the functor
\begin{equation}
-^\bull: \cv \to \Ab^\ZZ, \qquad X \mapsto X^\bull =(X^n)_{n\in\ZZ}, \qquad X^n =\cv(\1[-n],X).
\label{eq-V-AbZ}
\end{equation}
Equipped with the natural transformation
\[ \tens: X^n\tens_\ZZ Y^m =\cv(\1[-n],X)\tens_\ZZ\cv(\1[-m],Y) \to \cv(\1[-n-m],X\tens Y) =(X\tens Y)^{n+m},
\]
it becomes lax symmetric monoidal.
The graded abelian group $\kk=\1^\bull$, \(\kk^n=\1^\bull=\cv(\1[-n],\1)\) is actually a graded ring, graded commutative in the sense that \(ba=(-1)^{ab}ab\) for all \(a,b\in\kk^\bull\).
Moreover, the graded abelian groups coming from $\cv$ are actually commutative $\kk$\n-bimodules (whose category is denoted $\gr=\gr_\kk$), so lax monoidal functor~\eqref{eq-V-AbZ} has the latter category as the target: \(-^\bull:\cv\to\gr\).
We use the term `$\kk$\n-modules' instead of longer term `commutative $\kk$\n-bimodules'.

Being closed $\cv$ is enriched into itself (or rather $\und\cv$ is enriched into $\cv$).
The lax monoidal functor $-^\bull$ makes $\cv$ enriched in $\gr$.
Denote this $\gr$\n-enriched category \(\overline\cv\).
Thus for each pair $X,Y$ of objects of $\cv$ there is a graded $\kk$\n-module \(\overline\cv(X,Y)=(\und\cv(X,Y)^n)_{n\in\ZZ}\).
Elements of \(\und\cv(X,Y)^n=\cv(\1[-n],\und\cv(X,Y))\) are called morphisms $X\to Y$ of degree $n$.
For instance, \(\und\cv(X,Y)^0=\cv(X,Y)\) and \(\sigma_X^n\in\und\cv(X,X[n])^{-n}\) are morphisms of degree $-n$.
The $\gr$\n-category structure of $\overline\cv$ includes, in particular, the composition of morphisms of certain degrees.
\textit{E.g.} equation~\eqref{dia-sigma-dinatural} for \(f:X\to Y\in\cv\) can be presented as
\[ \bigl(X \rTTo^{\sigma_X^n} X[n] \rTTo^{f[n]} Y[n] \bigr) =\bigl(X \rTTo^f Y \rTTo^{\sigma_Y^n} Y[n] \bigr) \in \overline\cv,
\]
which can be converted to
\[ f[n] =\bigl( X[n] \rTTo^{\sigma_X^{-n}} X \rTTo^f Y \rTTo^{\sigma_Y^n} Y[n] \bigr) \in \overline\cv.
\]
Therefore, there are natural isomorphisms (of degree 0)
\[ \sigma^{-n}(1^{\tens(a-1)}\tens\sigma^n\tens1^{\tens(n-a)}): (\tens_{i=1}^kX_i)[n] \to X_1\tdt X_{a-1}\tens X_a[n]\tens X_{a+1}\tdt X_k
\]
which we use explicitly or implicitly.
Given \(X,Y,Z\in\Ob\cv\) and \(k\in\ZZ\) we have the following diagram.
More precisely, for each \(p:X\tens Y\to Z\in\cv\) there is a unique \(p^!:Y\to\und\cv(X,Z)\in\cv\) which makes the small triangle commutative
\begin{diagram}
X\tens\und\cv(X,Z)[k] &\rTTo^{(1\tens\sigma^{-k})\sigma^k} &(X\tens\und\cv(X,Z))[k] &\rTTo^{\ev[k]} &Z[k]
\\
&= &\uTTo<{(1\tens p^!)[k]} &\ruTTo^{\ttt=\quad}_{p[k]} &
\\
\uTTo<{1\tens p^![k]} &&(X\tens Y)[k] &&\uTTo>q
\\
&\ruTTo^{(1\tens\sigma^{-k})\sigma^k} &&= &
\\
X\tens Y[k] &&\rEq &&X\tens W
\end{diagram}
where \(W=Y[k]\).
For an arbitrary \(q:X\tens W\to Z[k]\in\cv\) there is a unique \(p:X\tens Y\to Z\in\cv\) which makes the big lower triangle commute.
Thus there is a unique \(q^!=p^![k]:W\to\und\cv(X,Z)[k]\in\cv\) which makes the exterior commutative.
Therefore, $\und\cv(X,Z)[k]$ is isomorphic to $\und\cv(X,Z[k])$ and the isomorphism $i$ satisfies
\begin{diagram}
X\tens\und\cv(X,Z)[k] &\rTTo^{(1\tens\sigma^{-k})\sigma^k} &(X\tens\und\cv(X,Z))[k]
\\
\dTTo<{1\tens i}>\wr &= &\dTTo>{\ev[k]}
\\
X\tens\und\cv(X,Z[k]) &\rTTo^\ev &Z[k]
\end{diagram}
so we have natural bijections
\begin{multline}
\und\cv(X,Z)^k =\cv(\1[-k],\und\cv(X,Z)) \simeq \cv(\1,\und\cv(X,Z)[k])
\\
\rTTo^{\cv(\1,i)} \cv(\1,\und\cv(X,Z[k])) \simeq \cv(X,Z[k]).
\label{eq-V(XZ)k-V(XZk)}
\end{multline}
Let us consider the case when \(-^\bull:\cv\to\gr\) is an equivalence.

\begin{example}
Let $\cv=\gr=\gr\text-\kk\modul$.
For any graded \(\kk\)-module $M$ and an integer $a$ denote by $M[a]$ the same module with the cohomological grading shifted by $a$: \(M[a]^k=M^{a+k}\).
The unit object of $\gr$ is $\1=\kk$ and its shifts are \(\1[a]=\kk[a]\).
We identify $M[a]$ with \(M\tens\kk[a]\) via \(x\mapsto x\tens1\) where \(1\in\kk[a]^{-a}\).
Then \(\sigma^a:M\to M[a]\) is the ``identity map'' \(M^k\ni x\mapsto x\in M[a]^{k-a}\) of degree \(\deg\sigma^a=-a\).
Write elements of $M[a]$ as \(m\sigma^a\).
When \(f:V\to X\) is a homogeneous map of certain degree, the map \(f[a]:V[a]\to X[a]\) is defined as \(f[a]=(-1)^{a\deg f}\sigma^{-a}f\sigma^a=(-1)^{af}\sigma^{-a}f\sigma^a\).
In particular, the differential \(d:M\to M[-1]\) in a $\dg$\n-module $M$ induces the differential \(d[a]:M[a]\to M[a-1]\) in $M[a]$.
The degree 0 isomorphisms \(\sigma^{-a}\cdot(\sigma^a\tens1):(V\tens W)[a]\to(V[a])\tens W\),
\((v\tens w)\sigma^a\mapsto(-1)^{wa}v\sigma^a\tens w\), and \(\sigma^{-a}\cdot(1\tens\sigma^a):(V\tens W)[a]\to V\tens(W[a])\), \((v\tens w)\sigma^a\mapsto v\tens w\sigma^a\), are graded natural.
This means that for arbitrary homogeneous maps \(f:V\to X\), \(g:W\to Y\) the following squares commute:
\begin{diagram}[w=5em]
(V[a])\tens W &\lTTo^{\sigma^{-a}\cdot(\sigma^a\tens1)}_\sim &(V\tens W)[a]
&\rTTo^{\sigma^{-a}\cdot(1\tens\sigma^a)}_\sim &V\tens(W[a])
\\
\dTTo<{(f[a])\tens g} &&\dTTo<{(f\tens g)[a]} &&\dTTo>{f\tens(g[a])}
\\
(X[a])\tens Y &\lTTo^{\sigma^{-a}\cdot(\sigma^a\tens1)}_\sim &(X\tens Y)[a]
&\rTTo^{\sigma^{-a}\cdot(1\tens\sigma^a)}_\sim &X\tens(Y[a])
\end{diagram}
Actually, the second isomorphism is ``more natural'' than the first one, not only because it does not have a sign, but also because it suits better the right operator system of notations, accepted in this paper.
We often identify \((V\tens W)[a]\) with \(V\tens(W[a])\) via \(\sigma^{-a}\cdot(1\tens\sigma^a)\).
\end{example}

We use some shorthand: $XT^>$ means $XT^{>0}=\oplus_{n>0}X^{\tens n}$ and $XT^\ge$ means $XT^{\ge0}=\oplus_{n\ge0}X^{\tens n}$.
The reason to use right handed notation is illustrated here: this is what we write
\begin{equation*}
V[1]T^>[2]T^>[-2]T^>[2]T^>[-3].
\end{equation*}
And this is what we mean
\[ (T^>((T^>((T^>((T^>(V[1]))[2]))[-2]))[2]))[-3].
\]
Very often (but not always) we shall write an object first and a functor operating on it second, an argument first and a function second, \textit{etc}.

Let $\bot:\cc\to\cc$ be a comonad.
The category of $\bot$\n-coalgebras is denoted $\cc_\bot$.
The following lemma is well known.

\begin{lemma}\label{lem-forgetful-functor-has-right-adjoint-T}
The forgetful functor $\cc_\bot\to\cc$ has the right adjoint $\bot:\cc\to\cc_\bot$: for any $V\in\Ob\cc$ and any
$\bot$\n-coalgebra \((X,\delta)\) there are mutually inverse bijections
\begin{align*}
\cc(X,V)\; &\longleftrightarrow \;\cc_\bot(X,V\bot),
\\
f\; &\rMapsTo \;\hat f=\delta\cdot(f\bot),
\\
\check g = g\cdot\varepsilon\; &\lMapsTo \;g.
\end{align*}
\end{lemma}

This is generalized to multicategories in \cite[Lemma~5.3]{BesLyuMan-book}.

Besides the ground category $\cv$ consider the strong symmetric Monoidal category $\cw=\NN\text-\gr\text-\cv$.
Objects of\index{TTsymb}{W@$\cw$} $\cw$ are $\NN$-graded objects \((X_k)_{k\in\NN}\) of $\cv$.
The tensor product of a family \((X^i)_{i\in I}\) of objects of $\cw$ is \((\tens_\cw^{i\in I}X^i)_k=\oplus_{\sum_{i\in I}k_i=k}\tens_\cv^{i\in I}X^i_{k_i}\).
The structure isomorphisms $\lambda_\cw$ (the associativity and the symmetry) are induced by $\lambda_\cv$.
The monoidal category $\cw$ is closed with the inner hom \(\und\cw(X,Y)\) given by \(\und\cw(X,Y)_c=\prod_{a\in\NN}\und\cv(X_a,Y_{a+c})\), $c\in\NN$.
The shifts in $\cw$ ignore the lower degree, coming from the functor \(\cz\rTTo^{\1[-]} \cv\to\cw\), where the image of latter functor is concentrated in lower degree~0.
For instance, when $\cv=\gr$ is the category of $\ZZ$\n-graded $\kk$\n-modules, the category $\cw$ consists of $\NN$-$\ZZ$-graded $\kk$\n-modules \(X=(X_n^k)_{n\in\NN}^{k\in\ZZ}\).

\chapter{Ordinary (co)operads}
We describe our conventions for trees.
This kind of trees is used to describe the monad algebras over which are operads.
Trees determine also certain colax Monoidal structure on the category of collections in $\cv$.
Coalgebras in this colax Monoidal category are precisely cooperads in $\cv$.

\section{Trees}
A \emph{rooted tree} $t$ can\index{TTindex}{rooted tree} be defined as a\index{TTindex}{parent map} \emph{parent map} \(\Pa_t:V(t)\to V(t)\), where $V(t)$ is a finite set (of vertices), such that \(|\im(\Pa_t^k)|=1\) for some $k\in\NN$.
The only element \(r\in\im(\Pa_t^k)\) is called the root.
An oriented graph without loops $G$ is constructed out of $\Pa$, whose set of vertices is $V(t)$ and arrows are \(v\to\Pa_t(v)\) if vertex $v$ is not the root.
Since $G$ is a connected graph, whose number of edges is one less than the number of vertices, it is a tree.
Thus the rooted tree is oriented towards the root. 
The only oriented path connecting a vertex $v$ with the root consists of $v$, $\Pa(v)$, $\Pa^2(v)$, \dots, the root.
This gives a partial ordering on the set of all vertices $V(t)$, namely, $u\preccurlyeq v$ iff $v$ lies on the oriented path connecting $u$ with the root. 
For each vertex \(p\in V(t)\) the non-negative number \(|p|=|\Pa_t^{-1}(p)\setminus\{\textup{root}\}|\) of its children is called its arity.

An \emph{ordered (planar) rooted tree} is\index{TTindex}{ordered rooted tree} a rooted tree with a chosen total ordering $\le$ of the set of incoming edges for each vertex.
Up to an isotopy there is only one embedding of an ordered rooted tree into the oriented plane with coordinates $x$, $y$, which agrees with the orientation and such that each descendant vertex $u$ is higher than its ancestor vertex $v$ ($u$ has bigger coordinate $y$). 
Thus, smaller incoming edges are to the left of the bigger ones.

The set of\index{TTindex}{staged tree} \emph{staged trees} $\str(m)$ consists\index{TTsymb}{str@$\str$} of sequences in $\co_\sk$
\begin{gather}
t =\bigl( t(0) \rTTo^{t_1} t(1) \rTTo^{t_2} t(2) \rTTo^{t_3} \dots \; t(m-1) \rTTo^{t_m} t(m) =\mb1 \bigr)
\label{eq-staged-tree-t(0)-t(1)-t(2)}
\\
\str(m) =\{ t\in \Cat([m],\co_\sk) \mid t(m) =\mb1 \}. \notag
\end{gather}
The number \(\height(t)=m\ge0\) is called the height of $t$.
In this presentation \(V(t)=\sqcup_{j=0}^mt(j)\), the restriction of the parent map is
\(\Pa_t|_{t(j)}=t_{j+1}:t(j)\to t(j+1)\) for \(0\le j\le m-1\), the root \(1\in t(m)\) is taken by $\Pa_t$ to itself.
The ordering of \(t_j^{-1}(i)\) is induced by that of \(t(j-1)\).
The set of all vertices \((V(t),\le)=t(0)\sqcup_<t(1)\sqcup_<\dots\sqcup_<t(m-1)\sqcup_<t(m)\) is lexicographically totally ordered.
Thus the root is the biggest element.
The set of isomorphism classes of ordered rooted trees is in bijection with the subset of staged trees consisting of sequences \eqref{eq-staged-tree-t(0)-t(1)-t(2)} such that \(t(0)\) is not empty.

Notice the following properties of the parent map \(\Pa_t:(V(t),\le)\to(V(t),\le)\):
\begin{myitemize}
\item[(a)] $\Pa_t$ is a non-decreasing map;

\item[(b)] \(x\le \Pa_tx\) for all \(x\in V(t)\);
\end{myitemize}
As a consequence of (b) $u\preccurlyeq v$ implies $u\le v$ for all \(u,v\in V(t)\).

On the other hand, a rooted tree \(\Pa_t:V(t)\to V(t)\) with totally ordered set $V(t)$ where the parent map satisfies conditions (a)--(b) admits a unique presentation as above staged tree.
Thus, the set of isomorphism classes of such \(\Pa_t:(V(t),\le)\to(V(t),\le)\) is in bijection with the set of staged trees (and with the set of isomorphism classes of ordered rooted trees).
One more total ordering $\trianglelefteqslant$ of the set $V(t)$ of vertices of an ordered rooted tree $t$ is defined as follows.
For any two vertices $u,v\in V(t)$ either they are comparable with respect to $\preccurlyeq$ and we set \(u\trianglelefteqslant v\) iff \(u\preccurlyeq v\), or they are not $\preccurlyeq$\n-comparable.
Then there are unique $n,m\in\NN_{>0}$ such that \(P^n(u)=P^m(v)\) and both \(P^{n-1}(u)\ne P^{m-1}(v)\) are distinct from the root.
We set \(u\triangleleft v\) iff \(P^{n-1}(u)<P^{m-1}(v)\) in the set of descendants (incoming edges) of $P^n(u)$.

A \emph{rooted tree with inputs} is\index{TTindex}{rooted tree with inputs} a rooted tree $t$ with a chosen subset $\Inp(t)$ of the set $L(t)$ of leaves, vertices without incoming edges. 
For instance, a 1\n-vertex tree has one leaf -- the root.
It gives rise to two rooted trees with inputs: \(\tau[0]=\bullet=(\mb1;\Inp\bullet=\varnothing)\) and
\(\theta_0=\circ=(\mb1;\Inp\circ=\mb1)\).
The \emph{set of internal vertices} is\index{TTindex}{set of internal vertices of a tree with inputs} defined as \(\IV(t)=V(t)-\Inp(t)\).
The subsets \(\IV(t)\) and $\Inp(t)$ of \(V(t)\) get their induced canonical total ordering.
For example, the picture
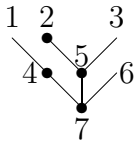
\[ \hstretch140
\vstretch70
\begin{tanglec}
\object1\step\object2\Step\object3\step \\
\node\Step \\
\nw1\Put(-5,0)[0,0]{\mbox{\normalsize4}}\node\nw1\node\Put(0,10)[0,0]{\mbox{\normalsize5}}\ne1
\Put(3,0)[0,0]{\mbox{\normalsize6}}\step \\
\nw1\n\ne1 \\
\object7
\end{tanglec}
\]
describes an ordered rooted tree with inputs $t$ with the root 7, nullary vertex (leaf) 2, unary vertex~4, the set of leaves $L(t)=\{1,2,3,6\}$, the subset of inputs $\Inp(t)=\{1,3,6\}$, non-input leaf~2 and the set of internal vertices \(\IV(t)=\{2,4,5,7\}\).
Combination of the two structures described above gives an\index{TTindex}{ordered rooted tree with inputs} \emph{ordered rooted tree with inputs}.
These coincide with the trees used by Leinster \cite[Example~2.3.3, Definition~E.1.3]{math.CT/0305049}.
The definition used by Muro \cite[Definition~3.4]{MR2821434} of planted 
planar tree with leaves coincides with ours up to different order $\trianglelefteqslant$ of vertices.
The origin of such version of trees comes from works of Kontsevich, \textit{e.g.} \cite[Section~6.6]{hep-th/9402147}.
The notion is brought to perfection by Borisov and Manin \cite[Section~1.1]{math/0609748}.
In the spirit of their approach we define a \emph{graph $G$ drawn on the plane} as a finite set $F$ of flags (half-edges) equipped with an involution $j:F\to F$ and a bijection $r:F\to F$, the clock-wise rotation.
Internal vertices of $G$ are orbits of $r$, edges are orbits of $j$: internal edges \(\{f,j(f)\}\) for $f\ne j(f)$ and tails for $f=j(f)$.
A connected graph drawn on the plane is a \emph{planar tree} if\index{TTindex}{planar tree} the number of internal vertices exceeds by 1 the number of internal edges.
A \emph{rooted planar tree} is\index{TTindex}{rooted planar tree} a planar tree with a distinguished tail, the root tail drawn downwards.
There is a bijection between the set of rooted planar trees $G$ defined in terms of flags and the set of ordered rooted trees with inputs $t$ defined above.
Notice that for $G$ and $t$ corresponding to each other, the sets of internal vertices $\IV(G)$ and $\IV(t)$ coincide except for
\(G=\thicklines\hstretch140\vstretch45\,\begin{tanglec}\n\\ \id\end{tanglec}\,\)
with \(|\IV(G)|=1\) and $t=\circ$ with \(\IV(t)=\emptyset\).
The sets of edges are related by \(E(G)=E(t)\sqcup\{\text{the root tail}\}\) with the same exceptional pair
\(\thicklines\hstretch140\vstretch45\,\begin{tanglec}\n\\ \id\end{tanglec}\,\)
and $\circ$.
All vertices of $t$ are defined as
\[ V(t) =\IV(t) \sqcup \{\text{tails } f=j(f)\in F\} - \{\text{the root tail}\}.
\]
When \(v\in\IV(G)\) has the outgoing flag $f$ and arity $n$, the set of incoming flags obtains the total ordering \(r(f)<r^2(f)<\dots<r^n(f)\).
Of course, \(r^{n+1}(f)=f\).

For instance, \(G=\thicklines\hstretch140\vstretch45\,\begin{tanglec}\s\end{tanglec}\,\) and $t=\bullet$ correspond to each other.

An ordered rooted tree with inputs can be conveniently presented as a uniquely determined \emph{staged tree with inputs}, which\index{TTindex}{staged tree with inputs} is a pair \((t;\Inp t)\), where $t$ is a staged tree and
\[ \Inp t \subset t(0)\cup \{ (j,i) \mid 1\le j\le m,\, i\in t(j),\, t_j^{-1}(i)=\varnothing \}
\]
is a subset of nullary vertices (leaves).
For any ordered rooted tree with inputs $t\ne\circ$ we have \(\Inp t=P_t^{-1}\IV(t)\setminus\IV(t)\), for \(P_t^{-1}\IV(t)=V(t)\).

A \emph{collection} is\index{TTindex}{collection} a function $f$ on the class of finite sets such that \(f(I)=f(J)\) if $I\simeq J$.
It is identified with the sequence \(f=(f(n))_{n\ge0}\).

Let $\tr$ be\index{TTsymb}{tr@$\tr$} the set of isomorphism classes of ordered rooted trees with inputs.
It splits up, \(\tr=\sqcup_{n\ge0}\tr(n)\), in components \(\tr(n)=\{t\in\tr\mid|\Inp t|=n\}\).
We view \(\tr=(\tr(n))_{n\ge0}\) as a collection of sets, thus, for any finite set $S$ the notation $\tr(S)$ means $\tr(|S|)$.
Also for any internal vertex $p$ the notation $\tr|p|$ means $\tr(|p|)$.
For each tree $t\in\tr$ there is an operation of substituting trees into internal vertices:
\begin{equation}
I_t: \prod_{p\in\IV(t)} \tr|p| \to \tr(\Inp t)
\label{eq-It-substituting-trees}
\end{equation}
which takes a family \((t_p)_{p\in\IV(t)}\) with \(|\Inp t_p|=|p|\) to the tree \(I_t(t_p\mid p\in\IV(t))\) obtained from $t$ by replacing each internal vertex \(p\in\IV(t)\) with the tree $t_p$.
Namely, if $t=\circ$, then \(I_\circ()=\circ\).
If $t\ne\circ$, then \(\tau=I_t(t_p\mid p\in\IV(t))\) has
\begin{equation}
V(\tau) =\bigsqcup_{p\in\IV(t)} V(t_p)/\sim,
\label{eq-V(tau)-V(t)}
\end{equation}
where the equivalence relation is generated by the following relation.
Let \(v\in\IV(t)\), \(v\ne\troot(t)\) and denote \(u=P_tv\ne v\).
Then \(\troot(t_v)\sim\phi_u(v)\), where $\phi_u$ is the order preserving bijection
\begin{equation}
\phi_u: (P_t^{-1}(u)\setminus\{\troot(t)\},\le) \to (\Inp t_u,\trianglelefteqslant).
\label{eq-phi-(P)-(Inp)}
\end{equation}
Equivalence classes of vertices are either single elements $(p,x)$, $p\in\IV(t)$, \(x\in\IV(t_p)\setminus\{\troot(t_p)\}\), or they are identified with maximal strings of consecutive vertices \(v\ne Pv\ne P^2v\ne\dots\ne P^mv\) such that \(t_{Pv}=t_{P^2v}=\dots=t_{P^mv}=\circ\), $m\ge0$, and, moreover, \(t_v\ne\circ\) or \(v\notin\IV(t)\).
Maximal means that either \(P^{m+1}v=P^mv\) is the root of $t$, or \(t_{P^{m+1}v}\ne\circ\).
If \(v\in\IV(t)\), then \(\troot(t_v)\) is equivalent to the only vertex of $t_{P^kv}$ for \(1\le k\le m\) and equivalent to \(\phi_{P^{m+1}v}(P^mv)\) if \(P^mv\) is not the root of $t$.
Moreover, no other vertex is equivalent to these.
When \(v\notin\IV(t)\), the vertices of $t_{P^kv}$ are equivalent for \(1\le k\le m\) and equivalent to \(\phi_{P^{m+1}v}(P^mv)\) if \(P^mv\) is not the root of $t$.

Whenever \(a,b\in V\bigl(I_t(t_p\mid p\in\IV(t))\bigr)\) are two such equivalence classes, there is no more than one triple \((p,x,y)\) such that \(p\in\IV(t)\), $(p,x)\in a$, $(p,y)\in b$.
We set \(P_\tau a=b\) iff \(P_{t_p}(x)=y\).
This defines \(\tau=I_t(t_p\mid p\in\IV(t))\) as an ordered rooted tree.
Its root is the equivalence class of \(\troot(t_{\troot(t)})\) if $t\ne\circ$ (that is, if \(\troot(t)\in\IV(t)\)).
Notice that the projection map
\begin{equation}
\pi: \bigsqcup_{p\in(\IV(t),\trianglelefteqslant)} (V(t_p),\trianglelefteqslant) \rEpi (V(\tau),\trianglelefteqslant)
\label{eq-Vtp<-Vtau<}
\end{equation}
is non-decreasing, where the first set is lexicographically totally ordered ($p\triangleleft q$ for \(p,q\in\IV(t)\) implies \((p,x)\triangleleft(q,y)\) for all \(x\in\IV(t_p)\), \(y\in\IV(t_q)\)).

The set $\Inp\tau$ of input vertices of $\tau$ by definition consists of single elements $(u,x)$, \(u\in\IV(t)\), \(x\in\Inp t_u\setminus \phi_u\bigl(\IV(t)\cap(P_t^{-1}u\setminus\troot(t))\bigr)\).
Clearly, \(\Inp\tau\subset L(\tau)\) and \(\Inp\tau=V(\tau)-\IV(\tau)\), where 
\begin{equation}
\IV(\tau) =\IV\bigl(I_t(t_p\mid p\in\IV(t))\bigr) \simeq \bigsqcup_{p\in\IV(t)} \IV(t_p),
\label{eq-v(It)-Uv(tp)}
\end{equation}
This formula relies on $\IV(\circ)=\emptyset$.
Moreover, for \(t\ne\circ\)
\begin{align}
\Inp\tau &=P_\tau^{-1}\IV(\tau)-\IV(\tau) =\bigl[\bigsqcup_{p\in\IV(t)} P_{t_p}^{-1}\IV(t_p)\bigr] -\bigl[\bigsqcup_{p\in\IV(t)}\IV(t_p)\bigr] \notag
\\
&=\bigl[\bigsqcup_{p\in\IV(t)}\Inp t_p\bigr]\setminus \bigl[\bigsqcup_{p\in\IV(t)}\IV(t_p)\bigr] \rto\sim \bigsqcup_{p\in\IV(t)}P_t^{-1}p -\IV(t) =P_t^{-1}\IV(t)-\IV(t) =\Inp t,
\label{eq-Inp(It)-Inp-t}
\end{align}
where $[S]$ means the set of equivalence classes containing elements of $S$.
For $t=\circ$ the equation \(\Inp\tau=\Inp t\) holds, since $\tau=\circ$ as well.
Since \eqref{eq-Vtp<-Vtau<} is a non-decreasing map, so is \((\Inp\tau,\trianglelefteqslant)\rto\sim (\Inp t,\trianglelefteqslant)\) given by \eqref{eq-Inp(It)-Inp-t}.

\begin{example}
Consider trees
\[ \hstretch140
\vstretch70
t=
\begin{tanglec}
\nw1\Put(-5,0)[0,0]{\mbox{\normalsize1}}\node\ne1 \\
\Put(-5,0)[0,0]{\mbox{\normalsize2}}\n \\
\step\Put(-5,0)[0,0]{\mbox{\normalsize3}}\n\step
\node\Put(5,0)[0,0]{\mbox{\normalsize4}} \\
\nw1\n\ne1 \\
\object5
\end{tanglec}
\quad, \quad t_1=
\begin{tanglec}
\nw1\node\ne1
\end{tanglec}
\;, \quad t_2 =t_3 =\circ, \quad t_4=
\begin{tanglec}
\node\Step\node \\
\nw1\node\ne1
\end{tanglec}
\;, \quad t_5 =
\begin{tanglec}
\nw1\node\ne1\step\ne1 \\
\nw1\node\ne1
\end{tanglec}
\;.
\]
Then
\[ \hstretch140
\vstretch70
\tau =I_t(t_1,t_2,t_3,t_4,t_5)=
\begin{tangles}{*4r}
&&\node &\node \\
&\nw1\n &\n &\ne1 \\
\nw1\node &\ne1 &\ne1 \\
&\nw1\node &\ne1\step
\end{tangles}
\;.
\]
\end{example}

There is a monad \(\TT:\Set^\NN\to\Set^\NN\), \(X=(X(n))_{n\ge0}\mapsto X\TT\),
\[ (X\TT)(n) =\bigsqcup_{t\in\tr(n)} \, \bigsqcap_{p\in\IV(t)} X|p|,
\]
on the category of collections of sets.
For any $p\in\IV(t)$ the notation $X|p|$ means $X(|p|)$.

The unit \(i:\Id\to\TT\) is given by the obvious isomorphism \(X(n)\to\prod_{p\in\IV(\tau[n])}X|p|\) into the summand indexed by the\index{TTindex}{n-corolla@$n$-corolla} \emph{$n$\n-corolla}, \(\tau[n]\overset{\text{def}}=(\bn\to\mb1;\Inp\tau[n]=\bn)\) for $n>0$ and \(\tau[0]=(\mb1;\varnothing)\) for $n=0$.
Notice that \(\IV(\tau[n])=\mb1\) for all $n\ge0$.
The multiplication \(m:\TT^2\to\TT\) takes a summand of
\begin{equation*}
(X\TT^2)(n) =\bigsqcup_{t\in\tr(n)} \, \bigsqcap_{p\in\IV(t)} \, \bigsqcup_{t_p\in\tr|p|} \, \bigsqcap_{q\in\IV(t_p)} X|q| \simeq
\bigsqcup_{t\in\tr(n)} \, \bigsqcup_{(t_p\in\tr|p|)_{p\in\IV(t)}} \, \bigsqcap_{(p,q)\in\sqcup_{r\in\IV(t)}\IV(t_r)} X|q|
\end{equation*}
identically to the summand indexed by the tree \(I_t(t_p\mid p\in\IV(t))\) of $X\TT$.
This monad is described by Leinster \cite[Example~4.1.11]{math.CT/0305049} as the monad substituting trees into vertices.

\begin{definition}
An \emph{operad} in $\Set$ is\index{TTindex}{operad in sets} an algebra over the monad $\TT$.
\end{definition}

In detail, an operad is a collection of sets \(\cO=(\cO(n))_{n\ge0}\), a multiplication map
\[ \mu_t: \prod_{p\in\IV(t)} \cO|p| \to \cO(\Inp t)
\]
for each \(t\in\tr\) such that for any corolla $\tau[n]$, $n\ge0$, the multiplication \(\mu_{\tau[n]}:\cO(n)\to\cO(n)\) is the identity map and for each tree $t\in\tr$ and each family \((t_p)_{p\in\IV(t)}\) with \(|\Inp t_p|=|p|\) the following equation holds
\begin{diagram}[LaTeXeqno,h=2.6em]
\prod_{p\in\IV(t)} \, \prod_{q\in\IV(t_p)} \cO|q| &\rTTo^{\prod_{p\in\IV(t)}\mu_{t_p}} &\prod_{p\in\IV(t)} \cO|p|
\\
\dTTo<\wr &= &\dTTo>{\mu_t}
\\
\prod_{x\in\IV(I_t(t_p\mid p\in\IV(t)))} \cO|x| &\rTTo^{\mu_{I_t(t_p\mid p\in\IV(t))}} &\cO(\Inp t)
\label{dia-OOOO-equation-ass}
\end{diagram}
The monad $\TT$ is also called the free operad monad.

Notice that the tree \(\circ=(\mb1;\Inp\circ=\mb1)\) with \(\IV(\circ)=\varnothing\) gives in the image of \(\eta=\mu_{\circ}:\mb1\to\cO(1)\) a distinguished element $1_\co$ called the unit of the operad.

For instance, applying $\TT$ to the terminal collection \(\mb1=(\mb1)_{n\ge0}\) we get \(\tr=\mb1\TT\), which automatically becomes an operad with the operation $I_t$ as the multiplication.
In particular, \(I_\circ()=\circ\in\tr(1)\) is the unit of the operad $\tr$.

\section{Operads}
An \emph{operad} $\co$ in $\cv$ is\index{TTindex}{operad} a collection of objects \(\co(n)\in\Ob\cv\) and multiplication morphism for each tree \(t\in\tr\)
\[ \mu_t^\co: \otimes_\cv^{p\in\IV(t)} \co|p| \to \co(\Inp t),
\]
such that
\begin{diagram}[h=2.2em,LaTeXeqno]
\otimes_\cv^{p\in\IV(t)} \, \otimes_\cv^{q\in\IV(t_p)} \co|q| &\rTTo^{\otimes_\cv^{p\in\IV(t)}\mu^\co_{t_p}}
&\otimes_\cv^{p\in\IV(t)} \co|p|
\\
\dTTo<\wr &= &\dTTo>{\mu^\co_t}
\\
\otimes_\cv^{x\in\IV(I_t(t_p\mid p\in\IV(t)))} \co|x| 
&\rTTo^{\mu^\co_{I_t(t_p\mid p\in\IV(t))}} &\co(\Inp t)
\label{dia-ass-operad-O}
\end{diagram}
and for all $n\ge0$ the morphism \(\mu^\co_{\tau[n]}=\rho_{\co(n)}:\tens_\cv^{\mb1}\co(n)\to\co(n)\) is the structural isomorphism.

A cooperad $C$ in $\cv$ (that is in the strong $\Cat$-operad \(\mcC=(\cv)_{n\ge0}\)) is a collection of objects \(C(n)\in\Ob\cv\) and comultiplication morphism for each tree \(t\in\tr\)
\[ \Delta_t: C(\Inp t) \to \otimes^{p\in\IV(t)} C|p|,
\]
which satisfies the coassociativity equation
\begin{diagram}[h=2.6em,LaTeXeqno]
C(\Inp t) &\rTTo^{\Delta_{I_t(t_p\mid p\in\IV(t))}\;} &\bigotimes_{x\in\IV(I_t(t_p\mid p\in\IV(t)))} C|x|
\\
\dTTo<{\Delta_t} &= &\dTTo>\wr
\\
\bigotimes_{p\in\IV(t)} C|p| &\rTTo^{\tens^{p\in\IV(t)}\Delta_{t_p}}
&\bigotimes_{p\in\IV(t)} \, \bigotimes_{q\in\IV(t_p)} C|q|
\label{dia-CCCC-equation-coass}
\end{diagram}
and the normalization condition
\begin{equation}
\Delta_{\tau[n]}=\rho_{C(n)}^{-1}:C(n)\rTTo^\sim \tens_\cv^{\mb1}C(n) \quad \textup{for } n\ge0.
\label{eq-Delta-tau-n-normalization}
\end{equation}

An arbitrary multiplication in an operad \(\mu_t^\co\), $t\in\tr$, can be presented as composition of several such multiplications \(\mu_{t'}^\co\), where $t'$ is taken from a proper subset of $\tr$.
On this way we come to several equivalent definitions of an operad, \textit{e.g.} from \cite{MR1898414}.

First of all, we consider the set of staged trees $\str$ which we map to $\tr$ as follows.
Staged tree $t$ from \eqref{eq-staged-tree-t(0)-t(1)-t(2)} is declared to have \(\Inp t=t(0)\).
Note that \(\circ=\theta_0\in\str\), but \(\bullet=\tau[0]\notin\str\), however, \((\emptyset\to\mb1)\in\str\).
The set $\str$ is partitioned into subsets \(\str(n,m)\) of staged trees $t$ of height $m$ with \(t(0)=\mb n\).
Whenever \(t\in\str(n,m)\) and for each \(k\in t(j)\), $j>0$, a tree \(t_{j,k}\in\str(|(j,k)|,m_j)\), \(|(j,k)|=|t_j^{-1}(k)|\) is given, the substitution result \(\tau=I_t(t_{j,k}\mid k\in t(j),\,j>0)\) is again a staged tree of height \(\sum_{j=1}^mm_j\) with \(t(0)=\mb n\).
In this situation $\mu_\tau$ expresses via $\mu_t$ and \(\mu_{t_{j,k}}\) accordingly to equation~\eqref{dia-ass-operad-O}.

\begin{proposition}\label{pro-operad-via-str}
An operad structure on a collection \(\co\in\Ob\cv^\NN\) amounts to multiplications $\mu_t$, $t\in\str$, such that equation~\eqref{dia-ass-operad-O} holds for each family $t\in\str(n,m)$, \(t_{j,k}\in\str(|(j,k)|,m_j)\), \(k\in t(j)\), $m_j\ge0$, $j\in\mb m$, and the normalization conditions \(\mu_{\emptyset\to\mb1}=\rho_{\co(0)}:\tens_\cv^{\mb1}\co(0)\to\co(0)\), \(\mu_{\tau[n]}=\rho_{\co(n)}:\tens_\cv^{\mb1}\co(n)\to\co(n)\) hold for $n>0$.
\end{proposition}

\begin{proof}
First of all, for an arbitrary \(t\in\str(n,m)\) there are two more staged trees \(\overline{t}\), \(\und t\in\str(n,m+1)\) of the form
\[ \overline{t} =\bigl( \overline{t}(0) =t(0) \rTTo^{\overline{t}_1}_\id \overline{t}(1) =t(0) \rTTo^{\overline{t}_2}_{t_1} \overline{t}(2) =t(1) \rTTo^{\overline{t}_3}_{t_2} \dots \; \rTTo^{\overline{t}_{m+1}}_{t_m} \overline{t}(m+1) =\mb1 \bigr),
\]
where $n$ unary vertices are added near leaves, and
\[ \und t =\bigl( t(0) \rTTo^{t_1} t(1) \rTTo^{t_2} t(2) \rTTo^{t_3} \dots \; t(m-1) \rTTo^{t_m} t(m) =\mb1 \rTTo^\id \und t(m+1) =\mb1 \bigr),
\]
where one unary vertex is added at the root.
Consider families of staged trees \((\circ)_{p\in\overline{t}(1)}\sqcup(t_p^>)_{p\in\IV(t)}\) indexed by \(\IV(\overline{t})\) and \((t_p^>)_{p\in\IV(t)}\sqcup(\circ)_{p\in\und t(m+1)}\) indexed by \(\IV(\und t)\).
We have
\[ I_{\overline{t}}\bigl((\circ)_{p\in\overline{t}(1)}\sqcup(t_p^>)_{p\in\IV(t)}\bigr) =t =I_{\und t}\bigl((t_p^>)_{p\in\IV(t)}\sqcup(\circ)_{p\in\und t(m+1)}\bigr).
\]
Equation~\eqref{dia-ass-operad-O} says that
\begin{multline}
\bigl[ \tens^{p\in\IV(t)}\co|p| \simeq \1^{\tens n}\tens\bigl(\tens^{p\in\IV(t)}\co|p|\bigr) \rTTo^{\eta^{\tens n}\tens\id} \co(1)^{\tens n}\tens\bigl(\tens^{p\in\IV(t)}\co|p|\bigr)
\\
\hfill \simeq \tens^{p\in\IV(\overline t)}\co|p| \rTTo^{\mu_{\overline t}} \co(n) \bigr] \quad
\\
\hskip\multlinegap =\mu_t =\bigl[ \tens^{p\in\IV(t)}\co|p| \simeq \bigl(\tens^{p\in\IV(t)}\co|p|\bigr)\tens\1 \rTTo^{\id\tens\eta} \bigl(\tens^{p\in\IV(t)}\co|p|\bigr)\tens\co(1) \hfill
\\
\simeq \tens^{p\in\IV(\und t)}\co|p| \rTTo^{\mu_{\und t}} \co(n) \bigr].
\label{eq-mu-overline-t-mu-t-mu-und-t}
\end{multline}
Applying this equation several times we prove the following statement.
Let \(\tau\in\tr\), \(\Inp\tau\ne\emptyset\).
There is a staged tree \(t\in\str\) and a subset \(M\subset\uv(t)\) such that $\tau$ is obtained from $t$ by removing all vertices from $M$ together with the incoming edge.
Formally, \(\tau=t^M\overset{\text{def}}=I_t\bigl((t_p^>)_{p\in\IV(t)-M},(\circ)_{p\in M}\bigr)\).
The claim is that for fixed $\tau$ all compositions
\begin{multline}
\bigl[ \tens^{p\in\IV(\tau)}\co|p| \simeq \1^{\tens M}\tens\bigl(\tens^{p\in\IV(t)-M}\co|p|\bigr) \rTTo^{\eta^{\tens M}\tens\id} \co(1)^{\tens M}\tens\bigl(\tens^{p\in\IV(t)-M}\co|p|\bigr)
\\
\simeq \tens^{p\in\IV(t)}\co|p| \rTTo^{\mu_t} \co(\Inp t) =\co(\Inp\tau) \bigr]
\label{eq-mu-tau-mu-t-M}
\end{multline}
are equal due to described relations between \(\mu_r\) for $r\in\str$ (thus, compositions do not depend on the choice of $t$ and $M$).
In fact, composition~\eqref{eq-mu-tau-mu-t-M} can be transformed via \eqref{eq-mu-overline-t-mu-t-mu-und-t} to such that no element of $M$ lies on the oriented path from a vertex \(p\in\IV(t)-M\) to the root.
Relations for $\mu_{\tr}$ imply that $\mu_\tau$ must be equal to \eqref{eq-mu-tau-mu-t-M}, and we take this expression as a definition of $\mu_\tau$.

When \(\Inp r=\emptyset\), \(r\in\tr\), hence, \(\IV(r)=V(r)\), we present $r$ as a staged tree $r'$ in a unique way as described after \eqref{eq-staged-tree-t(0)-t(1)-t(2)}.
We assign the staged tree \((\emptyset\to r')\in\str\) to $r$.
Respectively, we request that \(\mu_r=\mu_{\emptyset\to r'}:\otimes_\cv^{p\in\IV(r)}\co|p|\to\co(0)\).
Notice that \(\mu_{\emptyset\to\emptyset\to\dots\to\emptyset\to r'}=\mu_{\emptyset\to r'}\) for an arbitrary number $l\ge1$ of $\emptyset$'s preceding $r'$.
This follows from equation~\eqref{dia-ass-operad-O} for \(t=(\emptyset\to\emptyset\to\dots\to\emptyset\to r')\in\str\), \(m_1=m_2=\dots=m_{l-1}=0\), \(m_l=m_{l+1}=\dots=1\), \(t_{j,k}=\tau|(j,k-l)|\) for $k\ge l$, \(j\in r(k-l)\).
In fact, \(I_t(t_p\mid p\in\IV(t))=(\emptyset\to r')\) in this case.

In order to prove relation~\eqref{dia-ass-operad-O} for trees with inputs $t$, $(t_p)$, \(I_t(t_p\mid p\in\IV(t))\), we do the following:
\begin{myitemize}
\item[---] Represent $t$ as some \(\tilde{t}^M\), \(\tilde{t}\in\str\).
The equation to prove is equivalent to the equation written for \(\tilde{t}\), \((\circ)_{p\in M}\sqcup(t_p)_{p\in\IV(t)}\), \(I_{\tilde t}\bigl((\circ)_{p\in M}\sqcup(t_p)_{p\in\IV(t)}\bigr)\).
Thus we may assume $t\in\str$ without loss of generality.

\item[---] For each \(1\le j\le\height(t)\) represent \(t_p\in\tr\) for \(p\in t(j)\) as \(\tilde{t}_p^{M_p}\), \(\tilde{t}_p\in\str\), in such a way that \(\height(\tilde{t}_p)=m_j\) does not depend on \(p\in t(j)\).
\end{myitemize}
Then \(I_t(t_p\mid p\in\IV(t))\in\tr\) is represented as \(I_t(\tilde{t}_p\mid p\in\IV(t))^M\) for \(M=\sqcup_{p\in\IV(t)}M_p\).
Note that \(I_t(\tilde{t}_p\mid p\in\IV(t))\in\str\).
Simultaneously the required equation for \(\mu_{I_t(t_p\mid p\in\IV(t))}\) reduces to valid equation for \(\mu_{I_t(\tilde{t}_p\mid p\in\IV(t))}\).
\end{proof}

\begin{corollary}\label{cor-staged-operations}
An operad is unambiguously specified by operations $\mu_t$ for all \(t\in\str(-,2)\) and for the only element \(t=\circ\in\str(1,0)\) subject to associativity equation and two unitality equations.
The associativity equation is written for each \(t\in\str(n,3)\) with the help of \(t'=\bigl(\mb n\rTTo^{t_1} t(1)\rTTo^\con \mb1,\mb n\bigr)\in\str(n,2)\), \(t'_p=t_p^>\) for \(p\in t(1)\), \(t'_1=\bigl(t(1)\rTTo^{t_2} t(2)\rto\con \mb1,t(1)\bigr)\in\str(t(1),2)\) for \(1\in\mb1=t'(2)\) and \(t''=\bigl(\mb n\rTTo^{t_1\cdot t_2} t(2)\rto\con \mb1,\mb n\bigr)\in\str(n,2)\), \(t''_p=\bigl(t_1^{-1}t_2^{-1}p\rTTo^{t_1|} t_2^{-1}p\rTTo^\con \{p\},t_1^{-1}t_2^{-1}p\bigr)\in\str(-,2)\) for \(p\in t(2)=t''(1)\), \(t''_1=\tau[t(2)]\) for \(1\in\mb1=t''(2)\), namely,
\begin{diagram}[h=2.2em]
\otimes_\cv^{p\in\IV(t')} \, \otimes_\cv^{q\in\IV(t'_p)} \co|q| &\rTTo^{\otimes_\cv^{p\in\IV(t')}\mu^\co_{t'_p}} &\otimes_\cv^{p\in\IV(t')} \co|p|
\\
\uTTo<\wr &&\dTTo>{\mu^\co_{t'}}
\\
\otimes_\cv^{v\in\IV(t)} \co|v| &\rDashTo^{\mu^\co_t} &\co(\Inp t)
\\
\dTTo<\wr &&\uTTo>{\mu^\co_{t''}}
\\
\otimes_\cv^{p\in\IV(t'')} \, \otimes_\cv^{q\in\IV(t''_p)} \co|q| &\rTTo^{\otimes_\cv^{p\in\IV(t'')}\mu^\co_{t''_p}} &\otimes_\cv^{p\in\IV(t'')} \co|p|
\end{diagram}
since
\[ I_{t'}(t'_p\mid p\in\IV(t')) =t =I_{t''}(t''_p\mid p\in\IV(t'')).
\]
The first unitality equation uses \(t=\bigl(\mb n\rTTo^\con \mb1\rTTo^\id \mb1,\mb n\bigr)\in\str(n,2)\), \(t_{1'}=\tau[n]\) for \(1'\in t(1)\), \(t_{1''}=\circ\) for \(1''\in t(2)\).
Then \(I_t(\tau[n],\circ)=\tau[n]\) and
\[ \bigl[ \co(n) \simeq \co(n)\tens\1 \rTTo^{1\tens\eta} \co(n)\tens\co(1) \rTTo^{\mu_t} \co(n) \bigr] =1.
\]
The second unitality equation uses \(t=\bigl(\mb n\rTTo^\id \mb n\rTTo^\con \mb1,\mb n\bigr)\in\str(n,2)\), \(t_p=\circ\) for \(p\in t(1)=\mb n\), \(t_1=\tau[n]\) for \(1\in t(2)\).
Then \(I_t((\circ)_{p\in\mb n},\tau[n])=\tau[n]\) and
\[ \bigl[ \co(n) \simeq \1^{\tens n}\tens\co(n) \rTTo^{\eta^{\tens n}\tens1} \co(1)^{\tens n}\tens\co(n) \rTTo^{\mu_t} \co(n) \bigr] =1.
\]
\end{corollary}

The proof is by induction on height of a staged tree similarly to \cite[Proposition~2.28]{BesLyuMan-book}.
Furthermore, if countable coproducts exist in $\cv$ and $\tens_\cv$ commutes with countable coproducts, the above corollary follows immediately from the mentioned proposition.
In fact, the category $\cv^\NN$ admits in this case a strong Monoidal structure \(\odot^m:(\cv^\NN)^m\to\cv^\NN\)
\[ \bigl(\bigodot_{j\in\mb m}X_j\bigr)(n)
=\coprod_{t\in\str(n,m)} \, \bigotimes_{j\in\mb m} \, \bigotimes_{k\in t(j)} X_j|(j,k)|, \qquad |(j,k)|=|t_j^{-1}(k)|.
\]
An operad $A$ in the sense of \propref{pro-operad-via-str} is exactly the same as an algebra $A$ in \((\cv^\NN,\odot^m)\) in the sense of \cite[Definition~2.25]{BesLyuMan-book}, which presumes morphisms \(\odot^mA\to A\) in $\cv^\NN$.
Proposition~2.28 of [\textit{ibid.}] says that such an algebra structure amounts to the binary multiplication \(\odot^2A\to A\) and the unit \(\odot^0A\to A\) which satisfy the associativity and two unitality equations.

Another approach is sketched by Ginzburg and Kapranov in the symmetric case \cite[1.2]{GinzbKapranov} and presented in detail by Muro for non-symmetric operads \cite[Proposition~3.10]{MR2821434}.

We extend the notion of a right $\co$\n-module to the following set-up.
Let \(\mm=(\mm,+,0)\) be a commutative monoid equipped with a morphism of monoids \(|\cdot|:\mm\to\NN\), \(m\mapsto|m|\), such that $|m|=0$ implies $m=0$ and for any $n\in\NN$ the set \(\{m\in\mm\mid|m|=n\}\) is finite.
In particular, $\mm$ is a M\"obius monoid (as defined \textit{e.g.} in \cite{MR2720184}).
In this article we shall be interested in \(\mm=0\) or $\mm=\NN$, occasionally in \(\mm=\NN^n\) for $n>1$ with \(|(a_1,\dots,a_n)|=\sum_{i=1}^na_i\).
The monoidal category \((\cv^\NN,\odot)\) acts on the category $\cv^\mm$ of $\mm$\n-graded objects ($\mm$ is viewed as a discrete category with the set of objects $\mm$).
The monoidal action \(\odot:\cv^\mm\times\cv^\NN\to\cv^\mm\), \(\cp\mapsto\cp\odot\cb\), was denoted $\odot_0$ in \cite{Lyu-A8-several-entries}.
For $\ell\in\mm$
\[ (\cp\odot\cb)(\ell) =\coprod_{n_1+\dots+n_k=\ell}^{k\ge0,\,n_r\in\mm} \Bigl(\bigotimes_{r=1}^k\cp(n_r)\Bigr)\tens\cb(k).
\]
When $\co$ is an operad in $\cv$, the functor \(-\odot\co:\cv^\mm\to\cv^\mm\) is a monad.
Algebras over this monad are called right $\co$\n-modules.
They are objects $M$ of $\cv^\mm$ equipped with a unital associative action \(\alpha:M\odot\co\to M\).
Algebras over an operad $\co$ are right $\co$\n-modules with $\mm=0$.
Right $\co$\n-modules in the monoidal category \((\cv^\NN,\odot)\) are right $\co$\n-modules in the above sense for $\mm=\NN$.

\section{Operad $A_\infty^\hu$}
Homotopy unital \ainfm-algebras\index{TTindex}{homotopy unital $A_\infty$-algebra} are algebras over the operad $\mainf^\hu$, obtained as follows.

1) The operad $\mainf^\su$ is generated over $\mainf$ by 
a nullary degree 0 cycle $\one^\su$ subject to the following relations:
\[ (1\tens\one^\su)m_2 =1, \quad (\one^\su\tens1)m_2 =1, \quad
(1^{\tens a}\tens\one^\su\tens1^{\tens c})m_{a+1+c} =0
\text{ \ if \ } a+c>1.
\]

2) There is a standard trivial cofibration and a homotopy isomorphism
	$\mainf^\su\rCof~\Sim \mainf^\su\langle\one^\su-\sfi,
	\sfj\rangle=\mainf^\su\langle \sfi,\sfj\rangle$,
where $\sfi$, $\sfj$ are two nullary operations, $\deg \sfi=0$,
$\deg \sfj=-1$, with \(\sfi d=0\), \(\sfj d=\one^\su-\sfi\). 

3) A cofibrant replacement \(\mainf^\hu\to\Ass\) of the operad of unital associative algebras is constructed as a 
$\dg$\n-suboperad of \(\mainf^\su\langle\sfi,\sfj\rangle\) generated as 
a graded operad by $\sfi$ and $n$\n-ary operations of degree $4-n-2k$
\[ m_{n_1;n_2;\dots;n_k} =(1^{\tens n_1}\tens \sfj\tens1^{\tens n_2} 
\tens\sfj\tdt1^{\tens n_{k-1}}\tens 
\sfj\tens1^{\tens n_k})m_{n+k-1},
\]
where \(n=\sum_{q=1}^kn_q\), $k\ge1$, $n_q\ge0$, \(n+k\ge3\).
Notice that the graded operad \(\mainf^\hu\) is free.
See \cite[Section~1.11]{Lyu-Ainf-Operad} for the proofs.

In order to construct its version -- the operad\index{TTsymb}{Ahu@$A_\infty^\hu$} $A_\infty^\hu$:

1) Add to $A_\infty$ a nullary degree $-1$ cycle $\bone^\su$ subject to the relations:
\begin{equation*}
(1\tens\bone^\su)b_2 =1, \quad (\bone^\su\tens1)b_2 =-1, \quad
(1^{\tens a}\tens\bone^\su\tens1^{\tens c})b_{a+1+c} =0 \text{ \ if \ } a+c>1.
\end{equation*}
The resulting operad is denoted $A_\infty^\su$.

2) Add to $A_\infty^\su$ two nullary operations $\bi$, $\bj$, $\deg\bi=-1$, $\deg\bj=-2$, with \(\bi d=0\), \(\bj d=\bi-\bone^\su\).
The standard trivial cofibration $A_\infty^\su\rCof~\Sim A_\infty^\su\langle\bi,\bj\rangle$ is a homotopy isomorphism.

3) $A_\infty^\hu$ is a $\dg$\n-suboperad of $A_\infty^\su\langle\bi,\bj\rangle$ generated as a graded operad by $\bi$ and $n$\n-ary operations of degree $3-2k$
\[ b_{n_1;n_2;\dots;n_k} 
=(1^{\tens n_1}\tens\bj\tens1^{\tens n_2}\tens\bj\tdt1^{\tens n_{k-1}}\tens\bj\tens1^{\tens n_k})b_{n+k-1},
\]
where \(n=\sum_{q=1}^kn_q\), $k\ge1$, $n_q\ge0$, \(n+k\ge3\).

Muro and Tonks presented this operad as an operad of cells for certain topological operad \cite{1110.1959}.
Muro has studied this kind of operads in symmetric model monoidal categories \cite{1111.2723}.

\section{Cooperads}
The picture for cooperads is very much the same as for operads since one passes from cooperads in $\cv$ to operads in $\cv^\op$ just by considering the opposite category.
Subtle differences are explained below.

The concrete form of coassociativity equation for cooperads is presented here.
A \emph{cooperad} $C$ in the category $\cv$ is\index{TTindex}{cooperad} a collection \((C(n))_{n\ge0}\), \(C(n)\in\Ob\cv\), together with coassociative comultiplications \(\Delta_{n^1,\dots,n^k}:C(n^1+\dots+n^k)\to C(n^1)\tdt C(n^k)\tens C(k)\)
and a two-sided counit \(\eps:C(1)\to\1\in\cv\).
Coassociativity says that for any $k\in\NN$, any \((l_1,\dots,l_k)\in\NN^k\), an arbitrary
\((i_1^j,\dots,i^j_{l_j})\in\NN^{l_j}\) for each \(1\le j\le k\) the following equation holds:
\begin{diagram}[nobalance,h=1.8em,LaTeXeqno]
C(n) &&\hphantom{C(i_1^1)\tdt C(i^1_{l_1})\tdt C(i_1^k)\tdt C(i^k_{l_k})\tens C(l_1+\dots+l_k)}
\\
\dTTo<{\Delta_{n^1,\dots,n^k}} &\rdTTo^{\Delta_{i_1^1,\dots,i^1_{l_1},\dots,i_1^k,\dots,i^k_{l_k}}}
\\
&&C(i_1^1)\tdt C(i^1_{l_1})\tdt C(i_1^k)\tdt C(i^k_{l_k})\tens C(l_1+\dots+l_k)
\\
C(n^1)\tdt C(n^k)\tens C(k)\hspace*{-3em} &&\dTTo>{1\tens\Delta_{l_1,\dots,l_k}}
\\
\dTTo~{\Delta_{i_1^1,\dots,i^1_{l_1}}\tdt\Delta_{i_1^k,\dots,i^k_{l_k}}\tens1} &=
\\
&&\hspace*{-3.5em}
C(i_1^1)\tens...\tens C(i^1_{l_1})\tens...\tens C(i_1^k)\tens...\tens C(i^k_{l_k})\tens C(l_1)\tens...\tens C(l_k)\tens C(k)
\\
&\ruTTo[hug]^\sim_{\text{unshuffle}}
\\
C(i_1^1)\tdt C(i^1_{l_1})\tens C(l_1)\tdt C(i_1^k)\tdt C(i^k_{l_k})\tens C(l_k)\tens C(k)\hspace*{-22em}
\label{dia-coassociativity-components}
\end{diagram}
where \(n^j=i_1^j+\dots+i^j_{l_j}\) for each \(1\le j\le k\) and \(n=n^1+\dots+n^k\).
Counitality means that equations
\begin{gather}
\bigl( C(n) \rTTo^{\Delta_n} C(n)\tens C(1) \rTTo^{1\tens\eps} C(n) \bigr) =1,
\label{eq-counitality-C(n)OC(1)}
\\
\bigl( C(n) \rTTo^{\Delta_{\sS{^n}1}} C(1)^{\tens n}\tens C(n) \rTTo^{\eps^{\tens n}\tens1} C(n) \bigr) =1
\label{eq-counitality-C(1)nOC(n)}
\end{gather}
hold for all $n\ge0$.

Let $V$ be a finite set, let $I_v$ be a set for each $v\in V$ and let $X_\alpha$ be a graded $\kk$\n-module for each $\alpha\in I_v$.
Fix an element \(\beta\in\prod_{v\in V}I_v\), \(\beta=(\beta(v))_{v\in V}\).
Tensoring the projections \(\pr_{\beta(v)}:\prod_{\alpha\in I_v}X_\alpha\to X_{\beta(v)}\) we get maps \(\tens_{v\in V}\pr_{\beta(v)}:\tens_{v\in V}\prod_{\alpha\in I_v}X_\alpha\to\tens_{v\in V}X_{\beta(v)}\).
They are combined into maps
\[ \zeta =\Bigl( \bigotimes_{v\in V}\pr_{\beta(v)} \Bigr)_{\beta\in\prod_{v\in V}I_v}: \bigotimes_{v\in V}\prod_{\alpha\in I_v}X_\alpha \to \prod_{\beta\in\prod_{v\in V}I_v} \bigotimes_{v\in V}X_{\beta(v)}.
\]
A particular case of this morphism is
\begin{equation}
\zeta: M\tens\prod_{\alpha\in A}X_\alpha \to \prod_{\alpha\in A}M\tens X_\alpha.
\label{eq-zMX-MX}
\end{equation}
On the other hand, one can show that general $\zeta$ is a composition of direct products of morphisms $\zeta$ of type \eqref{eq-zMX-MX}.

\begin{exercise}
If $\kk$ is concentrated in degree~0, \(\cv=\gr\text-\kk\modul\) and $M$ is a free $\kk$\n-module, then \eqref{eq-zMX-MX} is injective.
\end{exercise}

\begin{warning}
Consider $\kk=\ZZ$, $M=\QQ$, $A=\NN_{\ge2}$, \(X_n=\ZZ/n\ZZ\) for $n\ge2$, all concentrated in degree~0.
Then \(M\tens X_n=0\), hence, \(\prod_{n\ge2}M\tens X_n=0\).
On the other hand, \(W=\QQ\tens_\ZZ\prod_{n\ge2}\ZZ/n\ZZ\) contains $\QQ$\n-vector subspace of $\QQ_p$ generated by $\ZZ_p$ for any prime $p$.
In particular, $W\ne0$ and $\zeta:W\to0$ is not injective.
\end{warning}

In terminology of \cite[Definition~2.5]{BesLyuMan-book} a lax Monoidal category (or a represented multicategory) is the context for studying algebras.
The same notion is called an oplax monoidal category by R.~Street and many other researchers.
The context for studying coalgebras is called a colax Monoidal category by \cite{BesLyuMan-book} (opposite to a represented multicategory) or a lax monoidal category by R.~Street \emph{et al.}
There is no such a discrepancy for (co)lax monoidal (Monoidal) functors.

Recall that $\cv$ possesses countable products.
Aguiar and Mahajan suggest in \cite[Appendix~B.4.4]{AguiarMahajan:Species} that $\cv^\NN$ has a colax Monoidal category structure.
Let us describe all its ingredients \((\cv^\NN,\bar\odot^I,\lambda^f,\rho^L)\):
\[ \Bigl(\bbodot_{i\in I}X_i\Bigr)(n)
=\prod_{t\in\str(n,I)} \, \bigotimes_{i\in I} \, \bigotimes_{k\in t(i)} X_i(t_i^{-1}k).
\]
The valency of \((i,k)\in\IV(t)\) is \(|(i,k)|=|t_i^{-1}k|\).
For each singleton \(L=\{l\}\) the natural transformation \(\rho^L:X\to\bar\odot^LX\) is the obvious isomorphism
\[ \rho^L(n): X(n) \to \prod_{t=\tau[n]} \, \bigotimes_{l\in L} \, \bigotimes_{1\in\mb1} X(n).
\]
For \(f:I\to J\in\co\) the natural transformation
\(\lambda^f:(\bar\odot^{j\in J}\bar\odot^{i\in f^{-1}j}X_i)(n)\to(\bar\odot^{i\in I}X_i)(n)\) is given by the composite
\begin{multline*}
\lambda^f =\bigl[ \prod_{t\in\str(n,J)} \, \bigotimes_{j\in J} \, \bigotimes_{k\in t(j)} \, 
\prod_{\sS{^k_j}t\in\str(t_j^{-1}k,f^{-1}j)} \, \bigotimes_{i\in f^{-1}j} \, \bigotimes_{l\in\sS{^k_j}t(i)} X_i(\sS{^k_j}t_i^{-1}l) \rTTo^{(\prod_t\tens_j\zeta)\cdot\prod_t\zeta}
\\
\prod_{t\in\str(n,J)} \, \prod_{(\sS{^k_j}t)\in\prod_{(j,k)\in\IV(t)}\str(t_j^{-1}k,f^{-1}j)} \, \bigotimes_{j\in J}
\, \bigotimes_{k\in t(j)} \, \bigotimes_{i\in f^{-1}j} \, \bigotimes_{l\in\sS{^k_j}t(i)} X_i(\sS{^k_j}t_i^{-1}l) \rTTo^\sim
\\
\prod_{\tau\in\str(n,I)} \, \bigotimes_{i\in I} \, \bigotimes_{p\in\tau(i)} X_i(\tau_i^{-1}p) \bigr].
\end{multline*}
The second isomorphism takes the factor indexed by \((t,(\sS{^k_j}t))\) to the isomorphic factor indexed by \(\tau=I_t\bigl(\sS{^k_j}t\mid(j,k)\in\IV(t)\bigr)\).
It is easy to prove (opposite to) equations (2.5.1), (2.5.2) and (2.5.4) of \cite{BesLyuMan-book}.
For instance, the latter equation between $\lambda$'s follows from the observation that moving $\prod$ to the left in two steps in the diagram
\begin{diagram}[w=6em]
\ttt \prod\tens\tens\prod\tens\tens\prod\tens\,\tens &\rTTo &\ttt \prod\prod\tens\tens\tens\tens\prod\tens\,\tens
\\
\dTTo &\rdTTo &\dTTo
\\
\ttt \prod\tens\tens\prod\prod\tens\tens\tens\,\tens &\rTTo &\ttt \prod\prod\prod\tens\tens\tens\tens\tens\,\tens
\end{diagram}
amounts to moving them at once (along the diagonal arrow) since all three paths come from the projections \(\prod\tens\tens\prod\tens\tens\prod\tens\,\tens\to\tens\tens\tens\tens\tens\,\tens\).

Denote $*$ the terminal object of $\cv$.
The unit object \(\1=\bar\odot^\varnothing\) of \((\cv^\NN,\odot)\) is computed as
\[ \1(n) =
\begin{cases}
\1_\cv \quad & \text{if } n=1,
\\
* \quad & \text{if } n\ne1.
\end{cases}
\]
Assume further that initial and terminal objects of $\cv$ are isomorphic, hence, constitute the zero object 0.
Then there are isomorphisms
\begin{equation}
\begin{split}
Y\bar\odot\1 \rTTo^{\rho^{\mb1}\bar\odot1}_\sim (\bar\odot^{\mb1}Y)\bar\odot(\bar\odot^\varnothing) 
\rTTo^{\lambda_\cc^{\msf{I{\lar.}}}}_\sim \bar\odot^{\mb1}Y \rTTo^{(\rho^{\mb1})^{-1}}_\sim Y,
\\
\1\bar\odot Y \rTTo^{1\bar\odot\rho^{\mb1}}_\sim (\bar\odot^\varnothing)\bar\odot(\bar\odot^{\mb1}Y)
\rTTo^{\lambda_\cc^{\msf{{\lar.}I}}}_\sim \bar\odot^{\mb1}Y \rTTo^{(\rho^{\mb1})^{-1}}_\sim Y.
\end{split}
\label{eq-YO1-Y-1OY-Y}
\end{equation}

By \cite[Definition~2.25]{BesLyuMan-book} a coalgebra in \((\cv^\NN,\bar\odot^I)\) is a colax Monoidal functor
\(C:(\1,\tens)\to(\cv^\NN,\bar\odot^I)\) from the one-morphism category $\1$.
Equivalently, it is an object $C$ of $\cv^\NN$ equipped with a morphism \(\Delta_I:C\to\bar\odot^IC\) for each $I\in\Ob\co$ such that \(\Delta_L=\rho_L\) for each singleton $L$ and for every map $f:I\to J\in\co$ the following equation holds:
\[ \Delta_I = \bigl( C \rTTo^{\Delta_J} \bar\odot^JC \rTTo^{\bar\odot^{j\in J}\Delta_{f^{-1}j}\;}
\bar\odot^{j\in J}\bar\odot^{f^{-1}j}C \rTTo^{\lambda^f} \bar\odot^IC \bigr).
\]

\begin{remark}
It is proven in \cite[Proposition~2.28]{BesLyuMan-book} that a coalgebra structure of an object $C$ of
\((\cv^\NN,\bar\odot^I)\) amounts to comultiplication \(\Delta:C\to C\bar\odot C\) and counit \(\eps:C\to\1\) such that counitality equations hold
\[ \bigl( C \rto\Delta C\bar\odot C \rTTo^{1\bar\odot\eps} C\bar\odot\1 \rto\sim C \bigr) =\id_C, \qquad
\bigl( C \rto\Delta C\bar\odot C \rTTo^{\eps\bar\odot1} \1\bar\odot C \rto\sim C \bigr) =\id_C,
\]
where isomorphisms are those of \eqref{eq-YO1-Y-1OY-Y}.
Due to invertibility of \(\rho^{\mb1}\) equations (2.27.4), (2.27.5) of \cite{BesLyuMan-book} simplify to the above form.
Futhermore, coassociativity equation holds:
\begin{multline*}
\bigl[ C \rto\Delta C\bar\odot C \rTTo^{1\bar\odot\Delta} C\bar\odot(C\bar\odot C) \rTTo^{\rho^{\mb1}\bar\odot\ell^b}
(\bar\odot^{\mb1}C)\bar\odot(\bar\odot^{\{2,3\}}(C,C)) \rto{\lambda^{\msf{IV}}} C\bar\odot C\bar\odot C \bigr]
\\
=\bigl[ C \rto\Delta C\bar\odot C \rTTo^{\Delta\bar\odot\rho^{\{3\}}} (C\bar\odot C)\bar\odot(\bar\odot^{\{3\}}C)
\rto{\lambda^{\msf{VI}}} C\bar\odot C\bar\odot C \bigr].
\end{multline*}
Here \(b:\{2,3\}\to\{1,2\}\) is the only increasing bijection and the isomorphism
\[ \ell^b =\bigl[ A\bar\odot B \overset{\text{def}}= \bar\odot^{\{1,2\}}(A,B) \rTTo^{\rho^{\{2\}}\bar\odot\rho^{\{3\}}}
(\bar\odot^{\{2\}}A)\bar\odot(\bar\odot^{\{3\}}B) \rto{\lambda^b} \bar\odot^{\{2,3\}}(A,B) \bigr]
\]
from \cite[Proposition~2.12]{BesLyuMan-book} is the obvious one (the indexing set is replaced with a bijective one).
\end{remark}

\begin{proposition}
A coalgebra in \((\cv^\NN,\bar\odot^I)\) is the same as a cooperad.
\end{proposition}

\begin{proof}
The both structures on \(C\in\cv^\NN\) can be specified by a counit \(\eps:C(1)\to\1\) and a comultiplication
\[ \Delta: C(n) \to \prod_{t=(\mb n\rto f\mb k\rto\con\mb1)} \bigl(\tens^{j\in\mb k}C(f^{-1}j)\bigr)\tens C(k) 
=\prod_{(n^j)\in\NN^k}^{k\ge0} \bigl(\tens^{j\in\mb k}C(n^j)\bigr)\tens C(k),
\]
where \(n^j=|f^{-1}j|\).
Counitality equations for coalgebra are precisely counitality equations \eqref{eq-counitality-C(n)OC(1)},
\eqref{eq-counitality-C(1)nOC(n)} for cooperad.
Coassociativity equation for coalgebra takes the form
\begin{multline*}
\Bigl[ C(n) \rto\Delta \prod_{f:\mb n\to\mb l} \bigl(\tens^{p\in\mb l}C(f^{-1}p)\bigr)\tens C(l) \rTTo^{\prod1\tens\Delta}
\\
\prod_{f:\mb n\to\mb l} \bigl(\tens^{p\in\mb l}C(f^{-1}p)\bigr)\tens
\prod_{g:\mb l\to\mb k} \bigl(\tens^{j\in\mb k}C(g^{-1}j)\bigr)\tens C(k) \rto{\lambda^{\msf{IV}}}
\\
\hfill \prod_{\mb n\rto f\mb l\rto g\mb k}
\bigl(\tens^{p\in\mb l}C(f^{-1}p)\bigr)\tens \bigl(\tens^{j\in\mb k}C(g^{-1}j)\bigr)\tens C(k) \Bigr] \quad
\\
\hskip\multlinegap =\Bigl[ C(n) \rto\Delta \prod_{h:\mb n\to\mb k} \bigl(\tens^{j\in\mb k}C(h^{-1}j)\bigr)\tens C(k) \rTTo^{\prod(\tens\Delta)\tens1} \hfill
\\
\prod_{h:\mb n\to\mb k} \bigl\langle \tens^{j\in\mb k} \prod_{f_j:h^{-1}j\to\mb{l_j}}
\bigl[(\tens^{p\in\mb{l_j}}C(f_j^{-1}p))\tens C(l_j) \bigr] \bigr\rangle\tens C(k) \rto{\lambda^{\msf{VI}}}
\\
\prod_{\mb n\rto f\mb l\rto g\mb k} \bigl\langle \tens^{j\in\mb k}
\bigl[(\tens^{p\in g^{-1}j}C(f_j^{-1}p))\tens C(g^{-1}j) \bigr] \bigr\rangle\tens C(k) \rTTo^{\text{unshuffle}}_\sim
\\
\prod_{\mb n\rto f\mb l\rto g\mb k}
\bigl(\tens^{p\in\mb l}C(f^{-1}p)\bigr)\tens \bigl(\tens^{j\in\mb k}C(g^{-1}j)\bigr)\tens C(k) \Bigr].
\end{multline*}
Here at step \(\lambda^{\msf{VI}}\) we change the indexing sets of products as follows.
Given a family \((l_j)_{j\in\mb k}\) we define \(l=\sum_{j=1}^kl_j\) and a map \(g:\mb l\to\mb k\) such that \(|g^{-1}j|=l_j\).
The map \(f:\mb n\to\mb l\) is the disjoint union of maps \((f_j)_{j\in\mb k}\):
\[ f =\bigsqcup_{j\in\mb k}f_j: \mb n =\bigsqcup_{j\in\mb k}h^{-1}j \to \bigsqcup_{j\in\mb k}g^{-1}j=\mb l.
\]
Clearly, \(h=f\cdot g\).
Denoting \(i^j_p=|f_j^{-1}p|\), \(n^j=|h^{-1}j|\) we identify the above equation with \eqref{dia-coassociativity-components}.
\end{proof}

Introduce a right colax Monoidal action \(\bar\odot:\cv^\mm\times(\cv^\NN)^m\to\cv^\mm\), \((X_0,X_1,\dots,X_m)\mapsto\bar\odot^{i\in[m]}X_i\),
\[ \Bigl(\bbodot_{i\in[m]}X_i\Bigr)(\ell)
=\prod_{t\in\str(-,m),\,n\in\mm^{t(0)}}^{n_1+\dots+n_{t(0)}=\ell} \, \bigotimes_{i\in[m]} \biggl\langle \bigotimes_{j\in t(0)} X_0(n_j), \Bigl(\bigotimes_{j\in t(i)} X_i(t_i^{-1}j)\Bigr)_{i=1}^m\biggr\rangle.
\]
The relevant structure is similar to that of \((\cv^\NN,\bar\odot^I,\lambda^f,\rho^L)\).

\begin{definition}
Let $C$ be a cooperad in $\cv$.
A \emph{right $C$\n-comodule} is\index{TTindex}{comodule over a cooperad} a coaction morphism \(\delta_r=\delta:N\to N\bar\odot C\) in $\cv^\mm$ such that
\[ \bigl( N \rto\delta N\bar\odot C \rTTo^{1\bar\odot\eps} N\bar\odot\1 \rto\sim N \bigr) =\id_N,
\]
\begin{multline*}
\bigl[ N \rto\delta N\bar\odot C \rTTo^{1\bar\odot\Delta} N\bar\odot(C\bar\odot C) \rto{\lambda^{\msf{IV}}} N\bar\odot C\bar\odot C \bigr]
\\
=\bigl[ N \rto\delta N\bar\odot C \rTTo^{\delta\bar\odot1} (N\bar\odot C)\bar\odot C \rto{\lambda^{\msf{VI}}} N\bar\odot C\bar\odot C \bigr].
\end{multline*}
\end{definition}

In detail a right $C$\n-comodule \(N\in\cv^\mm\) is a family \(\delta_{n^1,\dots,n^k}:N(n^1+\dots+n^k)\to N(n^1)\tdt N(n^k)\tens C(k)\) of morphisms of $\cv$ such that for all $n\in\mm$
\begin{equation}
\bigl[ N(n) \rTTo^{\delta_n} N(n)\tens C(1) \rTTo^{1\tens\eps} N(n)\tens\1 \rTTo^\sim N(n) \bigr] =\id
\label{eq-X-XC(1)-Xk-X}
\end{equation}
and for any $k\in\NN$, any \((l_1,\dots,l_k)\in\NN^k\), an arbitrary \((i_1^j,\dots,i^j_{l_j})\in\mm^{l_j}\) for each \(1\le j\le k\) the following equation holds:
\begin{diagram}[nobalance,h=1.8em,LaTeXeqno]
N(n) &&\hphantom{N(i_1^1)\tdt N(i^1_{l_1})\tdt N(i_1^k)\tdt N(i^k_{l_k})\tens C(l_1+\dots+l_k)}
\\
\dTTo<{\delta_{n^1,\dots,n^k}} &\rdTTo^{\delta_{i_1^1,\dots,i^1_{l_1},\dots,i_1^k,\dots,i^k_{l_k}}}
\\
&&N(i_1^1)\tdt N(i^1_{l_1})\tdt N(i_1^k)\tdt N(i^k_{l_k})\tens C(l_1+\dots+l_k)
\\
N(n^1)\tdt N(n^k)\tens C(k)\hspace*{-3em} &&\dTTo>{1\tens\Delta_{l_1,\dots,l_k}}
\\
\dTTo~{\delta_{i_1^1,\dots,i^1_{l_1}}\tdt\delta_{i_1^k,\dots,i^k_{l_k}}\tens1} &=
\\
&&\hspace*{-3.5em}
N(i_1^1)\tens...\tens N(i^1_{l_1})\tens...\tens N(i_1^k)\tens...\tens N(i^k_{l_k})\tens C(l_1)\tens...\tens C(l_k)\tens C(k)
\\
&\ruTTo[hug]^\sim_{\text{unshuffle}}
\\
N(i_1^1)\tdt N(i^1_{l_1})\tens C(l_1)\tdt N(i_1^k)\tdt N(i^k_{l_k})\tens C(l_k)\tens C(k)\hspace*{-22em}
\label{dia-N-NNC-NNNNC}
\end{diagram}
where \(n^j=i_1^j+\dots+i^j_{l_j}\) for each \(1\le j\le k\) and \(n=n^1+\dots+n^k\).
A \emph{coalgebra over a cooperad} $C$ is\index{TTindex}{coalgebra over a cooperad} a right $C$\n-comodule with $\mm=0$.

\section{Binary multiplications for operads}
Since we know by \corref{cor-staged-operations} that an operad admits a classical definition, we may use an equivalent one with binary multiplications only.
Equivalence is well known, see \textit{e.g.} \cite{MR1898414} or \cite[Remark~2.6]{MR2821434}.
In particular the second definition applies to the operad $\tr$ in $\Set$.
So we define grafting operation
\[ \cup_i: \tr(n) \times \tr(m) \to \tr(n+m-1), \quad i\in\mb m, \quad (t',t) \mapsto (t'\sqcup t)/\sim \; =t'\cup_it,
\]
which glues together precisely two vertices -- the root of $t'$ and the vertex \(i\in\Inp t\simeq\mb m\).
The ordering of the set of incoming edges for each vertex is inherited from $t'$ and $t$ and, by definition,
\[ \Inp(t'\cup_it) =\Inp t' \sqcup (\Inp t-\{i\}).
\]
In particular, all trees with 2 internal vertices are of the form \(T(x,y,z)=\tau[y]\cup_{x+1}\tau[x+1+z]\in\tr(x+y+z)\), \(x,y,z\ge0\).
Thus,
\[ T(x,y,z) =\bigl( \mb y \rTTo^f \mb{x+1+z} \rTTo^\con \mb1, T(0)\sqcup(T(1)-\{x+1\}) \bigr), \qquad f(\mb y)\subset \{x+1\}.
\]
The corresponding multiplications in an operad $\co$
\[ \underset{x,y,z}\bull=\mu^\co_{T(x,y,z)}: \co(y)\tens\co(x+1+z)\to\co(x+y+z),
\]
satisfy a system of identities:
\begin{enumerate}
\renewcommand{\labelenumi}{(\arabic{enumi})}
\item \(\bigl[\co(x+1+z) \rTTo^{\eta\tens1} \co(1)\tens\co(x+1+z) \rTTo^{\underset{x,1,z}\bull} \co(x+1+z)\bigr]=\id\) for \(x,z\ge0\);

\item \(\bigl[\co(y)\rTTo^{1\tens\eta} \co(y)\tens\co(1) \rTTo^{\underset{0,y,0}\bull} \co(y)\bigr]=\id\) for \(y\ge0\);

\item for \(v,w,x,y,z\ge0\)
\begin{diagram}[LaTeXeqno,nobalance]
\co(w)\tens\co(y)\tens\co(v+1+x+1+z) &\rTTo^{1\tens\underset{v+1+x,y,z}\bull\;} &\co(w)\tens\co(v+1+x+y+z)
\\
\dTTo<{(12)}>\wr &
\\
\co(y)\tens\co(w)\tens\co(v+1+x+1+z) &= &\dTTo>{\underset{v,w,x+y+z}\bull}
\\
\dTTo<{1\tens\underset{v,w,x+1+z}\bull}
\\
\co(y)\tens\co(v+w+x+1+z) &\rTTo^{\underset{v+w+x,y,z}\bull} &\co(v+w+x+y+z)
\label{dia-operad-3-OOOOOOOOOO}
\end{diagram}

\item for \(v,w,x,y,z\ge0\)
\begin{diagram}[LaTeXeqno,nobalance]
\co(x)\tens\co(w+1+y)\tens\co(v+1+z) &\rTTo^{1\tens\underset{v,w+1+y,z}\bull\;} &\co(x)\tens\co(v+w+1+y+z)
\\
\dTTo<{\underset{w,x,y}\bull\tens1} &= &\dTTo>{\underset{v+w,x,y+z}\bull}
\\
\co(w+x+y)\tens\co(v+1+z) &\rTTo^{\underset{v,w+x+y,z}\bull} &\co(v+w+x+y+z)
\label{dia-operad-4-OOOOOOO}
\end{diagram}
\end{enumerate}
Equation~(3) is illustrated by an equation between trees with 3 vertices
\begin{equation}
\thicklines
\vstretch 120
\hstretch 140
\qquad\qquad
\begin{tangles}{rl}
\nodel{v+w+x+y+z}\hln6 &\hln5 \\
\nodel{v}\id\hln1\id\step\id\hln1\id\noder{w}\Step\id\hln1 &\id\hln1\id\hln1\id\hln2\id\hln1\id\noder{x+y+z} \\
\nw1\nw1\nw1\n\Step\nodel{x}\id\hln1 &\id\step\id\hln1\id\noder{y}\step\ne1\ne1 \\
\nw1\nw1\nw1\step\id\step &\id\step\n\ne1\ne1\ne1 \\
\nodel{v+1+x}\hln4 &\Step\hln1\noder{z} \\
\HH\hstr{70} \nw4\step\nw3\step\nw2\step\nw1\step &\hstr{70} \id\step\ne1\step\ne2\step\ne3 \\
\hln2 &\hln{1.5}\step[.3]\noder{v+1+x+1+z} \\
\HH\hstr{70} \nw4\nw3\nw2\nw1 &\hstr{70} \id\ne1\ne2\ne3 \\
\node
\end{tangles}
\quad=\qquad\qquad
\begin{tangles}{rl}
\nodel{v+w+x+y+z}\hln6 &\hln5 \\
\nodel{v+w+x}\id\hln1\id\hln2\id\hln1\id\hln1\id\hln1 &\id\Step\nodel{y}\id\hln1\id\step\id\hln1\id\noder{z} \\
\nw1\nw1\step\id\hln1\id\noder{w}\step\id\hln1 &\id\noder{x}\Step\n\ne1\ne1\ne1 \\
\nw1\nw1\nw1\n\step\id\hln1 &\id\step\ne1\ne1\ne1 \\
\nodel{v}\hln1\Step\hln1 &\hln3\noder{x+1+z} \\
\HH\hstr{70} \nw4\step\nw3\step\nw2\step\nw1\step &\hstr{70} \id\step\ne1\step\ne2\step\ne3 \\
\hln2 &\hln{1.5}\step[.3]\noder{v+1+x+1+z} \\
\HH\hstr{70} \nw4\nw3\nw2\nw1 &\hstr{70} \id\ne1\ne2\ne3 \\
\node
\end{tangles}
\label{eq-tangles-3-vertices}
\end{equation}
Equation~(4) describes two ways to collapse the tree
\begin{equation}
\thicklines
\vstretch 120
\hstretch 140
\begin{tangles}{rl}
\nodel{x}\hln1 &\hln1 \\
\nw1\nw1\nw1\nodel{w}\hln1\step[-1]\nw1\nw1\n &\ne1\ne1\step[-1]\hln1\noder{y}\ne1\ne1\ne1 \\
\nw1\nodel{v}\hln1\step[-1]\nw1\nw2\nw1\n &\ne1\ne2\ne1\step[-1]\hln1\noder{z}\ne1 \\
\nw3\nw2\step &\n\step\ne2\ne3
\end{tangles}
\label{eq-tangle-3-vertices-(4)}.
\end{equation}

Notice that for any collection \(\co\) in $\cv$ a system of operations $\eta$, $\underset{x,y,z}\bull$ that satisfies (1)--(4) amounts to an operad structure.
Furthermore, if $\eta$ is not given and only (3), (4) are satisfied, this datum is called a pseudo-operad \cite[Definition~1.18]{MR1898414}.

\section{Binary comultiplications for cooperads}
Corresponding equations for cooperads are obtained by replacing $\cv$ with the opposite category.
We list them for further reference.
\begin{enumerate}
\renewcommand{\labelenumi}{(\arabic{enumi})}
\item \(\bigl[C(x+1+z) \rTTo^{\Delta_{T(x,1,z)}} C(1)\tens C(x+1+z) \rTTo^{\eps\tens1} C(x+1+z)\bigr]=\id\) for \(x,z\ge0\);

\item \(\bigl[C(y)\rTTo^{\Delta_{T(0,y,0)}} C(y)\tens C(1) \rTTo^{1\tens\eps} C(y)\bigr]=\id\) for \(y\ge0\);

\item for \(v,w,x,y,z\ge0\)
\begin{equation}
\begin{diagram}[inline]
C(v+w+x+y+z) &\rTTo^{\Delta_{T(v,w,x+y+z)}} &C(w)\tens C(v+1+x+y+z)
\\
&&\dTTo>{1\tens\Delta_{T(v+1+x,y,z)}}
\\
\dTTo<{\Delta_{T(v+w+x,y,z)}} &= &C(w)\tens C(y)\tens C(v+1+x+1+z)
\\
&&\dTTo<\wr>{(12)}
\\
C(y)\tens C(v+w+x+1+z) &\rTTo^{1\tens\Delta_{T(v,w,x+1+z)}} &C(y)\tens C(w)\tens C(v+1+x+1+z)
\end{diagram}
\label{dia-cooperad-3-OOOOOOO}
\end{equation}

\item for \(v,w,x,y,z\ge0\)
\begin{diagram}[LaTeXeqno,nobalance]
C(v+w+x+y+z) &\rTTo^{\Delta_{T(v+w,x,y+z)}} &C(x)\tens C(v+w+1+y+z)
\\
\dTTo<{\Delta_{T(v,w+x+y,z)}} &= &\dTTo<{1\tens\Delta_{T(v,w+1+y,z)}}
\\
C(w+x+y)\tens C(v+1+z) &\rTTo^{\Delta_{T(w,x,y)}\tens1} &C(x)\tens C(w+1+y)\tens C(v+1+z)
\label{dia-cooperad-4-OOOOO}
\end{diagram}
\end{enumerate}

\begin{proposition}\label{pro-inner-hom-VN}
The monoidal category \((\cv^\NN,\odot)\) is closed with the inner hom \(\und{\cv^\NN}(M,N)\) given by
\[ \und{\cv^\NN}(M,N)(k) =\prod_{(n_i)\in\NN^k} \und\cv\bigl(M(n_1)\tdt M(n_k),N(n_1+\dots+n_k)\bigr).
\]
\end{proposition}

\begin{proof}
There is an isomorphism natural in \(M,N,P\in\Ob\cv^\NN\)
\begin{multline*}
\cv^\NN(M\odot P,N) =\prod_{n\in\NN} \cv\biggl(\bigoplus^{k\in\NN}_{n_1+\dots+n_k=n}M(n_1)\tdt M(n_k)\tens P(k),N(n)\biggr)
\\
\simeq \prod_{k\in\NN} \prod_{(n_i)\in\NN^k} \cv\bigl(M(n_1)\tdt M(n_k)\tens P(k),N(n_1+\dots+n_k)\bigr)
\\
\simeq \prod_{k\in\NN} \prod_{(n_i)\in\NN^k} \cv\bigl(P(k),\und\cv(M(n_1)\tdt M(n_k),N(n_1+\dots+n_k))\bigr)
\\
\simeq \prod_{k\in\NN} \cv\bigl(P(k),\prod_{(n_i)\in\NN^k}\und\cv(M(n_1)\tdt M(n_k),N(n_1+\dots+n_k))\bigr) =\cv^\NN(P,\und{\cv^\NN}(M,N)).
\end{multline*}
This proves the claim.
\end{proof}

\begin{corollary}
\(\END M=\und{\cv^\NN}(M,M)\) is an operad.
\end{corollary}

\begin{corollary}
Let $M$ be an object of $\cv^\NN$ and let $\co$ be an operad in $\cv$.
In assumptions of \propref{pro-inner-hom-VN} a structure of right $\co$\n-module on $M$ amounts to an operad morphism \(a:\co\to\END M\), which is a mapping \(a\in\cv^\NN\) that takes the unit to the unit and satisfies for all \(x,y,z\in\NN\) the equation
\begin{diagram}
\co(y)\tens\co(x+1+z) &\rTTo^{a\tens a} &\END M(y)\tens\END M(x+1+z)
\\
\dTTo<{\mu_{T(x,y,z)}} &= &\dTTo>{\mu_{T(x,y,z)}}
\\
\co(x+y+z) &\rTTo^a &\END M(x+y+z)
\end{diagram}
\end{corollary}

The last statement can be rewritten as follows.
An $\co$\n-module $M$ amounts to a unital action \(\alpha_{n_1,\dots,n_k}:M(n_1)\tdt M(n_k)\tens\co(k)\to M(n_1+\dots+n_k)\), \(k,n_1,\dots,n_k\in\NN\), such that for all \(w,x,y\in\NN\) the equation holds
\begin{multline}
\bigl[ M(n_1)\tdt M(n_{w+x+y})\tens\co(x)\tens\co(w+1+y) \rto\sim
\\
\hskip\multlinegap M(n_1)\tdt M(n_w)\tens M(n_{w+1})\tdt M(n_{w+x})\tens\co(x) \hfill
\\
\hfill \tens M(n_{w+x+1})\tdt M(n_{w+x+y})\tens\co(w+1+y)\quad
\\
\rTTo^{1^{\tens w}\tens\alpha_{n_{w+1},\dots,n_{w+x}}\tens1^{\tens(y+1)}}
\\
M(n_1)\tdt M(n_w)\tens M(n_{w+1}+\dots+n_{w+x})\tens M(n_{w+x+1})\tdt M(n_{w+x+y})\tens\co(w+1+y)
\\
\hfill \rTTo^{\alpha_{n_1,\dots,n_w,n_{w+1}+\dots+n_{w+x},n_{w+x+1},\dots,n_{w+x+y}}} M(n_1+\dots+n_{w+x+y}) \bigr] \quad
\\
\hskip\multlinegap =\bigl[ M(n_1)\tdt M(n_{w+x+y})\tens\co(x)\tens\co(w+1+y) \rTTo^{1^{\tens(w+x+y)}\tens\mu_{T(w,x,y)}} \hfill
\\
M(n_1)\tdt M(n_{w+x+y})\tens\co(w+x+y) \rTTo^{\alpha_{n_1,\dots,n_{w+x+y}}} M(n_1+\dots+n_{w+x+y}) \bigr].
\label{eq-MMOO-MMO-M}
\end{multline}
Actually, this description of right $\co$\n-modules holds also for non-closed categories $\cv$.
In fact, any tree with more than one vertex can be presented as consecutive grafting of trees with two vertices.
Relation of type~\eqref{eq-tangles-3-vertices} holds automatically and another relation holds due to above equation.
Dually, one can describe right comodules over a cooperad $C$ as a counital coaction \(\delta_{n_1,\dots,n_k}:N(n_1+\dots+n_k)\to N(n_1)\tdt N(n_k)\tens C(k)\), \(k,n_1,\dots,n_k\in\NN\), such that for all \(w,x,y\in\NN\) the equation holds
\begin{multline}
\bigl[ N(n_1+\dots+n_{w+x+y}) \rTTo^{\delta_{n_1,\dots,n_w,n_{w+1}+\dots+n_{w+x},n_{w+x+1},\dots,n_{w+x+y}}}
\\
N(n_1)\tdt N(n_w)\tens N(n_{w+1}+\dots+n_{w+x})\tens N(n_{w+x+1})\tdt N(n_{w+x+y})\tens C(w+1+y)
\\
\rTTo^{1^{\tens w}\tens\delta_{n_{w+1},\dots,n_{w+x}}\tens1^{\tens(y+1)}}
\\
\hskip\multlinegap N(n_1)\tdt N(n_w)\tens N(n_{w+1})\tdt N(n_{w+x})\tens C(x) \hfill
\\
\hfill \tens N(n_{w+x+1})\tdt N(n_{w+x+y})\tens C(w+1+y)\quad
\\
\hfill \rto\sim N(n_1)\tdt N(n_{w+x+y})\tens C(x)\tens C(w+1+y) \bigr] \quad
\\
\hskip\multlinegap =\bigl[ N(n_1+\dots+n_{w+x+y}) \rTTo^{\delta_{n_1,\dots,n_{w+x+y}}} N(n_1)\tdt N(n_{w+x+y})\tens C(w+x+y) \hfill
\\
\rTTo^{1^{\tens(w+x+y)}\tens\Delta_{T(w,x,y)}} N(n_1)\tdt N(n_{w+x+y})\tens C(x)\tens C(w+1+y) \bigr].
\label{eq-NNNNNNNNNNNNNNNNNNNN}
\end{multline}

\chapter{Free operads and cofree cooperads}
We study cofree conilpotent non-counital cooperads, which amounts to considering coalgebras over certain comonad given by sum over trees.
We show also that the category of augmented cooperads is equivalent to the category of non-counital cooperads.

\section{Free operad}
Since countable coproducts exist in $\cv$ and $\tens_\cv$ commutes with countable coproducts we may consider one more monad \(\TT:\cv^\NN\to\cv^\NN\), \(P=(P(n))_{n\ge0}\mapsto P\TT\),
\[ (P\TT)(n) =\coprod_{t\in\tr(n)} \, \bigotimes_{v\in\IV(t)} P|v|.
\]
$\TT$-algebras are precisely operads in $\cv$.
Equivalence of this definition of an operad and the conventional one is also proven in \cite[Theorem~1.105]{MR1898414}.

\begin{proposition}
The forgetful functor \(U:\Op\to\cv^\NN\) from operads to collections has a left adjoint functor
\[ \Fo: \cv^\NN \leftrightarrows \Op: U,
\]
where $\Fo(P)=(P\TT,m:P\TT^2\to P\TT)$ is\index{TTsymb}{Fo@$\Fo$} the free operad generated by $P$.
\end{proposition}

Follows from \lemref{lem-forgetful-functor-has-right-adjoint-T} applied to the comonad \(\TT^\op:(\cv^\NN)^\op\to(\cv^\NN)^\op\).

\section{Cofree cooperads}
In order to have an explicit expression for cofree cooperads we restrict their class to coalgebras over certain comonad.
Denote by \(\bott:\cv^\NN\to\cv^\NN\) the endofunctor \(X=(X(n))_{n\ge0}\mapsto X\bott\),
\[ (X\bott)(n) =\coprod_{t\in\tr(n)}^{t\ne\circ} \, \bigotimes_{p\in\IV(t)} X|p|,
\]
Notice that \(X\TT=\1\oplus X\bott\), where \(\1(n)=0\) for $n\ne1$ and \(\1(1)=\1\) -- the unit object of $\cv$.
The endofunctor $\bott$ admits the following structure of a comonad.
The map $\Delta$ restricted to the summand indexed by \(\tau\ne\circ\) with \(|\Inp\tau|=n\)
\begin{equation*}
\bigotimes_{v\in\IV(\tau)} X|v| \to
\coprod_{t\in\tr(n)}^{t\ne\circ} \, \coprod_{(t_p\in\tr|p|)_{p\in\IV(t)}}^{\forall p\,t_p\ne\circ} \, \bigotimes_{(p,q)\in\sqcup_{r\in\IV(t)}\IV(t_r)} X|q| =(X\bott^2)(n)
\end{equation*}
is the sum of all canonical isomorphisms
\[ \otimes^{v\in\IV(\tau)} X|v| \rTTo^\sim \otimes^{(p,q)\in\sqcup_{r\in\IV(t)}\IV(t_r)} X|q| 
\]
to the summands indexed by \((t,(t_p)_{p\in\IV(t)})\in\sqcup_{t\in\tr(n)\setminus\circ}\sqcap_{p\in\IV(t)}(\tr|p|\setminus\circ)\) such that \(\tau=I_t(t_p\mid p\in\IV(t))\).
The number of such decompositions of $\tau$ is finite since \(\IV(t_p)\ne\varnothing\), see \eqref{eq-v(It)-Uv(tp)} and also \eqref{eq-Inp(It)-Inp-t}.
In fact all such decompositions are obtained as follows.
Partition the set $\IV(\tau)$ into non-empty subsets $\IV(t_p)$, where $p$ runs over some set $\IV(t)$, thus \(\IV(\tau)=\sqcup_{p\in\IV(t)}\IV(t_p)\).
For each $p\in\IV(t)$ it is required that the full subgraph of $\tau$ with the set of vertices \(\IV(t_p)\) were a tree (and not a mere forest).
Denote by $t_p$ the subtree of $\tau$ with the set of vertices \(V(t_p)=\Pa_t^{-1}(\IV(t_p))\cup\IV(t_p)\) whose subset of internal vertices is $\IV(t_p)$.
Thus \(\Inp t_p=\Pa_t^{-1}(\IV(t_p))\setminus\IV(t_p)\).
Glueing together all vertices of $\tau$ belonging to same $\IV(t_p)$ we obtain a graph which is necessarily a tree.
Its set of vertices is \(\Inp\tau\sqcup\IV(t)\), and we denote it by $t$ imposing \(\Inp t=\Inp\tau\).
Clearly, \(\tau=I_t(t_p\mid p\in\IV(t))\) and that gives all such decompositions of $\tau$.

Below we shall present $\bott$\n-coalgebras as a particular case of non-counital cooperads.
The counit \(e:\bott\to\Id\) is given by the projection to the summand indexed by $\tau[n]$ and the isomorphism \(\tens^{p\in\IV(\tau[n])}X|p|=\tens^{\mb1}X(n)\rto\rho X(n)\).

Define the linear tree of height $m$ as
\begin{equation*}
\theta_m =\bigl( \theta_m(0) =\mb1 \to \mb1 \to \mb1 \to \dots \; \to \mb1 \to \mb1 =\theta_m(m) \bigr)
\end{equation*}
with \(\Inp\theta_m=\theta_m(0)=\mb1\).

The unit cooperad $\1$, \(\1(1)=\1\), \(\1(n)=0\) for $n\ne1$, has the following comultiplication.
For $t$ with \(\Inp t\simeq\mb1\) the map \(\Delta_t:\1\to\otimes^{p\in\IV(t)}\1|p|\) has non-zero target iff \(\uv(t)=\IV(t)\) iff \(t=\theta_m\) for some \(m\ge0\) (including \(\theta_0=\circ\)).
For \(t=\theta_m\) the map \(\Delta_{\theta_m}:\1\to\otimes^{p\in\mb m}\1\) is the isomorphism inverse to the multiplication map \(\1^{\tens m}\to\1\).

Let us describe a way for obtaining new comonoids in a monoidal category \((\cc,\tens)\).
The following statement will be applied later.
In this section we use it only as an illustration.

\begin{proposition}\label{pro-cooperad-P-collection-J}
Let $P$ be a comonoid in \((\cc,\tens)\), let \(J\in\Ob\cc\).
Suppose given morphisms \(\eps:J\to\1\) and \(\vartheta:J\tens P\to J\tens P\tens J\).
Then cooperations on $J\tens P$
\begin{align}
\tilde{\Delta} &= \bigl( J\tens P \rTTo^{1\tens\Delta} J\tens P\tens P \rTTo^{\vartheta\tens1} J\tens P\tens J\tens P \bigr),
\label{eq-Delta-Delta-theta}
\\
\tilde{\eps} &= \bigl( J\tens P \rTTo^{\eps\tens\eps} \1\tens\1 \simeq \1\bigr)
\label{eq-eps-eps-O-eps}
\end{align}
turn $J\tens P$ into a comonoid if and only if
\begin{diagram}[LaTeXeqno]
J\tens P &\rTTo^{1\tens\Delta} &J\tens P\tens P &\rTTo^{\vartheta\tens1} &J\tens P\tens J\tens P
\\
\dTTo<\vartheta &&= &&\dTTo>{1\tens1\tens\vartheta}
\\
J\tens P\tens J &\rTTo^{1\tens\Delta\tens1} &J\tens P\tens P\tens J &\rTTo^{\vartheta\tens1\tens1} &J\tens P\tens J\tens P\tens J
\label{dia-JP-JPP-JPJP}
\end{diagram}
\begin{gather}
\bigl( J\tens P \rTTo^{\vartheta} J\tens P\tens J \rTTo^{1\tens1\tens\eps} J\tens P \bigr) =\id,
\label{eq-JP-JPJ-JP-1}
\\
\bigl( J\tens P \rTTo^{\vartheta} J\tens P\tens J \rTTo^{\eps\tens\eps\tens1} J \bigr) =1\tens\eps.
\label{eq-JP-JPJ-J-1e}
\end{gather}
\end{proposition}

\begin{proof}
Let us prove that the given equations suffice to turn \((J\tens P,\tilde{\Delta},\tilde{\eps})\) into a comonoid.
Coassociativity is proven below:
\begin{diagram}
J\tens P &\rTTo^{1\tens\Delta} &J\tens P\tens P &\rTTo^{\vartheta\tens1} &J\tens P\tens J\tens P
\\
\dTTo<{1\tens\Delta} &= &\dTTo>{1\tens1\tens\Delta} &= &\dTTo>{1\tens1\tens1\tens\Delta}
\\
J\tens P\tens P &\rTTo^{1\tens\Delta\tens1} &J\tens P\tens P\tens P &\rTTo^{\vartheta\tens1\tens1} &J\tens P\tens J\tens P\tens P
\\
\dTTo<{\vartheta\tens1} &&= &&\dTTo>{1\tens1\tens\vartheta\tens1}
\\
J\tens P\tens J\tens P &\rTTo^{1\tens\Delta\tens1\tens1} &J\tens P\tens P\tens J\tens P
&\rTTo^{\vartheta\tens1\tens1\tens1} &J\tens P\tens J\tens P\tens J\tens P
\end{diagram}
Counitality is proven as follows:
\begin{gather*}
\bigl( J\tens P \rto{\tilde{\Delta}} J\tens P\tens J\tens P \rTTo^{1\tens1\tens\tilde\eps} J\tens P \bigr)
=\bigl( J\tens P \rto{\vartheta} J\tens P\tens J \rTTo^{1\tens1\tens\eps} J\tens P \bigr) =\id,
\\
\begin{split}
&\bigl( J\tens P \rto{\tilde{\Delta}} J\tens P\tens J\tens P \rTTo^{\tilde\eps\tens1\tens1} J\tens P \bigr)
\\
&=\bigl( J\tens P \rTTo^{1\tens\Delta} J\tens P\tens P \rTTo^{\vartheta\tens1} J\tens P\tens J\tens P \rTTo^{\eps\tens\eps\tens1\tens1} J\tens P \bigr)
\\
&=\bigl( J\tens P \rTTo^{1\tens\Delta} J\tens P\tens P  \rTTo^{1\tens\eps\tens1} J\tens P \bigr) =\id.
\end{split}
\end{gather*}
Necessity of mentioned conditions is left to the reader.
\end{proof}

Let us apply \propref{pro-cooperad-P-collection-J} to the monoidal category \((\cc,\tens)=(\End\cv^\NN,\circ)\), the comonad $P=\bott$ and the endofunctor \(J:\cv^\NN\to\cv^\NN\), \(XJ=X\oplus\1\), \(fJ=f\oplus1_\1\) ($\1$ is concentrated in arity~1).
It does not apply literally, so in addition to the comonad\index{TTsymb}{bott@$\bott$}
\[ X\bott =\coprod_{t\in\tr}^{t\ne\circ} \, \bigotimes_{p\in\IV(t)} X|p|
\]
we introduce also the endofunctor\index{TTsymb}{botth@$\botth$}
\[ X\botth =\prod_{t\in\tr}^{t\ne\circ} \, \bigotimes_{p\in\IV(t)} X|p|
\]
with the obvious embedding \(\iota:\bott\hookrightarrow\botth\).
The last two endofunctors have companions\index{TTsymb}{botto@$\botto$}
\begin{gather}
X\botto =X\bott J =\coprod_{t\in\tr}\, \bigotimes_{p\in\IV(t)} X|p|, \notag
\\
X\bottho =X\botth J =\prod_{t\in\tr}\, \bigotimes_{p\in\IV(t)} X|p|,
\label{eq-XTo-XTJ-POXp}
\end{gather}
embedded\index{TTsymb}{bottoh@$\bottho$} via \(\iota_\circ=\iota J:\botto\hookrightarrow\bottho\).
The map \(\eps:J\to\Id\) required in \propref{pro-cooperad-P-collection-J} is the projection \(\pr_X:X\oplus\1\to X\).
It seems that a map \(\vartheta:\bott\cdot J\to J\bott J\) with good properties can not be defined.
Instead we take the natural transformation
\[ \vartheta_X: X\botto =X\bott J \to XJ\botth J =XJ\bottho, \quad
\vartheta_X: \coprod_{\tau\in\tr} \, \bigotimes_{v\in\IV(\tau)} X|v| \to \prod_{t\in\tr} \, \bigotimes_{p\in\IV(t)} (X\oplus\1)|p|
\]
defined as follows.
For any pair of trees \((\tau,t)\) find all subsets \(N\subset\uv(t)\) such that \(\tau=t^N\),   where $t^N$ means the tree $t$ with unary vertices from $N$ removed (and adjacent edges glued together), \(\IV(t^N)=\IV(t)-N\).
More precisely, \(t^N=I_t(t_p\mid p\in\IV(t))\), where \(t_p=\tau|p|\) for \(p\in\IV(t)-N\) and \(t_p=\circ\) for \(p\in N\).
In particular, \(t^\emptyset=t\).
With any such $N$ there is associated the obvious embedding
\[ \bigotimes_{v\in\IV(\tau)} X|v| \rTTo_\sim \bigotimes_{p\in\IV(t)} Y^N(p) \rMono \bigotimes_{p\in\IV(t)} (X\oplus\1)|p|,
\]
where
\[ Y^N(p)=
\begin{cases}
\1, &\quad \text{if } p\in N,
\\
X|p|, &\quad \text{if } p\in\IV(t)-N.
\end{cases}
\]
The matrix entry of $\vartheta_X$
\[ \bigotimes_{v\in\IV(\tau)} X|v| \to \bigotimes_{p\in\IV(t)} (X\oplus\1)|p|
\]
is the sum of all such embeddings.
In the particular case of \((\tau=\circ,t=\theta_m)\) for $m\ge0$, the summand $\1$ is sent by an obvious isomorphism to the summand \(\1^{\tens m}\) corresponding to \(N=\uv(\theta_m)=\mb m\).

Equation~\eqref{eq-JP-JPJ-JP-1} obtains the form
\[ \bigl( \botto =\bott\cdot J \rto\vartheta J\botth J \rTTo^{\eps\botth J} \botth J =\bottho \bigr) = \iota_\circ.
\]
It holds because $N=\emptyset$ implies $\tau=t$.
The counit \(\eps=\pr_{\tau[n]}:(X\bott)(n)\to X(n)\), \(n\ge0\), extends on three other endofunctors by the same projection
\begin{align*}
\eps &=\pr_{\tau[n]}: (X\botth)(n) \to X(n),
\\
(\eps &=\pr_{\tau[n]}: (X\botto)(n) \to X(n)) =(\eps\cdot\eps: (X\bott J)(n) \to X(n)),
\\
(\eps &=\pr_{\tau[n]}: (X\bottho)(n) \to X(n)) =(\eps\cdot\eps: (X\botth J)(n) \to X(n)),
\end{align*}
\textit{cf.} \eqref{eq-eps-eps-O-eps}.
Equation~\eqref{eq-JP-JPJ-J-1e} takes the form
\[ \bigl( (X\bott J)(n) \rto\vartheta (XJ\bottho)(n) \rto\eps (XJ)(n) \bigl) =\eps J. 
\]
It obviously holds true for $n\ne1$.
For $n=1$ we consider $t=\tau[1]$.
The set \(\uv(\tau[1])=\IV(\tau[1])\) consists of one element and either $N=\emptyset$, $\tau=\tau[1]$, or $N=\uv(\tau[1])$, $\tau=\circ$.

Let us define a natural transformation between two endofunctors of $\cv^\NN$:
\[ \hat{\Delta}_X: \prod_{\tau\in\tr} \, \bigotimes_{v\in\IV(\tau)} X|v| \to \prod_{t\in\tr} \, \prod_{(t_p)\in\prod_{p\in\IV(t)}\tr|p|} \, \bigotimes_{p\in\IV(t)} \, \bigotimes_{q\in\IV(t_p)} X|q|,
\]
``comultiplication''.
Its component indexed by the collection of trees $t\in\tr$, \((t_p)\in\prod_{p\in\IV(t)}\tr|p|\) is taken equal to
\[ \prod_{\tau\in\tr} \, \bigotimes_{v\in\IV(\tau)} X|v| \rTTo^{\pr_{I_t(t_p\mid p\in\IV(t))}} \bigotimes_{v\in\IV(I_t(t_p\mid p\in\IV(t)))} X|v| \rto\sim \bigotimes_{p\in\IV(t)} \, \bigotimes_{q\in\IV(t_p)} X|q|.
\]
One expects a kind of comultiplication for $\botto$ to be given by \eqref{eq-Delta-Delta-theta} which in our case takes the form
\[ \tilde{\Delta} = \bigl( X\botto =X\bott J \rTTo^{\Delta J} X\bott\bott J \rTTo^\vartheta X\bott J\botth J =X\botto\bottho \bigr).
\]
Comultiplications $\tilde{\Delta}$ and $\hat{\Delta}$ are related by the commutative diagram
\begin{diagram}[nobalance,LaTeXeqno,h=2.4em]
X\botto &\rMono^{\iota_\circ} &X\bottho &\rEq &\prod_{\tau\in\tr} \, \bigotimes_{v\in\IV(\tau)} X|v| \qquad
\\
\dTTo<{\Delta J} &\rdTTo^{\tilde{\Delta}} &&= &\dTTo>{\hat{\Delta}}
\\
X\bott\bott J &\rTTo^\vartheta &X\botto\bottho &\rMono^i &\prod_{t\in\tr} \, \prod_{(t_p)\in\prod_{p\in\IV(t)}\tr|p|} \, \bigotimes_{p\in\IV(t)} \, \bigotimes_{q\in\IV(t_p)} X|q|
\label{dia-tilde-Delta-hat-Delta}
\end{diagram}
where the embedding $i$ is
\begin{align*}
i =\Bigl( X\botto\bottho \rTTo^{\iota_\circ\bottho} X\bottho\bottho \rTTo^{\prod\zeta} &\prod_{t\in\tr} \, \prod_{(t_p)\in\prod_{p\in\IV(t)}\tr|p|} \, \bigotimes_{p\in\IV(t)} \, \bigotimes_{q\in\IV(t_p)} X|q| \Bigr)
\\
=\Bigl( \prod_{t\in\tr} \, \bigotimes_{p\in\IV(t)} \, \coprod_{t_p\in\tr|p|} \, \bigotimes_{q\in\IV(t_p)} X|q| \rto\sim &\prod_{t\in\tr} \, \coprod_{(t_p)\in\prod_{p\in\IV(t)}\tr|p|} \, \bigotimes_{p\in\IV(t)} \, \bigotimes_{q\in\IV(t_p)} X|q|
\\
\rMono^{\prod\iota_\circ} &\prod_{t\in\tr} \, \prod_{(t_p)\in\prod_{p\in\IV(t)}\tr|p|} \, \bigotimes_{p\in\IV(t)} \, \bigotimes_{q\in\IV(t_p)} X|q| \Bigr).
\end{align*}

Take any collection of trees $t\in\tr$, $t_p\in\tr|p|$ for $p\in\IV(t)$.
The only summand of $X\botto$ (and factor of $X\bottho$) taken by $\hat{\Delta}$ to $(t,(t_p))$-indexed factor is indexed by \(I_t(t_p\mid p\in\IV(t))\).
The left-bottom path composed with the projection to $(t,(t_p))$-indexed factor
\begin{multline*}
\coprod_{\tau\in\tr} \, \bigotimes_{v\in\IV(\tau)} X|v| \rTTo^{\Delta J} \coprod_{\theta\in\tr} \, \bigotimes_{p\in\IV(\theta)} \, \coprod_{t_p\in\tr|p|}^{t_p\ne\circ} \, \bigotimes_{q\in\IV(t_p)} X|q| \rto\vartheta \prod_{t\in\tr} \, \bigotimes_{p\in\IV(t)} \, \coprod_{t_p\in\tr|p|} \, \bigotimes_{q\in\IV(t_p)} X|q|
\\
\rto\sim \prod_{t\in\tr} \, \coprod_{(t_p)\in\prod_{p\in\IV(t)}\tr|p|} \, \bigotimes_{p\in\IV(t)} \, \bigotimes_{q\in\IV(t_p)} X|q| \rTTo^{\pr_{t.(t_p)}} \bigotimes_{p\in\IV(t)} \, \bigotimes_{q\in\IV(t_p)} X|q|
\end{multline*}
vanishes unless
\begin{equation*}
\tau =I_\theta(t_p\mid p\in\IV(\theta)), \quad t_p\ne\circ \ \ \forall \, p\in\IV(\theta),
\quad \theta =t^N, \quad N =\{p\in\IV(t) \mid t_p =\circ \}.
\end{equation*}
These data amount to two conditions:
\[ \tau =I_t(t_p\mid p\in\IV(t)), \qquad N =\{p\in\IV(t) \mid t_p =\circ \},
\]
which coincide with the previously obtained ones.
Under these conditions the both obtained maps
\[ \bigotimes_{v\in\IV(\tau)} X|v| \rto\sim \bigotimes_{p\in\IV(t)} \, \bigotimes_{q\in\IV(t_p)} X|q|
\]
coincide, hence, diagram~\eqref{dia-tilde-Delta-hat-Delta} is commutative.

The introduced comultiplication $\hat\Delta$ for $\bottho$ allows to define $\bottho$\n-coalgebras.
These are collections \(C\in\cv^\NN\) equipped with the coaction \(\delta:C\to C\bottho\) such that
\begin{equation}
\bigl( C \rTTo^\delta C\bottho \rTTo^\eps C \bigr) =\id,
\label{eq-delta-eps=id}
\end{equation}
\begin{diagram}[nobalance,LaTeXeqno,h=2.4em]
C &\rMono^\delta &C\bottho &\rEq &\prod_{\tau\in\tr} \, \bigotimes_{v\in\IV(\tau)} C|v| \qquad
\\
\dTTo<\delta &&= &&\dTTo>{\hat{\Delta}}
\\
C\bottho &\rMono^{\delta\bottho} &C\bottho\bottho &\rTTo^{\prod\zeta} &\prod_{t\in\tr} \, \prod_{(t_p)\in\prod_{p\in\IV(t)}\tr|p|} \, \bigotimes_{p\in\IV(t)} \, \bigotimes_{q\in\IV(t_p)} C|q|
\label{dia-delta-delta-delta-hat-Delta}
\end{diagram}
The coaction $\delta$ can be presented as the family \(\Delta_t:C(\Inp t)\to\otimes^{p\in\IV(t)}C|p|\), \(t\in\tr\).
Counitality equation \eqref{eq-delta-eps=id} expands to
\[ \bigl( C(n) \rTTo^{(\Delta_t)} \prod_{t\in\tr(n)} \, \bigotimes_{p\in\IV(t)} C|p| \rTTo^{\pr_{\tau[n]}} \bigotimes_{p\in\mb1} C(n) \rTTo^{\rho_{C(n)}}_\sim C(n) \bigr) =\id,
\]
which coincides with \eqref{eq-Delta-tau-n-normalization}, that is, \(\Delta_{\tau[n]}=\rho_{C(n)}^{-1}\).
Coassociativity equation~\eqref{dia-delta-delta-delta-hat-Delta} expands to
\begin{diagram}[nobalance,h=2.4em]
C &\rMono^{(\Delta_\tau)} &\prod_{\tau\in\tr} \, \bigotimes_{v\in\IV(\tau)} C|v| \qquad
\\
&&\dTTo>{\hat{\Delta}}
\\
\dMono<{(\Delta_t)} &= &\prod_{t\in\tr} \, \prod_{(t_p)\in\prod_{p\in\IV(t)}\tr|p|} \, \bigotimes_{p\in\IV(t)} \, \bigotimes_{q\in\IV(t_p)} C|q|
\\
&&\uTTo>{\prod\zeta}
\\
\prod_{t\in\tr} \, \bigotimes_{p\in\IV(t)} C|p| &\rTTo^{\prod\tens(\Delta_{t_p})} &\prod_{t\in\tr} \, \bigotimes_{p\in\IV(t)} \, \prod_{t_p\in\tr|p|} \, \bigotimes_{q\in\IV(t_p)} C|q|
\end{diagram}
Its composition with the projection to the \((t,(t_p))\)\n-indexed factor is precisely \eqref{dia-CCCC-equation-coass}.
Hence, $\bottho$\n-coalgebras are precisely cooperads.

Morphisms of $\bottho$\n-coalgebras are \(f:C\to D\in\cv^\NN\) such that
\[ \bigl( C \rto f D \rMono^{(\Delta_t)} D\bottho \bigr) =\bigl( C \rMono^{(\Delta_t)} C\bottho \rTTo^{f\bottho} D\bottho \bigr).
\]
That is, $f$ agrees with comultiplication \(\Delta_t:C(\Inp t)\to\otimes^{p\in\IV(t)}C|p|\) for each \(t\in\tr\).
Thus, morphisms of $\bottho$\n-coalgebras are precisely morphisms of cooperads.

All conditions of \propref{pro-cooperad-P-collection-J} are realized in  our setting, although not on the nose.
The remaining one \eqref{dia-JP-JPP-JPJP} takes the form of the following equation.

\begin{lemma}\label{lem-theta-theta-delta-hat-Delta}
The following diagram commutes
\begin{diagram}[nobalance,LaTeXeqno,h=2.4em]
X\botto &&\rTTo^\vartheta &&XJ\bottho &\rEq &\prod_{\tau\in\tr} \, \bigotimes_{v\in\IV(\tau)} (XJ)|v| \qquad\qquad\qquad\qquad
\\
\dTTo<{\Delta J} &\rdTTo^{\tilde{\Delta}} &&&= &&\dTTo>{\hat{\Delta}}
\\
X\bott\bott J &\rTTo^\vartheta &X\botto\bottho &\rTTo^{\vartheta\bottho} &XJ\bottho\bottho &\rTTo^{\prod\zeta} &\prod_{t\in\tr} \, \prod_{(t_p)\in\prod_{p\in\IV(t)}\tr|p|} \, \bigotimes_{p\in\IV(t)} \, \bigotimes_{q\in\IV(t_p)} (XJ)|q|
\label{dia-theta-theta-delta-hat-Delta}
\end{diagram}
\end{lemma}

\begin{proof}
Compose this diagram with the embedding of the summand of the source $X\botto$ indexed by $t'\in\tr$ and with the projection to the factor of the target indexed by $t$, \((t_p)\in\prod_{p\in\IV(t)}\tr|p|\).
Then the top-right path is the sum of structural embeddings over all
\[ M \subset \uv\bigl(I_t(t_p\mid p\in\IV(t))\bigr) \simeq \coprod_{p\in\IV(t)}\uv(t_p)
\]
such that \(t'=I_t(t_p\mid p\in\IV(t))^M\).
Actually, \(M\subset\uv(t)\) determines the direct summand of \(\otimes^{v\in\IV(\tau)}(X\oplus\1)|v|\), namely, the factor indexed by $v$ is $X$ if $v\notin M$ and $\1$ if $v\in M$.
The left-bottom path takes the summand indexed by $t'$ as follows
\begin{multline*}
\bigotimes_{x\in\IV(t')} X|x| \rTTo^{\Delta J} \coprod_{t''\in\tr} \, \bigotimes_{y\in\IV(t'')} \, \coprod_{t_y''\in\tr|y|}^{t_y''\ne\circ} \, \bigotimes_{z\in\IV(t_y'')} X|z| \rto\vartheta \prod_{t\in\tr} \, \bigotimes_{p\in\IV(t)} \, \coprod_{t_p''\in\tr|p|} \, \bigotimes_{z\in\IV(t_p'')} X|z|
\\
\rTTo^{\prod\tens\vartheta} \prod_{t\in\tr} \, \bigotimes_{p\in\IV(t)} \, \prod_{t_p\in\tr|p|} \, \bigotimes_{q\in\IV(t_p)} (XJ)|q| \rTTo^{\pr_t\cdot\tens_p\pr_{t_p}} \bigotimes_{p\in\IV(t)} \, \bigotimes_{q\in\IV(t_p)} (XJ)|q|.
\end{multline*}
The matrix entries vanish unless \(t'=I_{t''}(t_y''\mid y\in\IV(t''))\) where \(t_y''\ne\circ\), furthermore \(t''=t^N\) where \(t_p''=\circ\) iff \(p\in N=\IV(t)-\IV(t'')\), at last, \(t_p''=t_p^{K_p}\) for \(K_p\subset\uv(t_p)\) under the condition that the family \((K_p)_{p\in\IV(t)}\) determines the direct summand of \(\otimes^{p\in\IV(t)}\otimes^{q\in\IV(t_p)}(X\oplus\1)|q|\), namely, the factor indexed by $(p,q)$ is $X$ if $q\notin K_p$ and it is $\1$ if $q\in K_p$.
The direct summands coincide iff \(M=\coprod_{p\in\IV(t)}K_p\), equivalently, \(K_p=M\cap\uv(t_p)\).
Once $M$ and \((K_p)_{p\in\IV(t)}\) are given, the subset \(N\subset\uv(t)\) is determined unambiguously by \((t_p)_{p\in\IV(t)}\), namely, \(N=\{p\in\IV(t)\mid K_p=\IV(t_p)\}\).
Clearly, $p\in N$ implies \(K_p=\uv(t_p)=\IV(t_p)\), hence, \(t_p=\theta_m\) for some $m\ge0$.
In particular, \(N\subset\uv(t)\).
Finally, the both maps from the source summand to the direct summand of the target are the obvious isomorphisms, hence, coincide.
Commutativity of \eqref{dia-theta-theta-delta-hat-Delta} is proven.
\end{proof}

\begin{definition}
\emph{A non-counital cooperad (pseudo-cooperad)}\index{TTindex}{non-counital cooperad} \((\bar C,\bar\Delta)\) consists of a collection $\bar C$ in $\cv$, morphisms in $\cv$
\begin{equation}
\bar\Delta_t: \bar C(\Inp t) \to \otimes^{p\in\IV(t)} \bar C|p|
\label{eq-bar-Delta-t-def-pseudo-cooperad}
\end{equation}
given for each tree \(t\in\tr\setminus\circ\) such that \(\bar\Delta_{\tau[n]}=\rho^{-1}_{\bar C(n)}\) for all $n\ge0$ and coassociativity equation \eqref{dia-CCCC-equation-coass} holds for $\bar\Delta$ for all
\((t,(t_p)_{p\in\IV(t)})\in\sqcup_{t\in\tr(n)\setminus\circ}\sqcap_{p\in\IV(t)}(\tr|p|\setminus\circ)\).
Morphisms of the category $\nucoop$ of non-counital cooperads  are morphisms of collections that preserve the comultiplications $\bar\Delta_t$.
\end{definition}

\begin{proposition}\label{pro-aug-equiv-non-counital}
The category $\augcoop$ of augmented cooperads \(\eta:(\1,\Delta)\to(C,\Delta)\) is equivalent to the category $\nucoop$.
\end{proposition}

\begin{proof}
Consider an arbitrary augmented cooperad\index{TTindex}{augmented cooperad} \(\eta:(\1,\Delta)\to(C,\Delta)\), which is a morphism of cooperads.
In particular,
\[ \bigl[ \1 \rTTo^\eta C(1) \rTTo^{\Delta_\circ} \1 \bigr] =\Delta_\circ =\id.
\]
Define a collection $\bar C$ by
\[ \bar C(n)=
\begin{cases}
\Ker(\Delta_\circ:C(1)\to\1), &\quad \text{if } n=1,
\\
C(n), &\quad \text{if } n\ne1.
\end{cases}
\]
Thus, \(C=\1\oplus\bar C\).
Notice that $\bar C$ possesses comultiplication map
\begin{equation}
\bar\Delta_t = \bigl[ \bar C(\Inp t) \rMono^{\inj_2} C(\Inp t) \rTTo^{\Delta_t} \otimes^{p\in\IV(t)} C|p| \rTTo^{\otimes^{p\in\IV(t)}\pr_2\;} \otimes^{p\in\IV(t)} \bar C|p| \bigr]
\label{eq-bar-Delta-t-Delta-pr}
\end{equation}
for each \(t\in\tr\), \(t\ne\circ\).
Clearly, \(\bar\Delta_{\tau[n]}=\id_{\bar C(n)}\) for all $n\ge0$ and coassociativity equation \eqref{dia-CCCC-equation-coass} holds for $\bar\Delta$ for all \((t,(t_p)_{p\in\IV(t)})\in\sqcup_{t\in\tr(n)\setminus\circ}\sqcap_{p\in\IV(t)}(\tr|p|\setminus\circ)\).
This gives a functor \(\augcoop\to\nucoop\), \(C\mapsto\bar C\).

Let us construct a functor \(\nucoop\to\augcoop\), \(\bar C\mapsto\1\oplus\bar C=C\).
We have to define comultiplications
\[ \Delta_t: C(\Inp t) =(\1\oplus\bar C)(\Inp t) \to \bigoplus_{N\subset\uv(t)} \, \bigotimes_{p\in\IV(t)} Y^N(p) =\bigotimes_{p\in\IV(t)} C|p|,
\]
where
\[ Y^N(p)=
\begin{cases}
\1, &\quad \text{if } p\in N,
\\
\bar C|p|, &\quad \text{if } p\in\IV(t)-N.
\end{cases}
\]
The restriction of $\Delta_t$ to \(\1(\Inp t)\) vanishes unless \(\uv(t)=\IV(t)\), that is, \(t=\theta_m\) for some \(m\ge0\).
The restriction of \(\Delta_{\theta_m}\) to \(\1(1)\) is the isomorphism from $\1$ to the direct summand \(\1^{\tens\IV(t)}=\1^{\tens m}\), indexed by \(N=\uv(t)=\mb m\).

The restriction of $\Delta_t$ to \(\bar C(\Inp t)\) for an arbitrary $t$ is given by the sum over \(N\subset\uv(t)\) of maps
\begin{equation}
\bar C(\Inp t) \rTTo^{\bar\Delta_{t^N}} \otimes^{p\in\IV(t^N)} \bar C|p| \rTTo^\sim \otimes^{p\in\IV(t)} Y^N(p) \rMono^{\inj_N} \otimes^{p\in\IV(t)} C|p|,
\label{eq-C(Inpt)-YN(p)}
\end{equation}
where $t^N$ means the tree $t$ with unary vertices from $N$ removed.
In the case \(N=\IV(t)\), which implies \(t^N=\circ\), the first arrow is defined as \(\bar\Delta_\circ=0\).

Let us prove now coassociativity of $\Delta$.
Equation~\eqref{dia-CCCC-equation-coass} involves the right vertical isomorphism, which for \(C=\1\oplus\bar C\) identifies the summand indexed by \(M\subset\uv(\tau)\), \(\tau=I_t(t_p\mid p\in\IV(t))\) with the summand indexed by \((N,(K_p)_{p\in\IV(t)-N})\) as in \lemref{lem-theta-theta-delta-hat-Delta}.
Whiskering diagram~\eqref{dia-CCCC-equation-coass} with the projection to the latter summand we get the equation
\begin{diagram}[h=2.3em]
\bar C(\Inp t) &\rTTo^{\bar\Delta_{\tau^M}} &\otimes^{x\in\IV(\tau^M)} \bar C|x|
\\
\dTTo<{\bar\Delta_{t^N}} &= &\dTTo>\wr
\\
\otimes^{p\in\IV(t^N)} \bar C|p| &\rTTo^{\tens^{p\in\IV(t^N)}\bar\Delta_{t_p^{K_p}}}
&\otimes^{p\in\IV(t^N)} \, \otimes^{q\in\IV(t_p^{K_p})} \bar C|q|
\end{diagram}
which holds true due to isomorphism of trees \(\tau^M=I_{t^N}\bigl(t_p^{K_p}\mid p\in\IV(t^N)\bigr)\).
In fact,
\begin{multline*}
\IV(\tau^M) =\IV(\tau)-M =\bigl(\bigsqcup_{p\in\IV(t)} \IV(t_p)\bigr) -\bigl(\bigsqcup_{p\in N} \uv(t_p)\bigr) -\bigl(\bigsqcup_{p\in\IV(t)-N} K_p\bigr)
\\
=\bigsqcup_{p\in\IV(t)-N} \bigl(\IV(t_p)-K_p\bigr) =\bigsqcup_{p\in\IV(t^N)} \IV\bigl(t_p^{K_p}\bigr) =\IV\bigl(I_{t^N}\bigl(t_p^{K_p}\mid p\in\IV(t^N)\bigr)\bigr).
\end{multline*}
If \(\uv(\tau)=\IV(\tau)\), one proves also equation~\eqref{dia-CCCC-equation-coass} on the direct summand $\1$.

Clearly, the composition \(\nucoop\to\augcoop\to\nucoop\) is isomorphic to the identity functor.
Moreover, the composition \(\augcoop\to\nucoop\to\augcoop\) is isomorphic to the identity functor as well.
In fact, given \(M\subset\uv(t)\) write down equation~\eqref{dia-CCCC-equation-coass} with the following data: $t$, \(t_p=t_p^>\) if \(p\in\IV(t)-M\) and \(t_p=\theta_0=\circ\) if \(p\in M\).
Then \(\tau=I_t(t_p\mid p\in\IV(t))=t^M\), therefore,
\[ \bigl[ (\1\oplus\bar C)(\Inp t) \rTTo^{\Delta_t} \otimes^{p\in\IV(t)}(\1\oplus\bar C)|p| \rTTo^{\otimes^{p\in\IV(t)}f(p)\;} \otimes^{p\in\IV(t^M)}(\1\oplus\bar C)|p| \bigr] =\Delta_{t^M},
\]
where
\[ f(p)=
\begin{cases}
\Delta_\circ =\pr_\1, &\quad \text{if } p\in M,
\\
\Delta_{t_p^>} =\id, &\quad \text{if } p\in\IV(t)-M.
\end{cases}
\]
Thus, corestriction of $\Delta_t$ to the direct summand indexed by $M$ is determined by \(\bar\Delta_{t^M}\) as above.
\end{proof}

Summing up, an arbitrary non-counital cooperad is given by the maps
\[ (\bar\Delta_t)_t: \bar C(n) \rTTo \prod_{t\in\tr(n)\setminus\circ} \; \bigotimes_{p\in\IV(t)} \bar C|p|,
\]
$n\ge0$, that satisfy the coassociativity equation and the normalization condition.
$\bott$\n-coalgebras are distinguished among them by the condition that the above map factorizes as
\[ \bar C(n) \rTTo^{\sum\bar\Delta_t} \bigoplus_{t\in\tr(n)\setminus\circ} \; \bigotimes_{p\in\IV(t)} \bar C|p| \rMono \prod_{t\in\tr(n)\setminus\circ} \; \bigotimes_{p\in\IV(t)} \bar C|p|.
\]
In the case of graded $\kk$\n-modules this means that for any homogeneous element \(c\in\bar C(n)\) there is only a finite number of trees $t$ such that \(\bar\Delta_t(c)\ne0\).
This gives the reason to call the non-counital (resp. augmented) cooperads coming from $\bott$\n-coalgebras non-counital (resp. augmented) conilpotent cooperads\index{TTindex}{conilpotent cooperad}.

Due to \lemref{lem-forgetful-functor-has-right-adjoint-T} for any $\bott$\n-coalgebra $B$ an arbitrary morphism \(f:B\to X\bott\) has the form \(f=\bigl(B\rto\delta B\bott\rTTo^{\check f\bott} X\bott\bigr)\) for certain \(\check f:B\to X\).
Restriction of \(f:Y\bott\to X\bott\) to summand indexed by \(\tau\in\tr(n)\setminus\circ\) is
\[ f =\Bigl[ \bigotimes_{v\in\IV(\tau)} Y|v| \rto\Delta \coprod_{t\in\tr(n)\setminus\circ} \, \bigotimes_{p\in\IV(t)} \, \coprod_{t_p\in\tr|p|\setminus\circ} \, \bigotimes_{q\in\IV(t_p)} Y|q| \rTTo^{\coprod_t\tens_p\check f_{|p|}} \coprod_{t\in\tr(n)\setminus\circ} \, \bigotimes_{p\in\IV(t)} X|p| \Bigr].
\]
In particular, for \(f=\id_X\) the map \(\check f=\eps_\bott\) sends any summand indexed by \(\tens^{q\in\IV(t_p)}X|q|\) to 0 unless \(t_p=\tau|p|\), in which case \(\check f_{|p|}=\rho:\tens^{\mb1}X|p|\to X|p|\).

\begin{remark}\label{rem-functor-Tcoalg-augcoop}
One can extract from the proof of \propref{pro-aug-equiv-non-counital} the following statement.
There is a functor
\[ J: \bott\coalg \to \augcoop, \quad \bar C \mapsto \bigl( \bar CJ, \bar CJ \rTTo^{\delta J} \bar C\bott J \rto\vartheta \bar CJ\bottho \bigr), \quad f \mapsto fJ.
\]
Besides, this can be proven also directly.
The fact that $\bar CJ$ is a cooperad is the commutative diagram
\begin{diagram}[nobalance]
\bar CJ &\rTTo^{\delta J} &\bar C\bott J &\rTTo^\vartheta &\bar CJ\bottho &\rEq &\prod_{\tau\in\tr} \, \bigotimes_{v\in\IV(\tau)} (\bar CJ)|v| \qquad\qquad\qquad\qquad
\\
\dTTo<{\delta J} &= &\dTTo<{\Delta J} &&\eqref{dia-theta-theta-delta-hat-Delta} &&
\\
\bar C\bott J &\rTTo^{\delta\bott J} &\bar C\bott\bott J &&= &&\dTTo>{\hat{\Delta}}
\\
\dTTo<\vartheta &= &\dTTo<\vartheta &&&&
\\
\bar CJ\bottho &\rTTo^{\delta J\bottho} &\bar C\bott J\bottho &\rTTo^{\vartheta\bottho} &\bar CJ\bottho\bottho &\rTTo^{\prod\zeta} &\prod_{t\in\tr} \, \prod_{(t_p)\in\prod_{p\in\IV(t)}\tr|p|} \, \bigotimes_{p\in\IV(t)} \, \bigotimes_{q\in\IV(t_p)} (\bar CJ)|q|
\end{diagram}
When \(f:\bar C\to\bar D\in\bott\coalg\), the explicit expression of the coaction implies that \(fJ:\bar CJ\to\bar DJ\) is a morphism of cooperads.
The category of $\bott$\n-coalgebras has a zero object (initial and final object) \(0=0\bott\).
Therefore, each cooperad $\bar CJ$ for \(\bar C\in\bott\coalg\) is equipped with a morphism of cooperads \(\eta=0J=\inj_\1:\1=0J\to\bar CJ=\bar C\oplus\1\), the augmentation.
Clearly, for each \(f:\bar C\to\bar D\in\bott\coalg\) the morphism $fJ$ preserves the augmentation.
\end{remark}

\begin{proposition}
Let $C$ be a graded cooperad and let $X$ be a graded collection.
Then the map \(\psi:\Coop(C,X\botto)\to\gr^\NN(C,X)\), \(\psi(g)=\bigl(C\rto g X\botto\rto\eps X\bigr)\), is injective and
\[ \im\psi =\bigl\{ f: C\to X\in \gr^\NN \mid \forall \, c\in C(n)^\bull \; \Card\{t\in\tr(n) \mid c.\Delta_t.\tens_{p\in\IV(t)}f|p| \ne0\} <\infty \bigr\}.
\]
\end{proposition}

\begin{proof}
We have a commutative diagram
\begin{diagram}[nobalance,LaTeXeqno,h=2.4em]
C &\rTTo^g &X\botto &\rMono^{\iota_\circ} &X\bottho &\rEq &\prod_{\tau\in\tr} \, \bigotimes_{v\in\IV(\tau)} X|v| \qquad\quad
\\
\dTTo<\delta &= &\dTTo>{\tilde\Delta} &&\eqref{dia-tilde-Delta-hat-Delta} &= &\dTTo>{\hat{\Delta}}
\\
C\bottho &\rTTo^{g\bottho} &X\botto\bottho &&\rMono^i &&\prod_{t\in\tr} \, \prod_{(t_p)\in\prod_{p\in\IV(t)}\tr|p|} \, \bigotimes_{p\in\IV(t)} \, \bigotimes_{q\in\IV(t_p)} X|q|
\label{dia-g-tilde-Delta}
\end{diagram}
Fix an arbitrary collection $t\in\tr$, \((t_p)\in\prod_{p\in\IV(t)}\tr|p|\).
Compose this diagram with the projection to the factor indexed by $t,(t_p)$.
It follows that
\begin{multline}
\Bigl[ C(\Inp t) \rTTo^{g\cdot\pr_{I_t(t_p\mid p\in\IV(t))}} \bigotimes_{v\in\IV(I_t(t_p\mid p\in\IV(t)))} X|v| \rto\sim \bigotimes_{p\in\IV(t)} \, \bigotimes_{q\in\IV(t_p)} X|q| \Bigr]
\\
=\Bigl[ C(\Inp t) \rTTo^{\Delta_t} \bigotimes_{p\in\IV(t)} C|p| \rTTo^{\tens_{p\in\IV(t)}g\cdot\pr_{t_p}} \bigotimes_{p\in\IV(t)} \, \bigotimes_{q\in\IV(t_p)} X|q| \Bigr].
\label{eq-CXX-CCX}
\end{multline}
Consider the particular case of this equation, \(t_p=\tau|p|\).
Then \(I_t(t_p\mid p\in\IV(t))=t\), \(\IV(t_p)=\mb1\), \(g\cdot\pr_{\tau|p|}=g\cdot\eps=\psi(g):C|p|\to X|p|\).
Hence,
\begin{equation}
g\cdot\pr_t =\Bigl[ C \rTTo^{\Delta_t} \bigotimes_{p\in\IV(t)} C|p| \rTTo^{\tens_{p\in\IV(t)}\psi(g)|p|} \bigotimes_{p\in\IV(t)} X|p| \Bigr].
\label{eq-gprt}
\end{equation}
Thus $g$ is expressed via $\psi(g)$ explicitly, and injectivity of $\psi$ is proven.

Let us describe the image of $\psi$.
Clearly, $\im\psi$ is contained in
\[ \bigl\{ f: C\to X\in \gr^\NN \mid \forall \, c\in C(n)^\bull \; \Card\{t\in\tr(n) \mid c.\Delta_t.\tens_{p\in\IV(t)}f|p| \ne0\} <\infty \bigr\}.
\]
On the other hand, an element $f$ of the last set gives rise to a unique map \(g:C\to X\botto\in\gr^\NN\) with
\begin{equation}
g\cdot\pr_t =\Bigl[ C(\Inp t) \rTTo^{\Delta_t} \bigotimes_{p\in\IV(t)}C|p| \rTTo^{\tens_{p\in\IV(t)}f|p|} \bigotimes_{p\in\IV(t)}X|p| \Bigr].
\label{eq-}
\end{equation}
We have to prove that
\begin{diagram}
C &\rTTo^g &X\botto &\rTTo^{\Delta J} &X\bott\botto
\\
\dTTo<\delta &&= &&\dTTo>\vartheta
\\
C\bottho &&\rTTo^{g\bottho} &&X\botto\bottho
\end{diagram}
Composing with the embedding $i$ we represent the alleged equation as commutativity of the exterior of diagram~\eqref{dia-g-tilde-Delta}.
As we already know, this is equivalent to equation~\eqref{eq-CXX-CCX} holding for an arbitrary collection $t\in\tr$, \((t_p)\in\prod_{p\in\IV(t)}\tr|p|\).
Plugging \eqref{eq-gprt} into \eqref{eq-CXX-CCX} we get a valid identity
\begin{multline*}
\Bigl[ C(\Inp t) \rTTo^{\Delta_{I_t(t_p\mid p\in\IV(t))}} \bigotimes_{v\in\IV(I_t(t_p\mid p\in\IV(t)))} C|v| \rTTo^{\tens_vf|v|} \bigotimes_{v\in\IV(I_t(t_p\mid p\in\IV(t)))} X|v| \rto\sim \bigotimes_{p\in\IV(t)} \, \bigotimes_{q\in\IV(t_p)} X|q| \Bigr]
\\
=\Bigl[ C(\Inp t) \rTTo^{\Delta_t} \bigotimes_{p\in\IV(t)} C|p| \rTTo^{\tens_{p\in\IV(t)}\Delta_{t_p}} \bigotimes_{p\in\IV(t)} \, \bigotimes_{q\in\IV(t_p)} C|q| \rTTo^{\tens_{p\in\IV(t)}\tens_{q\in\IV(t_p)}f|q|} \bigotimes_{p\in\IV(t)} \, \bigotimes_{q\in\IV(t_p)} X|q| \Bigr]
\end{multline*}
due to \eqref{dia-CCCC-equation-coass}.

At last, \(g\cdot\pr_{\tau[n]}=f(n)\) for all \(n\in\NN\).
Hence, \(\psi(g)=f\).
\end{proof}

\chapter{Coderivations}
Here we obtain explicit formulae for coderivations with values in cofree conilpotent cooperads and the associated augmented cooperads.
We view coderivations as infinitesimal homomorphisms.
As a warm-up we begin with derivations whose domain is a free operad.

The main example of a cooperad $C$ is considered such that $\Cobar C$ will be the operad of homotopy unital \ainf-algebras.

\section{Derivations}
Denote by\index{TTsymb}{D@$\DD$} \(\DD=\1\langle p\rangle/(p^2)=\1\oplus\1 p=\1\oplus\1[a]\) the commutative algebra in $\cv$ generated by an element $p$ of degree $-a$.
There is a functor \(-\tens\DD:\cv=\comm\1\bimod\to\cv_\DD=\comm\DD\bimod\), \(X\mapsto\tilde{X}=X\tens\DD=X\oplus X[a]\).
This functor extends componentwise to functors \(\cv^\mm\to\cv^\mm_\DD\), \(M\mapsto\tilde{M}=M\tens\DD\).
The algebra $\DD$ is coaugmented by the homomorphism \(\eps:\DD\to\1\), \(p\mapsto0\).
The category $\cv_\DD$ is Monoidal, with the tensor product $\tens_\DD$ defined as the coequalizer
\[ M\tens\DD\tens N \pile{\rTTo^{\text{right\_action}\tens N}\\ \rTTo_{M\tens\text{left\_action}}} M\tens N \to M\tens_\DD N.
\]
Actually we use this tensor product as well as $\cv_\DD$ itself only to write down some equations in concise form.
Using long, expanded from we could get rid of assumption that $\cv$ has coequalizers.

Let \(\phi:\co\to\cp\in\Op_\cv\) be a morphism of operads in $\cv$.
A $\phi$\n-infinitesimal morphism\index{TTindex}{infinitesimal morphism of operads} is a morphism of operads \(\tilde\phi:\tilde\co\to\tilde\cp\) in $\cv_\DD$ such that
\[ \bigl( \tilde\co =\co\tens\DD \rTTo^{1\tens\eps} \co \rto\phi \cp \bigr) =\bigl( \tilde\co \rTTo^{\tilde\phi} \tilde\cp =\cp\tens\DD \rTTo^{1\tens\eps} \cp \bigr).
\]
Equivalently,
\[ \tilde\phi\tens_\DD\1 =\phi: \tilde\co\tens_\DD\1 =\co \to \tilde\cp\tens_\DD\1 =\cp.
\]
The map $\tilde\phi$ is determined by its restriction to $\co$:
\[ \tilde\phi =\phi +\xi p: \co \to \cp\oplus\cp\tens p =\cp\tens\DD =\tilde\cp,
\]
where \(\xi:\co\to\cp\) is a homogeneous map of degree $a$.
The degree 0 map \(\tilde\phi=\phi+\xi p\) is a morphism of operads iff $\xi$ is a $\phi$\n-derivation, that is,
\begin{diagram}[LaTeXeqno]
\co(n_1)\tdt\co(n_k)\tens\co(k) &\rTTo^{\sum_{i=0}^k\phi^{\tens i}\tens\xi\tens\phi^{\tens(k-i)}} &\cp(n_1)\tdt\cp(n_k)\tens\cp(k)
\\
\dTTo<\mu &= &\dTTo>\mu
\\
\co(n_1+\dots+n_k) &\rTTo^\xi &\cp(n_1+\dots+n_k)
\label{dia-OOO-PPP-O-P}
\end{diagram}
for all \(k,n_1,\dots,n_k\in\NN\).
In fact,
\begin{equation}
(\phi +\xi p)^{\tens_\DD(k+1)} =\phi^{\tens(k+1)} +\sum_{i=0}^k\phi^{\tens i}\tens\xi p\tens\phi^{\tens(k-i)}.
\label{eq-phi-xi-p-k1}
\end{equation}
An equivalent definition of $\phi$\n-derivations using binary compositions in $\co$ is given \textit{e.g.} by Vallette \cite[Section~4.7]{math/0609002}.

\begin{example}\label{exa-n-1-derivation}
Let $\co$ be a graded operad.
Then it admits an $\id_\co$\n-derivation \(\co(n)^\bull\ni x\mapsto(n-1)x\in\co(n)^\bull\) of degree~0.
In fact, for the multiplication \(m:\co(n_1)\tdt\co(n_k)\tens\co(k)\to\co(n_1+\dots+n_k)\) the required identity reduces to \(n_1-1+\dots+n_k-1+k-1=n_1+\dots+n_k-1\).
Furthermore, for an arbitrary element of the ground ring $a\in\kk^p$ of degree~$p$ the map \(\co(n)^q\ni x\mapsto(n-1)xa\in\co(n)^{q+p}\) is an $\id_\co$\n-derivation of degree~$p$. 
\end{example}

Let \(\tilde\phi=\phi+\xi p:\tilde\co\to\tilde\cp\) be as above and let \(f:M\to N\in\cv^\mm\) be a $\phi$\n-morphism of modules over operads, that is,
\[ \bigl( M\odot\co \rto\alpha M \rto f N \bigr) =\bigl( M\odot\co \rTTo^{f\odot\phi} N\odot\cp \rto\alpha N \bigr).
\]
An $(f;\phi,\xi)$\n-infinitesimal morphism of modules over operads\index{TTindex}{infinitesimal morphism of modules over operads} is a $\tilde\phi$\n-morphism of modules over operads \(\tilde f:\tilde M=M\tens\DD\to\tilde N=N\tens\DD\) such that
\begin{equation}
\bigl( \tilde M =M\tens\DD \rTTo^{1\tens\eps} M \rto f N \bigr) =\bigl( \tilde M \rto{\tilde f} \tilde N =N\tens\DD \rTTo^{1\tens\eps} N \bigr).
\label{eq-MMDMN-MNNDN}
\end{equation}
Equivalently,
\begin{equation}
\tilde f\tens_\DD\1 =f: \tilde M\tens_\DD\1 =M \to \tilde N\tens_\DD\1 =N.
\label{eq-fkf-MkM-NkN}
\end{equation}
The $\DD$\n-linear degree 0 map \(\tilde f\) is determined by its restriction to $M$:
\begin{equation}
\tilde f =f +rp: M \to N\oplus N\tens p =N\tens\DD =\tilde N,
\label{eq-ffrp-M-NNpNDN}
\end{equation}
where \(r:M\to N\) is a homogeneous map of degree $a$.
The map \(\tilde f\) is a morphism of modules over operads iff $r$ is an \((f;\phi,\xi)\)-derivation, that is,
\begin{diagram}[LaTeXeqno]
M(n_1)\tdt M(n_k)\tens\co(k) &\rTTo^{\sum_{i=1}^kf^{\tens(i-1)}\tens r\tens f^{\tens(k-i)}\tens\phi+f^{\tens k}\tens\xi} &N(n_1)\tdt N(n_k)\tens\cp(k)
\\
\dTTo<\alpha &= &\dTTo>\alpha
\\
M(n_1+\dots+n_k) &\rTTo^r &N(n_1+\dots+n_k)
\label{dia-MMO-NNP-M-N}
\end{diagram}
for all \(k,n_1,\dots,n_k\in\NN\).
In fact,
\begin{equation}
(f+rp)^{\tens_\DD k}\tens_\DD(\phi+\xi p) =f^{\tens k}\tens\phi +\sum_{i=1}^kf^{\tens(i-1)}\tens rp\tens f^{\tens(k-i)}\tens\phi+f^{\tens k}\tens\xi p.
\label{eq-(frp)k-phi}
\end{equation}
The set of derivations \((r;\xi)\), \(\deg r=\deg\xi\), over morphisms \((f;\phi)\) is a graded $\kk$\n-module.

\begin{remark}
Assume that \(\tilde f=f+rp:\tilde M\to\tilde N\) is a morphism over \(\tilde\phi=\phi+\xi p:\tilde\co\to\tilde\cp\) and \(\tilde g=g+sp:\tilde N\to\tilde L\) is a morphism over \(\tilde\psi=\psi+\eta p:\tilde\cp\to\tilde\cq\), \(\deg r=\deg\xi=\deg s=\deg\eta=a\).
Composing these infinitesimal morphisms we see that \(\tilde f\tilde g=fg+(fs+rg)p:\tilde M\to\tilde L\) is a morphism over \(\tilde\phi\tilde\psi=\phi\psi+(\phi\eta+\xi\psi)p:\tilde\co\to\tilde\cq\).
Assume in addition that $a$ is odd and \(\phi\eta=\xi\psi\), \(fs=rg\) hold.
Then \(\xi\eta:\co\to\cq\) is a $\phi\psi$\n-derivation.
In fact, paste together two diagrams~\eqref{dia-OOO-PPP-O-P} and use
\begin{equation}
\biggl( \sum_{i=0}^k\phi^{\tens i}\tens\xi\tens\phi^{\tens(k-i)} \biggr) \cdot \biggl( \sum_{j=0}^k\psi^{\tens j}\tens\eta\tens\psi^{\tens(k-j)} \biggr) =\sum_{i=0}^k(\phi\psi)^{\tens i}\tens\xi\eta\tens(\phi\psi)^{\tens(k-i)}.
\label{eq-phi-xi-phi}
\end{equation}
Furthermore, \(rs:M\to L\) is an \((fg;\phi\psi,\xi\eta)\)-derivation.
In fact, paste together two diagrams~\eqref{dia-MMO-NNP-M-N} and use
\begin{multline}
\biggl( \sum_{i=1}^kf^{\tens(i-1)}\tens r\tens f^{\tens(k-i)}\tens\phi+f^{\tens k}\tens\xi \biggr) \cdot \biggl( \sum_{j=1}^kg^{\tens(j-1)}\tens s\tens g^{\tens(k-j)}\tens\psi+g^{\tens k}\tens\eta \biggr)
\\
=\sum_{i=1}^k(fg)^{\tens(i-1)}\tens rs\tens(fg)^{\tens(k-i)}\tens\phi\psi+(fg)^{\tens k}\tens\xi\eta.
\label{eq-frf-f-xi-gsg-g}
\end{multline}
The same conclusion is reached when \(\deg r=\deg\xi=a\), \(\deg s=\deg\eta=b\) are odd and \(\phi\eta=c\xi\psi\), \(fs=crg\) for \(c\in\kk^{b-a}\), or \(\xi\psi=c'\phi\eta\), \(rg=c'fs\) for \(c'\in\kk^{a-b}\).
These reasonings are typically applied when \(\phi=\psi=\id_\co\), \(f=g=\id_M\), \(\eta=\xi\) is a derivation of $\co$ of odd degree $a$ and $s=r$ is a \((\id_M;\id_\co,\xi)\)-derivation of degree $a$.
Then $r^2$ is a \((\id_M;\id_\co,\xi^2)\)-derivation.
\end{remark}

Introduce the tensoring monad\index{TTsymb}{top@$\top$}
\begin{align*}
\top: \cv^\mm \times \Op_\cv &\longrightarrow \cv^\mm \times \Op_\cv,
\\
(L,\co) &\longmapsto (L\odot\co,\co),
\\
(f:L\to M,\phi:\co\to\cp) &\longmapsto (f\odot\phi:L\odot\co\to M\odot\cp,\phi:\co\to\cp).
\end{align*}
The unit and the multiplication for this monad are induced by those of the operad.
Algebras over this monad are pairs \((M,\co)\) where $\co$ is an operad in $\cv$ and \(M\in\cv^\mm\) is a right $\co$\n-module.
Morphisms of $\top$\n-algebras are pairs \((f:M\to N,\phi:\co\to\cp)\) where $\phi$ is a morphism of operads and $f$ is a morphism of modules over $\phi$.
Denote by $\ModOp$ the category of right modules over operads, equivalent to $\top\alg$.
The underlying functor $U$ has a left adjoint $F$,
\[ F: \cv^\mm \times \Op_\cv \leftrightarrows \ModOp: U,
\]
given by the same formulae as $\top$, see \lemref{lem-forgetful-functor-has-right-adjoint-T}.
Note that \(F\cdot U=\top\) and $U$ is monadic by definition \cite[Section~3]{BarrWells:TopTT}.

In particular, any infinitesimal morphism \(\tilde f=f+rp:\tilde L\odot\tilde\co\to\tilde N\) over \(\tilde\phi=\phi+\xi p:\tilde\co\to\tilde\cp\) is determined by \(\tilde g=g+sp=\check f+\check rp=(1\tens\eta)f+(1\tens\eta)rp:\tilde L\to\tilde N\), namely,
\begin{gather*}
\tilde f =f+rp =\bigl( \tilde L\odot\tilde\co \rTTo^{\tilde g\odot\tilde\phi} \tilde N\odot\tilde\cp \rto\alpha \tilde N \bigr) =\bigl( L\odot\co \rTTo^{g\odot\phi} N\odot\cp \rto\alpha N \bigr) +rp,
\\
\begin{split}
r =\bigl[ L(n_1)\tdt L(n_k)&\tens\co(k) \rTTo^{\sum_{i=1}^kg^{\tens(i-1)}\tens s\tens g^{\tens(k-i)}\tens\phi+g^{\tens k}\tens\xi}
\\
&N(n_1)\tdt N(n_k)\tens\cp(k) \rto\alpha N(n_1+\dots+n_k) \bigr].
\end{split}
\end{gather*}
A general derivation \(r:L\odot\co\to N\) has the above form.

\section{Ground category change}
Let \((F,\phi^I):(\cw,\tens_\cw^I,\lambda^f,\rho)\to(\cv,\tens_\cv^I,\lambda^f,\rho)\) be a colax symmetric Monoidal $\gr$\n-functor \cite[Definition~2.26]{BesLyuMan-book}, in particular, \(\phi^I:F\tens_\cw^{i\in I}X_i\to\tens_\cv^{i\in I}FX_i\) are natural transformations.
In our notations \(\rho=\id:\tens^{\mb1}X\rEq X\) in $\cw$ and $\cv$.
Assume that $F$ preserves countable coproducts.
These data allow to define a natural transformation
\begin{diagram}
\cw^\NN &\rTTo^\bott &\cw^\NN
\\
\dTTo<{F^\NN} &\ldTwoar^\gamma &\dTTo>{F^\NN}
\\
\cv^\NN &\rTTo^\bott &\cv^\NN
\end{diagram}
\begin{multline}
\gamma =\bigl[ (F^\NN\bott C)(n) =F \coprod_{t\in\tr(n)\setminus\circ} \otimes_\cw^{p\in\IV(t)} C|p| \lTTo^\sim \coprod_{t\in\tr(n)\setminus\circ} F \otimes_\cw^{p\in\IV(t)} C|p|
\\
\rTTo^{\coprod_t\phi^{\IV(t)}} \coprod_{t\in\tr(n)\setminus\circ} \otimes_\cv^{p\in\IV(t)} FC|p| =(\bott F^\NN C)(n) \bigr].
\label{eq-FC-FC-FC-FC-FC}
\end{multline}
The entwining transformation $\gamma$ enjoys the following properties.
Firstly,
\begin{equation}
\begin{diagram}[inline]
\cw^\NN &\rTTo^\bott &\cw^\NN
\\
\dTTo<{F^\NN} &\ldTwoar(2,1)_\gamma &\dTTo>{F^\NN}
\\
\cv^\NN &\pile{\rTTo^\bott \\ \sss\eps\Downarrow \\ \rTTo_\Id} &\cv^\NN
\end{diagram}
\quad=\quad
\begin{diagram}[inline]
\cw^\NN &\pile{\rTTo^\bott \\ \sss\eps\Downarrow \\ \rTTo_\Id} &\cw^\NN
\\
\dTTo<{F^\NN} &= &\dTTo>{F^\NN}
\\
\cv^\NN &\rTTo_\Id &\cv^\NN
\end{diagram}
\label{eq-dia-gamma-counital}
\end{equation}
This follows from equation \(\phi^{\IV(\tau[n])}=\id\) since \(|\IV(\tau[n])|=1\).
Secondly,
\begin{equation}
\begin{diagram}[inline]
\cw^\NN &&\rTTo^\bott &&\cw^\NN
\\
\dTTo<{F^\NN} &\rdTTo_\bott &\dTwoar<\Delta &\ruTTo^\bott \ldTwoar(2,4)_\gamma &\dTTo>{F^\NN}
\\
\cv^\NN &\lTwoar_\gamma &\cw^\NN &&\cv^\NN
\\
&\rdTTo_\bott &\dTTo<{F^\NN} &\ruTTo_\bott &
\\
&&\cv^\NN
\end{diagram}
\;=\;
\begin{diagram}[inline]
\cw^\NN &&\rTTo^\bott &&\cw^\NN
\\
\dTTo<{F^\NN} &&&\ldTwoar(4,2)^\gamma &\dTTo>{F^\NN}
\\
\cv^\NN &&\rTTo^\bott &&\cv^\NN
\\
&\rdTTo_\bott &\dTwoar<\Delta &\ruTTo_\bott &
\\
&&\cv^\NN
\end{diagram}
\label{eq-dia-gamma-multiplicative}
\end{equation}
This equation follows from the multiplicativity equation for $\phi^I$ \cite[Definition~2.26(ii)]{BesLyuMan-book} written for the map \(f:\IV(I_t(t_p\mid p\in\IV(t)))\simeq\sqcup_{p\in\IV(t)}\IV(t_p)\to\IV(t)\in\cs\) -- the canonical map to index set.
Notice that \(f^{-1}p=\IV(t_p)\).

\begin{exercise}\label{exe-FN-delta-gamma}
Let \(\delta:C\to\bott C\) be a $\bott$\n-coalgebra in $\cw^\NN$.
Deduce from equations \eqref{eq-dia-gamma-counital}, \eqref{eq-dia-gamma-multiplicative} that
\[ \delta' =\bigl(F^\NN C \rTTo^{F^\NN\delta} F^\NN\bott C \rTTo^\gamma \bott F^\NN C \bigr)
\]
is a $\bott$\n-coalgebra in $\cv^\NN$.
\end{exercise}

When $C$ is a non-counital cooperad (a pseudo-cooperad), a \emph{right comodule over $C$} is\index{TTindex}{comodule over a non-counital cooperad} defined as a family \(\delta_{n_1,\dots,n_k}:N(n_1+\dots+n_k)\to N(n_1)\tdt N(n_k)\tens C(k)\), \(k,n_1,\dots,n_k\in\NN\) that satisfies \eqref{eq-NNNNNNNNNNNNNNNNNNNN} only (without the counitality requirement).
Morphisms of comodules over $C$ are maps $M\to N$ compatible with coactions $\delta_{n_1,\dots,n_k}$.

\section{$\NN$-graded cooperads}\label{sec-N-graded-cooperads}
Recall that $\cw=\NN\text-\gr\text-\cv$ is\index{TTsymb}{W@$\cw$} the category of $\NN$-graded objects \((X_k)_{k\in\NN}\) of $\cv$.
The tensor product of a family \((X^i)_{i\in I}\) of objects of $\cw$ is \((\tens_\cw^{i\in I}X^i)_k=\oplus_{\sum_{i\in I}k_i=k}\tens_\cv^{i\in I}X^i_{k_i}\).
The structure isomorphisms $\lambda_\cw$ (the associativity and the symmetry) are induced by $\lambda_\cv$.
Denote by \(\cw^\NN_{++}\) the\index{TTsymb}{WN++@$\cw^\NN_{++}$} full subcategory of the category \(\cw^\NN\) containing collections $P$ such that \(P(0)_0=0=P(1)_0\).
Denote by\index{TTsymb}{bottcoalg++@$\bott\coalg_{++}$} $\bott\coalg_{++}$ (resp. $\nucoop_{++}$) the\index{TTsymb}{nuCoop++@$\nucoop_{++}$} full subcategory of the category $\bott\coalg$ (resp. $\nucoop$) consisting of $\bott$\n-coalgebras (resp. non-counital cooperads) \(\bar{C}\) in \(\cw^\NN\), whose underlying collection lies in \(\cw^\NN_{++}\).

Notice that \((\cw^\NN_{++})\bott\subset\bott\coalg_{++}\).
In fact, let a collection $X=(X(n)_k)_{n.k\in\NN}$ have \(X(0)_0=X(1)_0=0\).
Then
\[ (X\bott)(0)_0 =\bigoplus_{t\in\tr(0)} \, \bigotimes_{p\in\IV(t)} X|p|_0 =0
\]
since $t$ has at least one leaf \(p\in\IV(t)=V(t)\).
For the same reason summands of
\[ (X\bott)(1)_0 =\bigoplus_{t\in\tr(1)}^{t\ne\circ} \, \bigotimes_{p\in\IV(t)} X|p|_0
\]
vanish unless \(\mb1\simeq\Inp t=L(t)\).
In this case \(t=\theta_m\), $m>0$, hence, there is \(p\in\IV(t)\), $|p|=1$, and the corresponding summand vanishes as well.

\begin{proposition}\label{pro-Tcoalg+nuCoop+}
The full embedding $\bott\coalg_{++}\rMono \nucoop_{++}$ is an equivalence of categories.
\end{proposition}

\begin{proof}
Let $\bar C$ be an object of $\nucoop_{++}$.
Family of comultiplications \eqref{eq-bar-Delta-t-def-pseudo-cooperad} in $\cw$ takes in $\cv$ the form
\begin{equation}
(\bar\Delta_t)_t: \bar C(n)_k \rTTo \prod_{t\in\tr(n)\setminus\circ} \; \bigoplus_{\sum_{p\in\IV(t)}k_p=k}^{k_p\in\NN} \; \bigotimes_{p\in\IV(t)} \bar C|p|_{k_p}.
\label{eq-bar-Delta-t-in-W}
\end{equation}
By assumption, \(\bar C(0)_0=\bar C(1)_0=0\).
We claim that for any pair \((n,k)\in\NN^2\) there is only a finite number of trees \(t\in\tr(n)\) for which there exists a family \((k_p)_{p\in\IV(t)}\in\NN^{\IV(t)}\) such that \(\sum_{p\in\IV(t)}k_p=k\) and \((|p|,k_p)\ne(0,0),\,(1,0)\) for all \(p\in\IV(t)\).
In fact, \(p\in L(t)-\Inp t\) or \(p\in\uv(t)\) implies $k_p\ge1$.
Therefore,
\begin{equation}
|\Inp t|=n, \qquad |L(t)-\Inp t|+|\uv(t)|\le k.
\label{eq-Inpt-n-Lt-Inpt-ut}
\end{equation}
Consequently,
\begin{equation}
|L(t)| +|\uv(t)| \le z \equiv n+k.
\label{eq-Lt-ut-nk}
\end{equation}
The case of $n=k=0$ being easy we may assume that \(z>0\).
The claim follows from the statement that the number of ordered rooted trees (without a chosen subset \(\Inp t\subset L(t)\)) which satisfy restriction \eqref{eq-Lt-ut-nk} is finite.
Indeed, from each staged tree~\eqref{eq-staged-tree-t(0)-t(1)-t(2)} of height $m$ one can produce another staged tree in $\co_\sk$
\begin{equation}
L(t) =X_0 \rEpi^{r_1} X_1 \rEpi^{r_2} \dots \rEpi^{r_m} X_m =\mb1
\label{eq-Lt-X0-X1-Xm1}
\end{equation}
by replacing leaves with unary vertices until all leaves move to $X_0$.
All $r_j$ are surjections and there are no more than \(|\uv(t)|\) bijections between them.
Hence \(m\le|L(t)|-1+|\uv(t)|\le z-1\).
Fix \(m=\height t\), \(l=|L(t)|\).
The set of such trees $t$ is embedded into the product
\(\Set(L(t),\mb{1+m})\times\{\text{trees \eqref{eq-Lt-X0-X1-Xm1}}\}\).
The function \(L(t)\to\mb{1+m}\) returns the level at which a leaf occurs in $t$.
The first factor has no more than \(z^z\) elements.
The number of surjections $r_1$ starting at $L(t)$ is \(2^{|L(t)|-1}=2^{l-1}\).
Similarly, the number of surjections $r_j$ is not bigger than \(2^{l-1}\).
Therefore, the number of sequences~\eqref{eq-Lt-X0-X1-Xm1} is not greater than
\[ \sum_{l=1}^z \sum_{m=0}^{z-1} 2^{(l-1)m} \le z+\sum_{l=2}^z 2^{(l-1)(z-1)+1} < 2^{(z-1)^2+2}.
\]
Consequently, the number of trees $t$ satisfying \eqref{eq-Lt-ut-nk} does not exceed \(z^z2^{(z-1)^2+2}\).
Thus, \eqref{eq-bar-Delta-t-in-W} factors through
\[ \sum\bar\Delta_t: \bar C(n)_k \rTTo \coprod_{t\in\tr(n)\setminus\circ} \; \coprod_{\sum_{p\in\IV(t)}k_p=k}^{k_p\in\NN} \; \bigotimes_{p\in\IV(t)} \bar C|p|_{k_p}
\]
and $\bar C$ has a structure of a $\bott$\n-coalgebra.
\end{proof}

\begin{corollary}
The underlying functor below has a right adjoint
\[ U: \nucoop_{++} \leftrightarrows \cw^\NN_{++}: \Fc,
\]
where $\Fc(P)=(P\bott,\Delta:P\bott\to P\bott^2)$ is\index{TTsymb}{Fc@$\Fc$} the cofree conilpotent non-counital cooperad generated by $P$ (actually, a cofree $\bott$\n-coalgebra).
\end{corollary}

\begin{proof}
Equivalence of $\bott\coalg_{++}$ and $\nucoop_{++}$ reduces the statement to \lemref{lem-forgetful-functor-has-right-adjoint-T} written for the comonad $\bott$ in $\cw^\NN_{++}$.
\end{proof}

By definition, there is a symmetric Monoidal functor \((\oplus,\phi^I):(\cw,\tens_\cw^I)\to(\cv,\tens_\cv^I)\), \(\oplus X=\oplus_{k\in\NN}X_k\), for any family \((X^i)_{i\in I}\) of objects of $\cw$ 
\[ \phi^I =\bigl[ \bigoplus_{k\in\NN}\bigoplus_{\sum_{i\in I}k_i=k}\bigotimes_{i\in I}X^i_{k_i} \rTTo^\sim \bigoplus_{(k_i)_{i\in I}\in\NN^I}\bigotimes_{i\in I}X^i_{k_i} \rTTo^\sim \bigotimes_{i\in I}\bigoplus_{k_i\in\NN}X^i_{k_i} \bigr]
\]
is the natural morphism, invertible since we require $\tens_\cv$ to preserve countable direct sums.
Thus, we may construct $\gamma$ as in \eqref{eq-FC-FC-FC-FC-FC} and conclude by \exeref{exe-FN-delta-gamma} that any $\bott$\n-coalgebra \(\delta:C\to\bott C\in\cw^\NN\) induces a $\bott$\n-coalgebra in $\cv^\NN$
\[ \delta' =\bigl(\oplus^\NN C \rTTo^{\oplus^\NN\delta} \oplus^\NN\bott C \rTTo^\gamma \bott\oplus^\NN C \bigr).
\]

\begin{proposition}\label{pro-TX-coalgebra-structures}
Let \(W\in\Ob\cw\), \(X\in\Ob\cw^\NN_{++}\).
Then non-counital $X\bott$\n-comodule structures on $W$ are in bijection with collections of degree~0 maps \(\beta_n:W\to W^{\tens n}\tens X(n)\), $n\ge0$:
\begin{align}
\{W\in X\bott\comodul^{nu}\}\; &\longrightarrow \;\{(\beta_n:W\to W^{\tens n}\tens X(n) \in\cw)_{n\ge0}\} \hspace{-0.4em}
\label{eq-W-Xbott-comod}
\\
(\delta_n:W\to W^{\tens n}\tens(X\bott)(n))_{n\ge0}\; &\longmapsto \;\beta_n=\langle W \rto{\delta_n} W^{\tens n} \tens(X\bott)(n) \rTTo^{1\tens\eps_\bott} W^{\tens n}\tens X(n) \rangle. \notag
\end{align}
\end{proposition}

\begin{proof}
Let us show that map \eqref{eq-W-Xbott-comod} is injective.
In fact, \((\delta_n)_{n\ge0}\) is recovered from \((\beta_n)_{n\ge0}\) due to the following computation, valid for any \(t\in\tr(n)\setminus\circ\) given in form \eqref{eq-staged-tree-t(0)-t(1)-t(2)} of a staged tree with some \(\Inp(t)\subset L(t)\).
For \(0\le j\le m\) denote
\[ \Inp_j t =t(j) \cap \Inp t, \qquad \Inp_{\ge j}t =\bigcup_{k=j}^m \Inp_k t.
\]
Then
\begin{multline*}
\bigl\langle W\rto\delta W^{\tens n}\tens(X\bott)(n) \rTTo^{1\tens\pr_t} W^{\tens n}\tens \tens^{p\in\IV(t)}X|p| \bigr\rangle
\\
\hskip\multlinegap =\bigl\langle W\rto\delta W^{\tens n}\tens(X\bott)(n) \rTTo^{1\tens\Delta_t} W^{\tens n}\tens \tens^{p\in\IV(t)}(X\bott)|p| \rTTo^{1\tens\tens^p\pr_{\tau|p|}} W^{\tens n}\tens \tens^{p\in\IV(t)}X|p| \bigr\rangle \hfill
\\
\hskip\multlinegap =\bigl\langle W\rTTo^{\delta_{|\troot|}} W^{\tens|\troot|}\tens(X\bott)|\troot| \rTTo^{\tens^{p\in t(m-1)}a_p\tens1\cdot\simeq} \hfill
\\
W^{\tens t(m-2)\cup\Inp_{m-1}t} \tens[\tens^{p\in t(m-1)\setminus\Inp t}(X\bott)|p|]\tens(X\bott)|\troot| \rTTo^{\tens^{p\in t(m-2)}a_p\tens1\tens1\cdot\simeq}
\\
W^{\tens t(m-3)\cup\Inp_{\ge m-2}t} \tens[\tens^{p\in t(m-2)\setminus\Inp t}(X\bott)|p|] \tens[\tens^{p\in t(m-1)\setminus\Inp t}(X\bott)|p|] \tens(X\bott)|\troot|
\\
\to \dots \to W^{\tens t(0)\cup\Inp_{\ge1}t} \tens \tens^{j\in\mb m}[\tens^{p\in t(j)\setminus\Inp t}(X\bott)|p|] =W^{\tens\Inp t}\tens \tens^{p\in\IV(t)}(X\bott)|p|
\\
\hfill \rTTo^{1\tens\tens^{p\in\IV(t)}\pr_{\tau|p|}} W^{\tens n}\tens \tens^{p\in\IV(t)}X|p| \bigr\rangle \quad
\\
\hskip\multlinegap =\bigl\langle W\rTTo^{\beta_{|\troot|}} W^{\tens|\troot|} \tens X|\troot| \rTTo^{\tens^{p\in t(m-1)}c_p\tens1\cdot\simeq} \hfill
\\
W^{\tens t(m-2)\cup\Inp_{m-1}t} \tens[\tens^{p\in t(m-1)\setminus\Inp t}X|p|] \tens X|\troot| \rTTo^{\tens^{p\in t(m-2)}c_p\tens1\tens1\cdot\simeq}
\\
W^{\tens t(m-3)\cup\Inp_{\ge m-2}t} \tens[\tens^{p\in t(m-2)\setminus\Inp t}X|p|] \tens[\tens^{p\in t(m-1)\setminus\Inp t}X|p|] \tens X|\troot|
\\
\to \dots \to W^{\tens t(0)\cup\Inp_{\ge1}t} \tens \tens^{j\in\mb m}[\tens^{p\in t(j)\setminus\Inp t}X|p|] =W^{\tens\Inp t} \tens \tens^{p\in\IV(t)}X|p| \bigr\rangle,
\end{multline*}
where \(a_p=\delta_{|p|}\), \(c_p=\beta_{|p|}\) if \(p\notin\Inp t\) and \(a_p=c_p=\id:W\to W\) if \(p\in\Inp t\).
The isomorphisms are the expected permutations.

Based on the above the inverse to mapping \eqref{eq-W-Xbott-comod} is defined as follows.
Present \(t\in\tr(n)\setminus\circ\) as above.
Given a collection \(\beta_n:W\to W^{\tens n}\tens X(n)\), $n\ge0$, produce the following composition:
\begin{multline*}
\delta^t =\bigl\langle W\rTTo^{\beta_{|\troot|}} W^{\tens|\troot|} \tens X|\troot| \rTTo^{\tens^{p\in t(m-1)}a_p\tens1\cdot\simeq}
\\
W^{\tens t(m-2)\cup\Inp_{m-1}t} \tens[\tens^{p\in t(m-1)\setminus\Inp t}X|p|] \tens X|\troot| \rTTo^{\tens^{p\in t(m-2)}a_p\tens1\tens1\cdot\simeq}
\\
W^{\tens t(m-3)\cup\Inp_{\ge m-2}t} \tens[\tens^{p\in t(m-2)\setminus\Inp t}X|p|] \tens[\tens^{p\in t(m-1)\setminus\Inp t}X|p|] \tens X|\troot|
\\
\to \dots \to W^{\tens t(0)\cup\Inp_{\ge1}t} \tens \tens^{j\in\mb m}[\tens^{p\in t(j)\setminus\Inp t}X|p|] =W^{\tens\Inp t} \tens \tens^{p\in\IV(t)}X|p| \bigr\rangle,
\end{multline*}
In any degree \(k\in\NN\) the map \(\delta^t:W_k\to[W^{\tens\Inp t}\tens\tens^{p\in\IV(t)}X|p|]_k\) vanishes unless the total number of nullary and unary internal vertices does not exceed $k$, that is conditions~\eqref{eq-Inpt-n-Lt-Inpt-ut} are satisfied.
It is shown in the proof of \propref{pro-Tcoalg+nuCoop+} that the number of such trees with fixed \(|\Inp t|=n\) is finite.
Therefore, the top horizontal arrow in
\begin{diagram}[w=7em]
W_k &\rTTo^{\delta_{n,k}\;} &\prod_{t\in\tr(n)\setminus\circ} \bigl[W^{\tens n}\tens \bigotimes_{p\in\IV(t)}X|p| \bigr]_k
\\
&\rdTTo_{\exists!\delta_{n,k}} &\uMono
\\
&&\coprod_{t\in\tr(n)\setminus\circ} \bigl[W^{\tens n}\tens \bigotimes_{p\in\IV(t)}X|p| \bigr]_k
\end{diagram}
factorizes as shown through the coproduct.
That gives the composition
\begin{equation*}
\delta_n =\bigl\langle W \rto{\delta_n} \coprod_{t\in\tr(n)\setminus\circ} W^{\tens n}\tens \bigotimes_{p\in\IV(t)}X|p| \rto\sim W^{\tens n}\tens \coprod_{t\in\tr(n)\setminus\circ} \bigotimes_{p\in\IV(t)}X|p| =W^{\tens n}\tens(X\bott)(n) \bigr\rangle.
\end{equation*}
One can prove that these maps make $W$ into a $X\bott$\n-comodule.
\end{proof}

\begin{remark}\label{rem-number-trees-finite}
Let $\cv$ be additive and let \(X_1\in\cw^\NN_+\), \(X_i\in\cw^\NN_{++}\) for \(2\le i\le m\).
Then the canonical morphism (extension by zeroes)
\[ \Bigl(\bigodot_{i\in\mb m}X_i\Bigr)(n) =\coprod_{t\in\str(n,m)} \, \bigotimes_{i\in\mb m} \, \bigotimes_{j\in t(i)} X_i(t_i^{-1}j) \to \prod_{t\in\str(n,m)} \, \bigotimes_{i\in\mb m} \, \bigotimes_{j\in t(i)} X_i(t_i^{-1}j) =\Bigl(\bbodot_{i\in\mb m}X_i\Bigr)(n)
\]
is an isomorphism.
In fact, for each $k\ge0$ only trees $t\in\str(n,m)$ that satisfy \(|L(t)\cap\bigcup_{i\ge1}t(i)|+|\uv(t)\cap\bigcup_{i\ge2}t(i)|\le k\) contribute to the source and the target of \(\bigl(\odot^{i\in\mb m}X_i\bigr)(n)_k\to\bigl(\bar\odot^{i\in\mb m}X_i\bigr)(n)_k\).
Together with \(|\uv(t)\cap t(1)|\le|t(0)|\) this implies an estimate \(|L(t)|+|\uv(t)|\le2|t(0)|+k=2n+k\), cf. \eqref{eq-Lt-ut-nk}.
The number of such trees is finite by the proof of \propref{pro-Tcoalg+nuCoop+}, hence, this map is an isomorphism.

Generalizing, let \(X_0\in\cw^\mm_+\), \(X_i\in\cw^\NN_{++}\) for \(1\le i\le m\).
Then the canonical morphism (extension by zeroes)
\begin{multline}
\Bigl(\bigodot_{i\in[m]}X_i\Bigr)(\ell)
=\coprod_{t\in\str(-,m),\,n\in\mm^{t(0)}}^{n_1+\dots+n_{t(0)}=\ell} \, \bigotimes_{i\in[m]} \biggl\langle \bigotimes_{j\in t(0)} X_0(n_j), \Bigl(\bigotimes_{j\in t(i)} X_i(t_i^{-1}j)\Bigr)_{i=1}^m\biggr\rangle \to
\\
\prod_{t\in\str(-,m),\,n\in\mm^{t(0)}}^{n_1+\dots+n_{t(0)}=\ell} \, \bigotimes_{i\in[m]} \biggl\langle \bigotimes_{j\in t(0)} X_0(n_j), \Bigl(\bigotimes_{j\in t(i)} X_i(t_i^{-1}j)\Bigr)_{i=1}^m\biggr\rangle =\Bigl(\bbodot_{i\in[m]}X_i\Bigr)(\ell).
\label{eq-X-coprod-prod-X}
\end{multline}
is an isomorphism for all \(\ell\in\mm\).
In fact, for each $k\ge0$ only pairs $(t,(n_j)_{j\in t(0)})$ that satisfy
\begin{equation}
|\{j\in t(0)\mid n_j=0\}| +|L(t)\cap\bigcup_{i\ge1}t(i)| +|\uv(t)|\le k
\label{eq-nLuk}
\end{equation}
contribute to the source and the target of \(\bigl(\odot^{i\in[m]}X_i\bigr)(n)_k\to\bigl(\bar\odot^{i\in[m]}X_i\bigr)(n)_k\).
Define a new tree \(\tilde{t}\in\str(|n|,1+m)\), extending $t$ at $\Inp t$ by a new floor of vertices of valency $|n_1|$, \dots, $|n_{t(0)}|$.
Thus, \(\tilde{t}(i+1)=t(i)\) for \(0\le i\le m\), \(\tilde{t}_{i+1}=t_i\) for \(1\le i\le m\), \(|\tilde{t}_1^{-1}(j)|=|n_j|\) for \(j\in\tilde{t}(1)=t(0)\).
Inequality \eqref{eq-nLuk} can be written as \(|L(\tilde{t})\cap\bigcup_{i\ge1}\tilde{t}(i)|+|\uv(\tilde{t})\cap\bigcup_{i\ge2}\tilde{t}(i)|\le k\), which implies \(|L(\tilde{t})|+|\uv(\tilde{t})|\le2|\tilde{t}(0)|+k=2|n|+k=z\) as above.
The number of such trees $\tilde{t}$ is finite, hence, the number of suitable pairs $(t,(n_j)_{j\in t(0)})$ is finite.
Therefore, map~\eqref{eq-X-coprod-prod-X} is an isomorphism.
\end{remark}

\section{Coderivations}\label{sec-Coderivations}
Let \(\phi:B\to C\in\Coop_\cv\) be a morphism of cooperads in $\cv$.
A $\phi$\n-infinitesimal morphism\index{TTindex}{infinitesimal morphism of cooperads} is a morphism of cooperads \(\tilde\phi:\tilde B\to\tilde C\) in $\cv_\DD$ such that
\[ \bigl( \tilde B =B\tens\DD \rTTo^{1\tens\eps} B \rto\phi C \bigr) =\bigl( \tilde B \rto{\tilde\phi} \tilde C =C\tens\DD \rTTo^{1\tens\eps} C \bigr).
\]
Equivalently,
\[ \tilde\phi\tens_\DD\1 =\phi: \tilde B\tens_\DD\1 =B \to \tilde C\tens_\DD\1 =C.
\]
The map $\tilde\phi$ is determined by its restriction to $B$:
\[ \tilde\phi =\phi +\xi p: B \to C\oplus C\tens p =C\tens\DD =\tilde C,
\]
where \(\xi:B\to C\) is a homogeneous map of degree $a$.
The degree 0 map \(\tilde\phi=\phi+\xi p\) is a morphism of cooperads iff $\xi$ is a $\phi$\n-coderivation, that is,
\begin{diagram}[LaTeXeqno]
B(n_1+\dots+n_k) &\rTTo^\xi &C(n_1+\dots+n_k)
\\
\dTTo<\Delta &= &\dTTo>\Delta
\\
B(n_1)\tdt B(n_k)\tens B(k) &\rTTo^{\sum_{i=0}^k\phi^{\tens i}\tens\xi\tens\phi^{\tens(k-i)}} &C(n_1)\tdt C(n_k)\tens C(k)
\label{dia-B-C-BBB-CCC}
\end{diagram}
for all \(k,n_1,\dots,n_k\in\NN\), due to \eqref{eq-phi-xi-p-k1}.

Let \(\tilde\phi=\phi+\xi p:\tilde B\to\tilde C\) be as above and let \(f:M\to N\in\cv^\mm\) be a $\phi$\n-morphism of comodules over cooperads, that is,
\begin{equation}
\bigl( M \rto f N \rto\delta N\bar\odot C \bigr) =\bigl( M \rto\delta M\bar\odot B \rTTo^{f\bar\odot\phi} N\bar\odot C \bigr).
\label{eq-MNNC-MMBNC}
\end{equation}
An $(f;\phi,\xi)$\n-infinitesimal morphism of comodules over cooperads\index{TTindex}{infinitesimal morphism of comodules over cooperads} is a $\tilde\phi$\n-morphism of comodules over cooperads \(\tilde f:\tilde M=M\tens\DD\to\tilde N=N\tens\DD\) such that \eqref{eq-MMDMN-MNNDN}, equivalently, \eqref{eq-fkf-MkM-NkN} holds.
The $\DD$\n-linear degree 0 map \(\tilde f=f+rp\) is determined by its restriction to $M$, see \eqref{eq-ffrp-M-NNpNDN}, where \(r:M\to N\) is a homogeneous map of degree $a$.
The map \(\tilde f\) is a morphism of comodules over cooperads iff $r$ is an\index{TTindex}{coderivation of cooperads} \((f;\phi,\xi)\)-coderivation, that is,
\begin{diagram}
M(n_1+\dots+n_k) &\rTTo^r &N(n_1+\dots+n_k)
\\
\dTTo<\delta &= &\dTTo>\delta
\\
M(n_1)\tdt M(n_k)\tens B(k) &\rTTo^{\sum_{i=1}^kf^{\tens(i-1)}\tens r\tens f^{\tens(k-i)}\tens\phi+f^{\tens k}\tens\xi} &N(n_1)\tdt N(n_k)\tens C(k)
\end{diagram}
for all \(k,n_1,\dots,n_k\in\NN\), due to \eqref{eq-(frp)k-phi}.
The set of coderivations \((r;\xi)\), \(\deg r=\deg\xi\), over morphisms \((f;\phi)\) is a graded $\kk$\n-module.

\begin{remark}\label{rem-r2-idN:idC:xi2-coderivation}
Assume that \(\tilde f=f+rp:\tilde L\to\tilde M\) is a morphism over a morphism of cooperads \(\tilde\phi=\phi+\xi p:\tilde A\to\tilde B\) and \(\tilde g=g+sp:\tilde M\to\tilde N\) is a morphism over \(\tilde\psi=\psi+\eta p:\tilde B\to\tilde C\), \(\deg r=\deg\xi=\deg s=\deg\eta=a\).
Composing these infinitesimal morphisms we see that \(\tilde f\tilde g=fg+(fs+rg)p:\tilde L\to\tilde N\) is a morphism over \(\tilde\phi\tilde\psi=\phi\psi+(\phi\eta+\xi\psi)p:\tilde A\to\tilde C\).
Assume in addition that $a$ is odd and \(\phi\eta=\xi\psi\), \(fs=rg\) hold.
Then \(\xi\eta:A\to C\) is a $\phi\psi$\n-coderivation due to \eqref{eq-phi-xi-phi}.
Furthermore, \(rs:L\to N\) is an \((fg;\phi\psi,\xi\eta)\)-coderivation due to \eqref{eq-frf-f-xi-gsg-g}.
The same conclusion is reached when \(\deg r=\deg\xi=a\), \(\deg s=\deg\eta=b\) are odd and \(\phi\eta=c\xi\psi\), \(fs=crg\) for \(c\in\kk^{b-a}\), or \(\xi\psi=c'\phi\eta\), \(rg=c'fs\) for \(c'\in\kk^{a-b}\).
These reasonings are typically applied when \(\phi=\psi=\id_C\), \(f=g=\id_N\), \(\eta=\xi\) is a coderivation of $C$ of odd degree $a$ and $s=r$ is a \((\id_N;\id_C,\xi)\)-coderivation of degree $a$.
Then $r^2$ is a \((\id_N;\id_C,\xi^2)\)-coderivation.
\end{remark}

Recall that \(\cw^\NN_{++}\) denotes the full subcategory of the category \(\cw^\NN\) containing collections $P$ such that \(P(0)_0=0=P(1)_0\).
Denote by $\augcoopl$ the\index{TTsymb}{augCoop++@$\augcoopl$} category of augmented cooperads $C$ in $\cw$ with \(\Ker\eps\in\cw^\NN_{++}\).
Such $C$ are called\index{TTindex}{connected cooperad} \emph{connected}.
Denote by \(\cw^\mm_+\) the\index{TTsymb}{WM+@$\cw^\mm_+$} full subcategory of the category \(\cw^\mm\) containing objects $P$ such that \(P(0)_0=0\).
Denote by \(\comodplur C\) the\index{TTsymb}{comodC@$\comodplur C$} full subcategory of the category of right $C$\n-comodules which are contained in \(\cw^\mm_+\).
Such comodules are\index{TTindex}{connected comodule over a cooperad} called \emph{connected}.

Introduce the tensoring comonad\index{TTsymb}{bot@$\bot$}
\begin{align*}
\bot: \cw^\mm_+ \times \augcoopl &\longrightarrow \cw^\mm_+ \times \augcoopl,
\\
(N,C) &\longmapsto (N\odot C,C),
\\
(f:M\to N,\phi:B\to C) &\longmapsto (f\odot\phi:M\odot B\to N\odot C,\phi:B\to C).
\end{align*}
The counit and the comultiplication for this comonad are induced by those of the cooperad.
Coalgebras over this comonad are pairs \((N,C)\) where $C$ is a connected cooperad in $\cw$ and \(M\in\cw^\mm_+\) is a connected right $C$\n-comodule.
Morphisms of $\bot$\n-coalgebras are pairs \((f:M\to N,\phi:B\to C)\) where $\phi$ is a morphism of cooperads and $f$ is a morphism of comodules over $\phi$, see \eqref{eq-MNNC-MMBNC}, which by \remref{rem-number-trees-finite} is equivalent to
\[ \bigl( M \rto f N \rto\delta N\odot C \bigr) =\bigl( M \rto\delta M\odot B \rTTo^{f\odot\phi} N\odot C \bigr).
\]
Denote by \(\comodcoopl\) the above category of comodules over cooperads equivalent to $\bot\coalg$.
The underlying functor $U$ has a right adjoint $F$,
\[ U:\comodcoopl \leftrightarrows \cw^\mm_+\times\augcoopl: F,
\]
given by the same formulae as $\bot$, see \lemref{lem-forgetful-functor-has-right-adjoint-T}.
Note that \(F\cdot U=\bot\) and $U$ is comonadic by definition.
The adjunction bijection is
\begin{align*}
\cw^\mm_+(M,N) \times \augcoopl(B,C) &\longleftrightarrow \comodcoopl((M,B),(N\odot C,C)),
\\
(g:M\to N,\phi:B\to C) &\rMapsTo (M \rto\delta M\odot B \rTTo^{g\odot\phi} N\odot C,\phi:B\to C),
\\
(M \rto f N\odot C \rTTo^{1\tens\eps} N,\phi:B\to C) &\lMapsTo (f:M\to N\odot C,\phi:B\to C).
\end{align*}
In particular, an infinitesimal morphism \(\tilde f=f+rp:\tilde M\to\tilde N\odot\tilde C\) over \(\tilde\phi=\phi+\xi p:\tilde B\to\tilde C\) is determined by \(\tilde g=g+sp:\tilde M\to\tilde N\), namely,
\begin{gather}
\tilde f =f+rp =\bigl( \tilde M \rto\delta \tilde M\odot\tilde B \rTTo^{\tilde g\odot\tilde\phi} \tilde N\odot\tilde C \bigr) =\bigl( M \rto\delta M\odot B \rTTo^{g\odot\phi} N\odot C \bigr) +rp, \notag
\\
\begin{split}
r =\bigl[ M(n_1+\dots+n_k) &\rto\delta M(n_1)\tdt M(n_k)\tens B(k) 
\\
&\rTTo^{\sum_{i=1}^kg^{\tens(i-1)}\tens s\tens g^{\tens(k-i)}\tens\phi+g^{\tens k}\tens\xi}
N(n_1)\tdt N(n_k)\tens C(k) \bigr].
\end{split}
\label{eq-general-coderivation-rM-NC}
\end{gather}
A general coderivation \(r:M\to N\odot C\) has the above form.

An infinitesimal $\phi$\n-morphism is the same as an $\phi$\n-coderivation in the sense of

\begin{definition}
Let $\phi:B\to C$ be a morphism of non-counital cooperads in $\gr$.
A \emph{$\phi$\n-coderivation} is\index{TTindex}{coderivation of cooperads} a map $\xi:B\to C$ of degree $a$ such that for all trees $t\ne\circ$
\begin{diagram}[w=4em]
B(\Inp t) &\rTTo^{\Delta_t} &\otimes^{p\in\IV(t)} B|p|
\\
\dTTo<\xi &= &\dTTo>{\sum_{x+1+z=|\IV(t)|}\phi^{\tens x}\tens \xi\tens \phi^{\tens z}}
\\
C(\Inp t) &\rTTo^{\Delta_t} &\otimes^{p\in\IV(t)} C|p|
\end{diagram}
\end{definition}

Since $\Delta_t$ can be presented as composition of comultiplications \(\Delta_{T(x,y,z)}\) tensored with 1, a homogeneous map \(\xi:B\to C\) is a $\phi$\n-coderivation iff for all \(x,y,z\in\NN\)
\begin{diagram}
B(x+y+z) &\rTTo^{\Delta_{T(x,y,z)}} &B(y)\tens B(x+1+z)
\\
\dTTo<\xi &= &\dTTo>{\xi\tens \phi +\phi\tens \xi}
\\
C(x+y+z) &\rTTo^{\Delta_{T(x,y,z)}} &C(y)\tens C(x+1+z)
\end{diagram}

\begin{exercise}\label{exe-id-coderivations}
Let $\phi:B\to C$ be a morphism of non-counital cooperads in $\gr$.
Let \(d_B:B\to B\) (resp. \(d_C:C\to C\)) be an $\id_B$\n-coderivation (resp. $\id_C$\n-coderivation) of degree~$a$.
Then \(\phi\cdot d_C\), \(d_B\cdot \phi\) and \(\xi=\phi\cdot d_C-d_B\cdot \phi:B\to C\) are $\phi$\n-coderivations of degree~$a$.
Notice that when $d_B$, $d_C$ are differentials, then $\xi=0$ iff $\phi$ is a chain map.
\end{exercise}

For any $\bott$\n-coalgebra $B$ and any morphism \(\phi:B\to X\bott\) homogeneous $\phi$\n-coderivations $\xi$ are in bijection with homogeneous maps \(\check \xi=\bigl(B\rto{\xi} X\bott\rTTo^{\eps_\bott} X\bigr)\).
They can be recovered via
\begin{equation}
\xi =\Bigl[ B \rto\delta B\bott =\coprod_{t\in\tr\setminus\circ} \, \bigotimes_{p\in\IV(t)} B|p| \rTTo^{\coprod_t\sum_{i\in\IV(t)}\tens^{p\in\IV(t)}g_i(p)} \coprod_{t\in\tr\setminus\circ} \, \bigotimes_{p\in\IV(t)} X|p| \Bigr],
\label{eq-r-delta-TX}
\end{equation}
where \(g_i(p)=\check \phi\) if $p\ne i$ and \(g_i(p)=\check \xi\) if $p=i$.
In particular, \(\id_{X\bott}\)-coderivations \(\xi:X\bott\to X\bott\) are recovered from \(\check \xi:X\bott\to X\) via
\begin{multline}
\xi =\Bigl[ X\bott =\coprod_{\tau\in\tr\setminus\circ} \, \bigotimes_{v\in\IV(\tau)} X|v| \rto\Delta \coprod_{t\in\tr\setminus\circ} \, \bigotimes_{p\in\IV(t)} \coprod_{t_p\in\tr|p|\setminus\circ} \, \bigotimes_{q\in\IV(t_p)} X|q|
\\
\rTTo^{\coprod_t\sum_{i\in\IV(t)}\tens^{p\in\IV(t)}g_i(p)} \coprod_{t\in\tr\setminus\circ} \, \bigotimes_{p\in\IV(t)} X|p| =X\bott \Bigr],
\label{eq-r-B-TB}
\end{multline}
where \(g_i(p)=\eps_\bott\) if $p\ne i$ and \(g_i(p)=\check \xi\) if $p=i$.

\section{Coderivations with values in ``cofree'' cooperads}
Recall that \(\DD=\1\langle p\rangle/(p^2)\), \(\deg p=-a\) and \(\cv_\DD=\comm\DD\bimod\) denotes the category of $\DD$\n-modules (commutative $\DD$\n-bimodules) with $\DD$\n-linear homomorphisms of degree~0.
The tensor product $\tens_\DD$ and the associated comonads in $\cv_\DD$ are still denoted $\tens$, $\bott$, $\botto$, $\botth$, $\bottho$ by abuse of notation.
The homomorphism \(\eps:\DD\to\1\), \(p\mapsto0\), is denoted also \(\mod p\).
The map \(1\tens\eps:X\tens\DD\to X\) is denoted \(\mod p\) as well.

Let \((\bar C,\delta)\) be a $\bott$\n-coalgebra.
The cooperad \(C=\bar CJ=\bar C\oplus\1\) is associated with the non-counital cooperad $\bar C$.
They are extended to cooperads \(\bar C_\DD=\bar C\tens\DD\) and \(C_\DD= C\tens\DD\) over $\DD$.
Clearly, \(C_\DD=\bar C_\DD\oplus\DD\).
Let \(X\in\Ob\gr^\NN\) and \(X_\DD=X\tens\DD\).

\begin{proposition}
The map
\begin{gather*}
\psi: \bigl\{ g\in \Coop_\DD(C_\DD,X_\DD\botto) \mid \exists \, g':\bar C\to X\bott \in\bott\coalg: \; g\equiv (g'\tens\DD)\oplus1_\DD \mod p \bigr\}
\\
\begin{split}
&\to \bigl\{ f\in\gr_\DD^\NN(C_\DD,X_\DD) \mid \inj_\DD\cdot f\equiv0 \mod p \bigr\},
\\
(g:C_\DD\to X_\DD\botto) &\mapsto (C_\DD \rto g X_\DD\botto \rto\eps X_\DD),
\end{split}
\end{gather*}
is bijective.
\end{proposition}

\begin{proof}
Clearly, the composition \(f=(C_\DD \rto g X_\DD\botto \rto\eps X_\DD)\) taken modulo $p$ satisfies \(\inj_\DD\cdot f\equiv0\).
Let us construct a candidate for the map inverse to $\psi$.
Take a map \(f:C_\DD\to X_\DD\) such that \(\inj_\DD\cdot f\equiv0\mod p\).
Consider the composition
\begin{multline}
\bigl( \bar C_\DD\botto \rto\vartheta C_\DD\bottho \rTTo^{f\bottho} X_\DD\bottho \bigr)
\\
=\Bigl[ \coprod_{\tau\in\tr} \, \bigotimes_{v\in\IV(\tau)} \bar C_\DD|v| \rto\vartheta \prod_{t\in\tr}\, \bigotimes_{y\in\IV(t)} (\bar C_\DD\oplus\DD)|y| \rTTo^{\prod_{t\in\tr}\tens_{y\in\IV(t)}f|y|} \prod_{t\in\tr} \, \bigotimes_{y\in\IV(t)} X_\DD|y| \Bigr].
\label{eq-CDTo-fTo-XDTo}
\end{multline}
For each tree $t$ there is no more than one tensor factor, in which $f$ is restricted to $\DD$, since each such factor is a multiple of $p$.
Fix a tree $\tau$ indexing a summand of the source.
The relevant tree $t$ either coincides with $\tau$ or differs from $\tau$ by adding one unary vertex.
The number of such trees is finite.
Therefore, \eqref{eq-CDTo-fTo-XDTo} ends up in the submodule \(\coprod_{t\in\tr}\tens_{y\in\IV(t)}X_\DD|y|=X_\DD\botto\) and there is a unique $\DD$\n-linear mapping $Q(f)$ such that
\begin{equation}
\bigl( \bar C_\DD\botto \rTTo^{Q(f)} X_\DD\botto \rMono^{\iota_\circ} X_\DD\bottho \bigr) =\bigl( \bar C_\DD\botto \rto\vartheta C_\DD\bottho \rTTo^{f\bottho} X_\DD\bottho \bigr).
\label{eq-Qf-theta-fTo}
\end{equation}

Define a map
\begin{gather*}
\phi: \bigl\{ f\in\gr_\DD^\NN(C_\DD,X_\DD) \mid \inj_\DD\cdot f\equiv0 \mod p \bigr\} \to \Coop_\DD(C_\DD,X_\DD\botto),
\\
\phi(f) =\bigl( \bar C_\DD J \rTTo^{\delta J} \bar C_\DD\botto \rTTo^{Q(f)} X_\DD\botto \bigr).
\end{gather*}
Let us verify that $\phi(f)$ is indeed a homomorphism of cooperads.
Clearly, \(\delta:\bar C_\DD\to\bar C_\DD\bott\) is a morphism of $\bott$\n-coalgebras and $\delta J$ is a morphism of cooperads.
The property of $Q(f)$ being a morphism of cooperads is the left square in the following diagram
\begin{diagram}[nobalance,h=2.4em]
\bar C_\DD\botto &\rTTo^{Q(f)} &X_\DD\botto &\rMono^{\iota_\circ} &X_\DD\bottho &\rEq &\prod_{\tau\in\tr} \, \bigotimes_{v\in\IV(\tau)} X_\DD|v| \qquad\qquad\qquad\qquad
\\
\dTTo<{\tilde\Delta} &&\dTTo>{\tilde\Delta} &&\eqref{dia-tilde-Delta-hat-Delta} &= &\dTTo>{\hat{\Delta}}
\\
\bar C_\DD\botto\bottho &\rTTo^{Q(f)\bottho} &X_\DD\botto\bottho &\rTTo^{\iota_\circ\bottho} &X_\DD\bottho\bottho &\rTTo^{\prod\zeta} &\prod_{t\in\tr} \, \prod_{(t_y)\in\prod_{y\in\IV(t)}\tr|y|} \, \bigotimes_{y\in\IV(t)} \, \bigotimes_{q\in\IV(t_y)} X_\DD|q|
\end{diagram}
Since \(i=\iota_\circ\bottho\cdot\prod\zeta\) in the bottom row is an embedding, commutativity of the left square is equivalent to commutativity of the exterior of the above diagram.
Using definition \eqref{eq-Qf-theta-fTo} of $Q(f)$ we represent the equation in question as commutativity of
\begin{diagram}[nobalance,h=2.4em]
\bar C_\DD\botto &\rTTo^\vartheta &C_\DD\bottho &\rTTo^{f\bottho} &X_\DD\bottho &\rEq &\prod_{\tau\in\tr} \, \bigotimes_{v\in\IV(\tau)} X_\DD|v| \qquad\qquad\qquad\qquad
\\
\dTTo<{\tilde\Delta} &&&&&&\dTTo>{\hat{\Delta}}
\\
\bar C_\DD\botto\bottho &\rTTo^{\vartheta\bottho} &C_\DD\bottho\bottho &\rTTo^{f\bottho\bottho} &X_\DD\bottho\bottho &\rTTo^{\prod\zeta} &\prod_{t\in\tr} \, \prod_{(t_y)\in\prod_{y\in\IV(t)}\tr|y|} \, \bigotimes_{y\in\IV(t)} \, \bigotimes_{q\in\IV(t_y)} X_\DD|q|
\end{diagram}
Naturality of the transformation $\prod\zeta$ reduces the question to the diagram
\begin{diagram}[nobalance,h=2.5em]
\bar C_\DD\botto &\rTTo^\vartheta &C_\DD\bottho &\rTTo^{f\bottho} &X_\DD\bottho &\rTTo^{\hat{\Delta}} &\prod_{t\in\tr} \, \prod_{(t_y)\in\prod_{y\in\IV(t)}\tr|y|} \, \bigotimes_{y\in\IV(t)} \, \bigotimes_{q\in\IV(t_y)} X_\DD|q|
\\
\dTTo<{\tilde\Delta} &= &\eqref{dia-theta-theta-delta-hat-Delta} &\rdTTo(4,2)_{\hat{\Delta}} &&&\uTTo>{\prod\prod\tens\tens f|q|}
\\
\bar C_\DD\botto\bottho &\rTTo^{\vartheta\bottho} &C_\DD\bottho\bottho &&\rTTo^{\prod\zeta} &&\prod_{t\in\tr} \, \prod_{(t_y)\in\prod_{y\in\IV(t)}\tr|y|} \, \bigotimes_{y\in\IV(t)} \, \bigotimes_{q\in\IV(t_y)} X_\DD|q|
\end{diagram}
which obviously commutes.
Thus, $Q(f)$ and $\phi(f)$ are morphisms of cooperads.

Whenever \(f:C_\DD\to X_\DD\in\gr_\DD^\NN\) satisfies \(\inj_\DD\cdot f\equiv0\mod p\), it has the form \(f=f'\tens\DD+f''p\), where \(f'=\bigl(C\rTTo^\pr \bar C\rto{\bar f} X\bigr)\) for some \(\bar f\in\gr^\NN\).
Constructing a cooperad morphism $\tilde g$ from \(f'\tens\DD\) as in \eqref{eq-gprt} we see that it has the form \((g'\tens\DD)J\), where \(g'=\bigl(\bar C\rto\delta \bar C\bott\rTTo^{\bar f\bott} X\bott\bigr)\in\bott\coalg\).
In fact, the constructed cooperad morphism $\tilde g$ as well as \(\tilde g=(g'\tens\DD)J\) satisfies the equation \(\tilde g\cdot\eps=f'\tens\DD\), therefore, the two morphisms coincide by injectivity of $\psi$.
\end{proof}

\begin{corollary}\label{cor-coderivations-are-in-bijection-with-homogeneous-maps}
For any $\bott$\n-coalgebra $\bar C$ and any $\bott$\n-coalgebra morphism \(h:\bar C\to X\bott\) homogeneous $hJ$\n-coderivations $\xi:C\to X\botto$ are in bijection with homogeneous maps \(\check\xi=\bigl(\bar C\rto\xi X\botto \rTTo^\eps X\bigr)\).
They can be recovered as \(\xi=\bigl(C\rto{\delta J} \bar C\botto\rTTo^{q(\check\xi)} X\botto\bigr)\), where
\begin{multline*}
\bigl( \bar C\botto \rTTo^{q(\check\xi)} X\botto \rMono^{\iota_\circ} X\bottho \bigr)
\\
=\Bigl[ \coprod_{\tau\in\tr} \, \bigotimes_{v\in\IV(\tau)} \bar C|v| \rto\vartheta \prod_{t\in\tr} \, \bigotimes_{y\in\IV(t)} C|y| \rTTo^{\prod_t\sum_{i\in\IV(t)}\tens^{y\in\IV(t)}f_i(y)} \prod_{t\in\tr} \, \bigotimes_{y\in\IV(t)} X|y| \Bigr],
\end{multline*}
where \(f_i(y)=\bigl(C|y|\rto\pr \bar C|y|\rto h (X\bott)|y|\rto\eps X|y|\bigr)\) if $y\ne i$ and \(f_i(y)=\check\xi|y|\) if $y=i$.
\end{corollary}

In particular, for \(\bar C=Y\bott\) and a $\bott$\n-coalgebra morphism \(h=\bigl(Y\bott\rto\Delta Y\bott\bott\rTTo^{\check h\bott} X\bott\bigr)\) homogeneous $hJ$\n-coderivations $\xi:Y\botto\to X\botto$ are in bijection with homogeneous maps \(\check\xi=\bigl(Y\botto\rto\xi X\botto \rTTo^\eps X\bigr)\).
They can be recovered via
\begin{multline*}
\bigl( Y\botto \rto\xi X\botto \rMono^{\iota_\circ} X\bottho \bigr)
\\
=\Bigl[ \coprod_{\tau\in\tr} \, \bigotimes_{v\in\IV(\tau)} Y|v| \rto{\tilde\Delta} \prod_{t\in\tr} \, \bigotimes_{y\in\IV(t)} \, \coprod_{t_y\in\tr|y|} \, \bigotimes_{q\in\IV(t_y)} Y|q| \rTTo^{\prod_t\sum_{i\in\IV(t)}\tens^{y\in\IV(t)}f_i(y)} \prod_{t\in\tr} \, \bigotimes_{y\in\IV(t)} X|y| \Bigr],
\end{multline*}
where \(f_i(y)=\bigl(Y\botto|y|\rto\pr Y\bott|y|\rto{\check{h}} X|y|\bigr)\) if $y\ne i$ and \(f_i(y)=\check\xi|y|\) if $y=i$.

Taking into account \eqref{dia-tilde-Delta-hat-Delta} we get

\begin{corollary}\label{cor-coderivations-bottom}
Homogeneous $\id_{X\botto}$-coderivations $\xi:X\botto\to X\botto$ are in bijection with homogeneous maps \(\check\xi=\xi\cdot\eps:X\botto\to X\) and are recovered as
\begin{equation*}
\xi =\Bigl[ \coprod_{\tau\in\tr} \, \bigotimes_{v\in\IV(\tau)} X|v| \rto{\bar\Delta} \coprod_{t\in\tr} \, \coprod_{i\in\IV(t)} \, \bigotimes_{y\in\IV(t)} \, \coprod_{t_i^i\in\tr|i|} \, \bigotimes_{q\in\IV(t_y^i)} X|q| \rTTo^{\coprod_t\sum_{i\in\IV(t)}\tens^{y\in\IV(t)}f_i(y)} \coprod_{t\in\tr} \, \bigotimes_{y\in\IV(t)} X|y| \Bigr],
\end{equation*}
where \(t_y^i=\tau|y|\) and \(f_i(y)=\id:\tens_{q\in\IV(\tau|y|)}X|q|\to X|y|\) if $y\ne i$, and \(f_i(i)=\check\xi|i|:\coprod_{t_i^i\in\tr|i|} \, \tens_{q\in\IV(t_i^i)}X|q|\to X|y|\) if $y=i$, the component of $\bar\Delta$ indexed by \(t\in\tr\), \(i\in\IV(t)\), \((t_i^i)\in\prod_{i\in\IV(t)}\tr|i|\) is
\[ \coprod_{\tau\in\tr} \, \bigotimes_{v\in\IV(\tau)} X|v| \rTTo^{\pr_{I_t(t_y^i\mid y\in\IV(t))}} \bigotimes_{v\in\IV(I_t(t_y^i\mid y\in\IV(t)))} X|v| \rto\sim \bigotimes_{y\in\IV(t)} \, \bigotimes_{q\in\IV(t_y^i)} X|q|.
\]
\end{corollary}

Thus, the image of $\bar\Delta$ extends over decompositions \(\tau=I_t(t_y^i\mid y\in\IV(t))\) such that \(t_y^i=\tau|y|\) for \(y\ne i\in\IV(t)\).
Of course, the number of such presentations of $\tau$ is finite.
The coderivation $\xi$ can be split in two sums.
The first is the sum $\xi'$ over all non-empty subsets \(\IV(r)\subset\IV(\tau)\) which form a connected subgraph of $\tau$.
Then \(t=\tau/r\) is the tree $\tau$ with all vertices of $\IV(r)$ contracted to one, denoted $i$.
By definition, the subtree \(r=t_i^i\subset\tau\) has the set of internal vertices $\IV(r)$ and the set of inputs \(\Inp(r)=P_\tau^{-1}(\IV(r))\setminus\IV(r)=P_t^{-1}(i)\).
The second is the sum $\xi''$ over trees $t$ differing from $\tau$ by adding one unary vertex $i$ with the associated tree \(t_i^i=\circ\).
Thus,
\begin{multline}
\xi' =\Bigl[ \coprod_{\tau\in\tr} \, \bigotimes_{v\in\IV(\tau)} X|v| \rto{\Delta'} \coprod_{\tau\in\tr} \, \coprod_{\text{subtree }r\subset\tau}^{\IV(r)\ne\emptyset} \, \bigotimes_{y\in\IV(\tau/r)} \, \bigotimes_{q\in(\IV(\tau)\to\IV(\tau/r))^{-1}(y)}X|q|
\\
\rTTo^{\sum\tens^{y\in\IV(t)}f_\tau(y)} \coprod_{t\in\tr} \, \bigotimes_{y\in\IV(t)} X|y| \Bigr],
\label{eq-xi'-Delta'}
\end{multline}
where \(t=\tau/r\), \(f_\tau(y)=\id\) if \(y\ne[\IV(r)]\), and \(f_\tau(y)=\check\xi|y|\) if \(y=[\IV(r)]\),
\begin{equation}
\xi'' =\Bigl[ \coprod_{\tau\in\tr} \, \bigotimes_{v\in\IV(\tau)} X|v| \rto{\Delta''} \coprod_{t\in\tr} \, \coprod_{i\in\IV(t)} \, \bigotimes_{y\in\IV(t)} \, \bigotimes_{q\in\IV(t_y^i)}X|q|\rTTo^{\coprod_t\sum_i\tens^{y\in\IV(t)}g_i(y)} \coprod_{t\in\tr} \, \bigotimes_{y\in\IV(t)} X|y| \Bigr],
\label{eq-xi''-Delta''}
\end{equation}
where \(t_y^i=\tau[|y|]\) and \(g_i(y)=\id\) if $y\ne i$, and \(t_i^i=\circ\), \(g_i(i)=\check\xi_\circ(1):\1\to X(1)\), $\Delta''$ takes a summand indexed by $\tau$ to each summand indexed by \((t,i)\) such that \(t^{\{i\}}=\tau\).

\section{Infinitesimal comodule deformations}
Let $C$ be a (non-counital) cooperad and let \(\delta:N\to N\bar\odot C\) be a (non-counital) $C$\n-comodule.
An \emph{infinitesimal $C$\n-comodule deformation} of\index{TTindex}{infinitesimal comodule deformation} $N$ is a $\tilde{C}=C\tens\DD$-comodule structure \(\tilde\delta:\tilde{N}=N\tens\DD\to\tilde{N}\bar\odot_\DD\tilde{C}\) lifting $\delta$, that is,
\[ \tilde\delta\tens_\DD\1 =\delta: \tilde{N}\tens_\DD\1 =N \to (\tilde{N}\bar\odot_\DD\tilde{C})\tens_\DD\1 =N\bar\odot C.
\]
The deformed comodule structure $\tilde\delta$ splits up in two components
\[ \tilde\delta =\delta+rp: N\to N\bar\odot C \oplus N\bar\odot C\tens\1 p =\tilde{N}\bar\odot_\DD\tilde{C}.
\]
Thus, an infinitesimal $C$\n-comodule deformation is given by a family of degree $a$ maps \(r_{n_1,\dots,n_k}:N(n_1+\dots+n_k)\to N(n_1)\tdt N(n_k)\tens C(k)\), \(k,n_1,\dots,n_k\in\NN\), such that for all \(w,x,y\in\NN\) the equation holds
\begin{multline*}
\bigl[ N(n_1+\dots+n_{w+x+y}) \rTTo^{r_{n_1,\dots,n_w,n_{w+1}+\dots+n_{w+x},n_{w+x+1},\dots,n_{w+x+y}}}
\\
N(n_1)\tdt N(n_w)\tens N(n_{w+1}+\dots+n_{w+x})\tens N(n_{w+x+1})\tdt N(n_{w+x+y})\tens C(w+1+y)
\\
\rTTo^{1^{\tens w}\tens\delta_{n_{w+1},\dots,n_{w+x}}\tens1^{\tens(y+1)}}
\\
\hskip\multlinegap N(n_1)\tdt N(n_w)\tens N(n_{w+1})\tdt N(n_{w+x})\tens C(x) \hfill
\\
\hfill \tens N(n_{w+x+1})\tdt N(n_{w+x+y})\tens C(w+1+y)\quad
\\
\hfill \rto\sim N(n_1)\tdt N(n_{w+x+y})\tens C(x)\tens C(w+1+y) \bigr] \quad
\\
\hskip\multlinegap +\bigl[ N(n_1+\dots+n_{w+x+y}) \rTTo^{\delta_{n_1,\dots,n_w,n_{w+1}+\dots+n_{w+x},n_{w+x+1},\dots,n_{w+x+y}}} \hfill
\\
N(n_1)\tdt N(n_w)\tens N(n_{w+1}+\dots+n_{w+x})\tens N(n_{w+x+1})\tdt N(n_{w+x+y})\tens C(w+1+y)
\\
\rTTo^{1^{\tens w}\tens r_{n_{w+1},\dots,n_{w+x}}\tens1^{\tens(y+1)}}
\\
\hskip\multlinegap N(n_1)\tdt N(n_w)\tens N(n_{w+1})\tdt N(n_{w+x})\tens C(x) \hfill
\\
\hfill \tens N(n_{w+x+1})\tdt N(n_{w+x+y})\tens C(w+1+y)\quad
\\
\hfill \rto\sim N(n_1)\tdt N(n_{w+x+y})\tens C(x)\tens C(w+1+y) \bigr] \quad
\\
\hskip\multlinegap =\bigl[ N(n_1+\dots+n_{w+x+y}) \rTTo^{r_{n_1,\dots,n_{w+x+y}}} N(n_1)\tdt N(n_{w+x+y})\tens C(w+x+y) \hfill
\\
\rTTo^{1^{\tens(w+x+y)}\tens\Delta_{T(w,x,y)}} N(n_1)\tdt N(n_{w+x+y})\tens C(x)\tens C(w+1+y) \bigr].
\end{multline*}
see \eqref{eq-NNNNNNNNNNNNNNNNNNNN}.

\begin{exercise}\label{exe-infinitesimal-C-coalgebra-deformation}
In the above assumptions let \(d_N:N\to N\) be a map of degree $a$ and let \(d_C:C\to C\) be a coderivation of degree $a$.
Then
\begin{multline*}
r_{n_1,\dots,n_k} =\bigl[ N(n_1+\dots+n_k) \rTTo^{\delta_{n_1,\dots,n_k}} N(n_1)\tdt N(n_k)\tens C(k)
\\
\hfill \rTTo^{\sum_{i=1}^k1^{\tens(i-1)}\tens d_N\tens1^{\tens(k-i)}\tens1+1^{\tens k}\tens d_C} N(n_1)\tdt N(n_k)\tens C(k) \bigr] \hskip\multlinegap
\\
-\bigl[ N(n_1+\dots+n_k) \rto{d_N} N(n_1+\dots+n_k) \rTTo^{\delta_{n_1,\dots,n_k}} N(n_1)\tdt N(n_k)\tens C(k) \bigr]
\end{multline*}
is an infinitesimal $C$\n-comodule deformation.
When $(C,d_C)$ is a $\dg$\n-cooperad, then vanishing of $r$ tells precisely that $d_N$ is a \((\id_N;\id_C,d_C)\)-coderivation.
Actually, $r$ is a general form of a \emph{trivial} infinitesimal $C$\n-comodule deformation.
\textit{Hint:} definition of $r$ can be written as follows
\[ \bigl( N_\DD \rTTo^{1+d_Np} \tilde{N} \rTTo^{\delta+rp} \tilde{N}\bar\odot_\DD\tilde{C} \bigr) =\bigl( N_\DD \rto\delta N_\DD\bar\odot_\DD C_\DD \rTTo^{(1+d_Np)\bar\odot_\DD(1+d_Cp)} \tilde{N}\bar\odot_\DD\tilde{C} \bigr),
\]
where \(N_\DD=N\tens\DD\) is the non-deformed comodule over the non-deformed cooperad \(C_\DD=C\tens\DD\).
\end{exercise}

\begin{example}\label{exa-infinitesimal-deformations}
Let \(C=Y\botto\) for \(Y\in\Ob\cw^\NN_{++}\).
Then infinitesimal deformations of a $C$\n-comodule $N$ are in bijection with degree 0 maps \(\tilde b:\tilde{N}=N\tens\DD\to\tilde{N}\bar\odot_\DD\tilde{Y}\) by \propref{pro-TX-coalgebra-structures}, where \(\tilde{Y}=Y\tens\DD\).
This map splits in two components
\[ \tilde b =b+gp: N\to N\bar\odot Y \oplus (N\bar\odot Y)\tens\1 p =\tilde{N}\bar\odot_\DD\tilde{Y}.
\]
The first, fixed $b$, determines the $C$\n-comodule structure of $N$ and the second, varying \(g:N\to N\bar\odot Y\) of degree $a$, determines the infinitesimal deformation.
In the above assumptions, an infinitesimal deformation is trivial ($r=0$) iff \(\check r\overset{\text{def}}=\bigl(N \rto r N\bar\odot(Y\botto) \rTTo^{1\tens\eps_{\botto}} N\bar\odot Y\bigr)=0\).
\end{example}

\begin{proposition}\label{pro-id-coderivation}
For any non-counital cooperad $(C,\Delta)$ and any homogeneous map $\psi:C\to\1$ of degree $p$ (thus \(\psi\in\und\cv(C(1),\1)^p\)) the degree $p$ map
\[ \xi =\Delta_{T(0,n,0)}(1\tens\psi) -\sum_{e+1+f=n} \Delta_{T(e,1,f)}(\psi\tens1): C(n) \to C(n)
\]
is an $\id_C$\n-coderivation.
\end{proposition}

Such $\xi$ is called an inner coderivation.

\begin{proof}
The required identity
\begin{diagram}
C(a+b+c) &\rTTo^{\Delta_{T(a,b,c)}} &C(b)\tens C(a+1+c)
\\
\dTTo<\xi &= &\dTTo>{\xi\tens1+1\tens\xi}
\\
C(a+b+c) &\rTTo^{\Delta_{T(a,b,c)}} &C(b)\tens C(a+1+c)
\end{diagram}
for arbitrary $a$, $b$, $c\in\NN$ expands to
\begin{align}
&\bigl[ C(a+b+c) \rTTo^{\Delta_{T(0,a+b+c,0)}} C(a+b+c)\tens C(1) \rTTo^{\Delta_{T(a,b,c)}\tens\psi} C(b)\tens C(a+1+c) \bigr]
\label{eq-0a+b+c0-abc-1}
\\
&-\sum_{k+1+j=a+b+c} \bigl[ C(a+b+c) \rTTo^{\Delta_{T(k,1,j)}} C(1)\tens C(a+b+c)  \rTTo^{\psi\tens\Delta_{T(a,b,c)}} \notag
\\
&\hspace{21em} C(b)\tens C(a+1+c) \bigr]
\label{eq-k1j-abc-2}
\\
&=\bigl[ C(a+b+c) \rTTo^{\Delta_{T(a,b,c)}} C(b)\tens C(a+1+c) \rTTo^{\Delta_{T(0,b,0)}\tens1} \notag
\\
&\hspace{6em} C(b)\tens C(1)\tens C(a+1+c) \rTTo^{1\tens\psi\tens1} C(b)\tens C(a+1+c) \bigr]
\label{eq-abc-0b0-1psi1-3}
\\
&-\sum_{e+1+f=b} \bigl[ C(a+b+c) \rTTo^{\Delta_{T(a,b,c)}} C(b)\tens C(a+1+c) \rTTo^{\Delta_{T(e,1,f)}\tens1} \notag
\\
&\hspace{6em} C(1)\tens C(b)\tens C(a+1+c) \rTTo^{\psi\tens1\tens1} C(b)\tens C(a+1+c) \bigr]
\label{eq-abc-e1f-psi11-4}
\\
&+\bigl[ C(a+b+c) \rTTo^{\Delta_{T(a,b,c)}} C(b)\tens C(a+1+c) \rTTo^{1\tens\Delta_{T(0,a+1+c,0)}} \notag
\\
&\hspace{6em} C(b)\tens C(a+1+c)\tens C(1) \rTTo^{1\tens1\tens\psi} C(b)\tens C(a+1+c) \bigr]
\label{eq-abc-0a+1+c0-11psi-5}
\\
&-\sum_{e+f=a+c} \bigl[ C(a+b+c) \rTTo^{\Delta_{T(a,b,c)}} C(b)\tens C(a+1+c) \rTTo^{1\tens\Delta_{T(e,1,f)}} \notag
\\
&\hspace{6em} C(b)\tens C(1)\tens C(a+1+c) \rTTo^{1\tens\psi\tens1} C(b)\tens C(a+1+c) \bigr].
\label{eq-abc-e1f-1psi1-6}
\end{align}
Due to \eqref{dia-cooperad-4-OOOOO} with $v=z=0$, $w=a$, $x=b$, $y=c$ we have \eqref{eq-0a+b+c0-abc-1}=\eqref{eq-abc-0a+1+c0-11psi-5}.
Due to \eqref{dia-cooperad-4-OOOOO} with $w=y=0$, $v=a$, $x=b$, $z=c$ \eqref{eq-abc-0b0-1psi1-3} cancels against the summand of \eqref{eq-abc-e1f-1psi1-6} indexed by $e=a$, $f=c$.
Due to \eqref{dia-cooperad-4-OOOOO} with $v=a$, $w=e$, $x=1$, $y=f$, $z=c$ all summands of \eqref{eq-abc-e1f-psi11-4} cancel against summands of \eqref{eq-k1j-abc-2} for which $k\ge a$ and $j\ge c$, namely, $k=a+e$, $j=f+c$.
The summands of \eqref{eq-k1j-abc-2} with $k<a$ cancel against summands of \eqref{eq-abc-e1f-1psi1-6} with $e<a$ due to \eqref{dia-cooperad-3-OOOOOOO}  with $w=1$, $v=e=k$, $x=a-1-v$, $y=b$, $z=c$.
The summands of \eqref{eq-k1j-abc-2} with $f<c$ cancel against summands of \eqref{eq-abc-e1f-1psi1-6} with $j<c$ due to \eqref{dia-cooperad-3-OOOOOOO}  with $w=b$, $v=a$, $x=k-a-b$, $y=1$, $z=j=f$.
The identity is proven.
\end{proof}

\section{Comonoids as cooperads}
We equip $\cv^\NN$ with the monoidal product $\odot$ as in \secref{cor-staged-operations}.
A comonoid \((C,\Delta:C\to C\odot C,\eps:C\to\1)\) in \((\cv^\NN,\odot)\) is a cooperad in $\cv$.
In fact, components of the comultiplication
\[ \Delta:C(n) \to \bigoplus_{n^1+\dots+n^k=n}^{k\ge0} C(n^1)\tdt C(n^k)\tens C(k) =(C\odot C)(n) \in\cv
\]
satisfy coassociativity equation~\eqref{dia-coassociativity-components} and together with \(\eps:C(1)\to\1\in\cv\) counitality equations \eqref{eq-counitality-C(n)OC(1)}, \eqref{eq-counitality-C(1)nOC(n)}.

Assume that \((\bar C,\bar\Delta)\) is a $\bott$\n-coalgebra.
Then \(C=\1\oplus\bar C\) is an augmented comonoid in \((\cv^\NN,\odot)\).
In fact, for a height 2 staged tree \(t=\bigl(\mb n\rTTo^f \mb k\rTTo^\con \mb1,\mb n\bigr)\in\str(n,2)\), \(n_i=|f^{-1}i|\), \(i\in\mb k\), the value of $\Delta_t$ is presented in \eqref{eq-C(Inpt)-YN(p)} via \(\bar\Delta_{t^N}\) for \(N\subset\uv(t)\).
For each \(c\in\bar C(n)\) there is an integer $M>n$ such that \(\bar\Delta_\tau c=0\) for \(\tau=\bigl(\mb m\rTTo^g \mb l\rTTo^\con \mb1,\Inp\tau\bigr)\in\tr(2)\) if $m>M$ or $l>M$.
Then \(\Delta_tc=0\) if $n>2M$ or $k>M$, hence, the map
\[ (\Delta_{n^1,\dots,n^k})_{n_i\in\NN,\,\sum_{i=1}^kn_i=n}^{k\in\NN}: \bar C(n) \to \prod_{n^1+\dots+n^k=n}^{k\ge0} C(n^1)\tdt C(n^k)\tens C(k) 
\]
factors through \(\oplus_{n^1+\dots+n^k=n}^{k\ge0}C(n^1)\tdt C(n^k)\tens C(k)\), equipping $C$ with the comonoid structure.
Therefore, comonoids in \((\cv^\NN,\odot)\) is an intermediate notion between cooperads and $\bott$\n-coalgebras.

\begin{example}\label{exa-cooperad-partition-compositions-P}
Define the cooperad of partition-compositions $P$ by
\[ P(n)=
\begin{cases}
\1 f_n, \qquad &\text{if } n>0, \\
0, &\text{if } n=0,
\end{cases}
\]
where $\deg f_n=0$.
The comultiplication is given by the sum over partition-\hspace{0pt}compositions of $n$
\[ f_n\Delta =\sum_{i_1+\dots+i_l=n}^{l,i_1,\dots,i_l>0} (f_{i_1}\tens f_{i_2}\tdt f_{i_l})\tens f_l \in
\bigoplus_{i_1+\dots+i_l=n} P(i_1)\tdt P(i_l)\tens P(l).
\]
Our term partition-composition stands for any ordered sequence \((i_1,\dots,i_l)\) of positive integers such that \(i_1+\dots+i_l=n\).
In combinatorics such sequences are called compositions, but we prefer to reserve the term to set theory and category theory use.
The counit \(\eps:P\to\1\) is given by the isomorphism \(\eps:P(1)=\1 f_1\to\1\), \(f_1\mapsto1\).
$P$\n-coalgebras are conilpotent non-counital coalgebras in $\cv$.

Let $\cv=\ZZ\text-\gr\text-\kk\modul$, then $\cw=\NN\times\ZZ\text-\gr\text-\kk\modul$.
The augmented cooperad \(P'=\kk\{f_1,f_2,\dots\}=\kk f_1\oplus\bar P'\) in $\cw^\NN$ is concentrated in degree \((0,0)\).
The non-counital cooperad \(\bar P'\) satisfies \(\bar P'(0)=\bar P'(1)=0\), \(\bar P'(n)=\kk f_n\) for $n\ge2$.
By \propref{pro-Tcoalg+nuCoop+} \(\bar P'\) is a $\bott$\n-coalgebra.
Therefore, \(\bar P=\oplus\bar P'\) is a $\bott$\n-coalgebra as well.
The comultiplication \(\bar\Delta_t(f_n)\) does not vanish iff \(\mb n\simeq\Inp t=L(t)\) and \(\uv(t)=\varnothing\).
In that case \(\bar\Delta_t(f_n)=\tens^{p\in\IV(t)}f_{|p|}\).
\end{example}

\propref{pro-cooperad-P-collection-J} gives a way for obtaining new comonoids in \((\cv^\NN,\odot)\).

\begin{example}\label{exa-cooperad-C-JP}
A cooperad $C$ of graded abelian groups comes from an \(A_\infty^\hu\)-bimodule cooperad \(F_1^\hu\) \cite[Section~2.16]{Lyu-Ainf-Operad}, see further \cite[Proposition~3.17]{Lyu-A8-several-entries}.
Namely, if $n+t>0$, $t\ge0$, then $C(n)^{-2t}$ are free abelian groups spanned by the generators \(f_{n_0;n_1;\dots;n_t}\) of \(F_1^\hu\) such that \(n_i\in\NN\) for \(0\le i\le t\) and \(n=\sum_{i=0}^tn_i\).
If \(n=p=0\) or $p>0$ or $p$ is odd, then $C(n)^p=0$.
Notice, in particular, that \(C(0)^{-2}=\ZZ f_{0;0}\) and the generator $f_{0;0}$ is used in this article in place of $\bv$ distinguished in \cite{Lyu-Ainf-Operad,Lyu-A8-several-entries}.
Notice that $f_0$ does not appear among generators.
Thus, \(C\odot C\) is spanned by products
\begin{equation}
\Bigl( \bigotimes_{i=1}^k f_{n_0^i;n_1^i;\dots;n_{q_i}^i} \Bigr) \tens f_{k_0;k_1;\dots;k_r} \in (C\odot C)(n)^{-2t},
\label{eq-OfOf-COCnt}
\end{equation}
where \(k=\sum_{j=0}^rk_j\), \(n=\sum_{i=1}^k\sum_{j=0}^{q_i}n_j^i\), \(t=r+\sum_{i=1}^kq_i\).
The counit is the isomorphism \(\eps:C(1)^0=\ZZ f_1\to\ZZ\), \(f_1\mapsto1\).
The comultiplication in $C$ coincides with that of \(F_1^\hu\), computed on generators.

In order to describe the comultiplication we represent $C$ as a product of collections.
Define an object $J$ of $\gr^\NN$
\[ J(n)=
\begin{cases}
\ZZ\bj, \qquad &\text{if } n=0, \\
\ZZ, &\text{if } n=1, \\
0, &\text{otherwise,}
\end{cases}
\]
where $\deg\bj=-2$.

The isomorphism of collections \(C\rto\sim J\odot P\) is given by the formula
\begin{equation}
f_{n_0;n_1;\dots;n_t} \mapsto
(1^{\tens n_0}\tens\bj\tens1^{\tens n_1}\tens\bj\tdt1^{\tens n_{t-1}}\tens\bj\tens1^{\tens n_t})\tens f_{n+t},
\quad n=\sum_{i=0}^t n_i.
\label{eq-f-(1j1j1j1)}
\end{equation}
Let us prove that \(J\odot P\) has a structure of a cooperad.
Apply \propref{pro-cooperad-P-collection-J} to our \(J\odot P\) in $\gr$ over $\ZZ$.
Define
\begin{gather*}
\vartheta:J\odot P\to J\odot P\odot J, \quad c\mapsto c\tens1 +c\beta,
\\
c\beta=
\begin{cases}
\bj, \qquad &\text{for } c= \bj\tens f_1, \\
0, &\text{for other elements \eqref{eq-f-(1j1j1j1)}},
\end{cases}
\\
\eps:J\to\1, \quad \bj\mapsto 0, \quad 1\mapsto 1.
\end{gather*}

One can check easily equation~\eqref{dia-JP-JPP-JPJP} on elements $1\tens f_1$ and $\bj\tens f_1$ of the source (use \((\bj\tens f_1)\vartheta=\bj\tens f_1\tens1+\bj\) in the latter case).
Considering diagram~\eqref{dia-JP-JPP-JPJP} on elements $a$ of type~\eqref{eq-f-(1j1j1j1)} with $n+t>1$, we see that the right vertical arrow \(1\odot1\odot\vartheta\) does not include any summand \(1\odot1\odot\beta\).
Otherwise there should be a summand proportional to $\bj\tens f_1$ in the right upper corner, which can come only from \(\bj\tens f_1\tens f_1\in J\odot P\odot P\).
Then the source element (in the left upper corner) would be proportional to $\bj\tens f_1$, which has $n+t=1$.
Therefore, the element $a$ is taken along the diagram as in
\begin{diagram}
a &\rMapsTo^{1\odot\Delta} &b &\rMapsTo^{\vartheta\odot1} &c
\\
\dMapsTo<\vartheta &&= &&\dMapsTo>{1\odot1\odot\vartheta}
\\
a\tens1 &\rMapsTo^{1\odot\Delta\odot1} &b\tens1 &\rMapsTo^{\vartheta\odot1\odot1} &c\tens1
\end{diagram}
Thus, equation~\eqref{dia-JP-JPP-JPJP} holds true.
Equation~\eqref{eq-JP-JPJ-JP-1} is obvious.
Verification of equation~\eqref{eq-JP-JPJ-J-1e} on an element of the source reduces to three cases: $1\tens f_1$,
$\bj\tens f_1$ and \eqref{eq-f-(1j1j1j1)} with $n+t>1$.
Details are left to the reader.
Summing up, $J\odot P\simeq C$ is a cooperad.

In particular,
\[ f_{0;0}\Delta =f_{0;0}\tens f_1 +f_{0;0},
\]
which agrees with the analogous computation for $\bv$ in \cite[Section~2.16]{Lyu-Ainf-Operad}.
The value of \(f_{n_0;n_1;\dots;n_t}\Delta\) is obtained [\textit{loc.~cit.}] from the expression
\[ (1^{\tens n_0}\tens\bj\tens1^{\tens n_1}\tens\bj\tdt1^{\tens n_{t-1}}\tens\bj\tens1^{\tens n_t})
\sum_{i_1+\dots+i_l=n+t}^{l,i_1,\dots,i_l>0} (f_{i_1}\tens f_{i_2}\tdt f_{i_l})\tens f_l
\]
as follows.
Each factor $f_{i_m}$ absorbs $\bj$'s from the subinterval of arguments of $f_{i_m}$ in the first parentheses and becomes \(f_{*;*;\dots;*}\) in accordance with \eqref{eq-f-(1j1j1j1)}.
An exception occurs when $\bj$ is an argument of $f_1$.
The subexpression $\bj f_1$ is replaced not with \(f_{0;0}\) but with \(f_{0;0}+\bj\rho_\varnothing\).
The latter summand means that $\bj$ passes further to the right and becomes one of the arguments of $f_l$.
Several such \(\bj\rho_\varnothing\) turn $f_l$ into \(f_{*;*;\dots;*}\) as in \eqref{eq-f-(1j1j1j1)}.
The resulting expression is \(f_{n_0;n_1;\dots;n_t}\Delta\).
For instance, \(f_{0;1;0}\Delta\) is obtained from the following computation, where $\bv$ is used as a shorthand for $f_{0;0}$,
\begin{multline}
(\bj\tens1\tens\bj)[(f_1\tens f_1\tens f_1)\tens f_3 +(f_2\tens f_1)\tens f_2 +(f_1\tens f_2)\tens f_2 +f_3\tens f_1]
\\
\mapsto [(\bv+\bj\rho_\varnothing)\tens f_1\tens(\bv+\bj\rho_\varnothing)]\tens f_3 +[f_{0;1}\tens(\bv+\bj\rho_\varnothing)]\tens f_2
+[(\bv+\bj\rho_\varnothing)\tens f_{1;0}]\tens f_2 +f_{0;1;0}\tens f_1
\\
\mapsto (\bv\tens f_1\tens\bv)\tens f_3 +(\bv\tens f_1)\tens f_{2;0} +(f_1\tens\bv)\tens f_{0;2} +f_1\tens f_{0;1;0}
+(f_{0;1}\tens\bv)\tens f_2 +f_{0;1}\tens f_{1;0}
\\
+(\bv\tens f_{1;0})\tens f_2 +f_{1;0}\tens f_{0;1} +f_{0;1;0}\tens f_1 =f_{0;1;0}\Delta.
\label{eq-f010Delta-f-computation}
\end{multline}

Let $\cv=\ZZ\text-\gr\text-\kk\modul=\gr$, $\cw=\NN\times\ZZ\text-\gr\text-\kk\modul$.
Consider the lifting of $P$, $J$, $C$ to an augmented cooperad \(C'=J'\odot P'\) in $\cw$, where \(\deg_{\NN\times\ZZ}(f_n\in P'(n))=(0,0)\), \(\deg_{\NN\times\ZZ}(\bj\in J'(0))=(2,0)\), \(\deg_{\NN\times\ZZ}(1\in J'(1))=(0,0)\).
The original degree in $C$ is obtained as \(\deg=\deg_\ZZ-\deg_\NN\).
Then \(C'=\kk\{1\tens f_1\}\oplus\bar C'\).
The non-counital cooperad \(\bar C'=(\bar C'(n)_k^m)_{n,k\in\NN}^{m\in\ZZ}\) is concentrated in \(\bar C'(n)_{2t}^0\) spanned by \(f_{n_0;n_1;\dots;n_t}\) with \(\sum_{i=0}^tn_i=n\).
The special cases are \(\bar C'(0)_0^0=\bar C'(1)_0^0=0\).
By \propref{pro-Tcoalg+nuCoop+} \(\bar C'\) is a $\bott$\n-coalgebra in $\cw^\NN$.
Therefore, \(\bar C=\oplus\bar C'\) is a $\bott$\n-coalgebra in $\cv^\NN$.
\end{example}

Recall that a staged tree of height $r$ is a finite sequence
\[ X_0 \rTTo^{f_1} X_1 \rTTo^{f_2} \dots \rTTo^{f_{r-1}} X_{r-1} \rTTo^{f_r} \mb1
\]
of composable morphisms in the category $\co$ of finite ordinals and non-decreasing maps.

\begin{example}
The same $C$ admits indexing of generators by trees of height $2$
\[ \hstretch130
g_{n_0;n_1;\dots;n_t} =\;
\begin{tangles}{rlr}
\node &&\noder2 \\
\sw4\Step\node\sw1 &\object{\dots}\step\se2\step\object{\dots}\se4 &\noder1 \\
\sw1\s\noded{\underbrace{\Step}_{n_0}}\se1\step\sw1\noded{\underbrace{\Step}_{n_1}}\se1 
&\step[3]\sw1\s\noded{\underbrace{\Step}_{n_t}}\se1 &\noder0
\end{tangles}
\quad, \qquad \textit{e.g.} \quad
g_{0;0} =\;
\begin{tangles}{cr}
\node &\noder2 \\
\ne1\nw1 &\noder1 \\
\sw0 &\noder0
\end{tangles}
\quad.
\]
\vspace{0.5em}

\noindent In this isomorphic form we denote it $C_1$.
The isomorphism $C\to C_1$ is given by \(f_{n_0;n_1;\dots;n_t}\mapsto g_{n_0;n_1;\dots;n_t}\).
Then $C_1\odot C_1$ (resp. $C_1^{\odot p}$) has indexing of generators over $\ZZ$ by admissible (defined below) trees of height $4$ (resp. of height $2p$).
For instance, due to \eqref{eq-f010Delta-f-computation} \(g_{0;1;0}\Delta\) equals
\begin{equation}
\begin{tanglec}
\sw1\s\se1 \\
\s
\end{tanglec}
\;\Delta =\;
\begin{tangles}{rcl}
&\s \\
&\sw1\s\se1 \\
\sw1\s &\s &\s\se1 \\
&\s
\end{tangles}
\;+\;
\begin{tangles}{rrl}
&\s &\se1 \\
&\sw1\s \\
\sw1\s &\s \\
&\s
\end{tangles}
\;+\;
\begin{tangles}{rll}
\sw1 &\s \\
&\s\se1 \\
&\s &\s\se1 \\
&\s
\end{tangles}
\;+\;
\begin{tanglec}
\sw1\s\se1 \\
\s \\
\s \\
\s
\end{tanglec}
\;+\;
\begin{tangles}{rll}
&\s \\
&\s\se1 \\
\sw1 &\s &\s\se1 \\
&\s
\end{tangles}
\;+\;
\begin{tangles}{rl}
&\s\se1 \\
\s \\
\sw1\s \\
\s
\end{tangles}
\;+\;
\begin{tangles}{rrl}
&\s \\
&\sw1\s \\
\sw1\s &\s &\se1 \\
&\s
\end{tangles}
\;+\;
\begin{tangles}{rl}
\sw1\s \\
\s \\
&\s\se1 \\
\s
\end{tangles}
\;+\;
\begin{tangles}{cr}
&\noder4 \\
\s &\noder3 \\
\s &\noder2 \\
\sw1\s\se1 &\noder1 \\
\s &\noder0
\end{tangles}
\quad.
\label{eq-tan-elms}
\end{equation}
Here the tree between the second and the fourth floor is \(g_{k_0;k_1;\dots;k_r}\) corresponding to the last factor of \eqref{eq-OfOf-COCnt}.
The forest between the ground floor and the second floor is
\[ g_{n_0^1;n_1^1;\dots;n_{q_1}^1} \sqcup\dots\sqcup g_{n_0^k;n_1^k;\dots;n_{q_k}^k},
\]
which corresponds to all but the last factors in \eqref{eq-OfOf-COCnt}.
Roots of these trees (on the second floor) coincide with lowest leaves of the upper tree (also on the second floor).

A staged tree of height $r$ is \emph{admissible} if\index{TTindex}{admissible staged tree}
\begin{enumerate}
\renewcommand{\labelenumi}{(\roman{enumi})}
\item $r$ is even, $r=2p$, so the tree is a sequence
\[ X_0 \rTTo^{f_1} X_1 \rTTo^{f_2} \dots \rTTo^{f_{2p-1}} X_{2p-1} \rTTo^{f_{2p}} \mb1;
\]

\item all $f_2$, $f_4$, \dots, $f_{2p}$ are surjections;

\item if \(x\in X_{2m-1}-\im f_{2m-1}\), then there is \(y\in X_{2m-1}\) such that $y\ne x$ and \(f_{2m}(y)=f_{2m}(x)\).
\end{enumerate}
An admissible tree belongs to $(C_1^{\odot p})(n)^{-2q}$ iff its root is on the $2p$\n-th floor, the number of leaves on the ground floor is $n$, the number of downward angles whose vertices are on even floors is $q$.
For instance, $n=1$ and $q=2$ for all trees from \eqref{eq-tan-elms}.
\end{example}

\begin{example}\label{exa-cooperad-C2-commas}
The comonad $\ct$ comes also from the following cooperad $C_2$ over $R=\ZZ$, isomorphic to $C\simeq C_1$.
We consider sequences
\begin{equation}
(;c_1;c_2;\dots;c_{k-1};)
\label{eq-c1-c2-ck-1}
\end{equation}
where each semicolon ; is replaced either by comma , or by $|$.
Integers $c_0$ and $c_k$ are assumed to be very large odd numbers.
A sequence is \emph{admissible} if\index{TTindex}{admissible sequence}
\begin{enumerate}
\renewcommand{\labelenumi}{(\alph{enumi})}
\item at least one separating symbol ( , or $|$ ) is present;

\item for \(1\le i\le k\) if \(c_{i-1}|c_i\), then \(\min\{c_{i-1},c_i\}\) is even.
\end{enumerate}
The sequence $(,)$ is allowed, while $(|)$ is not.

The free abelian group $C_2(n)^{-2q}$ is spanned by admissible sequences \((;c_1;c_2;\dots;c_{k-1};)\), \(c_i\in\{1,2\}\), such that the number of commas is $n$ and the number of $2$'s is $q$.
Let us show that \((C_2\odot C_2)(n)^{-2q}\) (resp. \((C_2^{\odot p})(n)^{-2q}\)) is isomorphic to a free abelian group spanned by admissible sequences \((;c_1;c_2;\dots;c_{k-1};)\), \(c_i\in\{1,2,3,4\}\) (resp. \(1\le c_i\le2p\)), such that the number of commas is $n$ and the number of even integers is $q$.
The map \(C_1^{\odot p}\to C_2^{\odot p}\) sends a tree to a sequence of leaves written either as $|$ if the leaf is above the ground floor or as , if the leaf is on the ground floor.
Between two neighbouring leaves it is written the maximal height on the shortest path connecting the two leaves.
For instance, \eqref{eq-tan-elms} becomes
\begin{multline*}
(|2,2|)\Delta
\\
=(|2|3,3|2|) +(|2|3,4|) +(|4,3|2|) +(|4,4|) +(|2,3|2|) +(|2,4|) +(|2|3,2|) +(|4,2|) +(|2,2|).
\end{multline*}
Admissibility of the tree implies that for $i$\n-th leaf \(x\in X_{2m-1}-\im f_{2m-1}\) one of the numbers $c_{i-1}$, $c_i$ is $2m$ and the other is not smaller.
Hence, condition $(b)$ is satisfied for the obtained sequence \eqref{eq-c1-c2-ck-1}.

The map in the other direction \(C_2^{\odot p}\to C_1^{\odot p}\) is given by the following recipe.
Given an admissible sequence mark a point on the ground floor for each , in the sequence and a point on the \((2m-1)\)\n-th floor for each occurrence of \(c_{i-1}|c_i\) if \(\min\{c_{i-1},c_i\}=2m\), \(1\le i\le k\).
Connect each marked point with the root.
Identify upper parts of paths from neighbouring marked points (leaves) so that they merge at the prescribed height $c_i$.
The obtained tree is admissible.
In fact, it has no leaves on positive even floors.
Any leaf of the tree is in a marked point \(x\in X_{2m-1}\).
Then \(f_{2m}(x)\) is a vertex of a downward angle, hence, the last admissibility condition is fulfilled for the tree $t$.
Clearly, the constructed maps \(C_1^{\odot p}\leftrightarrow C_2^{\odot p}\) are bijections, inverse to each other.
\end{example}

\begin{proposition}[S.~Slobodianiuk]
Comultiplication for $C$ takes for $C_2$ the form of \eqref{eq-(c1c2)-(c1c2)}.
\end{proposition}

\section{Comonads}
Associate with a collection \(C\in\Ob\cv^\NN\) an endofunctor
\[ \bot:\cv\to\cv, \qquad V\mapsto \bigoplus_{n\in\NN} V[1]^{\tens n} \tens C(n)[-1] =V\bot.
\]
For any other collection $C'$ and the associated functor $\bot'$ we have
\begin{align*}
V{\bot\bot}' &=\bigoplus_{k\in\NN} (V\bot[1])^{\tens k} \tens C'(k)[-1]
\\
&\simeq \bigoplus_{k\in\NN} \bigoplus_{n_1,\dots,n_k\in\NN} \Bigl\langle
\tens^{i\in\mb k} \bigl( V[1]^{\tens n_i}\tens C(n_i)\bigr)\Bigr\rangle \tens C'(k)[-1]
\\
&\simeq \bigoplus_{k\in\NN} \bigoplus_{n_1,\dots,n_k\in\NN} V[1]^{\tens\sum_{i=1}^kn_i} \tens
\biggl\langle \Bigl( \bigotimes_{i=1}^k{\vphantom{\Big|}}_R C(n_i) \Bigr) \tens  C'(k) \biggr\rangle[-1]
\\
&\simeq \bigoplus_{n\in\NN} V[1]^{\tens n} \tens (C\odot C')(n)[-1].
\end{align*}
This gives a monoidal functor \(\Omega:(\cv^\NN,\odot)\to(\und\End\cv,\cdot)\), \(C\mapsto C\Omega=\bot\).
A comonoid $C$ in \((\cv^\NN,\odot)\) (which is automatically a cooperad in $\cv$) is sent by this functor to a comonoid $\bot$ in \((\und\End\cv,\cdot)\) (a comonad in $\cv$).
The comultiplication in the comonad\index{TTsymb}{bot@$\bot$} $\bot$ is defined as
\[ \Delta =\Bigl( V\bot =\bigoplus_{n\in\NN} V[1]^{\tens n} \tens  C(n)[-1] \rTTo^{\oplus1\tens\Delta[-1]} \bigoplus_{n\in\NN} V[1]^{\tens n} \tens (C\odot C)(n)[-1] \simeq V{\bot\bot} \Bigr).
\]
The counit is
\[ \eps =\bigl( V\bot \rTTo^{\pr_{1,0}} V\tens C(1)^0 \rTTo^{1\tens\eps} V \bigr).
\]

For instance, in \exaref{exa-cooperad-partition-compositions-P} \(V\bot=(V)(P\Omega)=V[1]T^>[-1]:\cv\to\cv\) is the comonad governing \ainf-algebras, see \cite{BesLyuMan-book}.

The category of coalgebras over the comonad $\bot$ is fully embedded into the category of coalgebras over the cooperad $C$ via the shift map \((X,\delta:X\to X\bot)\mapsto\bigl(X[1],X[1]\rTTo^{\delta[1]} \coprod_{n\in\NN}X[1]^{\tens n}\tens C(n)\hookrightarrow\prod_{n\in\NN}X[1]^{\tens n}\tens C(n)\bigr)\).

\begin{example}\label{exa-comonad-T-[12-3]}
Now we describe the common comonad $\bot$ obtained from the isomorphic cooperads of Examples \ref{exa-cooperad-C-JP}--\ref{exa-cooperad-C2-commas}.
This is the functor $\ct=\ct_\kk:\cv\to\cv$, given by the expression
\begin{align*}
V\ct &=V[1]T^\ge[2]T^>[-3]\ominus V[1]T^0[2]T^1[-3]
\\
&=\bigoplus^{k>0}_{\substack{(k>1)\vee(n_1>0)\\n_1.\dots,n_k\ge0}}
V[1]\Tu1^{n_1}[2]\tensu2\cdots\tensu2V[1]\Tu1^{n_{k-1}}[2]\tensu2V[1]\Tu1^{n_k}[-1].
\end{align*}
The summand \(V[1]T^0[2]T^1[-3]\simeq\kk[-1]\) corresponding to $k=1$, $n_1=0$ is excluded from the sum.
A summand of $V\ct$ has the form
\begin{equation}
X_1\tensu{c_1}\cdots\tensu{c_{k-2}}X_{k-1}\tensu{c_{k-1}}X_k,
\label{eq-X1OX2OXk}
\end{equation}
where \(c_i\in\{1,2\}\) for \(1\le i<k\), \(c_0=c_k=7\), and
\begin{enumerate}
\renewcommand{\labelenumi}{(\alph{enumi})}
\item $k\ge1$;

\item for \(1\le i\le k\) if \(X_i=\kk[b]\), then \(\min\{c_{i-1},c_i\}\) is even;

\item for \(1\le i<k\) either \(X_i=V[2+(-1)^{c_i}]\), or \(X_i=\kk[1+(-1)^{c_i}]\);

\item either \(X_k=V\), or \(X_k=\kk[-1]\).
\end{enumerate}
These conditions can be simplified.
For instance, factor \(\kk[b]\) is separated from other factors only by $\tens$ indexed by 2.
The conditions are put in this form in order to describe similarly $\ct^2$ and $\ct^3$.
Notice that condition (b) together with \(c_0=c_k=7\) implies that the expression \(\kk[-1]\), $k=1$, is not a summand of $V\ct$.

The graded $\kk$\n-module \(V\ct^2\subset V[1]\Tu1^\ge[2]\Tu2^>[-2]\Tu3^\ge[2]\Tu4^>[-3]\) also consists of
summands~\eqref{eq-X1OX2OXk} that satisfy conditions (a)--(d), however, with \(c_i\in\{1,2,3,4\}\) for \(1\le i<k\).
Furthermore,
\[ V\ct^3 \subset V[1]\Tu1^\ge[2]\Tu2^>[-2]\Tu3^\ge[2]\Tu4^>[-2]\Tu5^\ge[2]\Tu6^>[-3]
\]
as well consists of summands~\eqref{eq-X1OX2OXk} that satisfy conditions (a)--(d), although \(1\le c_i\le6\) for \(1\le i<k\).
The functor $\ct^p$ is presented similarly, except that \(1\le c_i\le2p\) for \(1\le i<k\) and \(c_0,c_k\ge2p+1\) are any odd numbers.

Summand~\eqref{eq-X1OX2OXk} is encoded also by sequence \eqref{eq-c1-c2-ck-1}, where semicolon ; on $i$\n-th place is replaced either by comma , iff \(X_i=V[*]\), or by $|$ iff \(X_i=\kk[*]\).
Condition (a) says that at least one separating symbol ( , or $|$ ) is present.
Condition (b) requires for \(1\le i\le k\) and \(c_{i-1}|c_i\) that \(\min\{c_{i-1},c_i\}\) be even.
Condition (c) and (d) prescribe how to recover \eqref{eq-X1OX2OXk} from \eqref{eq-c1-c2-ck-1}.

Let us equip $\ct$ with a structure of a comonad.
The natural transformation \(\Delta:V\ct\to V\ct^2\) restricted to summand~\eqref{eq-c1-c2-ck-1} with \(c_i\in\{1,2\}\) is the sum of all obvious isomorphisms
\begin{equation}
(;c_1;c_2;\dots;c_{k-1};) \to (;\overline{c_1};\overline{c_2};\dots;\overline{c_{k-1}};)
\label{eq-(c1c2)-(c1c2)}
\end{equation}
where
\begin{enumerate}
\renewcommand{\labelenumi}{(\roman{enumi})}
\item the value of all semicolons ; is preserved;

\item if \(c_i=1\), then \(\overline{c_i}=1\) or $3$;

\item if \(c_i=2\), then \(\overline{c_i}=2\) or $2|3$ or $3|2$ or $3|2|3$ or $4$.
\end{enumerate}
Of course, both source and target must be allowed by condition (b).
In other terms, \(\Delta:V\ct\to V\ct^2\) restricted to summand~\eqref{eq-X1OX2OXk} with \(c_i\in\{1,2\}\) is the sum of obvious isomorphisms with all summands in which
\begin{enumerate}
\renewcommand{\labelenumi}{(\roman{enumi})}
\item all original $V$'s and $\kk$'s keep their places;

\item $\tensu1$ is replaced with $\tensu1$ or $\tensu3$;

\item $\tensu2$ is replaced with $\tensu2$ or $\tensu2\kk\tensu3$ or $[-2]\tensu3\kk[2]\tensu2$ or $[-2]\tensu3\kk[2]\tensu2\kk\tensu3$ or $\tensu4$.
\end{enumerate}
This comultiplication is counital, the counit \(\eps:V\ct\to V\) is the projection \(\pr^V_\varnothing\) of \(V\ct\) on the direct summand \(V[1]T^1[2]T^1[-3]=V\).
\end{example}

\begin{example}
Let $C=P$, then \(\bot=P\Omega=[1]T^>[-1]\).
Any coderivation of degree~1 of the cofree $\bot$\n-coalgebra \(B=V[1]T^>[-1]\) is determined by a map \(\check{b}:V[1]T^>=B[1]\to V[1]\), that is, by a family of morphisms \(b_n:V[1]^{\tens n}\to V[1]\), $n\ge1$ of degree~1.
Recovering a coderivation \(b:V\bot[1]\to V\bot[1]\) we write its restriction to \(V[1]^{\tens n}\) due to \eqref{eq-general-coderivation-rM-NC} as
\[ b\big| =\sum_{x+y+z=n}^{x,z\ge0,\,y>0} 1^{\tens x}\tens b_y\tens1^{\tens z}: V[1]^{\tens n} \to V[1]T^>.
\]
The square \(b^2:V[1]T^>\to V[1]T^>\) is a coderivation of degree 2.
Vanishing of $b^2$ is equivalent to the equation \(b^2\eps_{T^>}=0:V[1]T^>\to V[1]\), that is, to
\[  \sum_{x+y+z=n}^{x,z\ge0,\,y>0} (1^{\tens x}\tens b_y\tens1^{\tens z})\cdot b_{x+1+z} =0,
\]
which is the defining equation for operations in an \ainf-algebra.
We recover the well-known fact that an \ainf-algebra can be defined as a cofree $\bot$\n-coalgebra $V\bot$ equipped with a degree 1 coderivation \(b:V\bot[1]\to V\bot[1]\) whose square is zero.
\end{example}

Let $C$ be a connected cooperad.
Special case of \eqref{eq-general-coderivation-rM-NC} are \((\id_{N\odot C};\id_C,0)\)-coderivations $\nabla:N\odot C\to N\odot C$ for $\mm=0$, \(N\in\cw_+=\cw^0_+\).
They are determined by degree $a$ maps \(\check\nabla:N\odot C\to N\) as follows:
\begin{multline}
\nabla =\Bigl[ \bigoplus_{n\ge0} N^{\tens n}\tens C(n) \rTTo^{1\tens\Delta}
\bigoplus_{l,i_1,\dots,i_l\ge0} N^{\tens(i_1+\dots+i_l)}\tens C(i_1)\tdt C(i_l)\tens C(l) \rto\sim \hfill
\\
\bigoplus_{l,i_1,\dots,i_l\ge0} N^{\tens i_1}\tens C(i_1)\tdt N^{\tens i_l}\tens C(i_l)\tens C(l)
\rTTo^{\oplus_l\sum_{j=1}^l\delta_{1i_1}\tens\eps\tdt\sS{_{i_j}}{\check\nabla}\tdt\delta_{1i_l}\tens\eps\tens1} \bigoplus_{l\ge0} N^{\tens l}\tens C(l) \Bigr]
\\
\hskip\multlinegap =\biggl[ \bigoplus_{n\ge0} N^{\tens n}\tens C(n) \rTTo^{1\tens\Delta}
\bigoplus^{x+1+z=l}_{x+y+z=n} N^{\tens n}\tens C(1)^{\tens x}\tens C(y)\tens C(1)^{\tens z}\tens C(l) \rto\sim \hfill
\\
\bigoplus_{x+1+z=l}^{l,x,y,z\ge0} (N\tens C(1))^{\tens x}\tens(N^{\tens y}\tens C(y))\tens(N\tens C(1))^{\tens z}\tens C(l)
\\
\hfill \rTTo^{\oplus_l\sum(1\tens\eps)^{\tens x}\tens\sS{_y}{\check\nabla}\tens(1\tens\eps)^{\tens z}\tens1}
\bigoplus_{l\ge0} N^{\tens l}\tens C(l) \biggr] \quad
\\
\hskip\multlinegap =\biggl[ \bigoplus_{n\ge0} N^{\tens n}\tens C(n) \rTTo^{1\tens\Delta_{T(x,y,z)}} 
\bigoplus^{x+1+z=l}_{x+y+z=n} N^{\tens n}\tens C(y)\tens C(l) \rto\sim \hfill
\\
\bigoplus_{x+1+z=l}^{l,x,y,z\ge0} N^{\tens x}\tens(N^{\tens y}\tens C(y))\tens N^{\tens z}\tens C(l)
\rTTo^{\oplus_l\sum1^{\tens x}\tens\sS{_y}{\check\nabla}\tens1^{\tens z}\tens1} 
\bigoplus_{l\ge0} N^{\tens l}\tens C(l) \biggr].
\label{eq-nabla-page}
\end{multline}

\chapter{Curved operads and curved cooperads}
The definitions of curved (co)operads given in this section are direct generalizations of definitions of curved (co)algebras obtained in \cite{Lyu-curved-coalgebras} in the process of transforming the definitions of $A_\infty$\n-(co)algebras.
The reader is invited to consult [\textit{loc. cit.}] for motivations.
We define cobar and bar constructions as functors between various categories of curved cooperads and curved operads.

\section{Curved operads}
We discuss several versions of the category of curved operads.
 
\begin{definition}\label{def-curved-operad}
A \emph{curved operad} \((\co,m,d,m_0,\eta)\) is\index{TTindex}{curved operad} an operad \((\co,m,\eta)\) in $\cv$ equipped with an $\id_\co$-derivation \(d:\co\to\co\) of degree~1, and the curvature element\index{TTsymb}{m0@$m_0$} $m_0\in\co(1)^2$ (a degree~2 map $m_0:\1\to\co$) such that for all $n\ge0$
\begin{align}
d^2 &=\sum_{x+1+z=n}(m_0\tens1)\cdot m_{T(x,1,z)} -(1\tens m_0)\cdot m_{T(0,n,0)} \notag
\\
&=\sum_{x+1+z=n}m_0\underset{x,1,z}\bull1 -1\underset{0,n,0}\bull m_0: \co(n) \to \co(n),
\label{eq-d2-(m01)m-(1m0)m}
\\
m_0d &=0. \notag
\end{align}
A \emph{morphism of curved operads} \(\sff:\co\to\cp\) is a pair \((\sff_1,\sff_0)\), where \(\sff_1:\co\to\cp\) is a morphism of graded operads and \(\sff_0\in\cp(1)^1\) (that is, a map \(\sff_0:\1\to\cp\) of degree~1) such that for all $n\in\NN$
\begin{align}
&\bigl\langle \co(n) \rto{\sff_1} \cp(n) \rto d \cp(n) \bigr\rangle +\bigl\langle \co(n) \rto{\sff_1} \cp(n) \rTTo^{1\tens\sff_0} \cp(n)\tens\cp(1) \rTTo^{m_{T(0,n,0)}} \cp(n) \bigr\rangle
\label{eq-OPP-OPPPP}
\\
&-\sum_{x+1+z=n}\bigl\langle \co(n) \rto{\sff_1} \cp(n) \rTTo^{\sff_0\tens1} \cp(1)\tens\cp(n) \rTTo^{m_{T(x,1,z)}} \cp(n) \bigr\rangle =\bigl\langle \co(n) \rto d \co(n) \rto{\sff_1} \cp(n) \bigr\rangle, \notag
\\
&\bigl\langle \1 \rto{m_0} \cp(1) \bigr\rangle -\bigl\langle \1 \rto{\sff_0} \cp(1) \rto d \cp(1) \bigr\rangle \notag
\\
&-\bigl\langle \1 =\1\tens\1 \rTTo^{\sff_0\tens\sff_0} \cp(1)\tens\cp(1) \rTTo^{m_{T(0,1,0)}} \cp(1) \bigr\rangle =\bigl\langle \1 \rto{m_0} \co(1) \rto{\sff_1} \cp(1) \bigr\rangle.
\label{eq-1P-1PP-111PPP-1OP}
\end{align}
The composition $\sfh:\co\to\cq$ of morphisms \(\sff:\co\to\cp\) and \(\sfg:\cp\to\cq\) is \(\sfh=(\sfh_1=\sff_1\sfg_1,\sfh_0=\sff_0\sfg_1+\sfg_0)\).
The unit morphisms are \((\id,0)\).
The category of such operads and their morphisms is denoted\index{TTsymb}{cOp@$\cOp$} $\cOp$.
\end{definition}

In the study of bar and cobar constructions for properads Vallette \cite[Section~4]{MR2320654} works with (co)augmented properads and coproperads.
Augmentation is assumed also in subsequent work by Merkulov and Vallette \cite{0707.0889}.
As we shall see augmentation is not obligatory for operads.

\begin{definition}
A \emph{unit-complemented curved operad} \((\co,m,d,m_0,\eta,\sfv)\) is\index{TTindex}{unit-complemented curved operad} a curved operad \((\co,m,d,m_0,\eta)\) equipped with a degree~0 map of collections \(\sfv:\co\to\1\) (splitting of the unit) such that
\begin{equation*}
\eta\sfv =1_\1.
\end{equation*}
A \emph{morphism of unit-complemented curved operads} \(\sff:\co\to\cp\) is a morphism in $\cOp$ ignoring the splitting $\sfv$ of the unit.
The corresponding category of such operads and their morphisms is denoted\index{TTsymb}{ucOp@$\ucOp$} $\ucOp$.
\end{definition}

\begin{definition}
The subcategory\index{TTsymb}{UCCOp@$\UCCOp$} $\UCCOp$ of $\ucOp$ consists of the same objects and its morphisms \(\sff:\co\to\cp\) are pairs \((\sff_1,\sff_0)\in\Mor\ucOp\) such that
\[ \sff_0 =\bigl( \1 \rto{\und f} \1 \rto\eta \cp \bigr)
\]
for some \(\und f\in\kk^1\), namely, for \(\und f=\sff_0\sfv\).
Thus, morphisms in $\UCCOp$ are pairs \((\sff_1,\und f)\) that satisfy the equations
\[ \sff_1d -(n-1)\sff_1\und f =d\sff_1: \co(n) \to \cp(n), \qquad m^\cp_0 -\und f^2\eta =m^\co_0\sff_1.
\]
\end{definition}

The element \(\und f^2=-\und f^2\in\kk^2\) vanishes in a graded strongly commutative ring $\kk$.
By definition a \emph{graded strongly commutative ring} is\index{TTindex}{graded strongly commutative ring} a graded ring $\kk$ such that \(ba=(-1)^{|a|\cdot|b|}ab\) for all homogeneous elements $a$, $b$ and $c^2=0$ for all elements $c$ of odd degree.
The composition $\sfh:\co\to\cq$ of morphisms \(\sff:\co\to\cp\) and \(\sfg:\cp\to\cq\) is \(\sfh=(\sfh_1=\sff_1\sfg_1,\und h=\und g+\und f)\).
The unit morphisms are \((\id,0)\).

\begin{definition}
The subcategory\index{TTsymb}{uccOp@$\uccOp$} $\uccOp$ of $\UCCOp\subset\ucOp$ consists of the same objects and its morphisms are pairs \((\sff,0)\in\Mor\ucOp\).
Thus, morphisms in $\uccOp$ are homomorphisms of operads $\sff:\co\to\cp$ that satisfy the equations
\[ \sff d =d\sff: \co \to \cp, \qquad m^\cp_0 =m^\co_0\sff.
\]
\end{definition}

An analogue of the above category for \ainf-algebras, namely $A_{[0,\infty[}$-algebras, is considered by Nicol\'as \cite{math.RT/0702449}.
He assumes that $\sff_0=0$.

\begin{definition}
A \emph{unit-complemented $\dg$-operad} is\index{TTindex}{unit-complemented dg-operad@unit-complemented $\dg$-operad} a unit-complemented curved operad \((\co,m,d,0,\eta,\sfv)\) with $m_0=0$.
Equivalently, it is a $\dg$\n-operad \((\co,m,d,\eta)\) with a degree 0 map \(\sfv:\co\to\1\) (splitting of the unit) such that \(\eta\cdot\sfv=1_\1\).
Morphisms of such operads are morphisms of $\dg$\n-operads.
They form a full subcategory of $\uccOp$ denoted\index{TTsymb}{ucdgOp@$\ucdgOp$} \(\ucdgOp\).
\end{definition}

\begin{exercise}
Applying equation \eqref{eq-OPP-OPPPP} we get
\[ d^2\sff_1 =\sff_1\bigl[ d +(1\tens\sff_0)m_{T(0,n,0)} -\sum_{x+1+z=n}(\sff_0\tens1)m_{T(x,1,z)} \bigr]^2.
\]
On the other hand,
\begin{multline*}
d^2\sff_1 =\sum_{x+1+z=n}(m_0\tens1)m_{T(x,1,z)}\sff_1 -(1\tens m_0) m_{T(0,n,0)}\sff_1
\\
=\sff_1\bigl[ \sum_{x+1+z=n}(m_0\sff_1\tens1)m_{T(x,1,z)} -(1\tens m_0\sff_1) m_{T(0,n,0)} \bigr].
\end{multline*}
Show that \eqref{eq-1P-1PP-111PPP-1OP} implies that these two expressions are identically equal to each other.
\end{exercise}

Clearly $\eta d=0$ for any operad \(\co\in\Ob\cOp\).
It is technically convenient to represent the same notions via operations in the shifted collection $\co[1]$.
We consider the binary multiplication $b$, the map $b_1$ and the element $b_0$, all of degree~1, such that
\begin{align*}
b_{T(x,y,z)} &= -(\sigma^{-1}\tens\sigma^{-1})m_{T(x,y,z)}\sigma: \co[1](y)\tens\co[1](x+1+z) \to \co[1](x+y+z),
\\
b_1 &= -\sigma^{-1}d\sigma: \co[1](n) \to \co[1](n),
\\
b_0 &= m_0\sigma: \1 \to \co[1](1).
\end{align*}
Also we consider the shifted unit $\bfeta$ of degree $-1$ and its splitting $\bv$ of degree~1:
\begin{align*}
\bfeta &=\bigl( \1 \rTTo^\eta \co \rTTo^\sigma \co[1] \bigr),
\\
\bv &=\bigl( \co[1] \rTTo^{\sigma^{-1}} \co \rTTo^\sfv \1 \bigr).
\end{align*}
These data satisfy the corresponding equations: signless versions of \eqref{dia-operad-3-OOOOOOOOOO} and \eqref{dia-operad-4-OOOOOOO} for \(b_{T(-,-,-)}\),
\begin{gather*}
b_{T(x,y,z)}b_1 +(1\tens b_1 +b_1\tens1)b_{T(x,y,z)} =0: \co[1](y)\tens\co[1](x+1+z) \to \co[1](x+y+z),
\\
b_1^2 +\sum_{x+1+z=n}(b_0\tens1)b_{T(x,1,z)} +(1\tens b_0)b_{T(0,n,0)} =0: \co[1](n) \to \co[1](n),
\\
b_0b_1 =0, \qquad (1\tens\bfeta)b_{T(0,n,0)} =1: \co[1](n) \to \co[1](n),
\\
(\bfeta\tens1)b_{T(x,1,z)} =-1: \co[1](n) \to \co[1](n) \text{ \ for \ } x+1+z=n, \quad \bfeta b_1 =0, \quad \bfeta\bv =1_\1.
\end{gather*}

A morphism in $\cOp$ is the same as a pair, consisting of a degree~0 map \(f_1=\sff_1[1]=\sigma^{-1}\sff_1\sigma:\co[1]\to\cp[1]\in\cv^\NN\) and an element \(f_0\in\cp(1)[1]^0\) such that
\begin{align*}
\sff_1 &=\bigl( \co \rTTo^\sigma \co[1] \rTTo^{f_1} \cp[1] \rTTo^{\sigma^{-1}} \cp \bigr),
\\
\sff_0 &=\bigl( \1 \rTTo^{f_0} \cp[1] \rTTo^{\sigma^{-1}} \cp \bigr).
\end{align*}
The equations satisfied by shifted data of a morphism are
\begin{gather*}
(f_1\tens f_1)b_{T(x,y,z)} =b_{T(x,y,z)}f_1, \qquad b_0 +f_0b_1 +(f_0\tens f_0)b_{T(0,1,0)} =b_0f_1,
\\
\bfeta_\co f_1 =\bfeta_\cp, \qquad f_1b_1 +(f_1\tens f_0)b_{T(0,n,0)} +\sum_{x+1+z=n}(f_0\tens f_1)b_{T(x,1,z)} =b_1f_1.
\end{gather*}

A morphism in $\UCCOp$ is the same as a pair, consisting of a degree~0 map \(f_1=\sff_1[1]=\sigma^{-1}\sff_1\sigma:\co[1]\to\cp[1]\) and an element \(\und f\in\kk^1\) which satisfy the equations
\begin{gather*}
(f_1\tens f_1)b_{T(x,y,z)} =b_{T(x,y,z)}f_1, \qquad \bfeta f_1 =\bfeta, \qquad b^\cp_0 -\und f^2\bfeta =b^\co_0f_1,
\\
f_1b_1 =b_1f_1 +(n-1)\und ff_1: \co[1](n) \to \cp[1](n).
\end{gather*}

\begin{example}\label{exa-END(Vd)}
Let $V$ be a graded $\kk$\n-module and let \(d=d_V:V\to V\) be a homogeneous map of degree~1.
Its square \(m_0=-d_V^2\) is a homogeneous map of degree~2.
We claim that the graded endomorphism operad \((\ce,m,\eta)\), \(\ce(n)=(\END V)(n)=\Hom_\kk(V^{\tens n},V)\), admits a curved operad structure \((\ce,m,d_\ce,m_0,\eta)\), where for all \(f\in\ce(n)^\bull\)
\[ (f)d_\ce =f\cdot d_V -(-)^f\sum_{x+1+z=n} (1^{\tens x}\tens d_V\tens1^{\tens z}) \cdot f.
\]
In fact, $d_\ce$ is an inner derivation of degree~1,
\begin{gather*}
(f)d_\ce^2 =f\cdot d_V^2 -\sum_{x+1+z=n} (1^{\tens x}\tens d_V^2\tens1^{\tens z}) \cdot f =-f\cdot m_0 +\sum_{x+1+z=n} (1^{\tens x}\tens m_0\tens1^{\tens z}) \cdot f,
\\
m_0d_\ce =m_0d_V -d_Vm_0 = -d_V^3 +d_V^3 =0.
\end{gather*}
In order to become unit-complemented the operad $\ce$ just requires a splitting of the unit \(\sfv:\End_\kk V\to\kk\), \((\id_V)\sfv=1\).
This exists, for instance, when $V$ has a direct summand $k$, \(\bigl(\kk\rto iV\rto p\kk\bigr)=1_\kk\), namely, \((f)\sfv=ifp\).
\end{example}

\section{Curved cooperads}
We study several categories of curved cooperads.

\begin{definition}
A \emph{curved cooperad} \((C,\Delta,d,\bdelt_0,\eps)\) is\index{TTindex}{curved cooperad} a cooperad \((C,\Delta,\eps)\) in $\cv$, equipped with an $\id_C$-coderivation \(d:C\to C\) of degree~1 and the curvature functional\index{TTsymb}{delta0@$\bdelt_0$} \(\bdelt_0:C(1)\to\1\) of degree~2 such that
\begin{align}
d^2 &=\Delta_{T(0,n,0)}\cdot(1\tens\bdelt_0) -\sum_{x+1+z=n}\Delta_{T(x,1,z)}\cdot(\bdelt_0\tens1): C(n) \to C(n),
\label{eq-d2-Delta(1delta0-Delta(delta01)}
\\
d\bdelt_0 &=0: C(1) \to \1. \notag
\end{align}

A \emph{morphism of curved cooperads} is a pair \((\sfg_1,\sfg_0)\) consisting of a morphism of cooperads \(\sfg_1:C\to D\) and a map \(\sfg_0:C(1)\to\1\) of degree~1 such that for all $n\ge0$
\begin{align}
d^C\sfg_1 +\sum_{x+1+z=n}\Delta_{T(x,1,z)}(\sfg_0\tens\sfg_1) -\Delta_{T(0,n,0)}(\sfg_1 &\tens\sfg_0) =\sfg_1d^D: C(n) \to D(n), \notag
\\
\bdelt^C_0 -d^C\sfg_0 -\Delta_{T(0,1,0)}(\sfg_0\tens\sfg_0) &=\sfg_1\bdelt^D_0: C(1) \to \1.
\label{eq-morphism-curved-cooperads}
\end{align}
The composition $\sfh:C\to E$ of morphisms \(\sff:C\to D\) and \(\sfg:D\to E\) is given by \(\sfh_1=\sff_1\sfg_1\), \(\sfh_0=\sff_0+\sff_1\sfg_0\).
The unit morphism is \((\id,0)\).
This category of curved cooperads and their morphisms is denoted\index{TTsymb}{cCoop@$\cCoop$} $\cCoop$.
\end{definition}

We verify that the composition defined above is indeed a morphism of $\cCoop$:
\begin{multline*}
d^C\sff_1\sfg_1 +\sum_{x+1+z=n}\Delta_{T(x,1,z)}[(\sff_0+\sff_1\sfg_0)\tens\sff_1\sfg_1] -\Delta_{T(0,n,0)}[(\sff_1\sfg_1\tens(\sff_0+\sff_1\sfg_0)] -\sff_1\sfg_1d^E
\\
\hskip\multlinegap =[d^C\sff_1 +\sum_{x+1+z=n}\Delta_{T(x,1,z)}(\sff_0\tens\sff_1) -\Delta_{T(0,n,0)}(\sff_1\tens\sff_0) -\sff_1d^D]\sfg_1 \hfill
\\
+\sff_1[d^D\sfg_1 +\sum_{x+1+z=n}\Delta_{T(x,1,z)}(\sfg_0\tens\sfg_1) -\Delta_{T(0,n,0)}(\sfg_1\tens\sfg_0) -\sfg_1d^E] =0
: C(n) \to E(n),
\\
\hskip\multlinegap \bdelt^C_0 -d^C(\sff_0+\sff_1\sfg_0) -\Delta_{T(0,1,0)}[(\sff_0+\sff_1\sfg_0)\tens(\sff_0+\sff_1\sfg_0)] -\sff_1\sfg_1\bdelt^E_0 \hfill
\\
=\bdelt^C_0 -d^C\sff_0 -\Delta_{T(0,1,0)}(\sff_0\tens\sff_0) -\sff_1\bdelt^D_0
+\sff_1\bdelt^D_0 -d^C\sff_1\sfg_0 +\Delta_{T(0,1,0)}(\sff_1\tens\sff_0)\sfg_0 -\Delta_{T(0,1,0)}(\sff_0\tens\sff_1)\sfg_0
\\
-\sff_1\Delta_{T(0,1,0)}(\sfg_0\tens\sfg_0) -\sff_1\sfg_1\bdelt^E_0
=\sff_1[\bdelt^D_0 -d^D\sfg_0 -\Delta_{T(0,1,0)}(\sfg_0\tens\sfg_0) -\sfg_1\bdelt^D_0] =0: C(1) \to \1.
\end{multline*}

\begin{definition}
A \emph{counit-complemented curved cooperad} \((C,\Delta,d,\bdelt_0,\eps,\sfw)\) is\index{TTindex}{counit-complemented curved cooperad} a curved cooperad \((C,\Delta,d,\bdelt_0,\eps)\) equipped with a \emph{splitting of the counit} \(\sfw:\1\to C\in\cv\), an element of $C(1)^0$ such that \(\sfw\cdot\eps=1_\1\).
Morphisms of such cooperads are morphisms of curved cooperads (ignoring the splitting).
The category of counit-complemented curved cooperads is denoted\index{TTsymb}{ucCoop@$\ucCoop$} $\ucCoop$.
\end{definition}

Denote \(\opr=\opr_C=1-\eps\sfw:C\to C\) the projection along $\1$ whose image is $\bar{C}=\Ker\eps$.
The projection \(\opr[-1]=\opr_C[-1]=1-\beps\bw:C[-1]\to C[-1]\) is also denoted \(\opr=\opr_C\) for the sake of brevity.

\begin{definition}
A \emph{curved augmented cooperad} is\index{TTindex}{curved augmented cooperad} a counit-complemented cooperad \((C,\Delta,d,\bdelt_0,\eps,\sfw)\) such that \(\sfw:\1\to C\) is an augmentation (a homomorphism of cooperads) and the non-counital cooperad \((\bar{C},\bar\Delta)\) is conilpotent.

A \emph{morphism of curved augmented cooperads} is a morphism of curved cooperads \((\sfg_1,\sfg_0)\) where \(\sfg_1:C\to D\) is a morphism of augmented cooperads (preserves the augmentation).
This category of curved augmented cooperads and their morphisms is denoted\index{TTsymb}{CACoop@$\CACoop$} $\CACoop$.
\end{definition}

Being augmented a cooperad $C$ has to satisfy
\[ \sfw\eps =1_\1, \qquad \sfw\Delta_{T(0,1,0)} =\sfw\tens\sfw, \qquad \sfw\Delta_{T(1,0,0)} =\sfw\Delta_{T(0,0,1)} =0.
\]
A morphism $\sfg_1$ of augmented cooperads satisfies \(\sfw\sfg_1=\sfw\).
Notice that on $\bar{C}$
\begin{alignat}2
\Delta_{T(0,n,0)} &= 1\tens\sfw +\bar\Delta_{T(0,n,0)}, &\qquad& \text{if } n\ne1, \notag
\\
\Delta_{T(0,1,0)} &= 1\tens\sfw +\sfw\tens1 +\bar\Delta_{T(0,1,0)}, &\qquad& \text{if } n=1, \notag
\\
\Delta_{T(x,1,z)} &= \sfw\tens1 +\bar\Delta_{T(x,1,z)}, &\qquad& \text{if } x+z\ne0, \notag
\\
\Delta_{T(x,y,z)} &= \bar\Delta_{T(x,y,z)}, &\qquad& \text{if } x+z\ne0, \quad y\ne1.
\label{eq-Delta-1w-w1-Delta}
\end{alignat}

\begin{definition}
Objects of the subcategory\index{TTsymb}{caCoop@$\caCoop$} $\caCoop$ of $\CACoop$ are the same and morphisms are pairs \((\sfg_1,\sfg_0)\in\CACoop\) such that \(\sfw\sfg_0=0\).
\end{definition}

\begin{definition}
\emph{Augmented curved cooperads} are\index{TTindex}{augmented curved cooperad} defined as curved augmented cooperads \((C,\Delta,d,\bdelt_0,\eps,\sfw)\) with
\begin{equation}
\sfw\bdelt_0 =0, \qquad \sfw d =0.
\label{eq-wd10-wd00}
\end{equation}
Such cooperads form a full subcategory of $\caCoop$ denoted\index{TTsymb}{acCoop@$\acCoop$} $\acCoop$.
\end{definition}

Equations~\eqref{eq-wd10-wd00} can be reformulated in Positselski style \cite{0905.2621} as \((\sfw,0):\1\to C\) being a morphism in $\caCoop$.

It is technically convenient to represent the same notions via (co)operations in a shifted collection $C[-1]$.
We consider the binary comultiplication $\xi$, the map $\xi_1$ and the functional $\xi_0$, all of degree~1, such that
\begin{align*}
\xi_{T(x,y,z)} &= -\sigma\Delta_{T(x,y,z)}(\sigma^{-1}\tens\sigma^{-1}): C[-1](x+y+z) \to C[-1](y)\tens C[-1](x+1+z),
\\
\xi_1 &= -\sigma d\sigma^{-1}: C[-1](n) \to C[-1](n),
\\
\xi_0 &= \sigma\bdelt_0: C[-1](1) \to \1.
\end{align*}
Also we consider the shifted counit $\beps$ of degree $-1$ and the shifted augmentation $\bw$ of degree~1:
\begin{align*}
\beps &=\sigma\eps: C[-1] \to \1,
\\
\bw &=\sfw\sigma^{-1}: \1 \to C[-1](1).
\end{align*}
These data satisfy the corresponding equations: signless versions of \eqref{dia-cooperad-3-OOOOOOO} and \eqref{dia-cooperad-4-OOOOO} for \(\xi_{T(-,-,-)}\),
\begin{gather*}
\begin{split}
\xi_1\xi_{T(x,y,z)} +\xi_{T(x,y,z)}(1\tens\xi_1 +\xi_1\tens1) &=0:
\\
C[-1](x+y+z) &\to C[-1](y)\tens C[-1](x+1+z),
\end{split}
\\
\xi_1^2 +\xi_{T(0,n,0)}(1\tens\xi_0) +\sum_{x+1+z=n}\xi_{T(x,1,z)}(\xi_0\tens1) =0: C[-1](n) \to C[-1](n),
\\
\xi_1\beps =0, \qquad \xi_1\xi_0 =0, \qquad \xi_{T(0,n,0)}(1\tens\beps) =-1: C[-1](n) \to C[-1](n),
\\
\bw\beps =1_\1, \qquad \xi_{T(x,1,z)}(\beps\tens1) =1: C[-1](n) \to C[-1](n) \text{ \ for \ } x+1+z=n.
\end{gather*}

A morphism $C\to D$ in $\cCoop$ is the same as a pair \((g_1,g_0)\) consisting of maps \(g_1=\sigma\sfg_1\sigma^{-1}:C[-1]\to D[-1]\) and \(g_0=\sigma\sfg_0:C[-1](1)\to\1\) of degree~0 such that
\begin{gather}
\xi_{T(x,y,z)}(g_1\tens g_1) =g_1\xi_{T(x,y,z)}, \qquad g_1\beps^D =\beps^C, \notag
\\
\xi_1g_1 +\sum_{x+1+z=n}\xi_{T(x,1,z)}(g_0\tens g_1) +\xi_{T(0,n,0)}(g_1\tens g_0) =g_1\xi_1: C[-1](n) \to D[-1](n), \notag
\\
\xi^C_0 +\xi_1g_0 +\xi_{T(0,1,0)}(g_0\tens g_0) =g_1\xi^D_0: C[-1](1) \to \1.
\label{eq-xi0-xi1g0-g1xi}
\end{gather}
A morphism of curved augmented cooperads \((g_1,g_0):C\to D\in\CACoop\) satisfies additional equation \(\bw^Cg_1 =\bw^D\).
A morphism $C\to D$ in $\caCoop$ satisfies one more equation: $\bw g_0=0$.

\section{Cobar construction}
Let us construct a functor\index{TTsymb}{Cobar@$\Cobar$} \(\Cobar:\ucCoop\to\ucOp\), the cobar construction.
Let \(C=(C,\xi,\xi_1,\xi_0,\beps,\bw)\) be a counit-complemented curved cooperad.
Decompose the idempotent \(1-\eps\sfw:C\to C\) into a projection \(\opr:C\to\bar{C}\) and an injection \(\oin:\bar{C}\to C\) so that \(\oin\opr=1_{\bar{C}}\).
Clearly, \(\bar{C}=\Ker\eps\).
The morphisms \(\sigma\opr\sigma^{-1}:C[-1]\to\bar{C}[-1]\) and \(\sigma\oin\sigma^{-1}:\bar{C}[-1]\to C[-1]\) are also denoted $\opr$ and $\oin$ by abuse of notations.
They split the idempotent \(1-\sigma\eps\sfw\sigma^{-1}=1-\beps\cdot\bw:C[-1]\to C[-1]\) and \(\bar{C}[-1]=\Ker\beps\).

\begin{proposition}\label{pro-map-curved-cooperad-curved-operad}
There is a map \(\Cobar:\Ob\ucCoop\to\Ob\ucOp\), the cobar construction\index{TTindex}{cobar construction}.
It takes a curved cooperad \(C=(C,\xi,\xi_1,\xi_0,\beps,\bw)\) to the operad \(\Cobar C=\bar{C}[-1]\TT\) equipped with its multiplication $m$, the unit \(\eta=\inj_\circ:\1\hookrightarrow\bar{C}[-1]\TT\), the splitting \(\sfv=\pr_\circ:\bar{C}[-1]\TT\to\1\), the degree~1 derivation \(d=\bar\xi:\bar{C}[-1]\TT\to\bar{C}[-1]\TT\) determined by its restriction \(\check{d}:\bar{C}[-1]\to\bar{C}[-1]\TT\), which is the sum of maps
\begin{align*}
\bar\xi_{T(x,y,z)} &= \bigl\langle \bar C[-1](x+y+z) \rMono^\oin C[-1](x+y+z) \rTTo^{\xi_{T(x,y,z)}}
\\
& C[-1](y)\tens C[-1](x+1+z) \rTTo^{\opr\tens\opr} \bar C[-1](y)\tens\bar C[-1](x+1+z) \bigr\rangle,
\\
\bar\xi_{\tau[n]} &= \bigl\langle \bar C[-1](n) \rMono^\oin C[-1](n) \rTTo^{\xi_1} C[-1](n) \rTTo^\opr \bar C[-1](n) \bigr\rangle,
\\
\bar\xi_\circ &= \bigl\langle \bar C[-1](1) \rMono^\oin C[-1](1) \rTTo^{\xi_0} \1 \bigr\rangle
\end{align*}
over \(x,y,z,n\in\NN\).
The curvature element \(m^{\Cobar C}_0:\1\to\bar{C}[-1]\TT\) is
\begin{multline*}
m^{\Cobar C}_0 =-\bw\xi_0 -\bw\xi_1\opr -[(\bw\tens\bw +\bw\xi_{T(0,1,0)}) +\bw\xi_{T(1,0,0)} +\bw\xi_{T(0,0,1)}](\opr\tens\opr)
\\
\in \1\oplus\bar{C}[-1](1)\oplus\bar{C}[-1](1)^{\tens2}\oplus\bar{C}[-1](0)\tens\bar{C}[-1](2)\oplus\bar{C}[-1](0)\tens\bar{C}[-1](2).
\end{multline*}
\end{proposition}

\begin{proof}
Notice that $\bw\xi_1\beps=0$ since \(d_C\cdot\eps=0\).
For the same reason \(\xi_1(1-\beps\bw)=\xi_1\).
One finds also
\begin{multline*}
\hfill \xi_{T(x,y,z)}[1\tens(1-\beps\bw)] =\xi_{T(x,y,z)} +\delta_{x+z,0}\tens\bw, \hfill
\\
\xi_{T(x,y,z)}[(1-\beps\bw)\tens1] =\xi_{T(x,y,z)} -\delta_{y,1}\bw\tens1,
\\
\xi_{T(x,y,z)}[(1-\beps\bw)\tens(1-\beps\bw)] =\xi_{T(x,y,z)} +\delta_{x+z,0}\tens\bw -\delta_{y,1}\bw\tens1 -\delta_{x+z,0}\delta_{y,1}\beps\bw\tens\bw:
\\
C[-1](x+y+z) \to C[-1](y)\tens C[-1](x+1+z).
\end{multline*}

Let us prove that so described $\Cobar C$ is a unit-complemented curved operad.
Both sides of the equation
\begin{multline}
(d^{\Cobar C})^2 =\sum_{x+1+z=n}(m^{\Cobar C}_0\tens1)\cdot m^{\Cobar C}_{T(x,1,z)} -(1\tens m^{\Cobar C}_0)\cdot m^{\Cobar C}_{T(0,n,0)}:
\\
\bar{C}[-1]\TT(n) \to \bar{C}[-1]\TT(n)
\label{eq-d-Cobar-C2-mm-mm}
\end{multline}
are degree~2 derivations.
In fact, the right hand side is an inner derivation (the notion dual to inner coderivation; dualize \propref{pro-id-coderivation}).
Equation~\eqref{eq-d-Cobar-C2-mm-mm} is equivalent to its restriction to generators $\bar{C}[-1](n)$:
\begin{multline}
\Bigl\langle \bar C[-1](n) \rto{\bar\xi_t} \coprod_{t\in\tr(n)} \, \bigotimes_{y\in\IV(t)} \bar{C}[-1]|y| \rTTo^{\coprod_t\sum_i\tens^{y\in\IV(t)}f_i(y)}
\\
\hfill \coprod_{t\in\tr(n)} \, \coprod_{i\in\IV(t)} \, \bigotimes_{y\in\IV(t)} \, \bigotimes_{q\in\IV(t_y^i)}\bar{C}[-1]|q| \rto m \coprod_{\tau\in\tr(n)} \, \bigotimes_{v\in\IV(\tau)} \bar{C}[-1]|v| \Bigr\rangle \quad
\\
\hskip\multlinegap =\Bigl(-\sum_{x+1+z=n}1+1\Bigr) (\bw\xi_0) \Bigl\langle \bar C[-1](n) \rTTo^{\inj_{\tau[n]}} \coprod_{\tau\in\tr(n)} \, \bigotimes_{v\in\IV(\tau)} \bar{C}[-1]|v| \Bigr\rangle \hfill
\\
-\sum_{x+1+z=n} \Bigl\langle \bar C[-1](n) \rTTo^{\bw\xi_1\opr\tens1} \bar C[-1](1)\tens\bar C[-1](n) \rTTo^{m^{\Cobar C}_{T(x,1,z)}} \coprod_{\tau\in\tr(n)} \, \bigotimes_{v\in\IV(\tau)} \bar{C}[-1]|v| \Bigr\rangle
\\
+\Bigl\langle \bar C[-1](n) \rTTo^{1\tens\bw\xi_1\opr} \bar C[-1](n)\tens\bar C[-1](1) \rTTo^{m^{\Cobar C}_{T(0,n,0)}} \coprod_{\tau\in\tr(n)} \, \bigotimes_{v\in\IV(\tau)} \bar{C}[-1]|v| \Bigr\rangle
\\
-\sum_{x+1+z=n} \Bigl\langle \bar C[-1](n) \rTTo^{(\bw\tens\bw +\bw\xi_{T(0,1,0)})(\opr\tens\opr)\tens1} \bigl(\bar C[-1](1)\tens\bar C[-1](1)\bigr)\tens\bar C[-1](n)
\\
\hfill \rTTo^{m^{\Cobar C}_{T(x,1,z)}} \coprod_{\tau\in\tr(n)} \, \bigotimes_{v\in\IV(\tau)} \bar{C}[-1]|v| \Bigr\rangle \quad
\\
-\sum_{x+1+z=n} \Bigl\langle \bar C[-1](n) \rTTo^{\bw\xi_{T(1,0,0)}(\opr\tens\opr)\tens1} \bigl(\bar C[-1](0)\tens\bar C[-1](2)\bigr)\tens\bar C[-1](n)
\\
\hfill \rTTo^{m^{\Cobar C}_{T(x,1,z)}} \coprod_{\tau\in\tr(n)} \, \bigotimes_{v\in\IV(\tau)} \bar{C}[-1]|v| \Bigr\rangle \quad
\\
-\sum_{x+1+z=n} \Bigl\langle \bar C[-1](n) \rTTo^{\bw\xi_{T(0,0,1)}(\opr\tens\opr)\tens1} \bigl(\bar C[-1](0)\tens\bar C[-1](2)\bigr)\tens\bar C[-1](n)
\\
\hfill \rTTo^{m^{\Cobar C}_{T(x,1,z)}} \coprod_{\tau\in\tr(n)} \, \bigotimes_{v\in\IV(\tau)} \bar{C}[-1]|v| \Bigr\rangle \quad
\\
+\Bigl\langle \bar C[-1](n) \rTTo^{1\tens(\bw\tens\bw +\bw\xi_{T(0,1,0)})(\opr\tens\opr)} \bar C[-1](n)\tens\bigl(\bar C[-1](1)\tens\bar C[-1](1)\bigr)
\\
\hfill \rTTo^{m^{\Cobar C}_{T(0,n,0)}} \coprod_{\tau\in\tr(n)} \, \bigotimes_{v\in\IV(\tau)} \bar{C}[-1]|v| \Bigr\rangle \quad
\\
+\Bigl\langle \bar C[-1](n) \rTTo^{1\tens\bw\xi_{T(1,0,0)}(\opr\tens\opr)} \bar C[-1](n)\tens\bigl(\bar C[-1](0)\tens\bar C[-1](2)\bigr)
\\
\hfill \rTTo^{m^{\Cobar C}_{T(0,n,0)}} \coprod_{\tau\in\tr(n)} \, \bigotimes_{v\in\IV(\tau)} \bar{C}[-1]|v| \Bigr\rangle \quad
\\
+\Bigl\langle \bar C[-1](n) \rTTo^{1\tens\bw\xi_{T(0,0,1)}(\opr\tens\opr)} \bar C[-1](n)\tens\bigl(\bar C[-1](0)\tens\bar C[-1](2)\bigr)
\\
\hfill \rTTo^{m^{\Cobar C}_{T(0,n,0)}} \coprod_{\tau\in\tr(n)} \, \bigotimes_{v\in\IV(\tau)} \bar{C}[-1]|v| \Bigr\rangle,
\label{eq-long-xi-wxim}
\end{multline}
where \(t_y^i=\tau|y|\) and \(f_i(y)=\id\) if $y\ne i$, and \(f_i(i)=\bar\xi_{t_i^i}\), $t_i^i\in\tr|i|$, $i\in\IV(t)$.
This family gives the tree \(\tau=\tau^i=I_t(t_y^i\mid y\in\IV(t))\) which indicates the summand of the target of the first composition.
Only summands with \(|\IV(\tau)|\le3\) can occur.
Let us consider all possible $\tau$ case by case.

The postcomposition of equation~\eqref{eq-long-xi-wxim} with $\pr_\tau$ for $\tau=\circ$ has the form
\[ \bigl\langle \bar C[-1](1) \rTTo^{\bar\xi_{\tau[1]}} \bar C[-1](1) \rTTo^{\bar\xi_\circ} \1 \bigr\rangle =0.
\]
This holds true due to $\xi_1\xi_0=0$.

The postcomposition of equation~\eqref{eq-long-xi-wxim} with $\pr_\tau$ for $\tau=\tau[n]$ has the form
\begin{align*}
&\bigl\langle \bar C[-1](n) \rTTo^{\bar\xi_{\tau[n]}} \bar C[-1](n) \rTTo^{\bar\xi_{\tau[n]}} \bar C[-1](n) \bigr\rangle
\\
&+\sum_{x+1+z=n} \Bigl\langle \bar C[-1](n) \rTTo^{\bar\xi_{T(x,1,z)}} \bar C[-1](1)\tens\bar C[-1](n) \rTTo^{\bar\xi_\circ\tens1} \bar{C}[-1](n) \Bigr\rangle
\\
&+\bigl\langle \bar C[-1](n) \rTTo^{\bar\xi_{T(0,n,0)}} \bar C[-1](n)\tens\bar C[-1](1) \rTTo^{1\tens\bar\xi_\circ} \bar{C}[-1](n) \bigr\rangle
\\
&=(-n+1) (\bw\xi_0) \id_{\bar C[-1](n)}.
\end{align*}
Its validity follows from the identity
\begin{equation*}
\xi_1^2 +\sum_{x+1+z=n}\xi_{T(x,1,z)}(\xi_0\tens1) -\sum_{x+1+z=n}\bw\xi_0 +\xi_{T(0,n,0)}(1\tens\xi_0) +\bw\xi_0 =(-n+1) (\bw\xi_0)
\end{equation*}
precomposed with $\oin$ and postcomposed with $\opr$.

The postcomposition of equation~\eqref{eq-long-xi-wxim} with $\pr_\tau$ for $\tau=T(x,y,z)$, \(n=x+y+z\), has the form
\begin{multline*}
\bigl\langle \bar C[-1](n) \rto{\bar\xi_\tau} \bar C[-1](y)\tens\bar C[-1](x+1+z) \rTTo^{\bar\xi_{\tau[y]}\tens1} \bar C[-1](y)\tens\bar C[-1](x+1+z) \bigr\rangle
\\
+\bigl\langle \bar C[-1](n) \rto{\bar\xi_\tau} \bar C[-1](y)\tens\bar C[-1](x+1+z) \rTTo^{1\tens\bar\xi_{\tau[x+1+z]}} \bar C[-1](y)\tens\bar C[-1](x+1+z) \bigr\rangle
\\
\hfill+\bigl\langle \bar C[-1](n) \rTTo^{\bar\xi_{\tau[n]}} \bar C[-1](n) \rto{\bar\xi_\tau} \bar C[-1](y)\tens\bar C[-1](x+1+z) \bigr\rangle \quad
\\
\hskip\multlinegap =-\delta_{y,1} \bigl\langle \bar C[-1](n) \rTTo^{\bw\xi_1\opr\tens1} \bar C[-1](1)\tens\bar C[-1](n) \bigr\rangle \hfill
\\
+\delta_{x+z,0} \bigl\langle \bar C[-1](n) \rTTo^{1\tens\bw\xi_1\opr} \bar C[-1](n)\tens\bar C[-1](1) \bigr\rangle.
\end{multline*}
This follows from the identity
\begin{equation*}
\xi_\tau(\xi_1\tens1) -\delta_{y,1}\bw\xi_1\tens1 +\xi_\tau(1\tens\xi_1) +\delta_{x+z,0}\tens\bw\xi_1 +\xi_1\xi_\tau =-\delta_{y,1}\bw\xi_1\tens1 +\delta_{x+z,0}\tens\bw\xi_1
\end{equation*}

The postcomposition of \eqref{eq-long-xi-wxim} with $\pr_\tau$ for $\tau$ given by \eqref{eq-tangle-3-vertices-(4)}, \(n=v+w+x+y+z\), is
\begin{multline*}
\bigl\langle \bar C[-1](n) \rTTo^{\bar\xi_{T(v,w+x+y,z)}} \bar C[-1](w+x+y)\tens\bar C[-1](v+1+z)
\\
\hfill \rTTo^{\bar\xi_{T(w,x,y)}\tens1} \bar C[-1](x)\tens\bar C[-1](w+1+y)\tens\bar C[-1](v+1+z) \bigr\rangle \quad
\\
\hskip\multlinegap +\bigl\langle \bar C[-1](n) \rTTo^{\bar\xi_{T(v+w,x,y+z)}} \bar C[-1](x)\tens\bar C[-1](v+w+1+y+z) \hfill
\\
\rTTo^{1\tens\bar\xi_{T(v,w+1+y,z)}} \bar C[-1](x)\tens\bar C[-1](w+1+y)\tens\bar C[-1](v+1+z) \bigr\rangle
\\
=-\delta_{w,0}\delta_{x,1}\delta_{y,0} \bigl\langle \bar C[-1](n) \rTTo^{(\bw\tens\bw +\bw\xi_{T(0,1,0)})(\opr\tens\opr)\tens1} \bar C[-1](1)\tens\bar C[-1](1)\tens\bar C[-1](n) \bigr\rangle
\\
-\delta_{w,1}\delta_{x,0}\delta_{y,0} \bigl\langle \bar C[-1](n) \rTTo^{\bw\xi_{T(1,0,0)}(\opr\tens\opr)\tens1} \bar C[-1](0)\tens\bar C[-1](2)\tens\bar C[-1](n) \bigr\rangle
\\
-\delta_{w,0}\delta_{x,0}\delta_{y,1} \bigl\langle \bar C[-1](n) \rTTo^{\bw\xi_{T(0,0,1)}(\opr\tens\opr)\tens1} \bar C[-1](0)\tens\bar C[-1](2)\tens\bar C[-1](n) \bigr\rangle
\\
+\delta_{v+w+y+z,0} \bigl\langle \bar C[-1](n) \rTTo^{1\tens(\bw\tens\bw +\bw\xi_{T(0,1,0)})(\opr\tens\opr)} \bar C[-1](n)\tens\bar C[-1](1)\tens\bar C[-1](1) \bigr\rangle.
\end{multline*}
This follows from
\begin{multline*}
(\xi_{T(v,w+x+y,z)} -\delta_{w+x+y,1}\bw\tens1) (\xi_{T(w,x,y)}\tens1) +(\xi_{T(v+w,x,y+z)} +\delta_{v+w+y+z,0}\tens\bw) (1\tens\xi_{T(v,w+1+y,z)}) 
\\
\hskip\multlinegap =-\delta_{w,0}\delta_{x,1}\delta_{y,0}\bw\xi_{T(0,1,0)}\tens1 -\delta_{w,1}\delta_{x,0}\delta_{y,0}\bw\xi_{T(1,0,0)}\tens1 \hfill
\\
-\delta_{w,0}\delta_{x,0}\delta_{y,1}\bw\xi_{T(0,0,1)}\tens1 +\delta_{v+w+y+z,0}1\tens\bw\xi_{T(0,1,0)}
\end{multline*}
precomposed with $\oin$ and postcomposed with $\opr\tens\opr\tens\opr$, see \eqref{dia-cooperad-4-OOOOO}.

The postcomposition of \eqref{eq-long-xi-wxim} with $\pr_\tau$ for $\tau$ given by the second tree of \eqref{eq-tau-wy-v1x1z-1}, \(n=v+w+x+y+z\), is
\begin{multline*}
\bigl\langle \bar C[-1](n) \rTTo^{\bar\xi_{T(v,w,x+y+z)}} \bar C[-1](w)\tens\bar C[-1](v+1+x+y+z)
\\
\rTTo^{1\tens\bar\xi_{T(v+1+x,y,z)}} \bar C[-1](w)\tens\bar C[-1](y)\tens\bar C[-1](v+1+x+1+z)
\\
\hfill \rTTo^{(12)}_\sim \bar C[-1](y)\tens\bar C[-1](w)\tens\bar C[-1](v+1+x+1+z) \bigr\rangle \quad
\\
\hskip\multlinegap +\bigl\langle \bar C[-1](n) \rTTo^{\bar\xi_{T(v+w+x,y,z)}} \bar C[-1](y)\tens\bar C[-1](v+w+x+1+z) \hfill
\\
\hfill \rTTo^{1\tens\bar\xi_{T(v,w,x+1+z)}} \bar C[-1](y)\tens\bar C[-1](w)\tens\bar C[-1](v+1+x+1+z) \bigr\rangle \quad
\\
\hskip\multlinegap =\delta_{v+x+y+z,0} \bigl\langle \bar C[-1](n) \rTTo^{1\tens\bw\xi_{T(1,0,0)}(\opr\tens\opr)} \bar C[-1](n)\tens\bar C[-1](0)\tens\bar C[-1](2) \hfill
\\
\hfill \rTTo^{(12)}_\sim \bar C[-1](0)\tens\bar C[-1](n)\tens\bar C[-1](2) \bigr\rangle \quad
\\
+\delta_{v+w+x+z,0}\bigl\langle \bar C[-1](n) \rTTo^{1\tens\bw\xi_{T(0,0,1)}(\opr\tens\opr)} \bar C[-1](n)\tens\bar C[-1](0)\tens\bar C[-1](2) \bigr\rangle.
\end{multline*}
This follows from
\begin{multline*}
(\xi_{T(v,w,x+y+z)} +\delta_{v+x+y+z,0}\tens\bw) (1\tens\xi_{T(v+1+x,y,z)}) (12)
\\
\hfill +(\xi_{T(v+w+x,y,z)} +\delta_{v+w+x+z,0}\tens\bw) (1\tens\xi_{T(v,w,x+1+z)}) \quad
\\
=\delta_{v+x+y+z,0} (1\tens\bw\xi_{T(1,0,0)}) (12) +\delta_{v+w+x+z,0}\tens\bw\xi_{T(0,0,1)}
\end{multline*}
precomposed with $\oin$ and postcomposed with $\opr\tens\opr\tens\opr$, taking into account \eqref{dia-cooperad-3-OOOOOOO}.

Let us prove that \(m^{\Cobar C}_0d^{\Cobar C}=0\), that is,
\begin{multline}
\bigl\langle \1 \rto\bw C[-1](1) \rTTo^{\xi_1\opr} \bar C[-1](1) \rto{\bar\xi_\circ} \1 \bigr\rangle +\bigl\langle \1 \rto\bw C[-1](1) \rTTo^{\xi_1\opr} \bar C[-1](1) \rto{\bar\xi_{\tau[1]}} \bar C[-1](1) \bigr\rangle
\\
+\bigl\langle \1 \rto\bw C[-1](1) \rTTo^{\xi_{T(0,1,0)}} C[-1](1)\tens C[-1](1) \rTTo^{\opr\tens\opr\bar\xi_\circ+\opr\bar\xi_\circ\tens\opr} \bar C[-1](1) \bigr\rangle_{T(0,1,0)}
\\
+\sum_{x+y+z=1} \bigl\langle \1 \rto\bw C[-1](1) \rTTo^{\xi_1\opr} \bar C[-1](1) \rto{\bar\xi_{T(x,y,z)}} \bar C[-1](y)\tens\bar C[-1](x+1+z) \bigr\rangle
\\
+\bigl\langle \1 \rto\bw C[-1](1) \rTTo^{\xi_{T(0,1,0)}} C[-1](1)\tens C[-1](1) \rTTo^{\opr\tens\opr\bar\xi_{\tau[1]}+\opr\bar\xi_{\tau[1]}\tens\opr} \bar C[-1](1)\tens\bar C[-1](1) \bigr\rangle
\\
+\bigl\langle \1 \rto\bw C[-1](1) \rTTo^{\xi_{T(1,0,0)}} C[-1](0)\tens C[-1](2) \rTTo^{\opr\tens\opr\bar\xi_{\tau[2]}+\opr\bar\xi_{\tau[0]}\tens\opr} \bar C[-1](0)\tens\bar C[-1](2) \bigr\rangle
\\
+\bigl\langle \1 \rto\bw C[-1](1) \rTTo^{\xi_{T(0,0,1)}} C[-1](0)\tens C[-1](2) \rTTo^{\opr\tens\opr\bar\xi_{\tau[2]}+\opr\bar\xi_{\tau[0]}\tens\opr} \bar C[-1](0)\tens\bar C[-1](2) \bigr\rangle
\\
\hskip\multlinegap +\sum_{x+y+z=1} \bigl\langle \1 \rto\bw C[-1](1) \rTTo^{\xi_{T(0,1,0)}} C[-1](1)\tens C[-1](1) \hfill
\\
\hfill \rTTo^{\opr\tens\opr\bar\xi_{T(x,y,z)}} \bar C[-1](1)\tens\bar C[-1](y)\tens\bar C[-1](x+1+z) \bigr\rangle \quad
\\
\hskip\multlinegap +\sum_{x+y+z=1} \bigl\langle \1 \rto\bw C[-1](1) \rTTo^{\xi_{T(0,1,0)}} C[-1](1)\tens C[-1](1) \hfill
\\
\hfill \rTTo^{\opr\bar\xi_{T(x,y,z)}\tens\opr} \bar C[-1](y)\tens\bar C[-1](x+1+z)\tens\bar C[-1](1) \bigr\rangle \quad
\\
\hskip\multlinegap +\sum_{x+y+z=2} \bigl\langle \1 \rto\bw C[-1](1) \rTTo^{\xi_{T(1,0,0)}} C[-1](0)\tens C[-1](2) \hfill
\\
\hfill \rTTo^{\opr\tens\opr\bar\xi_{T(x,y,z)}} \bar C[-1](0)\tens\bar C[-1](y)\tens\bar C[-1](x+1+z) \bigr\rangle \quad
\\
\hskip\multlinegap +\sum_{x+y+z=2} \bigl\langle \1 \rto\bw C[-1](1) \rTTo^{\xi_{T(0,0,1)}} C[-1](0)\tens C[-1](2) \hfill
\\
\hfill \rTTo^{\opr\tens\opr\bar\xi_{T(x,y,z)}} \bar C[-1](0)\tens\bar C[-1](y)\tens\bar C[-1](x+1+z) \bigr\rangle \quad
\\
+\bigl\langle \1 \rto\bw C[-1](1) \rTTo^{\xi_{T(1,0,0)}} C[-1](0)\tens C[-1](2) 
\rTTo^{\opr\bar\xi_{T(0,0,0)}\tens\opr} \bar C[-1](0)\tens\bar C[-1](1)\tens\bar C[-1](2) \bigr\rangle
\\
+\bigl\langle \1 \rto\bw C[-1](1) \rTTo^{\xi_{T(0,0,1)}} C[-1](0)\tens C[-1](2) 
\rTTo^{\opr\bar\xi_{T(0,0,0)}\tens\opr} \bar C[-1](0)\tens\bar C[-1](1)\tens\bar C[-1](2) \bigr\rangle
\\
=0.
\label{eq-m0d-Cobar-C}
\end{multline}

The postcomposition of \eqref{eq-m0d-Cobar-C} with $\pr_\tau$ for $\tau=\circ$ follows from \(\xi_1\xi_0=0\).
The postcomposition of \eqref{eq-m0d-Cobar-C} with $\pr_\tau$ for $\tau=\tau[1]$ follows from \(\xi_1^2 +\xi_{T(0,1,0)}(1\tens\xi_0 +\xi_0\tens1) =0\).

The postcomposition of \eqref{eq-m0d-Cobar-C} with $\pr_\tau$ for $\tau=T(0,1,0)$ follows from 
\begin{multline*}
\bw\xi_1\xi_{T(0,1,0)}(\opr\tens\opr) +\bw\xi_{T(0,1,0)}[1\tens(1-\beps\cdot\bw)\xi_1 +(1-\beps\cdot\bw)\xi_1\tens1](\opr\tens\opr)
\\
=(\bw\tens\bw)(1\tens\xi_1+\xi_1\tens1)(\opr\tens\opr) =0.
\end{multline*}

The postcomposition of \eqref{eq-m0d-Cobar-C} with $\pr_\tau$ for $\tau=T(1,0,0)$ follows from 
\[ \xi_1\xi_{T(1,0,0)} +\xi_{T(1,0,0)}(1\tens\xi_1 +\xi_1\tens1) =0.
\]

The postcomposition of \eqref{eq-m0d-Cobar-C} with $\pr_\tau$ for $\tau=T(0,0,1)$ follows from 
\[ \xi_1\xi_{T(0,0,1)} +\xi_{T(0,0,1)}(1\tens\xi_1 +\xi_1\tens1) =0.
\]

The postcomposition of \eqref{eq-m0d-Cobar-C} with $\pr_\tau$ for
\(\tau=
\vstretch 40
\begin{tangler}
\s \\ \nw1\s \\ \node
\end{tangler}
\)
follows from 
\[ \xi_{T(1,0,0)}(1\tens\xi_{T(1,1,0)}) +\xi_{T(1,0,0)}(\xi_{T(0,0,0)}\tens1) =0.
\]
This is equation~\eqref{dia-cooperad-4-OOOOO} for $v=1$, $w=x=y=z=0$.

The postcomposition of \eqref{eq-m0d-Cobar-C} with $\pr_\tau$ for
\(\tau=
\vstretch 40
\begin{tangle}
\s \\ \s\ne1 \\ \node
\end{tangle}
\)
follows from 
\[ \xi_{T(0,0,1)}(1\tens\xi_{T(0,1,1)}) +\xi_{T(0,0,1)}(\xi_{T(0,0,0)}\tens1) =0.
\]
This is equation~\eqref{dia-cooperad-4-OOOOO} for $v=w=x=y=0$, $z=1$.

The postcomposition of \eqref{eq-m0d-Cobar-C} with $\pr_\tau$ for
\(\tau=
\vstretch 40
\begin{tangles}{rl}
&\Step\node \\
\node &\step\sw1 \\
\nw1\n &\ne1
\end{tangles}
\)
follows from 
\[ \xi_{T(1,0,0)}(1\tens\xi_{T(2,0,0)})(12) +\xi_{T(1,0,0)}(1\tens\xi_{T(1,0,1)}) =0.
\]
This is equation~\eqref{dia-cooperad-3-OOOOOOO} for $v=1$, $w=x=y=z=0$.

The postcomposition of \eqref{eq-m0d-Cobar-C} with $\pr_\tau$ for
\(\tau=
\vstretch 40
\begin{tanglec}
\s \\
\node\Step \\
\nw1\n\ne1
\end{tanglec}
\)
follows from 
\[ \xi_{T(0,0,1)}(1\tens\xi_{T(0,0,2)}) +\xi_{T(0,0,1)}(1\tens\xi_{T(1,0,1)})(12) =0.
\]
This is equation~\eqref{dia-cooperad-3-OOOOOOO} for $v=w=x=y=0$, $z=1$.

The postcomposition of \eqref{eq-m0d-Cobar-C} with $\pr_\tau$ for
\(\tau=
\vstretch 40
\begin{tangles}{rl}
&\Step\node \\
\node\step &\step\sw1 \\
\nw1\n &\ne1
\end{tangles}
\)
follows from 
\[ \xi_{T(1,0,0)}(1\tens\xi_{T(0,0,2)}) +\xi_{T(0,0,1)}(1\tens\xi_{T(2,0,0)})(12) =0.
\]
This is equation~\eqref{dia-cooperad-3-OOOOOOO} for $v=w=0$, $x=1$, $y=z=0$.

The postcomposition of \eqref{eq-m0d-Cobar-C} with $\pr_\tau$ for
\(\tau=
\vstretch 40
\begin{tangler}
\n \\
\node\step\id \\
\nw1\n
\end{tangler}
\)
follows from 
\[ \bw\xi_{T(0,1,0)}[1\tens(1-\beps\cdot\bw)\xi_{T(0,0,1)}](\opr\tens\opr\tens\opr) +\bw\xi_{T(0,1,0)}(1\tens\xi_{T(1,1,0)})(\opr\tens\opr\tens\opr)(12) =0.
\]
This follows from equation~\eqref{dia-cooperad-3-OOOOOOO} for $v=w=x=0$, $y=1$, $z=0$.

The postcomposition of \eqref{eq-m0d-Cobar-C} with $\pr_\tau$ for
\(\tau=
\vstretch 40
\begin{tangler}
\nw1\node\step\s \\
\nw1\n
\end{tangler}
\)
follows from 
\[ \bw\xi_{T(0,1,0)}[1\tens(1-\beps\cdot\bw)\xi_{T(1,0,0)}](\opr\tens\opr\tens\opr)(12) +\bw\xi_{T(1,0,0)}(1\tens\xi_{T(0,1,1)})(\opr\tens\opr\tens\opr) =0.
\]
This follows from equation~\eqref{dia-cooperad-3-OOOOOOO} for $v=0$, $w=1$, $x=y=z=0$.

The postcomposition of \eqref{eq-m0d-Cobar-C} with $\pr_\tau$ for
\(\tau=
\vstretch 40
\begin{tangle}
\node \\
\n\sw1 \\
\n
\end{tangle}
\)
follows from 
\[ \bw\xi_{T(0,1,0)}[(1-\beps\cdot\bw)\xi_{T(0,0,1)}\tens1](\opr\tens\opr\tens\opr) +\bw\xi_{T(0,0,1)}(1\tens\xi_{T(0,2,0)})(\opr\tens\opr\tens\opr) =0.
\]
This follows from equation~\eqref{dia-cooperad-4-OOOOO} for $v=w=x=0$, $y=1$, $z=0$.

The postcomposition of \eqref{eq-m0d-Cobar-C} with $\pr_\tau$ for
\(\tau=
\vstretch 40
\begin{tangler}
\node \\
\nw1\n \\
\n
\end{tangler}
\)
follows from 
\[ \bw\xi_{T(0,1,0)}[(1-\beps\cdot\bw)\xi_{T(1,0,0)}\tens1](\opr\tens\opr\tens\opr) +\bw\xi_{T(0,0,1)}(1\tens\xi_{T(0,2,0)})(\opr\tens\opr\tens\opr) =0.
\]
This follows from equation~\eqref{dia-cooperad-4-OOOOO} for $v=0$, $w=1$, $x=y=z=0$.

The postcomposition of \eqref{eq-m0d-Cobar-C} with $\pr_\tau$ for
\(\tau=
\vstretch 40
\begin{tangler}
\n \\
\n \\
\n
\end{tangler}
\)
follows from 
\begin{multline*}
\bw\xi_{T(0,1,0)}[1\tens(1-\beps\cdot\bw)\xi_{T(0,1,0)}](\opr\tens\opr\tens\opr) 
\\
+\bw\xi_{T(0,1,0)}[(1-\beps\cdot\bw)\xi_{T(0,1,0)}\tens1](\opr\tens\opr\tens\opr) =0.
\end{multline*}
This follows from equation~\eqref{dia-cooperad-4-OOOOO} for $v=w=0$, $x=1$, $y=z=0$.

Thus $\Cobar C$ is a unit-complemented curved operad.
\end{proof}

\begin{proposition}
The map from \propref{pro-map-curved-cooperad-curved-operad} is a functor \(\Cobar:\ucCoop\to\ucOp\), the cobar construction.
The functor takes a morphism \(g=(g_1,g_0):C\to D\in\ucCoop\) to
\[ \sfCobar g =\sff =(\sff_1,\sff_0): \bar{C}[-1]\TT \to \bar{D}[-1]\TT,
\]
where the restriction of the algebra homomorphism \(\sfCobar_1 g=\sff_1=\bar{g}\) is the sum of maps
\begin{align*}
\bar g_1 =g_1\big| &=\bigl( \bar{C}[-1] \rMono^\oin C[-1] \rTTo^{g_1} D[-1] \rTTo^\opr \bar{D}[-1] \bigr),
\\
\bar g_0 =g_0' &=\bigl( \bar{C}[-1] \rMono^\oin C[-1] \rTTo^{g_0} \1 \bigr),
\end{align*}
and the degree~1 element is
\begin{equation*}
\sfCobar_0g =\sff_0 =\bigl( \1 \rTTo^\bw C[-1] \rTTo^{g_0+g_1} D[-1]\TT \rTTo^{\opr\TT} \bar{D}[-1]\TT \bigr).
\end{equation*}
\end{proposition}

\begin{proof}
Given a morphism \(g=(g_1,g_0):C\to D\in\ucCoop\) we check that
\[ \sfCobar g =\sff: \bar{C}[-1]\TT \to \bar{D}[-1]\TT
\]
is a morphism of $\ucOp$.
All terms of equation \eqref{eq-OPP-OPPPP} are $\bar{g}$\n-derivations.
Hence, the equation is equivalent to its restriction to $\bar{C}[-1]$:
\begin{multline*}
\bigl\langle \bar{C}[-1](n) \rMono^\oin C[-1](n) \rTTo^{g_1\opr} \bar{D}[-1](n) \rTTo^{\bar\xi_\circ+\bar\xi_{\tau[n]}+\sum_{x+y+z=n}\bar\xi_{T(x,y,z)}} \bar{D}[-1]\TT(n) \bigr\rangle
\\
\hskip\multlinegap +\bigl\langle \bar{C}[-1](n) \rMono^\oin C[-1](n) \rTTo^{g_1\opr} \bar{D}[-1](n) \rTTo^{1\tens\bw} \bar{D}[-1](n)\tens C[-1](1) \hfill
\\
\hfill \rTTo^{1\tens(g_0+g_1\opr)} \bar{D}[-1](n)\tens(\1\oplus\bar{D}[-1](1)) \rMono^{\inj_{\tau[n]}+\inj_{T(0,n,0)}} \bar{D}[-1]\TT(n) \bigr\rangle \quad
\\
\hskip\multlinegap -\sum_{x+1+z=n} \bigl\langle \bar{C}[-1](n) \rMono^\oin C[-1](n) \rTTo^{g_1\opr} \bar{D}[-1](n) \rTTo^{\bw\tens1} C[-1](1)\tens\bar{D}[-1](n) \hfill
\\
\hfill \rTTo^{(g_0+g_1\opr)\tens1} (\1\oplus\bar{D}[-1](1))\tens\bar{D}[-1](n) \rMono^{\inj_{\tau[n]}+\inj_{T(x,1,z)}} \bar{D}[-1]\TT(n) \bigr\rangle \quad		
\\
\hskip\multlinegap =\delta_{n,1} \bigl\langle \bar{C}[-1](n) \rMono^\oin C[-1](n) \rto{\xi_0} \1 \rMono^{\inj_\circ} \bar{D}[-1]\TT(n) \bigr\rangle \hfill
\\
\hskip\multlinegap +\bigl\langle \bar{C}[-1](n) \rMono^\oin C[-1](n) \rto{\xi_1} C[-1](n) \rTTo^{g_0+g_1\opr} \1\oplus\bar{D}[-1](n) \rMono^{\inj_\circ+\inj_{\tau[n]}} \bar{D}[-1]\TT(n) \bigr\rangle \hfill		
\\
\hskip\multlinegap +\sum_{x+y+z=n} \bigl\langle \bar{C}[-1](n) \rMono^\oin C[-1](n) \rTTo^{\xi_{T(x,y,z)}(\opr\oin\tens\opr\oin)} C[-1](y)\tens C[-1](x+1+z) \hfill
\\
\rTTo^{(g_0+g_1\opr)\tens(g_0+g_1\opr)} (\1\oplus\bar{D}[-1](y))\tens(\1\oplus\bar{D}[-1](x+1+z)) \hookrightarrow \bar{D}[-1]\TT(n) \bigr\rangle.
\end{multline*}
Let us expand this equation to individual summands:
\begin{multline*}
\delta_{n,1}\oin g_1\xi_0\inj_\circ +\oin g_1\xi_1\opr\inj_{\tau[n]} +\sum_{x+y+z=n}\oin g_1\xi_{T(x,y,z)}(\opr\tens\opr)\inj_{T(x,y,z)}
\\
+\oin g_1\opr\inj_{\tau[n]}(\bw g_0) +(\oin g_1\opr\tens\bw g_1\opr)\inj_{T(0,n,0)} -n(\bw g_0)\oin g_1\opr\inj_{\tau[n]}
\\
\hfill -\sum_{x+1+z=n} (\bw g_1\opr\tens\oin g_1\opr)\inj_{T(x,1,z)} \quad
\\
\hskip\multlinegap =\delta_{n,1}\oin\xi_0\inj_\circ +\delta_{n,1}\oin\xi_1g_0\inj_\circ +\oin\xi_1g_1\opr\inj_{\tau[n]} +\delta_{n,1}\oin\xi_{T(0,1,0)}(g_0\tens g_0)\inj_\circ \hfill
\\
+\sum_{x+1+z=n}\oin\xi_{T(x,1,z)}(g_0\tens g_1\opr)\inj_{\tau[n]} +\oin\xi_{T(0,n,0)}(g_1\opr\tens g_0)\inj_{\tau[n]}
\\
+\sum_{x+y+z=n}\oin\xi_{T(x,y,z)}(g_1\opr\tens g_1\opr)\inj_{T(x,y,z)} +\delta_{n,1}\oin g_0(\bw g_0)\inj_\circ +\delta_{n,1}\oin g_0\bw g_1\opr\inj_{\tau[n]}
\\
+\oin g_1\opr(\bw g_0)\inj_{\tau[n]} +(\oin g_1\opr\tens\bw g_1\opr)\inj_{T(0,n,0)} -\delta_{n,1}(\bw g_0)\oin g_0\inj_\circ 
\\
-\sum_{x+1+z=n}(\bw g_0)\oin g_1\opr\inj_{\tau[n]} -\delta_{n,1}\oin g_0\bw g_1\opr\inj_{\tau[n]} -\sum_{x+1+z=n} (\bw g_1\opr\tens\oin g_1\opr)\inj_{T(x,1,z)}
\\
=0.
\end{multline*}
The summands starting with $\oin$ and ending in the component indexed by $\circ$ cancel each other due to identity
\[ g_1\xi_0 =\xi_0 +\xi_1g_0 +\xi_{T(0,1,0)}(g_0\tens g_0) +g_0(\bw g_0) -(\bw g_0)g_0.
\]
The summands starting with $\oin$ and ending with $\opr\inj_{\tau[n]}$ cancel each other due to identity
\begin{multline*}
g_1\xi_1 +g_1(\bw g_0) -n(\bw g_0)g_1 =\xi_1g_1 +\sum_{x+1+z=n}\xi_{T(x,1,z)}(g_0\tens g_1)
\\
+\xi_{T(0,n,0)}(g_1\tens g_0) +\delta_{n,1}g_0\bw g_1 +g_1(\bw g_0) -n(\bw g_0)g_1 -\delta_{n,1}g_0\bw g_1.
\end{multline*}
The six remaining summands cancel each other for obvious reasons.

We also have to prove equation~\eqref{eq-1P-1PP-111PPP-1OP}.
In detail,
\begin{multline*}
\bw\xi^D_0\inj_\circ +\bw\xi^D_1\opr\inj_{\tau[1]} +\bw\xi^D_{T(0,1,0)}(\opr\tens\opr)\inj_{T(0,1,0)} +\bw\xi^D_{T(1,0,0)}(\opr\tens\opr)\inj_{T(1,0,0)}
\\
+\bw\xi^D_{T(0,0,1)}(\opr\tens\opr)\inj_{T(0,0,1)} +(\bw g_1-\bw)\xi^D_0\inj_\circ +(\bw g_1-\bw)\xi^D_1\opr\inj_{\tau[1]}
\\
+\sum_{x+y+z=1}(\bw g_1-\bw)\xi^D_{T(x,y,z)}(\opr\tens\opr)\inj_{T(x,y,z)} +(\bw g_0\tens\bw g_0)\inj_\circ +\bw g_1\opr(\bw g_0)\inj_{\tau[1]}
\\
\hfill +(\bw g_0)\bw g_1\opr\inj_{\tau[1]} +(\bw g_1\opr\tens\bw g_1\opr)\inj_{T(0,1,0)} \quad
\\
\hskip\multlinegap =\bw\xi^C_0\inj_\circ +\bw\xi^C_1g_0\inj_\circ +\bw\xi^C_1g_1\opr\inj_{\tau[1]} +(\bw\xi^C_{T(0,1,0)}+\bw\tens\bw)\times \hfill
\\
\times[(g_0\tens g_0)\inj_\circ +(g_0\tens g_1\opr)\inj_{\tau[1]} +(g_1\opr\tens g_0)\inj_{\tau[1]} +(g_1\opr\tens g_1\opr)\inj_{T(0,1,0)}]
\\
+\bw\xi^C_{T(1,0,0)}(g_1\opr\tens g_1\opr)\inj_{T(1,0,0)} +\bw\xi^C_{T(0,0,1)}(g_1\opr\tens g_1\opr)\inj_{T(0,0,1)}.
\end{multline*}
We have used here the following identities:
\begin{gather*}
\bw\xi^C_{T(0,1,0)}[(1-\beps\bw)\tens(1-\beps\bw)] =\bw\xi^C_{T(0,1,0)} +\bw\tens\bw,
\\
\bw g_1\opr\oin =\bw g_1-\bw^D.
\end{gather*}
Using equations~\eqref{eq-xi0-xi1g0-g1xi} we prove the above identity.
Thus, $\Cobar g$ is a morphism of $\UCCOp$.

The identity morphism $g=(\id,0)$ is mapped to the identity morphism $\sfCobar g=(\id,0)$.
Let us verify that $\Cobar$ agrees with the composition.
If $h=\bigl(C\rto fD\rto gE\bigr)$ in $\ucCoop$, \(h_1=f_1g_1\), \(h_0=f_0+f_1g_0\), then
\[ \sfCobar_1f\cdot\sfCobar_1g =\bar f\bar g =\bar h =\sfCobar_1h.
\]
In fact, on generators we have
\begin{align*}
\bar f\bar g &=\bigl\langle \bar C[-1] \rTTo^{\oin(f_0,f_1\opr)} \1\oplus\bar D[-1]
\rTTo^{\bigl(\begin{smallmatrix}
        1 & 0 \\ \oin g_0 & \oin g_1\opr
       \end{smallmatrix}\bigr)}
\1\oplus\bar E[-1] \bigr\rangle
\\
&=(\oin f_0+\oin f_1\opr\oin g_0,\oin f_1\opr\oin g_1\opr) =\oin(f_0+f_1g_0,f_1g_1\opr) =\oin(h_0,h_1\opr) =\bar h.
\end{align*}

Furthermore,
\[ \sfCobar_0h =\sfCobar_0f\sfCobar_1g +\sfCobar_0g.
\]
In fact,
\begin{multline*}
\bigl\langle \1 \rto\bw C[-1] \rTTo^{(h_0,h_1\opr)} \1\oplus\bar E[-1] \bigr\rangle
\\
=\bigl\langle \1 \rto\bw C[-1] \rTTo^{(f_0,f_1\opr)} \1\oplus\bar D[-1]
\rTTo^{\bigl(\begin{smallmatrix}
        1 & 0 \\ \oin g_0 & \oin g_1\opr
       \end{smallmatrix}\bigr)}
\1\oplus\bar E[-1] \bigr\rangle
\\
+\bigl\langle \1 \rto\bw D[-1] \rTTo^{(g_0,g_1\opr)} \1\oplus\bar E[-1] \bigr\rangle
\end{multline*}
due to
\begin{multline*}
\bw^C(f_0+f_1\opr\oin g_0,f_1\opr\oin g_1\opr) +\bw^D(g_0,g_1\opr)
\\
=(\bw^C(f_0+f_1g_0) -\bw^Dg_0 +\bw^Dg_0, \bw^Cf_1g_1\opr -\bw^Dg_1\opr +\bw^Dg_1\opr)
\\
=\bw^C(f_0+f_1g_0,f_1g_1\opr) =\bw^C(h_0,h_1\opr).
\end{multline*}
The functor \(\Cobar:\ucCoop\to\ucOp\) is constructed.
\end{proof}

\section{Bar construction}
Let us describe a functor\index{TTsymb}{Bar@$\Bbar$} \(\Bbar:\UCCOp\to\CACoop\), the bar construction.
Let \((\co,m,d,m_0,\eta,\sfv)\) be a unit-complemented curved operad.
Decompose the idempotent \(1-\sfv\cdot\eta:\co\to\co\) into a projection \(\opr:\co\to\bar\co\) and an injection \(\oin:\bar\co\to\co\) so that \(\oin\opr=1_{\bar\co}\).
The morphisms \(\sigma^{-1}\opr\sigma:\co[1]\to\bar\co[1]\) and \(\sigma^{-1}\oin\sigma:\bar\co[1]\to\co[1]\) are also denoted $\opr$ and $\oin$ by abuse of notations.
They split the idempotent \(1-\bv\cdot\bfeta:\co[1]\to\co[1]\).
The source and the target of $\opr$ and $\oin$ can be recovered from the context.

\begin{proposition}\label{pro-bar-construction}
There is a map \(\Bbar:\Ob\UCCOp\to\Ob\CACoop\).
It assigns to a unit-complemented curved operad $\co$ the graded cooperad \(\Bbar\co=(\bar\co[1]\botto,\Delta_{T(-,-,-)},\eps=\pr_\circ)\) corresponding via \propref{pro-aug-equiv-non-counital} to the conilpotent non-counital graded cooperad \((\bar\co[1]\bott,\bar\Delta_{T(-,-,-)})\).
The splitting of the counit is \(\sfw^{\Bbar\co}=\inj_\circ:\1\hookrightarrow\bar\co[1]\botto\), the degree~1 coderivation \(d^{\Bbar\co}=\bar{b}:\bar\co[1]\botto\to\bar\co[1]\botto\) corresponds via \corref{cor-coderivations-bottom} to the map \(\tilde{b}=\bar{b}\cdot\eps:\bar\co[1]\botto\to\bar\co[1]\), whose non-vanishing components are
\begin{align*}
\bar b_{T(x,y,z)} &= \bigl\langle \bar\co[1](y)\tens\bar\co[1](x+1+z) \rMono^{\oin\tens\oin} \co[1](y)\tens\co[1](x+1+z)
\\
&\hspace{8em} \rTTo^{b_{T(x,y,z)}} \co[1](x+y+z) \rTTo^\opr \bar\co[1](x+y+z) \bigr\rangle,
\\
\bar b_{\tau[n]} &= \bigl\langle \bar\co[1](n) \rMono^\oin \co[1](n) \rTTo^{b_1} \co[1](n) \rTTo^\opr \bar\co[1](n) \bigr\rangle,
\\
\bar b_\circ &= \bigl\langle \1 \rTTo^{b_0} \co[1](1) \rTTo^\opr \bar\co[1](1) \bigr\rangle.
\end{align*}
The curvature functional \(\bdelt_0^{\Bbar\co}:\bar\co[1]\botto\to\1\) is
\[ \bdelt_0^{\Bbar\co} =-\bigl( \bar\co[1]\botto \rMono^{\oin\botto} \co[1]\botto \rTTo^{\check{b}} \co[1] \rTTo^\bv \1 \bigr),
\]
where the relevant non-vanishing components of $\check{b}$ are
\begin{align*}
b_{T(x,y,z)} &: \co[1](y)\tens\co[1](x+1+z) \to \co[1](1), \qquad x+y+z=1,
\\
b_{\tau[1]} =b_1 &: \co[1](1) \to \co[1](1),
\\
b_\circ =b_0 &:\1 \to \co[1](1).
\end{align*}
\end{proposition}

\begin{proof}
Note that \(\sfw^{\Bbar\co}\) is a graded cooperad homomorphism and the cooperad \(\overline{\bar\co[1]\botto}=\bar\co[1]\bott\) is conilpotent.
It is important also that
\begin{multline*}
\hfill [1\tens(1-\bv\bfeta)]b_{T(x,y,z)} =b_{T(x,y,z)} -\delta_{x+z,0}\tens\bv, \hfill
\\
[(1-\bv\bfeta)\tens1]b_{T(x,y,z)} =b_{T(x,y,z)} +\delta_{y,1}\bv\tens1,
\\
[(1-\bv\bfeta)\tens(1-\bv\bfeta)]b_{T(x,y,z)} =b_{T(x,y,z)} -\delta_{x+z,0}\tens\bv +\delta_{y,1}\bv\tens1 -\delta_{x+z,0}\delta_{y,1}(\bv\tens\bv)\bfeta:
\\
\co[1](y)\tens\co[1](x+1+z) \to \co[1](x+y+z).
\end{multline*}

Let us verify the necessary identities.
Both sides of the equation
\[ (d^{\Bbar\co})^2 =\Delta_{T(0,n,0)}(1\tens\bdelt^{\Bbar\co}_0) -\sum_{x+1+z=n} \Delta_{T(x,1,z)}(\bdelt^{\Bbar\co}_0\tens1): \bar\co[1]\botto(n) \to \bar\co[1]\botto(n)
\]
are degree 2 coderivations.
Due to \corref{cor-coderivations-bottom} the equation is equivalent to its composition with \(\pr_{\tau[n]}:\bar\co[1]\botto(n)\to\bar\co[1](n)\).
Taking into account \eqref{eq-xi'-Delta'} and \eqref{eq-xi''-Delta''}, to
\begin{multline}
\Bigl[ \coprod_{\tau\in\tr(n)} \, \bigotimes_{v\in\IV(\tau)} \bar\co[1]|v| \rto{\Delta'} \coprod_{\tau\in\tr(n)} \, \coprod_{\text{subtree }r\subset\tau}^{\IV(r)\ne\emptyset} \, \bigotimes_{y\in\IV(\tau/r)} \, \bigotimes_{q\in(\IV(\tau)\to\IV(\tau/r))^{-1}(y)}\bar\co[1]|q|
\\
\hfill \rTTo^{\sum\tens^{y\in\IV(t)}f_\tau(y)} \coprod_{t\in\tr(n)} \, \bigotimes_{y\in\IV(t)} \bar\co[1]|y| \rto{\tilde b} \bar\co[1](n) \Bigr] \quad
\\
\hskip\multlinegap +\Bigl[ \coprod_{\tau\in\tr(n)} \, \bigotimes_{v\in\IV(\tau)} \bar\co[1]|v| \rto{\Delta''} \coprod_{t\in\tr(n)} \, \coprod_{i\in\IV(t)} \, \bigotimes_{y\in\IV(t)} \, \bigotimes_{q\in\IV(t_y^i)}\bar\co[1]|q| \hfill
\\
\hfill \rTTo^{\coprod_t\sum_i\tens^{y\in\IV(t)}g_i(y)} \coprod_{t\in\tr(n)} \, \bigotimes_{y\in\IV(t)} \bar\co[1]|y| \rto{\tilde b} \bar\co[1](n) \Bigr] \quad
\\
\hskip\multlinegap =\Bigl[ \coprod_{\tau\in\tr(n)} \, \bigotimes_{v\in\IV(\tau)} \bar\co[1]|v| \rTTo^{\Delta_{T(0,n,0)}} \coprod_{\theta\in\tr(1)}^{\tau[n]\underset{0,n,0}\bull\theta=\tau} \bar\co[1](n)\tens \bigotimes_{v\in\IV(\theta)} \bar\co[1]|v| \rTTo^{1\tens\bdelt^{\Bbar\co}_0} \bar\co[1](n) \Bigr] \hfill
\\
-\sum_{x+1+z=n} \Bigl[ \coprod_{\tau\in\tr(n)} \, \bigotimes_{v\in\IV(\tau)} \bar\co[1]|v| \rTTo^{\Delta_{T(x,1,z)}} \coprod_{\theta\in\tr(1)}^{\theta\underset{x,1,z}\bull\tau[n]=\tau} \Bigl( \bigotimes_{v\in\IV(\theta)} \bar\co[1]|v| \Bigr)\tens\bar\co[1](n) \rTTo^{\bdelt^{\Bbar\co}_0\tens1} \bar\co[1](n) \Bigr],
\label{eq-long-eqn-d2}
\end{multline}
where \(t=\tau/r\), \(f_\tau(y)=\id\) if \(y\ne[\IV(r)]\), and \(f_\tau(y)=\bar b_r\) if \(y=[\IV(r)]\); \(t_y^i=\tau[|y|]\) and \(g_i(y)=\id\) if $y\ne i$, and \(t_i^i=\circ\), \(g_i(i)=\bar b_\circ:\1\to\bar\co[1](1)\), $\Delta''$ takes a summand indexed by $\tau$ to each summand indexed by \((t,i)\) such that \(t^{\{i\}}=\tau\).
Both hand sides vanish on all summands indexed by $\tau$ unless $\tau$ has at most 3 internal vertices.
Let us consider all possible $\tau$ separately.

For $\tau=\circ$ equation \eqref{eq-long-eqn-d2} reduces to
\[ \bar b_\circ \cdot \bar b_{\tau[1]} =0: \1 \to \bar\co[1](1),
\]
which follows from the obvious equation \(b_0(1-\bv\cdot\bfeta)b_1\opr=0\).

For $\tau=\tau[n]$ equation \eqref{eq-long-eqn-d2} reduces to
\begin{align*}
&\bigl\langle \bar\co[1](n) \rto{\bar b_{\tau[n]}} \bar\co[1](n) \rto{\bar b_{\tau[n]}} \bar\co[1](n) \bigr\rangle +\bigl\langle \bar\co[1](n) \rTTo^{1\tens\bar b_\circ} \bar\co[1](n)\tens\bar\co[1](1) \rTTo^{\bar b_{T(0,n,0)}} \bar\co[1](n) \bigr\rangle
\\
&\hspace{11.9em}+\sum_{x+1+z=n} \bigl\langle \bar\co[1](n) \rTTo^{\bar b_\circ\tens1} \bar\co[1](1)\tens\bar\co[1](n) \rTTo^{\bar b_{T(x,1,z)}} \bar\co[1](n) \bigr\rangle
\\
&=-\bigl\langle \bar\co[1](n) \rTTo^{1\tens b_\circ} \bar\co[1](n)\tens\co[1](1) \rTTo^{1\tens\bv} \bar\co[1](n) \bigr\rangle
\\
&\hspace{11.9em}+\sum_{x+1+z=n} \bigl\langle \bar\co[1](n) \rTTo^{b_\circ\tens1} \co[1](1)\tens\bar\co[1](n) \rTTo^{\bv\tens1} \bar\co[1](n) \bigr\rangle.
\end{align*}
This equation follows from the computation
\begin{multline*}
b_1(1-\bv\bfeta)b_1 +[1\tens b_0(1-\bv\bfeta)]b_{T(0,n,0)}
+\sum_{x+1+z=n} [b_0(1-\bv\bfeta)\tens1]b_{T(x,1,z)}
\\
=-1\tens b_0\bv +\sum_{x+1+z=n} b_0\bv\tens1: \co[1](n) \to \co[1](n).
\end{multline*}

Let \(x,y,z\in\NN\) and \(n=x+y+z\).
For $\tau=T(x,y,z)$ equation \eqref{eq-long-eqn-d2} reduces to
\begin{multline*}
\bigl\langle \bar\co[1](y)\tens\bar\co[1](x+1+z) \rTTo^{\bar b_{\tau[y]}\tens1} \bar\co[1](y)\tens\bar\co[1](x+1+z) \rTTo^{\bar b_\tau} \bar\co[1](n) \bigr\rangle
\\
+\bigl\langle \bar\co[1](y)\tens\bar\co[1](x+1+z) \rTTo^{1\tens\bar b_{\tau[x+1+z]}} \bar\co[1](y)\tens\bar\co[1](x+1+z) \rTTo^{\bar b_\tau} \bar\co[1](n) \bigr\rangle
\\
\hfill +\bigl\langle \bar\co[1](y)\tens\bar\co[1](x+1+z) \rTTo^{\bar b_\tau} \bar\co[1](n) \rTTo^{\bar b_{\tau[n]}} \bar\co[1](n) \bigr\rangle \quad
\\
=-\delta_{x,0}\delta_{z,0} \bigl\langle \bar\co[1](n)\tens\bar\co[1](1) \rTTo^{1\tens\oin b_1\bv} \bar\co[1](n) \bigr\rangle +\delta_{y,1} \bigl\langle \bar\co[1](1)\tens\bar\co[1](n) \rTTo^{\oin b_1\bv\tens1} \bar\co[1](n) \bigr\rangle.
\end{multline*}
This holds true due to valid identity:
\begin{multline*}
[b_1(1-\bv\bfeta)\tens1]b_\tau +[1\tens b_1(1-\bv\bfeta)]b_\tau +b_\tau(1-\bv\bfeta)b_1 =b_1\bv\tens1 -1\tens b_1\bv
\\
=-\delta_{x,0}\delta_{z,0}\tens b_1\bv +\delta_{y,1}b_1\bv\tens1: \co[1](y)\tens\co[1](x+1+z) \to \co[1](n),\end{multline*}
precomposed with $\oin\tens\oin$ and postcomposed with $\opr$.

There are two types of trees $\tau$ with \(|\IV(t)|=3\).
The tree from \eqref{eq-tangle-3-vertices-(4)} is
\begin{equation}
\tau =\bigl( \mb x \rto{f} \mb{w+1+y} \rto{g} \mb{v+1+z} \rto\con \mb1, \tau(0)\sqcup(\tau(1)-\{w+1\})\sqcup(\tau(2)-\{v+1\}) \bigr),
\label{eq-tau-x-w1y-v1z-1}
\end{equation}
where \(v,w,x,y,z\in\NN\), \(f(\mb x)\subset\{w+1\}\), \(g(\mb{w+1+y})=\{v+1\}\).
Equation~\eqref{eq-long-eqn-d2} takes the form
\begin{multline*}
\bigl\langle \bar\co[1](x)\tens\bar\co[1](w+1+y)\tens\bar\co[1](v+1+z) \rTTo^{\bar b_{T(w,x,y)}\tens1}
\\
\hfill \bar\co[1](w+x+y)\tens\bar\co[1](v+1+z) \rTTo^{\bar b_{T(v,w+x+y,z)}} \bar\co[1](v+w+x+y+z) \bigr\rangle \quad
\\
\hskip\multlinegap +\bigl\langle \bar\co[1](x)\tens\bar\co[1](w+1+y)\tens\bar\co[1](v+1+z) \rTTo^{1\tens\bar b_{T(v,w+1+y,z)}} \hfill
\\
\hfill \bar\co[1](x)\tens\bar\co[1](v+w+1+y+z) \rTTo^{\bar b_{T(v+w,x,y+z)}} \bar\co[1](v+w+x+y+z) \bigr\rangle \quad
\\
\hskip\multlinegap =-\chi(v=w=y=z=0) \bigl\langle \bar\co[1](x)\tens\bar\co[1](1)\tens\bar\co[1](1) \rTTo^{1\tens(\oin\tens\oin)b_{T(0,1,0)}\bv} \bar\co[1](x) \bigr\rangle \hfill
\\
\hfill +\delta_{w+y,0} \delta_{x,1} \bigl\langle \bar\co[1](1)\tens\bar\co[1](1)\tens\bar\co[1](v+1+z) \rTTo^{(\oin\tens\oin)b_{T(0,1,0)}\bv\tens1} \bar\co[1](v+1+z) \bigr\rangle \quad
\\
\hfill +\delta_{x+y,0} \delta_{w,1} \bigl\langle \bar\co[1](0)\tens\bar\co[1](2)\tens\bar\co[1](v+1+z) \rTTo^{(\oin\tens\oin)b_{T(1,0,0)}\bv\tens1} \bar\co[1](v+1+z) \bigr\rangle \quad
\\
+\delta_{w+x,0} \delta_{y,1} \bigl\langle \bar\co[1](0)\tens\bar\co[1](2)\tens\bar\co[1](v+1+z) \rTTo^{(\oin\tens\oin)b_{T(0,0,1)}\bv\tens1} \bar\co[1](v+1+z) \bigr\rangle.
\end{multline*}
Using equation~\eqref{dia-operad-4-OOOOOOO} we deduce the above from the computation
\begin{align*}
[b_{T(w,x,y)}(1 &-\bv\cdot\bfeta)\tens1]b_{T(v,w+x+y,z)} +[1\tens b_{T(v,w+1+y,z)}(1-\bv\cdot\bfeta)]b_{T(v+w,x,y+z)}
\\
&= \delta_{w+x+y,1}b_{T(w,x,y)}\bv\tens1 -\chi(v=w=y=z=0)\tens b_{T(0,1,0)}\bv
\\
&= \chi(w=y=0) \delta_{x,1}b_{T(0,1,0)}\bv\tens1 +\chi(x=y=0) \delta_{w,1}b_{T(1,0,0)}\bv\tens1
\\
&+\chi(w=x=0) \delta_{y,1} b_{T(0,0,1)}\bv\tens1 -\chi(v=w=y=z=0)\tens b_{T(0,1,0)}\bv,
\end{align*}
precomposed with $\oin\tens\oin\tens\oin$ and postcomposed with $\opr$.

The second type of a tree $\tau$ with \(|\IV(t)|=3\) depicted in \eqref{eq-tangles-3-vertices} is
\begin{equation}
\tau =\bigl( \mb{w+y} \rto{f} \mb{v+1+x+1+z} \rto\con \mb1, \tau(0)\sqcup(\tau(1)-\{v+1,v+1+x+1\}) \bigr),
\label{eq-tau-wy-v1x1z-1}
\end{equation}
where \(v,w,x,y,z\in\NN\), \(f(i)=v+1\) if \(1\le i\le w\) and \(f(i)=v+1+x+1\) if \(w<i\le w+y\).
Equation~\eqref{eq-long-eqn-d2} takes the form:
\begin{multline*}
\bigl\langle \bar\co[1](w)\tens\bar\co[1](y)\tens\bar\co[1](v+1+x+1+z) \rTTo^{1\tens\bar b_{T(v+1+x,y,z)}}
\\
\hfill \bar\co[1](w)\tens\bar\co[1](v+1+x+y+z) \rTTo^{\bar b_{T(v,w,x+y+z)}} \bar\co[1](v+w+x+y+z) \bigr\rangle \quad
\\
\hskip\multlinegap +\bigl\langle \bar\co[1](w)\tens\bar\co[1](y)\tens\bar\co[1](v+1+x+1+z) \rTTo^{(12)}_\sim \bar\co[1](y)\tens\bar\co[1](w)\tens\bar\co[1](v+1+x+1+z) \hfill
\\
\hfill \rTTo^{1\tens\bar b_{T(v,w,x+1+z)}} \bar\co[1](y)\tens\bar\co[1](v+w+x+1+z) \rTTo^{\bar b_{T(v+w+x,y,z)}} \bar\co[1](v+w+x+y+z) \bigr\rangle \quad
\\
\hskip\multlinegap =-\chi(v=x=y=z=0) \bigl\langle \bar\co[1](w)\tens\bar\co[1](0)\tens\bar\co[1](2) \rTTo^{1\tens(\oin\tens\oin)b_{T(1,0,0)}\bv} \bar\co[1](w) \bigr\rangle \hfill
\\
\hskip\multlinegap -\chi(v=w=x=z=0) \bigl\langle \bar\co[1](0)\tens\bar\co[1](y)\tens\bar\co[1](2) \rTTo^{(12)}_\sim \hfill
\\
\bar\co[1](y)\tens\bar\co[1](0)\tens\bar\co[1](2) \rTTo^{1\tens(\oin\tens\oin)b_{T(0,0,1)}\bv} \bar\co[1](y) \bigr\rangle.
\end{multline*}
Using equation~\eqref{dia-operad-3-OOOOOOOOOO} we deduce the above from:
\begin{multline*}
[1\tens b_{T(v+1+x,y,z)}(1-\bv\bfeta)]b_{T(v,w,x+y+z)} +(12)[1\tens b_{T(v,w,x+1+z)}(1-\bv\bfeta)]b_{T(v+w+x,y,z)}
\\
=-\chi(v=x+y+z=0)\tens b_{T(1,0,0)}\bv -\chi(v+w+x=z=0)(12)(1\tens b_{T(0,0,1)}\bv):
\\
\co[1](w)\tens\co[1](y)\tens\co[1](v+1+x+1+z) \to \co[1](v+w+x+y+z),
\end{multline*}
precomposed with $\oin\tens\oin\tens\oin$ and postcomposed with $\opr$.

One more identity to prove is
\[ d^{\Bbar\co}\bdelt^{\Bbar\co}_0 =0: \bar\co[1]\botto(1) \to \1.
\]
Taking into account \eqref{eq-xi'-Delta'} and \eqref{eq-xi''-Delta''} we rewrite this equation as
\begin{multline}
\Bigl\langle \coprod_{\tau\in\tr(1)} \, \bigotimes_{v\in\IV(\tau)} \bar\co[1]|v| \rto{\Delta'} \coprod_{\tau\in\tr(1)} \, \coprod_{\text{subtree }r\subset\tau}^{\IV(r)\ne\emptyset} \, \bigotimes_{y\in\IV(\tau/r)} \, \bigotimes_{q\in(\IV(\tau)\to\IV(\tau/r))^{-1}(y)}\bar\co[1]|q|
\\
\hfill \rTTo^{\sum\tens^{y\in\IV(t)}f_\tau(y)} \coprod_{t\in\tr(1)} \, \bigotimes_{y\in\IV(t)} \bar\co[1]|y| \rTTo^{\oin^{\tens\IV(t)}b_t\bv} \1 \Bigr\rangle \quad
\\
\hskip\multlinegap +\Bigl\langle \coprod_{\tau\in\tr(1)} \, \bigotimes_{v\in\IV(\tau)} \bar\co[1]|v| \rto{\Delta''} \coprod_{t\in\tr(1)} \, \coprod_{i\in\IV(t)} \, \bigotimes_{y\in\IV(t)} \, \bigotimes_{q\in\IV(t_y^i)}\bar\co[1]|q| \hfill
\\
\rTTo^{\coprod_t\sum_i\tens^{y\in\IV(t)}g_i(y)} \coprod_{t\in\tr(1)} \, \bigotimes_{y\in\IV(t)} \bar\co[1]|y| \rTTo^{\oin^{\tens\IV(t)}b_t\bv} \1 \Bigr\rangle =0,
\label{eq-long-d-delta-0=0}
\end{multline}
where \(t=\tau/r\), \(f_\tau(y)=\id\) if \(y\ne[\IV(r)]\), and \(f_\tau(y)=\bar b_r\) if \(y=[\IV(r)]\); \(t_y^i=\tau[|y|]\) and \(g_i(y)=\id\) if $y\ne i$, and \(t_i^i=\circ\), \(g_i(i)=\bar b_\circ:\1\to\bar\co[1](1)\), $\Delta''$ takes a summand indexed by $\tau$ to each summand indexed by \((t,i)\) such that \(t^{\{i\}}=\tau\).
The left hand side vanishes on all summands indexed by $\tau$ unless $\tau$ has at most 3 internal vertices.
Let us consider all possible $\tau$ separately.

For $\tau=\circ$ equation \eqref{eq-long-d-delta-0=0} reduces to
\[ \bar b_\circ\oin b_{\tau[1]} \bv =0: \1 \to \1,
\]
which follows from the obvious equation \(b_0(1-\bv\cdot\bfeta)b_1\bv=0\).

For $\tau=\tau[1]$ equation \eqref{eq-long-d-delta-0=0} reduces to
\begin{multline*}
\bigl\langle \bar\co[1](1) \rto{\bar b_{\tau[1]}} \bar\co[1](1) \rTTo^{\oin b_{\tau[1]}\bv} \1 \bigr\rangle +\bigl\langle \bar\co[1](1) \rTTo^{1\tens\bar b_\circ} \bar\co[1](1)\tens\bar\co[1](1) \rTTo^{(\oin\tens\oin)b_{T(0,1,0)}\bv} \1 \bigr\rangle
\\
+\bigl\langle \bar\co[1](1) \rTTo^{\bar b_\circ\tens1} \bar\co[1](1)\tens\bar\co[1](1) \rTTo^{(\oin\tens\oin)b_{T(0,1,0)}\bv} \1 \bigr\rangle =0.
\end{multline*}
This equation follows from the computation
\begin{multline*}
b_1(1-\bv\bfeta)b_1\bv +[1\tens b_0(1-\bv\bfeta)]b_{T(0,1,0)}\bv
+[b_0(1-\bv\bfeta)\tens1]b_{T(0,1,0)}\bv
\\
=-(b_0\bv)\bv +(b_0\bv)\bv =0: \co[1](1) \to \1.
\end{multline*}

Let \(x,y,z\in\NN\) satisfy \(x+y+z=1\).
For $\tau=T(x,y,z)$ equation \eqref{eq-long-d-delta-0=0} reduces to
\begin{multline*}
\bigl\langle \bar\co[1](y)\tens\bar\co[1](x+1+z) \rTTo^{\bar b_{\tau[y]}\tens1} \bar\co[1](y)\tens\bar\co[1](x+1+z) \rTTo^{(\oin\tens\oin)b_\tau\bv} \1 \bigr\rangle
\\
+\bigl\langle \bar\co[1](y)\tens\bar\co[1](x+1+z) \rTTo^{1\tens\bar b_{\tau[x+1+z]}} \bar\co[1](y)\tens\bar\co[1](x+1+z) \rTTo^{(\oin\tens\oin)b_\tau\bv} \1 \bigr\rangle
\\
+\bigl\langle \bar\co[1](y)\tens\bar\co[1](x+1+z) \rTTo^{\bar b_\tau} \bar\co[1](1) \rTTo^{(\oin\tens\oin)b_1\bv} \1 \bigr\rangle =0.
\end{multline*}
This is satisfied as the computation shows:
\begin{multline*}
(\oin\tens\oin)\bigl([b_1(1-\bv\bfeta)\tens1]b_\tau\bv +[1\tens b_1(1-\bv\bfeta)]b_\tau\bv +b_\tau(1-\bv\bfeta)b_1\bv\bigr)
\\
=(\oin\tens\oin)(b_1\bv\tens\bv -\bv\tens b_1\bv) =0 -0 =0: \bar\co[1](y)\tens\bar\co[1](x+1+z) \to \1.
\end{multline*}

Let \(v,w,x,y,z\in\NN\) satisfy \(v+w+x+y+z=1\).
For tree~\eqref{eq-tau-x-w1y-v1z-1} depicted in \eqref{eq-tangle-3-vertices-(4)} equation \eqref{eq-long-d-delta-0=0} takes the form
\begin{multline*}
\bigl\langle \bar\co[1](x)\tens\bar\co[1](w+1+y)\tens\bar\co[1](v+1+z) \rTTo^{\bar b_{T(w,x,y)}\tens1}
\\
\hfill \bar\co[1](w+x+y)\tens\bar\co[1](v+1+z) \rTTo^{(\oin\tens\oin)b_{T(v,w+x+y,z)}\bv} \1 \bigr\rangle \quad
\\
\hskip\multlinegap +\bigl\langle \bar\co[1](x)\tens\bar\co[1](w+1+y)\tens\bar\co[1](v+1+z) \rTTo^{1\tens\bar b_{T(v,w+1+y,z)}} \hfill
\\
\bar\co[1](x)\tens\bar\co[1](v+w+1+y+z) \rTTo^{(\oin\tens\oin)b_{T(v+w,x,y+z)}\bv} \1 \bigr\rangle =0.
\end{multline*}

This equation follows from equation~\eqref{dia-operad-4-OOOOOOO}:
\begin{multline*}
(\oin^{\tens3})\bigl([b_{T(w,x,y)}(1-\bv\cdot\bfeta)\tens1]b_{T(v,w+x+y,z)}\bv +[1\tens b_{T(v,w+1+y,z)}(1-\bv\cdot\bfeta)]b_{T(v+w,x,y+z)}\bv\bigr)
\\
=(\oin^{\tens3})[\delta_{w+x+y,1}b_{T(w,x,y)}\bv\tens\bv -\chi(v=w=y=z=0)\bv\tens b_{T(0,1,0)}\bv] =0 -0 =0:
\\
\bar\co[1](x)\tens\bar\co[1](w+1+y)\tens\bar\co[1](v+1+z) \to \1.
\end{multline*}

Keeping the assumption \(v+w+x+y+z=1\) consider the tree $\tau$ from \eqref{eq-tau-wy-v1x1z-1} depicted in \eqref{eq-tangles-3-vertices}.
Equation \eqref{eq-long-d-delta-0=0} on the summand indexed by $\tau$ takes the form
\begin{multline*}
\bigl\langle \bar\co[1](w)\tens\bar\co[1](y)\tens\bar\co[1](v+1+x+1+z) \rTTo^{1\tens\bar b_{T(v+1+x,y,z)}}
\\
\hfill \bar\co[1](w)\tens\bar\co[1](v+1+x+y+z) \rTTo^{(\oin\tens\oin)b_{T(v,w,x+y+z)}\bv} \1 \bigr\rangle \quad
\\
+\bigl\langle \bar\co[1](w)\tens\bar\co[1](y)\tens\bar\co[1](v+1+x+1+z) \rTTo^{(12)}_\sim \bar\co[1](y)\tens\bar\co[1](w)\tens\bar\co[1](v+1+x+1+z)
\\
\rTTo^{1\tens\bar b_{T(v,w,x+1+z)}} \bar\co[1](y)\tens\bar\co[1](v+w+x+1+z) \rTTo^{(\oin\tens\oin)b_{T(v+w+x,y,z)}\bv} \1 \bigr\rangle =0.
\end{multline*}
Due to equation~\eqref{dia-operad-3-OOOOOOOOOO} the above holds true:
\begin{multline*}
(\oin^{\tens3})\bigl([1\tens b_{T(v+1+x,y,z)}(1-\bv\bfeta)]b_{T(v,w,x+y+z)}\bv +(12)[1\tens b_{T(v,w,x+1+z)}(1-\bv\bfeta)]b_{T(v+w+x,y,z)}\bv\bigr)
\\
=-(\oin^{\tens3})[\chi(v=x+y+z=0)\bv\tens b_{T(1,0,0)}\bv +\chi(v+w+x=z=0)(12)(\bv\tens b_{T(0,0,1)}\bv)]
\\
=0: \bar\co[1](w)\tens\bar\co[1](y)\tens\bar\co[1](v+1+x+1+z) \to \1.
\end{multline*}
Thus the curved augmented cooperad \(\Bbar\co=\bar\co[1]\botto\) is constructed.
\end{proof}

\begin{proposition}\label{pro-bar-construction-functor}
There is a functor \(\Bbar:\UCCOp\to\CACoop\), the bar construction\index{TTindex}{bar construction}.
The functor \(\Bbar\) assigns to a morphism \(f=(f_1,\und f):\co\to\cp\in\UCCOp\) the morphism
\[ \sfBar f =\sfg =(\sfg_1,\sfg_0): \bar\co[1]\botto \to \bar\cp[1]\botto,
\]
where the graded cooperad morphism \(\sfBar_1f=\sfg_1=\bar{f}=\bar f_1\botto\) comes from the map of degree~0
\begin{equation*}
\bar f_1 =\bigl( \bar\co[1] \rMono^\oin \co[1] \rTTo^{f_1} \cp[1] \rTTo^\opr \bar\cp[1] \bigr)
\end{equation*}
and the degree~1 functional is
\begin{equation*}
\sfBar_0f =\sfg_0 =\bigl( \bar\co[1]\botto(1) \rMono^{\oin\botto} \co[1]\botto(1) \rTTo^{\check{f}} \cp[1](1) \rTTo^\bv \1 \bigr).
\end{equation*}
Here possibly non-vanishing components of $\check{f}$ are $f_1$ and $f_0=\und f\bfeta:\1\to\cp[1]$.
\end{proposition}

\begin{proof}
Given a morphism \(f=(f_1,\und f):\co\to\cp\in\UCCOp\) let us check that
\[ \sfBar f =\sfg =(\sfBar_1f=\sfg_1=\bar{f}=\bar f_1\botto,\sfBar_0f=\sfg_0=\check{f}\bv): \bar\co[1]\botto \to \bar\cp[1]\botto
\]
is indeed a morphism of $\CACoop$.
Notice that \(\sfw^{\Bbar\co}\bar f=\sfw^{\Bbar\cp}\).
It is required that
\[ d^{\Bbar\co}\bar f +\sum_{x+1+z=n}\Delta_{T(x,1,z)}(\sfg_0\tens\bar f) -\Delta_{T(0,n,0)}(\bar f\tens\sfg_0) =\bar fd^{\Bbar\cp}: \bar\co[1]\botto(n) \to \bar\cp[1]\botto(n).
\]
All parts of this equation are $\bar{f}$\n-coderivations.
By \corref{cor-coderivations-are-in-bijection-with-homogeneous-maps} the equation follows from its composition with $\pr_{\tau[n]}$: namely, for all \(\tau\in\tr(n)\)
\begin{align}
&\Bigl\langle \bigotimes_{v\in\IV(\tau)} \bar\co[1]|v| \rto{\bar b_\tau} \bar\co[1](n) \rTTo^{\bar f_1} \bar\cp[1](n) \Bigr\rangle \notag
\\
&+\sum_{x+1+z=n} \sum_{t\in\tr(1)}^{t\underset{x,1,z}\bull\tau[n]=\tau} \biggl\langle \Bigl( \bigotimes_{v\in\IV(t)} \bar\co[1]|v| \Bigr)\tens \bar\co[1](n) \rTTo^{\oin^{\tens\IV(t)}f_t\bv\tens\bar f_1} \bar\cp[1](n) \biggr\rangle \notag
\\
&-\sum_{t\in\tr(1)}^{\tau[n]\underset{0,n,0}\bull t=\tau} \biggl\langle \bar\co[1](n)\tens \bigotimes_{v\in\IV(t)} \bar\co[1]|v| \rTTo^{\bar f_1\tens\oin^{\tens\IV(t)}f_t\bv} \bar\cp[1](n) \biggr\rangle \notag
\\
&=\Bigl\langle \bigotimes_{v\in\IV(\tau)} \bar\co[1]|v| \rTTo^{\tens^{\IV(\tau)}\bar f_1} \bigotimes_{v\in\IV(\tau)}\bar\cp[1]|v| \rTTo^{\bar b_\tau} \bar\cp[1](n) \Bigr\rangle.
\label{eq-bf-ff-ff-fb-OP}
\end{align}
Both hand sides vanish unless \(|\IV(\tau)|\le2\).
Let us verify this equation case by case.

When $\tau=\circ$, the equation says that
\[ \bigl\langle \1 \rto{b_0} \co[1](1) \rto\opr \bar\co[1](1) \rto{\bar f_1} \bar\cp[1](1) \bigr\rangle =\bigl\langle \1 \rto{b_0} \cp[1](1) \rto\opr \bar\cp[1](1) \bigr\rangle,
\]
which follows from the computation
\[ b^\co_0(1-\bv^\co\bfeta^\co)f_1\opr =(b^\cp_0 -\und f^2\bfeta^\cp)\opr -b^\co_0\bv^\co\bfeta^\cp\opr =b^\cp_0\opr.
\]

When $\tau=\tau[n]$ equation~\eqref{eq-bf-ff-ff-fb-OP} claims that
\begin{multline*}
\bigl\langle \bar\co[1](n) \rTTo^{\bar b_{\tau[n]}} \bar\co[1](n) \rto{\bar f_1} \bar\cp[1](n) \bigr\rangle +\sum_{x+1+z=n} \bigl\langle \bar\co[1](n) \rTTo^{f_0\bv\tens\bar f_1} \bar\cp[1](n) \bigr\rangle
\\
-\bigl\langle \bar\co[1](n) \rTTo^{\bar f_1\tens f_0\bv} \bar\cp[1](n) \bigr\rangle =\bigl\langle \bar\co[1](n) \rto{\bar f_1} \bar\cp[1](n) \rTTo^{\bar b_{\tau[n]}} \bar\cp[1](n) \bigr\rangle,
\end{multline*}
which follows from the obvious
\[ b_1(1-\bv\bfeta)f_1\opr +(n-1)\und ff_1\opr =f_1(1-\bv\bfeta)b_1\opr.
\]

When $\tau=T(x,y,z)$, \(n=x+y+z\), equation~\eqref{eq-bf-ff-ff-fb-OP} becomes
\begin{multline*}
\bigl\langle \bar\co[1](y)\tens\bar\co[1](x+1+z) \rTTo^{\bar b_{T(x,y,z)}} \bar\co[1](n) \rto{\bar f_1} \bar\cp[1](n) \bigr\rangle
\\
\hfill +\delta_{y,1}\bigl\langle \bar\co[1](1)\tens\bar\co[1](n) \rTTo^{\oin f_1\bv\tens\bar f_1} \bar\cp[1](n) \bigr\rangle -\delta_{x+z,0}\bigl\langle \bar\co[1](n)\tens\bar\co[1](1) \rTTo^{\bar f_1\tens\oin f_1\bv} \bar\cp[1](n) \bigr\rangle \quad
\\
=\bigl\langle \bar\co[1](y)\tens\bar\co[1](x+1+z) \rTTo^{\bar f_1\tens\bar f_1} \bar\cp[1](y)\tens\bar\cp[1](x+1+z) \rTTo^{\bar b_{T(x,y,z)}} \bar\cp[1](n) \bigr\rangle,
\end{multline*}
which is proven by the computation:
\begin{multline*}
b_{T(x,y,z)}(1-\bv\bfeta)f_1\opr +\delta_{y,1}f_1\bv\tens f_1\opr -\delta_{x+z,0}(f_1\tens f_1\bv)\opr
\\
\hfill -(f_1\tens f_1)[(1-\bv\bfeta)\tens(1-\bv\bfeta)]b_{T(x,y,z)}\opr \quad
\\
=\delta_{y,1}f_1\bv\tens f_1\opr -\delta_{x+z,0}(f_1\tens f_1\bv)\opr -\delta_{y,1}(f_1\tens f_1)(\bv\tens1)\opr +\delta_{x+z,0}(f_1\tens f_1)(1\tens\bv)\opr =0.
\end{multline*}

The second equation to verify
\[ \bdelt^{\Bbar\co}_0 -d^{\Bbar\co}\sfg_0 -\Delta_{T(0,1,0)}(\sfg_0\tens\sfg_0) =\bar f\bdelt^{\Bbar\cp}_0: \bar\co[1]\botto(1)\to\1
\]
takes on a summand indexed by $\tau$ with \(|\Inp\tau|=1\) the form
\begin{align}
&-\Bigl\langle \bigotimes_{v\in\IV(\tau)} \bar\co[1]|v| \rTTo^{\oin^{\tens\IV(\tau)}b_\tau} \co[1](1) \rto\bv \1 \Bigr\rangle -\Bigl\langle \bigotimes_{v\in\IV(\tau)} \bar\co[1]|v| \rto{\bar b_\tau} \bar\co[1](1) \rTTo^{\oin f_1} \cp[1](1) \rto\bv \1 \Bigr\rangle \notag
\\
&-\sum_{t,\theta\in\tr(1)}^{t\underset{0,n,0}\bull\theta=\tau} \biggl\langle \bigotimes_{v\in\IV(\tau)} \bar\co[1]|v| \rto\sim  \Bigl( \bigotimes_{q\in\IV(t)} \bar\co[1]|q| \Bigr)\tens \Bigl( \bigotimes_{y\in\IV(\theta)} \bar\co[1]|y| \Bigr) \rTTo^{\oin^{\tens\IV(t)}f_t\bv\tens\oin^{\tens\IV(\theta)}f_\theta\bv} \1\tens\1 \biggr\rangle \notag
\\
&=-\Bigl\langle \bigotimes_{v\in\IV(\tau)} \bar\co[1]|v| \rTTo^{\tens^{\IV(\tau)}\bar f_1} \bigotimes_{v\in\IV(\tau)}\bar\cp[1]|v| \rTTo^{\oin^{\tens\IV(\tau)}b_\tau} \cp[1](1) \rto\bv \1 \Bigr\rangle.
\label{eq-bv-bfv-fvfv-fbv-OP}
\end{align}
Both hand sides vanish unless \(|\IV(\tau)|\le2\).
Let us prove the equation case by case.

When $\tau=\circ$, the equation becomes
\begin{multline*}
-\bigl\langle \1 \rto{b_0} \co[1](1) \rto\bv \1 \bigr\rangle -\bigl\langle \1 \rto{\bar b_0} \bar\co[1](1) \rTTo^{\oin f_1} \cp[1](1) \rto\bv \1 \bigr\rangle
\\
-\bigl\langle \1 =\1\tens\1 \rTTo^{\und f\tens\und f} \1\tens\1 =\1 \bigr\rangle =-\bigl\langle \1 \rto{b_0} \cp[1](1) \rto\bv \1 \bigr\rangle
\end{multline*}
which follows from
\[ -b^\co_0\bv^\co -b^\co_0(1-\bv^\co\bfeta^\co)f_1\bv^\cp -\und f^2 =-b^\co_0\bv^\co -b^\cp_0\bv^\cp +\und f^2\bfeta^\cp\bv^\cp +b^\co_0\bv^\co -\und f^2 =-b^\cp_0\bv^\cp.
\]

When $\tau=\tau[1]$ equation~\eqref{eq-bv-bfv-fvfv-fbv-OP} is
\begin{multline*}
-\bigl\langle \bar\co[1](1) \rTTo^{\oin b_1} \co[1](1) \rto\bv \1 \bigr\rangle -\bigl\langle \bar\co[1](1) \rto{\bar b_1} \bar\co[1](1) \rTTo^{\oin f_1} \cp[1](1) \rto\bv \1 \bigr\rangle
\\
\hfill -\bigl\langle \1\tens\bar\co[1](1) \rTTo^{\und f\tens\oin f_1\bv} \1\tens\1 \bigr\rangle -\bigl\langle \bar\co[1](1)\tens\1 \rTTo^{\oin f_1\bv\tens\und f} \1\tens\1 \bigr\rangle \quad
\\
=-\bigl\langle \bar\co[1](1) \rto{\bar f_1} \bar\cp[1](1) \rTTo^{\oin b_1} \cp[1](1) \rto\bv \1 \bigr\rangle.
\end{multline*}
This follows from the computation
\begin{multline*}
-b^\co_1\bv^\co -b^\co_1(1-\bv^\co\bfeta^\co)f_1\bv^\cp -\und f(f_1\bv) -(f_1\bv)\und f =-b^\co_1f_1\bv^\cp -\und f(f_1\bv) +\und f(f_1\bv)
\\
=-f_1b^\cp_1\bv^\cp =-f_1(1-\bv^\cp\bfeta^\cp)b^\cp_1\bv^\cp.
\end{multline*}

When $\tau=T(x,y,z)$, \(x+y+z=1\), equation~\eqref{eq-bv-bfv-fvfv-fbv-OP} becomes
\begin{multline*}
-\bigl\langle \bar\co[1](y)\tens\bar\co[1](x+1+z) \rTTo^{(\oin\tens\oin)b_{T(x,y,z)}} \co[1](1) \rto\bv \1 \bigr\rangle
\\
-\bigl\langle \bar\co[1](y)\tens\bar\co[1](x+1+z) \rTTo^{\bar b_{T(x,y,z)}} \bar\co[1](1) \rTTo^{\oin f_1} \cp[1](1) \rto\bv \1 \bigr\rangle
\\
\hfill -\delta_{y,1}\bigl\langle \bar\co[1](1)\tens\bar\co[1](1) \rTTo^{\oin f_1\bv\tens\oin f_1\bv} \1\tens\1 =\1 \bigr\rangle \quad
\\
=-\bigl\langle \bar\co[1](y)\tens\bar\co[1](x+1+z) \rTTo^{\bar f_1\tens\bar f_1} \bar\cp[1](y)\tens\bar\cp[1](x+1+z) \rTTo^{(\oin\tens\oin)b_{T(x,y,z)}} \cp[1](1) \rto\bv \1 \bigr\rangle.
\end{multline*}
This follows from the computation:
\begin{multline*}
-b_{T(x,y,z)}\bv^\co -b_{T(x,y,z)}(1-\bv^\co\bfeta^\co)f_1\bv^\cp -\delta_{y,1}f_1\bv^\cp\tens f_1\bv^\cp
\\
\hfill +(f_1\tens f_1)[(1-\bv^\cp\bfeta^\cp)\tens(1-\bv^\cp\bfeta^\cp)]b_{T(x,y,z)}\bv^\cp \quad
\\
\hskip\multlinegap =-(f_1\tens f_1)b_{T(x,y,z)}\bv^\cp -\delta_{y,1}f_1\bv^\cp\tens f_1\bv^\cp +(f_1\tens f_1)b_{T(x,y,z)}\bv^\cp \hfill
\\
+\delta_{y,1}f_1\bv^\cp\tens f_1\bv^\cp +\delta_{x+z,0}f_1\bv^\cp\tens f_1\bv^\cp -\delta_{x+z,0}\delta_{y,1}f_1\bv^\cp\tens f_1\bv^\cp =0.
\end{multline*}

The identity morphism $f=(\id,0)$ is mapped to the identity morphism $\sfBar f=(\id,0)$.
Let us verify that $\Bbar$ agrees with the composition.
If $h=fg$ in $\UCCOp$, \(h_1=f_1g_1\), \(\und h=\und f+\und g\), then $\bar h=\bar h_1\botto=\bar f_1\botto\bar g_1\botto=\bar f\bar g$, since
\[ \bar h_1 =\oin f_1g_1\opr =\oin f_1(1-\bv\bfeta)g_1\opr =\bar f_1\bar g_1.
\]
Furthermore,
\[ \sfBar_0f +(\sfBar_1f) \sfBar_0g =\sfBar_0h,
\]
that is,
\[ \oin\botto\check f\bv +\bar f(\oin\botto)\check g\bv =\oin\botto\check{h}\bv: \bar\co[1]\botto(1) \to \1.
\]
In fact, this equation has to be verified on two summands only.
For the summand indexed by $\tau=\circ$ it reduces to \(\und f+\und g=\und h:\1\to\1\), and for \(\tau=\tau[1]\) to
\[ \oin f_1\bv +\bar f_1\oin g_1\bv =\oin f_1\bv +\oin f_1(1-\bv\bfeta)g_1\bv =\oin f_1g_1\bv: \bar\co[1](1) \to \1.
\]
The functor \(\Bbar:\UCCOp\to\CACoop\) (the bar construction) is described.
\end{proof}

\section{Restrictions of bar and cobar constructions}
It is easy to see that the obtained functors $\Bbar$ and $\Cobar$ restrict to subcategories as the following diagram of functors and inclusions of categories shows:
\begin{diagram}[h=1.2em]
&\cCoop &&\cOp &
\\
&\uMono &&\uMono &
\\
\Cobar:{} &\ucCoop &\rTTo &\ucOp &
\\
&\uMono &&\uMono &
\\
\Cobar:{} &\CACoop &\pile{\lTTo\\ \rTTo} &\UCCOp &:\Bbar
\\
&\uMono &&\uMono &
\\
\Cobar:{} &\caCoop &\pile{\lTTo\\ \rTTo} &\uccOp &:\Bbar
\\
&\uMono &&\uMono &
\\
\Cobar:{} &\acCoop &\pile{\lTTo\\ \rTTo} &\ucdgOp &:\Bbar
\end{diagram}
Squares of this diagram with the same direction of opposite edges commute.

\chapter{Twisting cochains}
We use twisting cochains in order to prove that cobar and bar constructions are adjoint functors.
The remaining details are given of the example of an augmented curved cooperad $C$ whose $\Cobar C$ is the operad of homotopy unital \ainf-algebras.

\section{Definition and properties of twisting cochains}
\begin{definition}
A degree~1 map \(\theta:C\to\co\) between a curved cooperad $(C,\sfw)$ and a curved operad $\co$ that satisfies 
\begin{multline}
\bigl\langle C(n) \rto d C(n) \rto\theta \co(n) \bigr\rangle
+\bigl\langle C(n) \rto\theta \co(n) \rto d \co(n) \bigr\rangle
\\
+\sum_{x+y+z=n} \bigl\langle C(n) \rTTo^{\Delta_{T(x,y,z)}} C(y)\tens C(x+1+z) \rto{\theta\tens\theta} \co(y)\tens\co(x+1+z) \rto{\underset{x,y,z}\bull} \co(n) \bigr\rangle
\\
=\delta_{n,1}\bigl\langle C(n) \rto\eps \1 \rto{m_0} \co(n) \bigr\rangle
+\delta_{n,1}\bigl\langle C(n) \rto{\bdelt_0} \1 \rto\eta \co(n) \bigr\rangle
\label{eq-CCO-COO-CCCOOO}
\end{multline}
is called a\index{TTindex}{twisting cochain} \emph{twisting cochain}, cf. Berger and Moerdijk \cite[Section~8.5.3]{math/0502155}.
The summands containing curvature are similar to that of Positselski \cite{0905.2621} for algebra and coalgebra case.
The set of twisting cochains is denoted \(\Tw(C,A)\).
\end{definition}

Assume additionally that $C$ is counit-complemented.
The reason to consider twisting cochains is given by the following diagram where the dashed arrow has to be constructed
\begin{diagram}[LaTeXeqno]
\cOp(\Cobar C,\co) &\rDashTo &\Tw(C,\co) &\rMono &\und{\cv^\NN}(C,\co)^1
\\
\dMono &&&&\uTTo>\wr<{\sigma^{-1}\cdot-}
\\
\Op(\bar{C}[-1]\TT,\co)\times\co(1)^1 &\rTTo_\sim^{i_1} &\cv^\NN(\bar{C}[-1],\co)\times\cv^\NN(\1[-1],\co) &\rTTo_\sim^{i_2} &\cv^\NN(C[-1],\co)
\label{dia-cOp-C-1T-O}
\end{diagram}
The bijection $i_2$ is due to direct sum decomposition $C=\bar C\oplus\1$ and bijection $i_1$ describes homomorphisms from a free algebra.
A pair of morphisms of $\cv^\NN$ \((\check\sff_1:\bar{C}[-1]\to\co,\sigma\sff_0:\1[-1]\to\co)\) is mapped to a degree~1 morphism
\[ \sigma^{-1}[(\check\sff_1,\sigma\sff_0)i_2] =\theta=
\Bigl(
\begin{matrix}
\sigma^{-1}\check\sff_1
\\
\sff_0
\end{matrix}
\Bigr): \bar C\oplus\1 \to \co.
\]

\begin{proposition}\label{pro-unique-top-horizontal-bijections}
There is unique bijection in place of dashed arrow in diagram~\eqref{dia-cOp-C-1T-O} which makes the diagram commutative.
\end{proposition}

\begin{proof}
Denote by \(\sff_1:\bar{C}[-1]\TT\to\co\) morphism coming from $\check\sff_1$.
We are going to show that system of equations \eqref{eq-OPP-OPPPP}--\eqref{eq-1P-1PP-111PPP-1OP} on \((\sff_1,\sff_0)\) is equivalent to equation \eqref{eq-CCO-COO-CCCOOO} on
\[ \theta=
\Bigl(
\begin{matrix}
\oin\theta
\\
\sfw\theta
\end{matrix}
\Bigr)=
\Bigl(
\begin{matrix}
\sigma^{-1}\check\sff_1
\\
\sff_0
\end{matrix}
\Bigr).
\]
Three combinations of terms of \eqref{eq-OPP-OPPPP} are $\sff_1$\n-derivations $\bar{C}[-1]\TT\to\co$, hence, the equation is equivalent to its precomposition with \(\inj_1:\bar{C}[-1]\hookrightarrow\bar{C}[-1]\TT\).
System \eqref{eq-OPP-OPPPP}--\eqref{eq-1P-1PP-111PPP-1OP} becomes
\begin{equation}
\begin{split}
\check\sff_1d^\co +(\check\sff_1\tens\sff_0)m_{T(0,n,0)}^\co -\sum_{x+1+z=n}(\sff_0\tens\check\sff_1)m_{T(x,1,z)}^\co &=\check d^{\Cobar C}\sff_1: \bar{C}[-1](n)\to \co(n),
\\
m^\co_0 -\sff_0d^\co -(\sff_0\tens\sff_0)m_{T(0,1,0)}^\co &=m^{\Cobar C}_0\sff_1: \1 \to \co(1).
\end{split}
\label{eq-fm-mf-m-mf}
\end{equation}
In more detail equations~\eqref{eq-fm-mf-m-mf} read
\begin{subequations}
\begin{align}
\check\sff_1d^\co &+(\check\sff_1\tens\sff_0)m_{T(0,n,0)}^\co -\sum_{x+1+z=n}(\sff_0\tens\check\sff_1)m_{T(x,1,z)}^\co \notag
\\
&=\delta_{n,1}\bar\xi_\circ\eta^\co +\bar\xi_{\tau[n]}\check\sff_1 +\sum_{x+y+z=n}\bar\xi_{T(x,y,z)} (\check\sff_1\tens\check\sff_1)m_{T(x,y,z)}^\co: \bar{C}[-1](n)\to \co(n),
\\
m^\co_0 &-\sff_0d^\co -(\sff_0\tens\sff_0)m_{T(0,1,0)}^\co =-\bw\xi_0\eta^\co -\bw\xi_1\opr\check\sff_1 \notag
\\
&-\sum_{x+y+z=1}\bw\xi_{T(x,y,z)}(\opr\check\sff_1\tens\opr\check\sff_1)m_{T(x,y,z)}^\co: \1 \to \co(1).
\end{align}
\label{eq-fm-xe-xf-xffm}
\end{subequations}

We have
\begin{equation}
\theta=(\opr\oin+\eps\sfw)\theta =\opr\sigma^{-1}\check\sff_1+\eps\sff_0: C\to\co.
\label{eq-theta=(prinepsw)theta}
\end{equation}
Equation~\eqref{eq-CCO-COO-CCCOOO} can be rewritten as
\begin{equation*}
-\sigma\theta d^\co +\xi_1\sigma\theta +\xi_0\eta +\beps m_0 +\sum_{x+y+z=n}\xi_{T(x,y,z)}(\sigma\theta\tens\sigma\theta)m_{T(x,y,z)} =0: C[-1] \to \co.
\end{equation*}
Let us plug in it the expression \(\sigma\theta=\opr\check\sff_1+\beps\sff_0:C[-1]\to\co\).
The obtained equation is
\begin{multline}
-\opr\check\sff_1d^\co -\beps\sff_0d^\co +\xi_1\opr\check\sff_1 +\xi_0\eta +\beps m_0 +\sum_{x+y+z=n}\xi_{T(x,y,z)}(\opr\check\sff_1\tens\opr\check\sff_1)m_{T(x,y,z)}
\\
+\sum_{x+1+z=n}(\sff_0\tens\opr\check\sff_1)m_{T(x,1,z)} -(\opr\check\sff_1\tens\sff_0)m_{T(0,n,0)} -\delta_{n,1}\beps(\sff_0\tens\sff_0)m_{T(0,1,0)} =0:
\\
C[-1](n)\to \co(n).
\label{eq-dovge-theta-ff}
\end{multline}
This equation is equivalent to the system of two equations obtained by precomposing \eqref{eq-dovge-theta-ff} with \(\oin:\bar C[-1]\hookrightarrow C[-1]\) and \(\bw:\1\to C[-1](1)\).
One verifies immediately that these two equations are nothing else but (\ref{eq-fm-xe-xf-xffm}a) and (\ref{eq-fm-xe-xf-xffm}b).
\end{proof}

Assume that $C$ is a curved augmented cooperad and $\co$ is a unit-complemented curved operad.
We are going to close up the following diagram of sets
\begin{diagram}[LaTeXeqno]
\{ \theta\in\Tw(C,\co) \mid \sfw\theta\opr =0 \} &\lDashTo &\CACoop(C,\Bbar\co)
\\
\dMono &&\dMono
\\
\{ \theta\in\und{\cv^\NN}(C,\co)^1 \mid \sfw\theta\opr =0 \} &&\nucoop(\bar C,\bar\co[1]\bott)\times\und{\cv^\NN}(C,\1)^1
\\
\uTTo<{-\cdot\sigma^{-1}}>\wr &&\dTTo>{i_3}<\wr
\\
\{ \theta\sigma\in\cv^\NN(C,\co[1]) \mid \sfw\theta\sigma\opr =0 \} &&\cv^\NN(\bar C,\bar\co[1])\times\cv^\NN(C,\1[1])
\\
\uTTo>\wr &&\dTTo<\wr
\\
\cv^\NN(\bar C,\co[1])\times\cv^\NN(\1,\1[1]) &\lTTo_\sim &\cv^\NN(\bar C,\bar\co[1])\times\cv^\NN(\bar C,\1[1])\times\cv^\NN(\1,\1[1])
\label{dia-C-1T-A-g-caCoop}
\end{diagram}
Here we use that \(\augcoop(C,\bar\co[1]\botto)\simeq\nucoop(\bar C,\bar\co[1]\bott)\), isomorphism $i_3$ is due to description of morphisms to cofree conilpotent cooperad.
The remaining isomorphisms come from decompositions \(C=\bar C\oplus\1\), \(\co=\bar\co\oplus\1\).
A pair of morphisms of $\cv^\NN$ \((\check\sfg_1:\bar C\to\bar\co[1],\sfg_0\sigma:C\to\1[1])\) is mapped to a degree~1 map
\begin{equation}
\theta\sigma=
\Bigl(
\begin{matrix}
\check\sfg_1 & \oin\sfg_0\sigma
\\
0 & \sfw\sfg_0\sigma
\end{matrix}
\Bigr)
: \bar C\oplus\1 \to \bar\co[1]\oplus\1[1],
\label{eq-theta-sigma-C1-O1}
\end{equation}
where the second column is $\sfg_0\sigma$.

\begin{proposition}\label{pro-unique-top-horizontal-bijection}
There is a unique top horizontal bijection in diagram~\eqref{dia-C-1T-A-g-caCoop} which makes the diagram commutative.
\end{proposition}

\begin{proof}
Denote by \(\sfg_1:C\to\bar\co[1]\botto\in\augcoop\) and \(\bar\sfg_1:\bar C\to\bar\co[1]\bott\in\nucoop\) morphisms coming from \(\check\sfg_1:\bar C\to\bar\co[1]\in\cv^\NN\).
Extend the latter notation to map
\[ \check\sfg_1 =\opr\check\sfg_1=
\Bigl(
\begin{matrix}
\check\sfg_1
\\
0
\end{matrix}
\Bigr)
:C =\bar C\oplus\1 \to \bar\co[1].
\]
Let us show that system of equations \eqref{eq-morphism-curved-cooperads} on \((\sfg_1,\sfg_0)\) rewritten with the use of \corref{cor-coderivations-are-in-bijection-with-homogeneous-maps} as
\begin{align}
d^C\check\sfg_1 +\sum_{x+1+z=n}\Delta_{T(x,1,z)} (\sfg_0\tens\check\sfg_1) -\Delta_{T(0,n,0)}(\check\sfg_1 &\tens\sfg_0) =\sfg_1\check d^{\Bbar\co}: C(n)\to \bar \co[1](n), \notag
\\
\bdelt^C_0 -d^C\sfg_0 -\Delta_{T(0,1,0)}(\sfg_0\tens\sfg_0) &=\sfg_1\bdelt^{\Bbar\co}_0: C(1)\to \1,
\label{eq-dg-dg11gg-gd-d-dg}
\end{align}
is equivalent to equation \eqref{eq-CCO-COO-CCCOOO} for $\theta$ given by \eqref{eq-theta-sigma-C1-O1}.
Expanded form of \eqref{eq-dg-dg11gg-gd-d-dg} is
\begin{align}
&d^C\check\sfg_1 +\sum_{x+1+z=n}\Delta_{T(x,1,z)} (\sfg_0\tens\check\sfg_1) -\Delta_{T(0,n,0)}(\check\sfg_1\tens\sfg_0) \notag
\\
&=\delta_{n,1}\eps^C\bar b^\co_0 +\check\sfg_1\bar b_{\tau[n]} +\sum_{x+y+z=n}\opr\bar\Delta_{T(x,y,z)}(\check\sfg_1\tens\check\sfg_1)\bar b_{T(x,y,z)}: C(n)\to \bar\co[1](n),
\label{eq-dg-dg11gg-ebp-pgbp-dpgpgbp-n}
\\
&\bdelt^C_0 -d^C\sfg_0 -\Delta_{T(0,1,0)}(\sfg_0\tens\sfg_0) \notag
\\
&=-\eps^Cb^\co_0\bv -\check\sfg_1\oin b^\co_1\bv -\sum_{x+y+z=1}\opr\bar\Delta_{T(x,y,z)}(\check\sfg_1\oin\tens\check\sfg_1\oin)b_{T(x,y,z)}\bv: C(1)\to \1.
\label{eq-dg-dg11gg-ebp-pgbp-dpgpgbp-1}
\end{align}

Now let us plug in equation \eqref{eq-CCO-COO-CCCOOO} the expression
\[ \theta\sigma =(\opr\oin+\eps\sfw)\theta\sigma(\opr\oin+\bv\bfeta) =\opr\check\sfg_1\oin +\opr\oin\sfg_0\bfeta +\eps\sfw\sfg_0\bfeta =\opr\check\sfg_1\oin +\sfg_0\bfeta: C \to \co[1],
\]
rewriting the equation preliminarily as
\begin{equation*}
-d\theta\sigma +\theta\sigma b_1 +\sum_{x+y+z=n}\Delta_{T(x,y,z)}(\theta\sigma\tens\theta\sigma)b_{T(x,y,z)} +\delta_{n,1}\eps b_0 +\delta_{n,1}\bdelt_0\bfeta =0: C(n) \to \co[1](n).
\end{equation*}
We get
\begin{multline*}
-d\opr\check\sfg_1\oin -d\sfg_0\bfeta +\opr\check\sfg_1\oin b_1
+\sum_{x+y+z=n}\Delta_{T(x,y,z)}(\opr\check\sfg_1\oin\tens\opr\check\sfg_1\oin)b_{T(x,y,z)}
\\
+\Delta_{T(0,n,0)}(\opr\check\sfg_1\oin\tens\sfg_0) -\sum_{x+1+z=n}\Delta_{T(x,1,z)} (\sfg_0\tens\opr\check\sfg_1\oin)
\\
-\delta_{n,1}\Delta_{T(0,1,0)}(\sfg_0\tens\sfg_0)\bfeta +\delta_{n,1}\eps b_0 +\delta_{n,1}\bdelt_0\bfeta =0: C(n)\to \co[1](n),
\end{multline*}
Postcomposing this equation with \(\opr:\co[1](n)\to\bar\co[1](n)\) and with \(\bv:\co[1](1)\to\1\) we get, respectively, \eqref{eq-dg-dg11gg-ebp-pgbp-dpgpgbp-n} and \eqref{eq-dg-dg11gg-ebp-pgbp-dpgpgbp-1}.
For this identification we use the following computation:
\begin{multline*}
\opr\bar\Delta_{T(x,y,z)} =\opr\oin\Delta_{T(x,y,z)}(\opr\tens\opr) =(1-\eps\sfw)\Delta_{T(x,y,z)}(\opr\tens\opr)
\\
=\Delta_{T(x,y,z)}(\opr\tens\opr) -\eps\chi(x=z=0,y=1)(\sfw\opr\tens\sfw\opr) =\Delta_{T(x,y,z)}(\opr\tens\opr).
\end{multline*}
Thus, the required bijection is constructed.
\end{proof}

\begin{example}
For any counit-complemented curved cooperad $C$ there is a twisting cochain
\[ \theta =\bigl( C \rTTo^\opr \bar C \rTTo^{\sigma^{-1}} \bar C[-1] \rTTo^{\eta_\TT} \bar{C}[-1]\TT =\Cobar C \bigr),
\]
corresponding to \(\id_{\Cobar C}\in\cOp\), see diagram~\eqref{dia-cOp-C-1T-O} with \(\co=\Cobar C\).

For any unit-complemented curved operad $\co$ there is a twisting cochain
\[ \theta =\bigl( \Bbar\co =\bar\co[1]\botto \rTTo^{\eps_{\botto}} \bar\co[1] \rTTo^{\sigma^{-1}} \bar\co \rTTo^\oin \co \bigr),
\]
see diagram~\eqref{dia-C-1T-A-g-caCoop} with \(C=\Bbar\co\).
It corresponds to \(\id_{\Bbar\co}\in\CACoop\).
\end{example}

\begin{proposition}
The actions
\begin{alignat*}3
\cCoop(C,D) \times \Tw(D,\co) &\longrightarrow \Tw(C,\co), &\quad (\sfg,\theta) &\longmapsto \sfg\rightharpoonup\theta &&=\sfg_1\theta +\sfg_0\eta,
\\
\Tw(C,\co) \times \cOp(\co,\cp) &\longrightarrow \Tw(C,\cp), &\quad (\theta,\sff) &\longmapsto \theta\leftharpoonup\sff &&=\theta\sff_1 +\eps\sff_0
\end{alignat*}
make $\Tw$ into a $\cCoop\text-\cOp$-bimodule, in other words into a functor \(\Tw:\cCoop^\op\) \(\times\cOp\to\Set\).
\end{proposition}

\begin{proof}
One easily verifies that \(\sfg\rightharpoonup\theta\) and \(\theta\leftharpoonup\sff\) satisfy \eqref{eq-CCO-COO-CCCOOO}, both actions are associative and \((\sfg\rightharpoonup\theta)\leftharpoonup\sff=\sfg\rightharpoonup(\theta\leftharpoonup\sff)\).
\end{proof}

\begin{proposition}
Let $C$ run over counit-complemented curved cooperads and let $\co$ run over curved operads.
The bijection of \propref{pro-unique-top-horizontal-bijections}
\begin{equation}
\cOp(\Cobar C,\co) \to \Tw(C,\co)
\label{eq-cAlg(CobarCA)-Tw(CA)}
\end{equation}
is an isomorphism of $\ucCoop$-$\cOp$-bimodules.
\end{proposition}

\begin{proof}
First of all, \eqref{eq-cAlg(CobarCA)-Tw(CA)} is a homomorphism with respect to the right action of \(\sfh\in\cOp(\co,B)\).
In fact, for an arbitrary \(\sff\in\cOp(\Cobar C,\co)\)
\[ 
\begin{pmatrix}
\sigma^{-1}\widecheck{(\sff\cdot\sfh)}_1
\\
(\sff\cdot\sfh)_0
\end{pmatrix}
=
\begin{pmatrix}
\sigma^{-1}\check\sff_1\sfh_1
\\
\sff_0\sfh_1 +\sfh_0
\end{pmatrix}
=
\begin{pmatrix}
\sigma^{-1}\check\sff_1
\\
\sff_0
\end{pmatrix}
\leftharpoonup\sfh: \bar C\oplus\1 \to \co.
\]

Secondly, \eqref{eq-cAlg(CobarCA)-Tw(CA)} is a homomorphism with respect to the left action of \(\sfj\in\ucCoop(D,C)\), that is, with the use of \eqref{eq-theta=(prinepsw)theta} for an arbitrary \(\sff\in\cOp(\Cobar C,\co)\)
\begin{equation}
\opr\sigma^{-1}\widecheck{[(\sfCobar j)\sff]}_1+\eps[(\sfCobar j)\sff]_0 =\sfj\rightharpoonup(\opr\sigma^{-1}\check\sff_1+\eps\sff_0): D\to \co.
\label{eq-pr(Cobarj)f1-eps(Cobarj)f0}
\end{equation}
In fact, the left hand side is
\begin{align*}
&\opr\sigma^{-1}\widecheck{(\sfCobar_1j)}\sff_1 +\eps[(\sfCobar_0j)\sff_1+\sff_0]
\\
&=\opr\sigma^{-1}\oin j_0\eta +\opr\sigma^{-1}\oin j_1\opr\check\sff_1 +\eps[\bw j_0\eta +\bw j_1\opr\check\sff_1 +\sff_0]
\\
&=(1-\eps\sfw)\sigma^{-1}j_0\eta +(1-\eps\sfw)\sigma^{-1}j_1\opr\check\sff_1 +\eps\sfw\sigma^{-1}j_0\eta +\eps\sfw\sigma^{-1}j_1\opr\check\sff_1 +\eps\sff_0
\\
&=\sigma^{-1}j_0\eta +\sigma^{-1}j_1\opr\check\sff_1 +\sfj_1\eps\sff_0 =\sfj_1\opr\sigma^{-1}\check\sff_1 +\sfj_1\eps\sff_0 +\sfj_0\eta,
\end{align*}
which is the right hand side of \eqref{eq-pr(Cobarj)f1-eps(Cobarj)f0}.
\end{proof}

\begin{corollary}\label{cor-uccOp(CobarCO)-wtheta0}
Let $C$ run over curved augmented cooperads and let $\co$ run over unit-complemented curved operads.
Then bijection \eqref{eq-cAlg(CobarCA)-Tw(CA)} restricts to isomorphisms
\begin{diagram}[h=1.8em]
\uccOp(\Cobar C,\co) &\rTTo^\sim &\{ \theta\in\Tw(C,\co) \mid \sfw\theta =0 \}
\\
\dMono &&\dMono
\\
\UCCOp(\Cobar C,\co) &\rTTo^\sim &\{ \theta\in\Tw(C,\co) \mid \sfw\theta\opr =0 \}
\end{diagram}
of $\caCoop$-$\uccOp$-bimodules, resp. of $\CACoop$-$\UCCOp$-bimodules.
When $C$ runs over augmented curved cooperads and $\co$ runs over unit-complemented $\dg$-operads, the former isomorphism becomes an isomorphism of $\acCoop$-$\ucdgOp$-bimodules
\[ \ucdgOp(\Cobar C,\co) \rTTo^\sim \{ \theta\in\Tw(C,\co) \mid \sfw\theta =0 \}.
\]
\end{corollary}

\begin{proof}
Just notice that \(\{\theta\in\Tw(C,\co)\mid\sfw\theta=0\}\) is a $\caCoop$-$\uccOp$-subbimodule (resp. $\acCoop$-$\ucdgOp$-subbimodule) of \(\{\theta\in\Tw(C,\co)\mid\sfw\theta\opr=0\}\).
\end{proof}

\begin{proposition}\label{pro-wthetapr-CACoop(CBarO)}
Let $C$ run over curved augmented cooperads and let $\co$ run over unit-complemented curved operads.
The bijection of \propref{pro-unique-top-horizontal-bijection}
\begin{equation}
\{ \theta\in\Tw(C,\co) \mid \sfw\theta\opr =0 \} \lTTo^\sim \CACoop(C,\Bbar\co)
\label{eq-Tw(CO)-CACoop(CBarO)}
\end{equation}
is an isomorphism of $\CACoop$-$\UCCOp$-bimodules.
\end{proposition}

\begin{proof}
First of all, \eqref{eq-Tw(CO)-CACoop(CBarO)} is a homomorphism with respect to the left action of \(\sfj\in\CACoop(D,C)\).
In fact, for an arbitrary \((\sfg_1,\sfg_0)\in\CACoop(C,\Bbar\co)\) the element \(\sfj\cdot(\sfg_1,\sfg_0)=(\sfj_1\sfg_1,\sfj_0+\sfj_1\sfg_0)\) is mapped to
\[ 
\begin{pmatrix}
\sfj_1\check\sfg_1\sigma^{-1} & \sfj_0+\sfj_1\sfg_0
\end{pmatrix}
=\sfj\rightharpoonup
\begin{pmatrix}
\check\sfg_1\sigma^{-1} & \sfg_0
\end{pmatrix}
:C\to \bar \co\oplus\1.
\]

Secondly, \eqref{eq-Tw(CO)-CACoop(CBarO)} is a homomorphism with respect to the right action of \(\sfh\in\UCCOp(\co,\cp)\).
That is, for an arbitrary \(\sfg=(\sfg_1,\sfg_0)\in\CACoop(C,\Bbar\co)\) whose image in \(\Tw(C,\co)\) is \(\theta=\opr\check\sfg_1\sigma^{-1}\oin+\sfg_0\eta\) we have
\begin{equation}
\opr\widecheck{(\sfg\cdot\sfBar h)}_1\sigma^{-1}\oin+(\sfg\cdot\sfBar h)_0\eta
=(\opr\check\sfg_1\sigma^{-1}\oin+\sfg_0\eta)\leftharpoonup\sfh.
\label{eq-gBarh-(gin)h}
\end{equation}
In fact, noting that \(\check\sff=\oin\sff\pr_1\) for \(\sff=\sfg_1\sfBar_1h\) we find
\[ \widecheck{(\sfg_1\sfBar_1h)} =\oin\sfg_1\pr_1\oin h_1\opr =\check\sfg_1\oin h_1\opr.
\]
Therefore, the left hand side of \eqref{eq-gBarh-(gin)h} is
\begin{align*}
&\opr\widecheck{(\sfg_1\sfBar_1h)}\sigma^{-1}\oin +(\sfg_0+\sfg_1\sfBar_0h)\eta
\\
&=\opr\check\sfg_1\oin h_1\opr\sigma^{-1}\oin +\sfg_0\eta +\opr\check\sfg_1\oin h_1\bv\eta +\eps h_0\bv\eta
\\
&=\opr\check\sfg_1\oin h_1\sigma^{-1}(1-\sfv\eta) +\sfg_0\eta +\opr\check\sfg_1\oin h_1\sigma^{-1}\sfv\eta +\eps\sfh_0\sfv\eta
\\
&=\opr\check\sfg_1\oin h_1\sigma^{-1} +\sfg_0\eta\sfh_1 +\eps\und h\eta\sfv\eta =\opr\check\sfg_1\sigma^{-1}\oin\sfh_1 +\sfg_0\eta\sfh_1 +\eps\sfh_0,
\end{align*}
which is the right hand side of \eqref{eq-gBarh-(gin)h}.
\end{proof}

\begin{corollary}\label{cor-wtheta0-caCoop(CBarO)}
Let $C$ run over curved augmented cooperads and let $\co$ run over unit-complemented curved operads.
Then bijection \eqref{eq-Tw(CO)-CACoop(CBarO)} restricts to an isomorphism
\begin{equation*}
\{ \theta\in\Tw(C,\co) \mid \sfw\theta =0 \} \lTTo^\sim \caCoop(C,\Bbar\co)
\end{equation*}
of $\caCoop$-$\uccOp$-bimodules.
Furthermore, when $C$ runs over augmented curved cooperads and $\co$ runs over unit-complemented $\dg$-operads, the above bijection is an isomorphism
\[ \{ \theta\in\Tw(C,\co) \mid \sfw\theta =0 \} \lTTo^\sim \acCoop(C,\Bbar\co)
\]
of $\acCoop$-$\ucdgOp$-bimodules.
\end{corollary}

\begin{theorem}
The rows of the following diagram give rise to bar-cobar adjunctions\index{TTindex}{bar-cobar adjunction}
\begin{diagram}[h=1.2em]
\Cobar:{} &\CACoop &\pile{\lTTo\\ \rTTo} &\UCCOp &:\Bbar
\\
&\uMono &&\uMono &
\\
\Cobar:{} &\caCoop &\pile{\lTTo\\ \rTTo} &\uccOp &:\Bbar
\\
&\uMono &&\uMono &
\\
\Cobar:{} &\acCoop &\pile{\lTTo\\ \rTTo} &\ucdgOp &:\Bbar
\end{diagram}
\end{theorem}

\begin{proof}
Follows from \corref{cor-uccOp(CobarCO)-wtheta0}, \propref{pro-wthetapr-CACoop(CBarO)} and \corref{cor-wtheta0-caCoop(CBarO)}.
\end{proof}

\begin{example}
Let $C=(C,\Delta,\eps)$ denote the $\gr$\n-cooperad over $\kk=\ZZ$ from \exaref{exa-cooperad-C-JP}.
It admits a non-trivial augmented curved cooperad structure $(C,d_C=0,\bdelt_0,\sfw)$.
Namely, \(\bdelt_0:C(1)^{-2}=\ZZ\{f_{n_0;n_1}\mid n_0,n_1\in\NN\}\to\ZZ\) is given by
\[ (f_{0;1})\bdelt_0 =-1, \quad (f_{1;0})\bdelt_0 =1, \quad (f_{n_0;n_1})\bdelt_0 =0\; \text{ in other cases.}
\]
The augmentation \(\sfw:\ZZ\to C(1)^0\) is the bijection taking $1$ to $f_1$.
This augmented curved cooperad coincides with the one constructed by Hirsh and Mill{\`e}s \cite[Lemma 6.1.7]{MR2993002}.

We have to check that for all elements \(c=f_{n_0;n_1;\dots;n_t}\in\bar{C}(n)^{-2t}\), \(n=\sum_{i=0}^tn_i\), we have
\begin{equation}
c\Delta_{T(0,n,0)}(1\tens\bdelt_0) =\sum_{x+1+z=n}c\Delta_{T(x,1,z)}(\bdelt_0\tens1).
\label{eq-cT(0n0)1h-cT(x1z)h1}
\end{equation}
If $t=0$ both sides vanish.
If $t>0$, then the left hand side is computed as
\begin{align*}
&(1^{\tens n_0}\tens\bj\tens1^{\tens n_1}\tens\bj\tdt\bj\tens1^{\tens n_t})\tens (f_1\tens f_{n_0+\dots+n_t+t-1} +f_{n_0+\dots+n_t+t-1}\tens f_1)\tens f_2
\\
&\mapsto \delta_{n_0,0} f_{n_1;\dots;n_t}\tens f_{0;1}\bdelt_0 +\delta_{n_t,0} f_{n_0;\dots;n_{t-1}}\tens f_{1;0}\bdelt_0
\\
&=\delta_{n_0,0} f_{n_1;\dots;n_t} -\delta_{n_t,0} f_{n_0;\dots;n_{t-1}} =f_{n_0;n_1;\dots;n_t}\Delta_{T(0,n,0)}(1\tens \bdelt_0).
\end{align*}
The right hand side of \eqref{eq-cT(0n0)1h-cT(x1z)h1} is computed as
\begin{align*}
&(1^{\tens n_0}\tens\bj\tens1^{\tens n_1}\tens\bj\tdt\bj\tens1^{\tens n_t})\tens\sum_{a+2+b=n_0+\dots+n_t+t} (f_1^{\tens a}\tens f_2\tens f_1^{\tens b})\tens f_{a+1+b}
\\
&\mapsto \sum_{i=1}^t \chi(n_{i-1}>0)(f_1^{\tens(n_0+\dots+n_{i-1}-1)}\tens f_{1;0}\tens f_1^{\tens(n_i+\dots+n_t)})\tens f_{n_0;\dots;n_{i-1}+n_i;\dots;n_t}
\\
&\; +\sum_{i=1}^t \chi(n_i>0)(f_1^{\tens(n_0+\dots+n_{i-1})}\tens f_{0;1}\tens f_1^{\tens(-1+n_i+\dots+n_t)})\tens f_{n_0;\dots;n_{i-1}+n_i;\dots;n_t}
\\
&\mapsto -\sum_{i=1}^t \chi(n_{i-1}>0) f_{n_0;\dots;n_{i-1}+n_i;\dots;n_t}
+\sum_{i=1}^t \chi(n_i>0) f_{n_0;\dots;n_{i-1}+n_i;\dots;n_t}
\\
&=\sum_{x+1+z=n}f_{n_0;n_1;\dots;n_t}\Delta_{T(x,1,z)}(\bdelt_0\tens1).
\end{align*}
Here for any condition $P$ the value $\chi(P)=1$ if $P$ holds true and $\chi(P)=0$ otherwise.
The right hand side can be also written as
\begin{align*}
&-\sum_{i=1}^t (1-\delta_{n_{i-1},0}) f_{n_0;\dots;n_{i-1}+n_i;\dots;n_t}
+\sum_{i=1}^t (1-\delta_{n_i,0}) f_{n_0;\dots;n_{i-1}+n_i;\dots;n_t}
\\
&=\sum_{k=0}^{t-1} \delta_{n_k,0} f_{n_0;\dots;\wh{n_k};\dots;n_t}
-\sum_{i=1}^t \delta_{n_i,0} f_{n_0;\dots;\wh{n_i};\dots;n_t}
=\delta_{n_0,0} f_{n_1;\dots;n_t} -\delta_{n_t,0} f_{n_0;\dots;n_{t-1}},
\end{align*}
where indices under $\wh-$ are omitted.
This coincides with the left hand side.
Obviously, $\sfw\bdelt_0=(f_1)\bdelt_0=0$.
Thus, $C$ is an augmented curved cooperad.

Let us show that for the above augmented curved cooperad $C$ the unit-complemented $\dg$\n-operad \(\Cobar C=\bar{C}[-1]\TT\) is isomorphic to $A_\infty^\hu$.
This is already proven by Hirsh and Mill{\`e}s \cite[Theorem 6.1.8]{MR2993002}.
In fact, the graded operad \(\bar{C}[-1]\TT\) is freely generated by \(f_{n_0;n_1;\dots;n_t}\sigma^{-1}\in\bar{C}[-1](n)^{1-2t}\), \(n=\sum_{i=0}^tn_i\), either $n+t>1$ or $t=1$, $n_0=n_1=0$.
These elements are mapped to generators \(b_{n_0;n_1;\dots;n_t}\in A_\infty^\hu(n)^{1-2t}\) if $n+t>1$ and \(f_{0;0}\sigma^{-1}\) is mapped to $\bi$, which we denote also by \(b_{0;0}\in A_\infty^\hu(0)^{-1}\) for uniformity.
This establishes an isomorphism of graded operads \(\bar{C}[-1]\TT\rto\sim A_\infty^\hu\), since the latter is freely generated by all \(b_{n_0;n_1;\dots;n_t}\) including \(b_{0;0}\).
Let us verify that under this identification the differentials in \(\bar{C}[-1]\TT\) and $A_\infty^\hu$ coincide.

The differential in \(\bar{C}[-1]\TT\) is given by
\[ (f_{n_0;n_1;\dots;n_t}\sigma^{-1})d = f_{n_0;n_1;\dots;n_t}\bdelt_0 -\sum_{x+y+z=n}f_{n_0;n_1;\dots;n_t}\bar\Delta_{T(x,y,z)}(\sigma^{-1}\underset{x,y,z}\bull\sigma^{-1}).
\]
In presentation \(C\simeq J\odot P\) of \exaref{exa-cooperad-C-JP} the element \(f_{n_0;n_1;\dots;n_t}\Delta\in(C\odot C)(n)\) is
\[ \sum_{i_1+\dots+i_k=n+t}^{k,i_1,\dots,i_k>0} [(1^{\tens n_0}\tens\bj\tens1^{\tens n_1}\tens\bj\tdt1^{\tens n_{t-1}}\tens\bj\tens1^{\tens n_t})\tens(f_{i_1}\tens f_{i_2}\tdt f_{i_k})]\vartheta\tens f_k.
\]
Only certain summands of its expansion will contribute to \(\bar\Delta_{T(\_,\_,\_)}\).
All of them appear in
\begin{multline*}
[(1^{\tens n_0}\tens\bj\tens1^{\tens n_1}\tens\bj\tdt1^{\tens n_{t-1}}\tens\bj\tens1^{\tens n_t})\tens f_1^{\tens(n+t)})]\vartheta\tens f_{n+t}
\\
+\sum_{a+r+c=n+t}^{r>1,\,a+c>0} [(1^{\tens n_0}\tens\bj\tens1^{\tens n_1}\tens\bj\tdt1^{\tens n_{t-1}}\tens\bj\tens1^{\tens n_t})\tens(f_1^{\tens a}\tens f_r\tens f_1^{\tens c})]\vartheta\tens f_{a+1+c}.
\end{multline*}
Namely, \(f_{0;0}\sigma^{-1}d=0\) and if $n+t>1$, then \(\sum_{x+y+z=n}f_{n_0;n_1;\dots;n_t}\bar\Delta_{T(x,y,z)}\) is the sum of
\begin{equation}
\sum_{i=1}^t (1^{\tens(n_0+\dots+n_{i-1})}\tens f_{0;0}\tens1^{\tens(n_i+\dots+n_t)})\tens f_{n_0;\dots;n_{i-2};n_{i-1}+1+n_i;n_{i+1};\dots;n_t}
\label{eq-fnnn1nnn}
\end{equation}
and of terms of
\begin{equation}
\sum_{a+r+c=n+t}^{r>1,\,a+c>0} (1^{\tens n_0}\tens\bj\tens1^{\tens n_1}\tens\bj\tdt1^{\tens n_{t-1}}\tens\bj\tens1^{\tens n_t})\tens(1^{\tens a}\tens f_r\tens1^{\tens c})\tens f_{a+1+c}
\label{eq-1a-fr-1c-fa1c}
\end{equation}
transformed by moving all $\bj$'s on first $a$ and last $c$ positions to $f_{a+1+c}$.
The remaining $\bj$'s are absorbed by $f_r$.

On the other hand, \(b_{0;0}d=0\) and if $n+t>1$, then
\begin{align}
&b_{n_0;n_1;\dots;n_t}d = [(1^{\tens n_0}\tens\bj\tens1^{\tens n_1}\tens\bj\tdt1^{\tens n_{t-1}}\tens\bj\tens1^{\tens n_t})b_{n+t}]d \notag
\\
&= \sum_{i=1}^t (1^{\tens n_0}\tens\bj\tdt1^{\tens n_{i-1}}\tens\bone^\su\tens1^{\tens n_i}\tdt\bj\tens1^{\tens n_t})b_{n+t}
\label{eq-1j11su1j1}
\\
&- \sum_{i=1}^t (1^{\tens(n_0+\dots+n_{i-1})}\tens b_{0;0}\tens1^{\tens(n_i+\dots+n_t)})\tens b_{n_0;\dots;n_{i-2};n_{i-1}+1+n_i;n_{i+1};\dots;n_t}
\label{eq-1b1-bnnn1nnn}
\\
&- \sum_{a+r+c=n+t}^{r>1,\,a+c>0} (1^{\tens n_0}\tens\bj\tens1^{\tens n_1}\tens\bj\tdt1^{\tens n_{t-1}}\tens\bj\tens1^{\tens n_t})\tens(1^{\tens a}\tens b_r\tens1^{\tens c})\tens b_{a+1+c},
\label{eq-1a-br-1c-ba1c}
\end{align}
where each summand of the last sum is transformed via the same pattern as above: all $\bj$'s on first $a$ and last $c$ positions are moved to $b_{a+1+c}$ and the remaining $\bj$'s are absorbed by $b_r$.

One verifies immediately that \eqref{eq-1j11su1j1} equals \(f_{n_0;n_1;\dots;n_t}\bdelt_0\).
In fact, both vanish unless $n+t=2$ and it remains to consider the cases of $b_2$, $b_{0;1}$, $b_{1;0}$, $b_{0;0;0}$ separately.
For instance, in the case of $b_{0;0;0}$ we have
\[ (\bone^\su\tens\bj)b_2+(\bj\tens\bone^\su)b_2=-\bj+\bj=0.
\]
Sum \eqref{eq-1b1-bnnn1nnn} agrees with \eqref{eq-fnnn1nnn} and transformed sums \eqref{eq-1a-fr-1c-fa1c} and \eqref{eq-1a-br-1c-ba1c} also agree.
\end{example}

Expanding \eqref{eq-1a-br-1c-ba1c} we find that for $n+t>1$
\begin{align*}
&b_{n_0;n_1;\dots;n_t}d =f_{n_0;n_1;\dots;n_t}\bdelt_0
\\
&- \sum_{i=1}^t (1^{\tens(n_0+\dots+n_{i-1})}\tens b_{0;0}\tens1^{\tens(n_i+\dots+n_t)})\tens b_{n_0;\dots;n_{i-2};n_{i-1}+1+n_i;n_{i+1};\dots;n_t}
\\
&- \sum_{\substack{1\le i\le t-1\\ n_i-p-q\ge2}}^{p,q\ge0} (1^{\tens(n_0+\dots+n_{i-1}+p)}\tens b_{n_i-p-q}\tens 1^{\tens(q+n_{i+1}+\dots+n_t)})\tens b_{n_0;\dots;n_{i-1};p+1+q;n_{i+1};\dots;n_t}
\\
&- \sideset{}{''}\sum_{\substack{0\le p\le n_i\\ 0\le q\le n_{k-1}}}^{\substack{0\le i,k\le t+1\\ k-i>1}} (1^{\tens(n_0+\dots+n_{i-1}+p)}\tens b_{n_i-p;n_{i+1};\dots;n_{k-1}-q}\tens 1^{\tens(q+n_k+\dots+n_t)})\tens b_{n_0;\dots;n_{i-1};p+1+q;n_k;\dots;n_t}.
\end{align*}
Here $''$ reminds that two kinds of terms are not included in the sum: one corresponding to $i=k-2$, $p=n_i$, $q=n_{k-1}$ ($b_{0;0}$ would appear in the place of the first $b$), another with $i=0$, $k=t+1$, $p=0=q$ ($b_1$ would appear for the last $b$).

\section{Curved coalgebras over curved cooperads}\label{sec-Curved-coalgebras-over-curved-cooperads}
\begin{definition}\label{def-curved-coalgebra-over-curved-augmented-cooperad}
A \emph{curved comodule} $N$ over\index{TTindex}{curved comodule} a curved cooperad $C$ is a triple \((N,d_N,\delta)\), consisting of a plain right $C$\n-comodule \((N,\delta:N\to N\bar\odot C)\in\cv^\mm\) and a linear map $d=d_N:N\to N$ of degree~1 (an \((\id_N;\id_C,d_C)\)-coderivation) such that
\begin{diagram}
N &\rTTo^\delta &N\bar\odot C
\\
\dTTo<{1+d_Np} &= &\dTTo>{(1+d_Np)\bar\odot(1+d_Cp)}
\\
N &\rTTo^\delta &N\bar\odot C
\end{diagram}
\begin{equation}
d_N^2 =\bigl\langle N(n) \rTTo^{\delta_n} N(n)\tens C(1) \rTTo^{1\tens\bdelt_0} N(n) \bigr\rangle.
\label{eq-dN2-delta-n}
\end{equation}
\emph{Morphism \(f:N\to K\) of curved comodules} over a curved cooperad $C$ is a morphism of plain $C$\n-comodules commuting with derivations:
\[ fd_K =d_Nf.
\]
The category of curved comodules over $C$ is denoted $\comodur C$.

A\index{TTindex}{curved coalgebra over a curved augmented cooperad} \emph{curved coalgebra} over a curved augmented cooperad $C$ is a curved comodule \((V,d_V,\delta)\) with $\mm=0$.
It consists of an object $V$, map $d=d_V:V\to V$ of degree~1 and a family \(\delta=(\delta_l:V\to V^{\tens l}\tens C(l))_{l\ge0}\) coinciding with a coaction \(\delta:V\to V\bar\odot C\).
The category of curved coalgebras over $C$ is denoted $C\cCoalg$.
\end{definition}

Notice that both sides of equation~\eqref{eq-dN2-delta-n} are \((\id_N;\id_C,d_C^2)\)-coderivations.
For the left hand side this follows from \remref{rem-r2-idN:idC:xi2-coderivation}.
For the right hand side this is proven by the following computation which uses \eqref{dia-N-NNC-NNNNC} for \(l_1=\dots=l_k=1\) and for \(l_1=k\in\NN^1\).
For any \(n_1,\dots,n_k\in\mm\) and \(n=n_1+\dots+n_k\) we have
\begin{multline*}
\sum_{x+1+z=k} \bigl\langle N(n) \rTTo^{\delta_{n_1,\dots,n_k}} N(n_1)\tdt N(n_k)\tens C(k) \rTTo^{1^{\tens x}\tens\delta_{n_{x+1}}(1\tens\bdelt_0)\tens1^{\tens z+1}}
\\
\hfill N(n_1)\tdt N(n_{x+1})\tdt N(n_k)\tens C(k) \bigr\rangle \quad
\\
+\bigl\langle N(n) \rTTo^{\delta_{n_1,\dots,n_k}} N(n_1)\tdt N(n_k)\tens C(k) \rTTo^{1^{\tens k}\tens d_C^2} N(n_1)\tdt N(n_k)\tens C(k) \bigr\rangle
\\
\hskip\multlinegap =\sum_{x+1+z=k} \bigl\langle N(n) \rTTo^{\delta_{n_1,\dots,n_k}} N(n_1)\tdt N(n_k)\tens C(k) \rTTo^{1^{\tens k}\tens\Delta_{T(x,1,z)}} \hfill
\\
\hfill N(n_1)\tdt N(n_k)\tens C(1)\tens C(k) \rTTo^{1^{\tens k}\tens\bdelt_0\tens1} N(n_1)\tdt N(n_k)\tens C(k) \bigr\rangle \quad
\\
+\bigl\langle N(n) \rTTo^{\delta_{n_1,\dots,n_k}} N(n_1)\tdt N(n_k)\tens C(k) \rTTo^{1^{\tens k}\tens d_C^2} N(n_1)\tdt N(n_k)\tens C(k) \bigr\rangle
\\
\hskip\multlinegap =\bigl\langle N(n) \rTTo^{\delta_{n_1,\dots,n_k}} N(n_1)\tdt N(n_k)\tens C(k) \rTTo^{1^{\tens k}\tens\Delta_{T(0,k,0)}} \hfill
\\
\hfill N(n_1)\tdt N(n_k)\tens C(k)\tens C(1) \rTTo^{1^{\tens(k+1)}\tens\bdelt_0} N(n_1)\tdt N(n_k)\tens C(k) \bigr\rangle \quad
\\
=\bigl\langle N(n) \rto{\delta_n} N(n)\tens C(1) \rTTo^{1\tens\bdelt_0} N(n) \rTTo^{\delta_{n_1,\dots,n_k}} N(n_1)\tdt N(n_k)\tens C(k) \bigr\rangle.
\end{multline*}

Produce from a connected curved augmented cooperad $C$ the augmented comonad \(\bot=\bot_C\) in\index{TTsymb}{W+@$\cw_+$} $\cw_+=\cw^0_+$, $\mm=0$.
The category of $\bot$\n-coalgebras $\cw_{+\bot}$ contains the full subcategory $\cw_{+\bot}^f$ of cofree $\bot$\n-coalgebras.
They are of the form \((V\bot,\Delta:V\bot\to V\bot\bot)\) for \(V\in\Ob\cw\),
equivalently, of the form \((V[1]\odot C[-1],1\odot\Delta[-1]:V[1]\odot C[-1]\to V[1]\odot C\odot C[-1])\).
According to \lemref{lem-forgetful-functor-has-right-adjoint-T} the $\kk$\n-module of morphisms is
\[ \cw_{+\bot}^f(V\bot,W\bot) \simeq \cw_+(V\bot,W).
\]

We could take for $N\in\cw^\NN_+$ an arbitrary collection, but for simplicity we take \(N=V[1]\) concentrated in zero arity for $V\in\cw_+$, $\mm=0$.
Consider an \((\id_{N\odot C};\id_C,d_C)\)\n-coderivation \(\nabla:N\odot C\to N\odot C\) of degree~1, see \secref{sec-Coderivations}.
It can be presented as \(\nabla=1_N\odot d_C+b\), where $b$ is an \((\id_{N\odot C};\id_C,0)\)\n-coderivation.
The latter is determined by a degree~1 map \(\check{b}:N\odot C\to N\) through the equation
\[ 1+bp =\bigl(N\odot C \rTTo^{1_N\odot\Delta} N\odot C\odot C \rTTo^{(1_N\odot\eps+\check{b}p)\odot1_C} N\odot C\bigr).
\]
As shown in \eqref{eq-nabla-page} the restriction of $b$ to $n$-th summand is given by
\begin{multline}
b =\biggl\langle V[1]^{\tens n}\tens C(n) \rTTo^{(1\tens\Delta_{T(x,y,z)})\;} \bigoplus^{x+1+z=l}_{x+y+z=n} V[1]^{\tens n}\tens C(y)\tens C(l) \rto\sim
\\
\bigoplus^{x+1+z=l}_{x+y+z=n} V[1]^{\tens x}\tens(V[1]^{\tens y}\tens C(y))\tens V[1]^{\tens z}\tens C(l)
\rTTo^{\sum1^{\tens x}\tens\sS{_y}{\check b}\tens1^{\tens z}\tens1} \bigoplus_{l\ge0} V[1]^{\tens l}\tens C(l) \biggr\rangle.
\label{eq-b-V[1]C-V[1]CC-V[1]V[1]CV[1]C-V[1]C}
\end{multline}
The degree~1 maps \(\sS{_i}{\check{b}}:V[1]^{\tens i}\tens C(i)\to V[1]\) induces by closedness of $\cw$ degree~1 maps \(\check{b}_i:C(i)\to\und\cw(V[1]^{\tens i},V[1])\).
They combine to a map of collections \(\bb=(\sigma\check{b}_i)_{i\in\NN}:C[-1]\to\END(V[1])\) in $\cw$ preserving the degree.
Equivalently, $V[1]$ is an algebra over the free operad generated by $C[-1]$: there is an action \(\alpha:C[-1]\Fo\to\END(V[1])\).

Due to \remref{rem-r2-idN:idC:xi2-coderivation} $\nabla^2$ is an \((\id_{N\odot C};\id_C,d_C^2)\)-coderivation of degree~2.
Hence, the difference \(\nabla^2-1_N\odot d_C^2\) is an \((\id_{N\odot C};\id_C,0)\)\n-coderivation of degree~2.
Another \((\id_{N\odot C};\id_C,0)\)-coderivation $\hat\bdelt_0$ of degree~2 is determined by $\bdelt_0$ from the infinitesimal morphism of comodules
\[ 1+\hat\bdelt_0p =\bigl(N\odot C \rTTo^{1_N\odot\Delta} N\odot C\odot C \rTTo^{1_N\odot(\eps+\bdelt_0p)\odot1_C} N\odot C\bigr),
\]
\(\deg p=-2\).
The comodule morphism property is obvious from the diagram
\begin{diagram}
N\odot C &\rTTo^{1_N\odot\Delta} &N\odot C\odot C &\rTTo^{1_N\odot(\eps+\bdelt_0p)\odot1_C} &N\odot C
\\
\dTTo<{1\odot\Delta} &= &\dTTo<{1\odot1\odot\Delta} &= &\dTTo>{1\odot\Delta}
\\
N\odot C\odot C &\rTTo^{1_N\odot\Delta\odot1_C} &N\odot C\odot C\odot C &\rTTo^{1_N\odot(\eps+\bdelt_0p)\odot1_C\odot1_C} &N\odot C\odot C
\end{diagram}

Thus the restriction of inner coderivation $\hat\bdelt_0$ to $n$-th summand is given by
\begin{equation*}
\hat\bdelt_0 =\bigl\langle V[1]^{\tens n}\tens C(n) \rTTo^{1\tens\sum_{x+1+z=n}\Delta_{T(x,1,z)}\;} V[1]^{\tens n}\tens C(1)\tens C(n) \rTTo^{1\tens\bdelt_0\tens1} V[1]^{\tens n}\tens C(n) \bigr\rangle.
\end{equation*}

\begin{proposition}
Let \(V\in\cw^\NN_+\) and let $C$ be a connected curved augmented cooperad.
The following are equivalent:
\begin{myitemize}
\item[1)] A degree~0 map \(\bb=(\sigma\check{b}_i)_{i\in\NN}:C[-1]\to\END(V[1])\) such that \(\nabla=1\odot d_C+b:V[1]\odot C\to V[1]\odot C\) satisfies the equation
\begin{equation}
\nabla^2-1\odot d_C^2=\hat\bdelt_0,
\label{eq-nabla2-1dC2-h}
\end{equation}
where $b$ is given by \eqref{eq-b-V[1]C-V[1]CC-V[1]V[1]CV[1]C-V[1]C}.

\item[2)] A collection of degree~1 maps \(\sS{_i}{\check{b}}:V[1]^{\tens i}\tens C(i)\to V[1]\) that determines $b$ via \eqref{eq-b-V[1]C-V[1]CC-V[1]V[1]CV[1]C-V[1]C} and satisfies
\begin{equation}
\delta_{n,1} \bigl\langle V[1]^{\tens n}\tens C(n) \rTTo^{1\tens\bdelt_0} V[1] \bigr\rangle =\bigl\langle V[1]^{\tens n}\tens C(n) \rTTo^{1\tens d+b} \oplus_{l\ge0} V[1]^{\tens l}\tens C(l) \rto{(\sS{_l}{\check{b}})} V[1] \bigr\rangle.
\label{eq-deltan1-V[1]C-V[1]}
\end{equation}

\item[3)] A curved $C$\n-coalgebra \((V[1]\odot C,\nabla:V[1]\odot C\to V[1]\odot C)\), cofree\index{TTindex}{curved cofree coalgebra} as a $C$\n-coalgebra.
\end{myitemize}
\end{proposition}

\begin{proof}
The both hand sides of \eqref{eq-nabla2-1dC2-h} are \((\id_{V[1]\odot C};\id_C,0)\)-coderivations.
Hence, this equation is equivalent to $\nabla^2\eps_\bot=\delta_{n,1}1\tens\bdelt_0:V[1]^{\tens n}\tens C(n)\to V[1]$.
In detail this is \eqref{eq-deltan1-V[1]C-V[1]}.
Thus, 1) and 2) are equivalent.

We have
\[ 1\odot d_C^2 +\hat\bdelt_0 =\Delta_{T(0,n,0)}\cdot(1\tens\bdelt_0): V[1]^{\tens n}\tens C(n) \to V[1]^{\tens n}\tens C(n)
\]
due to \eqref{eq-d2-Delta(1delta0-Delta(delta01)}.
Therefore, \eqref{eq-nabla2-1dC2-h} is equivalent to \eqref{eq-dN2-delta-n} for \(N=V[1]\odot C\), thus, 1) and 3) are equivalent.
\end{proof}

Using \eqref{eq-nabla-page} we rewrite \eqref{eq-deltan1-V[1]C-V[1]} as a system of equations for arbitrary $n\in\NN$:
\begin{multline}
\delta_{n,1} \bigl\langle C(n) \rto{\bdelt_0} \1 \rTTo^{\dot1_{V[1]}} \und\cv(V[1],V[1]) \bigr\rangle =\bigl\langle C(n) \rto d C(n) \rto{\sS{_n}{\check b}} \und\cv(V[1]^{\tens n},V[1]) \bigr\rangle
\\
\hskip\multlinegap +\biggl\langle C(n) \rTTo^{\Delta_{T(x,y,z)}\;} \bigoplus^{x+1+z=l}_{x+y+z=n} C(y)\tens C(l) \rTTo^{\oplus\check{b}_y\tens\check{b}_l} \hfill
\\
\bigoplus^{x+1+z=l}_{x+y+z=n} \und\cv(V[1]^{\tens y},V[1]) \tens \und\cv(V[1]^{\tens l},V[1])
\rTTo^{\underset{x,y,z}\bull} \und\cv(V[1]^{\tens n},V[1]) \biggr\rangle,
\label{eq-Dn-Cn-C1Fn}
\end{multline}
where \(\underset{x,y,z}\bull\) is the binary multiplication in the operad \(\END(V[1])\).
Similar binary multiplications can be used in the definition of an arbitrary operad, in particular, for $C[-1]\Fo$.
The composite \(\sfb=\bigl(C \rTTo^{\sigma^{-1}} C[-1]\hookrightarrow C[-1]\Fo\bigr)\) allows to write down a map
\begin{multline*}
D(n) =-\bigl\langle C(n) \rto{\bdelt_0} \1 \rto\eta C[-1]\Fo(n) \bigr\rangle +\bigl\langle C(n) \rto d C(n) \rto{\sfb(n)} C[-1]\Fo(n) \bigr\rangle
\\
\hskip\multlinegap \phantom{D(n)} +\Bigl\langle C(n) \rTTo^{\Delta_{T(x,y,z)}\;} \bigoplus_{x+y+z=n} C(y)\tens C(x+1+z) \rTTo^{\oplus\sfb(y)\tens\sfb(x+1+z)} \hfill
\\
\bigoplus_{x+y+z=n} C[-1]\Fo(y)\tens C[-1]\Fo(x+1+z) \rTTo^{\underset{x,y,z}\bull} C[-1]\Fo(n) \Bigr\rangle.
\end{multline*}
Its image consists of ``elements'', whose action in $V[1]$ vanishes.
Define an operad in $\cv$ as the quotient \(\tilde{\co}=C[-1]\Fo/(\im D)\) over the ideal generated by \(\im D(n)\) for all $n\ge0$.
Clearly, a curved cofree $\bot$\n-coalgebra is the same as an \(\tilde{\co}\)\n-algebra in $\cv$.

Now let us use the augmentation of $C$.
The identity \(\bigl(\1\rto\sfw C\rto\eps \1\bigr)=\id\) implies that \(C=\1\oplus\bar C\) for the collection \(\bar{C}=\Ker\eps\).
The map \(\bb:C[-1]\to\END(V[1])\), restricts to a map \(\1(1)\rto\sfw C(1)\rto{b} \und\cw(V[1],V[1])\), which by closedness is equivalent to a degree~1 map \(\bd:V[1]\to V[1]\).
The map $\bd$ satisfies the equation $\bd^2=\sfw\bdelt_0-(\sfw d_C)\check{b}_1$, which is just \eqref{eq-Dn-Cn-C1Fn} applied to $\sfw$.
There is also restriction \(\overline\bb=\bb|:\bar C[-1]\to\END(V[1])\) and the canonical degree~1 map
\[ \overline\sfb =\bigl(\bar C \rTTo^{\sigma^{-1}} \bar C[-1]\hookrightarrow \bar C[-1]\Fo\bigr).
\]

\begin{proposition}
There is a functor from the category of curved cofree $C$\n-coalgebras to category $\Cobar C\cAlg$ of curved $\Cobar C$-algebras,
\[ \bigl(V[1],\bb:C[-1]\to\END(V[1])\bigr) \mapsto \bigl(V[1],(\sfw\sigma^{-1})\bb,\alpha:\Cobar C\to\END(V[1])\bigr).
\]
\end{proposition}

\begin{proof}
A coderivation $b$ in $V\bot$ amounts to a degree~1 map $\bd=\bw\bb:V[1]\to V[1]$ such that $\bd^2=\sfw\bdelt_0-(\sfw d_C)\check{b}_1$ and an action \(\bar{C}[-1]\Fo\to\END(V[1])\).
Equation~\eqref{eq-nabla2-1dC2-h} or equivalently vanishing of \(\im D\) in \(\END(V[1])\) can be written as chain property for the action map $\alpha$:
\begin{diagram}[LaTeXeqno]
C[-1](n) &\rTTo^\alpha &\und\cw(V[1]^{\tens n},V[1])
\\
\dTTo<d &&\dTTo>{[-,\bd]}
\\
C[-1]\Fo(n) &\rTTo^\alpha &\und\cw(V[1]^{\tens n},V[1])
\label{dia-d-alpha-alpha-[-d]}
\end{diagram}
where \([f,\bd]=f\bd-(-1)^f\sum_{x+1+z=n}(1^{\tens x}\tens\bd\tens1^{\tens z})f\) is the usual commutator for an ``element'' \(f:X\to\und\cw(V[1]^{\tens n},V[1])\) and
\begin{align}
d &=\_\bull \bd -\sum_{x+1+z=n} \bd\underset{x,1,z}\bull\_ -\sigma D(n)
\label{eq-a-partial-ad}
\\
&=\_\bull\bd -\sum_{x+1+z=n} \bd\underset{x,1,z}\bull\_ +\sigma\bdelt_0 -\sigma d\sigma^{-1} -\sum_{x+y+z=n}\sigma\Delta_{T(x,y,z)}(\sigma^{-1}\underset{x,y,z}\bull \sigma^{-1}), \notag
\end{align}
where \(\bd=\sfw\sigma^{-1}\in C(1)[-1]^1\) is the composition \(\1\rto\sfw C(1)\rto{\sigma^{-1}} C(1)[-1]\), and \(\underset{x,y,z}\bull\) means the multiplication in \(C[-1]\Fo\):
\[ \underset{x,y,z}\bull: C[-1]\Fo(y)\tens C[-1]\Fo(x+1+z) \to C[-1]\Fo(x+y+z).
\]
For instance,
\[ \bd d =\bd\bull\bd +\bd\bull\bd +\sfw\bdelt_0 -\sfw d_C\sigma^{-1} -\bd\bull\bd =\bd\bull\bd +\sfw\bdelt_0 -\sfw d_C\sigma^{-1}.
\]

Due to \eqref{eq-Delta-1w-w1-Delta} the restriction of the derivation $d$ to \(\bar{C}(n)[-1]\) has its image in \(\bar{C}[-1]\Fo(n)\).
In fact, \(\sigma d\sigma^{-1}\) takes \(\bar{C}(n)[-1]\) to itself.
Summands of \eqref{eq-a-partial-ad} containing $\bd$ cancel against summands of $\Delta_{T(x,y,z)}$ containing $\sfw$.
Thus,
\begin{align}
d &=\sigma\bdelt_0 -\sigma d_C\sigma^{-1} -\sum_{x+y+z=n}\sigma\bar\Delta_{T(x,y,z)}(\sigma^{-1}\underset{x,y,z}\bull \sigma^{-1}) \notag
\\
&=\xi_0 +\xi_1 +\sum_{x+y+z=n} \bar\xi_{T(x,y,z)}\inj_{T(x,y,z)}: \bar{C}(n)[-1] \to \bar{C}[-1]\Fo(n).
\label{eq-a-partial-bar-Delta}
\end{align}
Equip $V[1]$ with the differential $\bw\bb$.
Then \eqref{dia-d-alpha-alpha-[-d]} shows that \(\bigl(V[1],\bw\bb,\alpha:\Cobar C\to\END(V[1])\bigr)\) is a \(\Cobar C\)-algebra.

Clearly, morphisms of curved cofree $C$\n-coalgebras are taken to morphisms of curved $\Cobar C$-algebras.
\end{proof}

\section{Curved algebras over curved operads}
\begin{definition}
A \emph{curved module} $M$ over\index{TTindex}{curved module} a curved operad $\co$ is a triple $(M,d_M,\alpha)$, consisting of a plain right $\co$\n-module $(M,\alpha:M\odot\co\to M)\in\cv^\mm$ and a map $d=d_M:M\to M$ of degree~1 (an \((\id_M;\id_\co,d_\co)\)-derivation) such that
\begin{diagram}
M\odot\co &\rTTo^\alpha &M
\\
\dTTo<{(1+d_Mp)\odot(1+d_\co p)} &= &\dTTo>{1+d_Mp}
\\
M\odot\co &\rTTo^\alpha &M
\end{diagram}
\begin{equation}
(1_M\tens m^\co_0)\alpha +d_M^2 =0: M(n) \to M(n).
\label{eq-(1m)alpha-d2}
\end{equation}
\emph{Morphism \(f:M\to L\) of curved modules} over a curved operad $\co$ is a morphism of plain $\co$\n-modules commuting with derivations:
\[ fd_L =d_Mf.
\]
The category of curved modules over $\co$ is denoted $\modur\co$.

A\index{TTindex}{curved algebra over a curved operad} \emph{curved $\co$\n-algebra} is a curved $\co$\n-module with $\mm=0$.
Equivalently, it is a triple $(V,d,\alpha)$, consisting of an object $V$, a map $d=d_V:V\to V$ of degree~1 and a morphism \((\alpha,0):\co\to\END(V,d)\in\cOp\), where curved structure of $\END(V,d)$ is that of \exaref{exa-END(Vd)}.
Thus, $\alpha$ is an action and
\[ \alpha d_{\END V} =d_\co \alpha, \qquad (1_V\tens m^\co_0)\alpha +d_V^2 =0.
\]
The category of curved algebras over $\co$ is denoted $\co\cAlg$.
\end{definition}

Dually to \secref{sec-Curved-coalgebras-over-curved-cooperads} one can prove that the both terms of equation~\eqref{eq-(1m)alpha-d2} are \((\id_M;\id_\co,d_\co^2)\)-derivations.

Recall that $(\cw,\tens)$ denotes the symmetric monoidal category $\cv^\NN$ of $\NN$\n-graded objects \(X=(X_n)_{n\in\NN}\) of $\cv$ equipped with the usual tensor product \((X\tens Y)_m=\oplus_{k+l=m}X_k\tens Y_l\).
The symmetry in $\cw$ is induced by that of $\cv$.

Let us describe another side of the bar construction: unit-complemented curved operad \(\co=(\1\oplus\bar\co,d)\) in $\cw$ with \(\bar\co\in\cw^\NN_{++}\) is taken to augmented curved cooperad \(C=(\1\oplus\bar C,d)\) in $\cw$ with \(\bar C\in\cw^\NN_{++}\).
Start with a unit-complemented curved operad $\co$ in $\cw$.
Produce the monad \(\top=\top_\co\) in\index{TTsymb}{top@$\top$} $\cw$:
\[ V\top =\bigoplus_{n\in\NN} V[-1]^{\tens n}\tens\co(n)[1].
\]

A $\top$\n-algebra consists of $\bigl(A,\alpha=(\alpha_n:A[-1]^{\tens n}\tens\co(n)[1]\to A)_{n\ge0}\bigr)$. 
The category of algebras over the monad $\top$ is isomorphic to the category of algebras over the operad $\co$ via the shift map \((A,\alpha)\mapsto\bigl(A[-1],(\alpha_n[-1]:A[-1]^{\tens n}\tens\co(n)\to A[-1])_{n\ge0}\bigr)\).
The category of $\top$\n-algebras $\cw^\top$ contains the full subcategory $\cw^\top_f$ of free $\top$\n-algebras.
They are of the form \((V\top,m:V\top\top\to V\top)\) for \(V\in\Ob\cw\), equivalently, of the form \((V[-1]\odot\co[1],1\odot\mu[1]:V[-1]\odot\co\odot\co[1]\to V[-1]\odot\co[1])\).

According to \lemref{lem-forgetful-functor-has-right-adjoint-T} the module of morphisms is
\[ \cw^\top_f(V\top,W\top) \simeq \cw(V,W\top).
\]
More generally, the following bijections are inverse to each other:
\begin{align*}
\cw(V,A)\; &\longleftrightarrow \;\cw^\top(V\top,A),
\\
f\; &\rMapsTo \;\hat f=\bigl( V\top \rTTo^{f\top} A\top \rTTo^\alpha A\bigr),
\\
\check g = \eta_\top\cdot g\; &\lMapsTo \;g.
\end{align*}

Consider an \((\id_{V\top[-1]};\id_\co,d)\)-derivation \(\nabla:V\top[-1]\to V\top[-1]\), \(V\in\cw\), see \eqref{dia-MMO-NNP-M-N}.
An example of such derivation is \(1_N\odot d=1\tens d\), given by \(1^{\tens n}\tens d:V[-1]^{\tens n}\tens\co(n)\to V[-1]^{\tens n}\tens\co(n)\), $n\ge0$.
Here the collection $N$ is $V[-1]$ concentrated in arity~0.
The difference of two such coderivations \(b=\nabla-1\tens d\) is an \((\id_{V\top[-1]};\id_\co,0)\)-derivation.
Furthermore, $\nabla^2$ and $1\tens d^2$ are \((\id_{V\top[-1]};\id_\co,d^2)\)-derivations of degree~2.
Hence, \(\nabla^2-1\tens d^2\) is an \((\id_{V\top[-1]};\id_\co,0)\)-derivation of degree~2.
Another \((\id_{V\top[-1]};\id_\co,0)\)-derivation $\hat{m}_0$ of degree~2 is determined by $m_0$ from the infinitesimal morphism of modules
\begin{equation}
1 +\hat{m}_0p =\bigl( N\odot\co \rTTo^{1_N\odot(\eta+m_0p)\odot1_\co} N\odot\co\odot\co \rTTo^{1_N\odot m} N\odot\co \bigr),
\label{eq-1m0p-NO-NOO}
\end{equation}
\(\deg p=-2\).
The module morphism property is obvious from
\begin{diagram}
N\odot\co\odot\co &\rTTo^{1_N\odot(\eta+m_0p)\odot1_\co\odot1_\co} &N\odot\co\odot\co\odot\co &\rTTo^{1_N\odot m\odot1_\co} &N\odot\co\odot\co
\\
\dTTo<{1\odot m} &= &\dTTo>{1\odot1\odot m} &= &\dTTo>{1\odot m}
\\
N\odot\co &\rTTo^{1_N\odot(\eta+m_0p)\odot1_\co} &N\odot\co\odot\co &\rTTo^{1_N\odot m} &N\odot\co
\end{diagram}

\begin{proposition}
Let \(V\in\cw\) and let $\co$ be a unit-complemented curved operad.
The following data are equivalent:
\begin{myitemize}
\item[1)] An \((\id_{V\top[-1]};\id_\co,0)\)-derivation $b:V\top[-1]\to V\top[-1]$ of degree~1 which satisfies the equation
\begin{equation}
(1\tens d+b)^2 -1\tens d^2 +\hat{m}_0 =0: V\top[-1] \to V\top[-1]
\label{eq-(1db)2-1d2-m0}
\end{equation}
where $\hat{m}_0$ is given by \eqref{eq-1m0p-NO-NOO}.

\item[2)] An \((\id_{V\top[-1]};\id_\co,0)\)-derivation $b:V\top[-1]\to V\top[-1]$ of degree~1 which satisfies the equations, $n\in\NN$,
\begin{equation}
\eta_\top[-1]b(1\tens d+b) +\delta_{n,1}1\tens m_0 =0:V[-1]\to V[-1]^{\tens n}\tens\co(n). 
\label{eq-nTb-(1db)-d1m0}
\end{equation}

\item[3)] A curved $\co$\n-algebra \((V[-1]\odot\co,1\odot d+b:V[-1]\odot\co\to V[-1]\odot\co)\), free\index{TTindex}{curved free algebra over an operad} as an $\co$\n-algebra.
\end{myitemize}
\end{proposition}

\begin{proof}
Equivalence of 1) and 2) is due to freeness of $V\top$.

Properties 1) and 3) are equivalent since
\[ \hat{m}_0 -1\odot d^2 =(1\tens m_0^\co)\alpha: V[-1]\odot\co \to V[-1]\odot\co
\]
thanks to \eqref{eq-d2-(m01)m-(1m0)m}.
\end{proof}

\begin{proposition}
Let $\co$ be a unit-complemented curved operad in $\cw$ with \(\bar\co\in\cw^\NN_{++}\).
Then there is a functor
\[ (V\top,1\tens d+b:V\top[-1]\to V\top[-1])\mapsto(V[-1],\delta:V[-1]\to V[-1]\odot\bar\co[1]\botto)
\]
from the category of curved free $\co$\n-algebras to category \(\Bbar(\co,m,-d,m_0,\eta,\sfv)\cCoalg\) of curved $\Bbar(\co,m,-d,m_0,\eta,\sfv)$-coalgebras.
\end{proposition}

\begin{proof}
A derivation $b$ determines a sequence of degree 1 maps
\(\sS{_n}{\check{b}}:V[-1]\to V[-1]^{\tens n}\tens\co(n)\), $n\in\NN$.
Equivalently, it determines a sequence of degree 0 maps
\[ \bigl\langle V[-1] \rTTo^{\sS{_n}{\check{b}}} V[-1]^{\tens n}\tens\co(n) \rTTo^{1\tens\sigma} V[-1]^{\tens n}\tens\co(n)[1] \bigr\rangle. 
\]
Namely, an equation dual to \eqref{eq-nabla-page} holds:
\begin{multline*}
b =\biggl\langle V\top[-1] =\bigoplus_{n\ge0} V[-1]^{\tens n}\tens\co(n)
\rTTo^{\oplus_n\sum1^{\tens x}\tens\sS{_y}{\check b}\tens1^{\tens z}\tens1}
\\
\bigoplus^{x+1+z=n}_{x,y,z\ge0} V[-1]^{\tens x}\tens V[-1]^{\tens y}\tens\co(y)\tens V[-1]^{\tens z}\tens\co(n) \rto\sim
\\
\bigoplus_{x,y,z\ge0} V[-1]^{\tens(x+y+z)}\tens\co(y)\tens\co(x+1+z) \rTTo^{(1\tens\underset{x,y,z}\bull)} \bigoplus_{l\ge0}V[-1]^{\tens l}\tens\co(l) =V\top[-1] \biggr\rangle.
\end{multline*}
The sequence \(\sS{_n}{\check{b}}\) splits into two parts: a map \(\bd=-\sS{_1}{\check{b}}(1\tens\sfv):V[-1]\to V[-1]\) and a sequence \(\sS{_n}{\check{\und b}}=\sS{_n}{\check{b}}(1\tens\opr):V[-1]\to V[-1]^{\tens n}\tens\bar\co(n)\), $n\in\NN$.
Assume that \(\bar\co\in\cw^\NN_{++}\).
The corresponding sequence of degree 0 maps
\[ \bb(n) =\bigl\langle V[-1] \rTTo^{\sS{_n}{\check{\und b}}} V[-1]^{\tens n}\tens\bar\co(n) \rTTo^{1\tens\sigma} V[-1]^{\tens n}\tens\bar\co[1](n) \bigr\rangle
\]
is equivalent to a non-counital \(\bar\co[1]\bott\)-coalgebra structure \((V[-1],(\bar\delta_n)_{n\in\NN})\) due to \propref{pro-TX-coalgebra-structures}, hence, to a counital \(\bar\co[1]\botto\)-coalgebra structure \((V[-1],(\delta_n)_{n\in\NN})\), \(\delta_1=(\bar\delta_1,\id):V[-1]\to(V[-1]\tens\bar\co[1]\bott(1))\oplus V[-1]\), \(\delta_n=\bar\delta_n\) for $n\ne1$.

Impose on $b$ restriction~\eqref{eq-(1db)2-1d2-m0}, or equivalently \eqref{eq-nTb-(1db)-d1m0}.
This means that the equations hold:
\begin{multline*}
\bigl\langle V[-1] \rTTo^{(\sS{_n}{\check{b}})} \oplus_{n\ge0} V[-1]^{\tens n}\tens\co(n) \rTTo^{1\tens d+b} \oplus_{l\ge0} V[-1]^{\tens l}\tens\co(l) \bigr\rangle
\\
+\bigl\langle V[-1] \rTTo^{1\tens m_0} V[-1]\tens\co(1) \rMono \oplus_{l\ge0} V[-1]^{\tens l}\tens\co(l) \bigr\rangle =0.
\end{multline*}
It is expanded to
\begin{multline*}
-\bigl\langle V[-1] \rTTo^{1\tens m_0} V[-1]\tens\co(1) \rMono \oplus_{l\ge0} V[-1]^{\tens l}\tens\co(l) \bigr\rangle
\\
\hskip\multlinegap =\Bigl\langle V[-1] \rTTo^{(\sS{_l}{\check{b}})} \bigoplus_{l\ge0} V[-1]^{\tens l}\tens\co(l) \rTTo^{1\tens d} \bigoplus_{l\ge0}V[-1]^{\tens l}\tens\co(l) \Bigr\rangle \hfill
\\
\hskip\multlinegap +\biggl\langle V[-1] \rTTo^{(\sS{_n}{\check{b}})} \bigoplus_{n\ge0} V[-1]^{\tens n}\tens\co(n)
\rTTo^{\oplus_n\sum1^{\tens x}\tens\sS{_y}{\check b}\tens1^{\tens z}\tens1} \hfill
\\
\bigoplus^{x+1+z=n}_{x,y,z\ge0} V[-1]^{\tens x}\tens V[-1]^{\tens y}\tens\co(y)\tens V[-1]^{\tens z}\tens\co(n) \rto\sim
\\
\hfill \bigoplus_{x,y,z\ge0}^{l=x+y+z} V[-1]^{\tens(x+y+z)}\tens\co(y)\tens\co(x+1+z) \rTTo^{(1\tens\underset{x,y,z}\bull)} \bigoplus_{l\ge0}V[-1]^{\tens l}\tens\co(l) \biggr\rangle \quad
\\
\hskip\multlinegap =\Bigl\langle V[-1] \rTTo^{(\sS{_l}{\check{\und b}})} \bigoplus_{l\ge0} V[-1]^{\tens l}\tens\bar\co(l) \rTTo^{1\tens d} \bigoplus_{l\ge0}V[-1]^{\tens l}\tens\co(l) \Bigr\rangle \hfill	
\\
\hskip\multlinegap +\biggl\langle V[-1] \rTTo^{(\sS{_n}{\check{\und b}})} \bigoplus_{n\ge0} V[-1]^{\tens n}\tens\bar\co(n)
\rTTo^{\oplus_n\sum1^{\tens x}\tens\sS{_y}{\check{\und b}}\tens1^{\tens z}\tens1} \hfill
\\
\bigoplus^{x+1+z=n}_{x,y,z\ge0} V[-1]^{\tens x}\tens V[-1]^{\tens y}\tens\bar\co(y)\tens V[-1]^{\tens z}\tens\bar\co(n) \rto\sim
\\
\hfill \bigoplus_{x,y,z\ge0}^{l=x+y+z} V[-1]^{\tens(x+y+z)}\tens\bar\co(y)\tens\bar\co(x+1+z) \rTTo^{(1\tens\underset{x,y,z}\bull)} \bigoplus_{l\ge0}V[-1]^{\tens l}\tens\co(l) \biggr\rangle \quad
\\
\hskip\multlinegap +\bigl\langle V[-1] \rTTo^\bd V[-1] \rTTo^{\bd\tens\eta} V[-1]\tens\co(1) \rMono \oplus_{l\ge0} V[-1]^{\tens l}\tens\co(l) \bigr\rangle \hfill
\\
\hskip\multlinegap -\Bigl\langle V[-1] \rTTo^\bd V[-1]  \rTTo^{(\sS{_l}{\check{\und b}})} \bigoplus_{l\ge0} V[-1]^{\tens l}\tens\bar\co(l) \Bigr\rangle \hfill
\\
\hskip\multlinegap -\Bigl\langle V[-1] \rTTo^{(\sS{_l}{\check{\und b}})} \bigoplus_{l\ge0} V[-1]^{\tens l}\tens\bar\co(l)
\rTTo^{\oplus_l\sum1^{\tens x}\tens\bd\tens1^{\tens z}\tens1}
\bigoplus_{l\ge0}V[-1]^{\tens l}\tens\bar\co(l) \Bigr\rangle. \hfill
\end{multline*}
One of the equations obtained by postcomposing with $1\tens\sfv$ says that
\begin{multline*}
\bigl\langle V[-1] \rTTo^{1\tens m_0} V[-1]\tens\co(1) \rTTo^{1\tens\sfv} V[-1] \bigr\rangle
\\
\hskip\multlinegap +\Bigl\langle V[-1] \rTTo^{\sS{_1}{\check{\und b}}} V[-1]\tens\bar\co(1) \rTTo^{1\tens d} V[-1]\tens\co(1) \rTTo^{1\tens\sfv} V[-1] \Bigr\rangle \hfill
\\
\hskip\multlinegap +\biggl\langle V[-1] \rTTo^{(\sS{_n}{\check{\und b}})} \bigoplus_{n=1}^2 V[-1]^{\tens n}\tens\bar\co(n) \rTTo^{\oplus_n\sum1^{\tens x}\tens\sS{_y}{\check{\und b}}\tens1^{\tens z}\tens1} \hfill
\\
\bigoplus^{x+y+z=1}_{x,y,z\ge0} V[-1]^{\tens x}\tens V[-1]^{\tens y}\tens\bar\co(y)\tens V[-1]^{\tens z}\tens\bar\co(x+1+z) \rto\sim
\\
\hfill \bigoplus_{x,y,z\ge0}^{x+y+z=1} V[-1]\tens\bar\co(y)\tens\bar\co(x+1+z) \rTTo^{(1\tens\underset{x,y,z}\bull)} V[-1]\tens\co(1) \rTTo^{1\tens\sfv} V[-1] \biggr\rangle \quad
\\
\hskip\multlinegap +\bigl\langle V[-1] \rTTo^{\bd^2} V[-1] \bigr\rangle =0. \hfill
\end{multline*}
In equivalent form
\begin{multline*}
d_V^2 +1_V\tens b_0\bv +\bigl\langle V[-1] \rTTo^{\bb(1)} V[-1]\tens\bar\co[1](1) \rTTo^{-1\tens b_1^\co} V[-1]\tens\co[1](1) \rTTo^{1\tens\bv} V[-1] \bigr\rangle
\\
\hskip\multlinegap +\biggl\langle V[-1] \rTTo^{(\bb(n))} \bigoplus_{n=1}^2 V[-1]^{\tens n}\tens\bar\co[1](n) \rTTo^{\oplus_n\sum1^{\tens x}\tens\bb(y)\tens1^{\tens z}\tens1} \hfill
\\
\bigoplus^{x+y+z=1}_{x,y,z\ge0} V[-1]^{\tens x}\tens V[-1]^{\tens y}\tens\bar\co[1](y)\tens V[-1]^{\tens z}\tens\bar\co[1](x+1+z) \rto\sim
\\
\bigoplus_{x,y,z\ge0}^{x+y+z=1} V[-1]\tens\bar\co[1](y)\tens\bar\co[1](x+1+z) \rTTo^{(1\tens b_{T(x,y,z)})} V[-1]\tens\co[1](1) \rTTo^{1\tens\bv} V[-1] \biggr\rangle =0.
\end{multline*}
Summing up,
\[ d_V^2 =\bigl\langle V \rTTo^{\delta_1} V\tens C(1) \rTTo^{1\tens\bdelt^C_0} V\bigr\rangle,
\]
where \(C=\Bbar(\co,m,-d,m_0,\eta,\sfv)\).

The remaining equations obtained by postcomposing with $1\tens\opr$ tell that for each $l\in\NN$
\begin{multline*}
\bigl\langle V[-1] \rTTo^\bd V[-1]  \rTTo^{\sS{_l}{\check{\und b}}} V[-1]^{\tens l}\tens\bar\co(l) \bigr\rangle \hfill
\\
\hskip\multlinegap =-\bigl\langle V[-1] \rTTo^{\sS{_l}{\check{\und b}}} V[-1]^{\tens l}\tens\bar\co(l)
\rTTo^{\sum_{x+1+z=l}1^{\tens x}\tens\bd\tens1^{\tens z}\tens1}
V[-1]^{\tens l}\tens\bar\co(l) \bigr\rangle \hfill
\\
\hskip\multlinegap +\bigl\langle V[-1] \rTTo^{\sS{_l}{\check{\und b}}} V[-1]^{\tens l}\tens\bar\co(l) \rTTo^{1\tens d} V[-1]^{\tens l}\tens\co(l) \rTTo^{1\tens\opr} V[-1]^{\tens l}\tens\bar\co(l) \bigr\rangle \hfill
\\
\hskip\multlinegap +\biggl\langle V[-1] \rTTo^{(\sS{_n}{\check{\und b}})} \bigoplus_{n\ge0} V[-1]^{\tens n}\tens\bar\co(n)
\rTTo^{\oplus_n\sum1^{\tens x}\tens\sS{_y}{\check{\und b}}\tens1^{\tens z}\tens1} \hfill
\\
\bigoplus^{x+1+z=n}_{x,y,z\ge0} V[-1]^{\tens x}\tens V[-1]^{\tens y}\tens\bar\co(y)\tens V[-1]^{\tens z}\tens\bar\co(n) \rto\sim
\\
\bigoplus_{x,y,z\ge0}^{x+y+z=l} V[-1]^{\tens l}\tens\bar\co(y)\tens\bar\co(x+1+z) \rTTo^{(1\tens\underset{x,y,z}\bull)} V[-1]^{\tens l}\tens\co(l) \rTTo^{1\tens\opr} V[-1]^{\tens l}\tens\bar\co(l) \biggr\rangle
\\
\hskip\multlinegap +\delta_{l,1}\bigl\langle V[-1] \rTTo^{1\tens m_0} V[-1]\tens\co(1) \rTTo^{1\tens\opr} V[-1]\tens\bar\co(1) \bigr\rangle =0. \hfill
\end{multline*}
Equivalent form is
\begin{multline*}
\bigl\langle V[-1] \rTTo^\bd V[-1]  \rTTo^{\bb(l)} V[-1]^{\tens l}\tens\bar\co[1](l) \bigr\rangle
\\
\hskip\multlinegap =\bigl\langle V[-1] \rTTo^{\bb(l)} V[-1]^{\tens l}\tens\bar\co[1](l)
\rTTo^{\sum_{x+1+z=l}1^{\tens x}\tens\bd\tens1^{\tens z}\tens1}
V[-1]^{\tens l}\tens\bar\co[1](l) \bigr\rangle \hfill
\\
\hskip\multlinegap +\delta_{l,1}\bigl\langle V[-1] \rTTo^{1\tens b_0} V[-1]\tens\co[1](1) \rTTo^{1\tens\opr} V[-1]\tens\bar\co[1](1) \bigr\rangle \hfill
\\
\hskip\multlinegap -\bigl\langle V[-1] \rTTo^{\bb(l)} V[-1]^{\tens l}\tens\bar\co[1](l) \rTTo^{1\tens d[1]} V[-1]^{\tens l}\tens\co[1](l) \rTTo^{1\tens\opr} V[-1]^{\tens l}\tens\bar\co[1](l) \bigr\rangle \hfill
\\
\hskip\multlinegap +\biggl\langle V[-1] \rTTo^{(\bb(n))} \bigoplus_{n\ge0} V[-1]^{\tens n}\tens\bar\co[1](n)
\rTTo^{\oplus_n\sum1^{\tens x}\tens\bb(y)\tens1^{\tens z}\tens1} \hfill
\\
\bigoplus^{x+1+z=n}_{x,y,z\ge0} V[-1]^{\tens x}\tens V[-1]^{\tens y}\tens\bar\co[1](y)\tens V[-1]^{\tens z}\tens\bar\co[1](n) \rto\sim
\\
\bigoplus_{x,y,z\ge0}^{l=x+y+z} V[-1]^{\tens(x+y+z)}\tens\bar\co[1](y)\tens\bar\co[1](x+1+z) \rTTo^{-1\tens(\sigma^{-1}\tens\sigma^{-1})\cdot(\underset{x,y,z}\bull)\cdot\sigma}
\\
V[-1]^{\tens l}\tens\co[1](l) \rTTo^{1\tens\opr} V[-1]^{\tens l}\tens\bar\co[1](l) \biggr\rangle.
\end{multline*}
This equation can be rewritten as commutativity of the following diagram
\begin{diagram}[nobalance,LaTeXeqno]
V[-1] &\rTTo^{\delta_l} &V[-1]^{\tens l}\tens\bar\co[1]\botto(l)
\\
\dTTo<\bd &= &\dTTo>{\sum_{x+1+z=l}1^{\tens x}\tens\bd\tens1^{\tens z}\tens\eps_{\botto}+1\tens\check d}
\\
V[-1] &\rTTo^{\bb(l)} &V[-1]^{\tens l}\tens\bar\co[1](l)
\label{dia-V[-1]-db(l)}
\end{diagram}
where \(\check d:\bar\co[1]\botto(l)\to\bar\co[1](l)\) is given by its restrictions $\check d_t$ to components \(\tens^{p\in\IV(t)}\bar\co[1]|p|\) indexed by $t$:
\begin{align}
d_\circ &=b_0\opr =\bar b_\circ: \1 \to \bar\co[1](1), \notag
\\
\check d_{\tau[l]} &=-d[1]\opr =-\bar b_{\tau[l]}: \bar\co[1](l) \to \bar\co[1](l), \qquad l\ge0, \notag
\\
\check d_{T(x,y,z)} &=-(\sigma^{-1}\tens\sigma^{-1})\cdot\underset{x,y,z}\bull\cdot\sigma\opr =\bar b_{T(x,y,z)}: \notag
\\
&\hspace*{6em} \bar\co[1](y)\tens\bar\co[1](x+1+z) \to \bar\co[1](x+y+z),
\label{eq-check-partial-OO-O}
\end{align}
for all \(x,y,z\in\NN\) and \(\check d_t=0\) for other \(t\in\tr\).

The degree (0,1) map $\check d$ extends to a coderivation \(d=d^{\Bbar(\co,-d)}:\bar\co[1]\botto\to\bar\co[1]\botto\) by \corref{cor-coderivations-bottom}.
The map $\delta_l$ is a chain map with respect to $\bd$ and $d$, that is, the square in the following diagram commutes:
\begin{diagram}[nobalance]
V[-1] &\rTTo^{\delta_l} &V[-1]^{\tens l}\tens\bar\co[1]\botto(l)
\\
\dTTo<\bd &&\dTTo>{\sum_{x+1+z=l}1^{\tens x}\tens\bd\tens1^{\tens z}\tens1+1\tens d}
\\
V[-1] &\rTTo^{\delta_l} &V[-1]^{\tens l}\tens\bar\co[1]\botto(l) &\rTTo^{1\tens\eps_{\botto}\;\;\;} &V[-1]^{\tens l}\tens\bar\co[1](l).
\end{diagram}
In fact, the difference of two paths in the square is an infinitesimal deformation of \(\bar\co[1]\botto\)-coalgebra $V[-1]$ by \exeref{exe-infinitesimal-C-coalgebra-deformation}.
It vanishes iff its postcomposition with $1\tens\eps_{\botto}$ vanishes by \exaref{exa-infinitesimal-deformations}.
And the last requirement is precisely commutativity of \eqref{dia-V[-1]-db(l)}.
As shown in \propref{pro-bar-construction} \(C=\Bbar(\co,m,-d,m_0,\eta,\sfv)\) is a curved augmented cooperad.
It follows from \defref{def-curved-coalgebra-over-curved-augmented-cooperad} that $V[-1]$ is a curved coalgebra over $C$.

Clearly, morphisms of curved free $\co$\n-algebras are mapped to morphisms of curved $\Bbar(\co,m,-d,m_0,\eta,\sfv)$-coalgebras.
\end{proof}

\section{Functors determined by twisting cochains}
We use twisting cochains in order to construct functors between categories of right (co)modules over (co)operads, cf. Berger and Moerdijk \cite[Section~8.5.3]{math/0502155}.

Let \(\theta:C\to\co\) be a twisting cochain between a curved cooperad $C$ and a curved operad $\co$.

\begin{proposition}\label{pro-twisting-cochain-functor}
A twisting cochain $\theta$ determines a functor \(\text-\odot^\theta C:\modur\co\to\comodur C\), \(M\mapsto M\odot^\theta C\), which is the cofree right $C$\n-comodule \(M\odot C\) equipped with the derivation
\begin{multline*}
d_{M\odot^\theta C} =S_1 +S_2 +S_3,
\\
\hskip\multlinegap S_1 =\bigl[ M(n_1)\tdt M(n_k)\tens C(k) \rTTo^{1^{\tens k}\tens d_C} M(n_1)\tdt M(n_k)\tens C(k) \bigr], \hfill
\\
\hskip\multlinegap S_2 =\bigl[ M(n_1)\tdt M(n_k)\tens C(k) \rTTo^{\sum_{i=0}^{k-1}1^{\tens i}\tens d\tens1^{\tens(k-i)}} M(n_1)\tdt M(n_k)\tens C(k) \bigr], \hfill
\\
S_3 = \!\!\sum_{x+y+z=k} \!\! \bigl[ M(n_1)\tdt M(n_k)\tens C(k) \rTTo^{1^{\tens k}\tens\Delta_{T(x,y,z)}} M(n_1)\tdt M(n_k)\tens C(y)\tens C(x+1+z)
\\
\rTTo^{1^{\tens k}\tens\theta\tens1} M(n_1)\tdt M(n_k)\tens\co(y)\tens C(x+1+z) \rto\sim
\\
M(n_1)\tdt M(n_x)\tens M(n_{x+1})\tdt M(n_{x+y})\tens\co(y) \tens M(n_{x+y+1})\tdt M(n_k)\tens C(x+1+z)
\\
\rTTo^{1^{\tens x}\tens\alpha_{n_{x+1},\dots,n_{x+y}}\tens1^{\tens(z+1)}}
\\
M(n_1)\tdt M(n_x)\tens M(n_{x+1}+\dots+n_{x+y})\tens M(n_{x+y+1})\tdt M(n_k)\tens C(x+1+z) \bigr].
\end{multline*}
\end{proposition}

\begin{proof}
Let us prove that \(1\odot\Delta:M\odot C\to M\odot C\odot C\) is a chain map.
That is, the equation
\begin{diagram}
M\odot C &\rTTo^{d_{M\odot C}} &M\odot C
\\
\dTTo<{1\odot\Delta} &= &\dTTo>{1\odot\Delta}
\\
M\odot C\odot C &\rTTo^{1\odot1\odot d_C+\sum1_{M\odot C}^{\tens-}\tens d_{M\odot C}\tens1_{M\odot C}^{\tens-}\tens1_C} &M\odot C\odot C
\end{diagram}
holds.
Clearly, it holds if we replace $d_{M\odot C}$ with $S_1+S_2$ -- the sum of differentials on all places.
Let us show that it holds (without the summand \(1\odot1\odot d_C\)) for the third summand $S_3$ in place of $d_{M\odot C}$.
Restrict the source to the summand \(M(p_1)\tdt M(p_n)\tens C(n)\) and postcompose the equation with the projection \(1\odot1\odot\pr_{C(k)}\).
Let us verify that
\begin{equation}
(1\odot\Delta) \cdot \sum_{s=1}^k 1_{M\odot C}^{\tens(s-1)}\tens S_3\tens1_{M\odot C}^{\tens(k-s)}\tens1_C =S_3 \cdot (1\odot\Delta).
\label{eq-1-Delta}
\end{equation}
The left hand side is the sum over $n^1,\dots,n^k\in\NN$, $1\le s\le k$ and $x,y,z\in\NN$ such that $n^1+\dots+n^k=n$ and $x+y+z=n^s$ of maps
\begin{multline}
\bigl[ M(p_1)\tdt M(p_n)\tens C(n) \rTTo^{1^{\tens n}\tens\Delta_{n^1,\dots,n^k}} M(p_1)\tdt M(p_n)\tens C(n^1)\tdt C(n^k)\tens C(k)
\\
\rTTo^{1^{\tens(n+s-1)}\tens\Delta_{T(x,y,z)}\tens1^{\tens(k-s+1)}}
\\
M(p_1)\tdt M(p_n)\tens C(n^1)\tdt C(n^{s-1})\tens C(y)\tens C(x+1+z)\tens C(n^{s+1})\tdt C(n^k)\tens C(k)
\\
\rTTo^{1^{\tens(n+s-1)}\tens\theta\tens1^{\tens(1+k-s+1)}}
\\
M(p_1)\tdt M(p_n)\tens C(n^1)\tdt C(n^{s-1})\tens\co(y)\tens C(x+1+z)\tens C(n^{s+1})\tdt C(n^k)\tens C(k)
\\
\rto\sim
M(p_1)\tdt M(p_{n^1+\dots+n^{s-1}+x})\tens M(p_{n^1+\dots+n^{s-1}+x+1})\tdt M(p_{n^1+\dots+n^{s-1}+x+y})\tens\co(y)
\\
\tens M(p_{n^1+\dots+n^{s-1}+x+y+1})\tdt M(p_n)
\\
\hfill \tens C(n^1)\tdt C(n^{s-1})\tens C(x+1+z)\tens C(n^{s+1})\tdt C(n^k)\tens C(k) \quad
\\
\rTTo^{1^{\tens(n^1+\dots+n^{s-1}+x)}\tens\alpha_{p_{n^1+\dots+n^{s-1}+x+1},\dots,p_{n^1+\dots+n^{s-1}+x+y}}\tens1^{\tens(z+n^{s+1}+\dots+n^k+k+1)}}
\\
\hskip\multlinegap M(p_1)\tdt M(p_{n^1+\dots+n^{s-1}+x})\tens M(p_{n^1+\dots+n^{s-1}+x+1}+\dots+p_{n^1+\dots+n^{s-1}+x+y}) \hfill
\\
\tens M(p_{n^1+\dots+n^{s-1}+x+y+1})\tdt M(p_n)
\\
\tens C(n^1)\tdt C(n^{s-1})\tens C(x+1+z)\tens C(n^{s+1})\tdt C(n^k)\tens C(k)
\bigr].
\label{eq-L-522}
\end{multline}
The right hand side of \eqref{eq-1-Delta} is the sum over $a,y,c\in\NN$ and $l_1,\dots,l_k\in\NN$ such that $a+y+c=n$ and $l_1+\dots+l_k=a+1+c$ of maps
\begin{multline}
\bigl[ M(p_1)\tdt M(p_n)\tens C(n) \rTTo^{1^{\tens n}\tens\Delta_{T(a,y,c)}} M(p_1)\tdt M(p_n)\tens C(y)\tens C(a+1+c)
\\
\rTTo^{1^{\tens n}\tens\theta\tens\Delta_{l_1,\dots,l_k}} M(p_1)\tdt M(p_n)\tens\co(y)\tens C(l_1)\tdt C(l_k)\tens C(k) \rto\sim
\\
\hskip\multlinegap M(p_1)\tdt M(p_a)\tens M(p_{a+1})\tdt M(p_{a+y})\tens\co(y) \hfill
\\
\hfill \tens M(p_{a+y+1})\tdt M(p_n)\tens C(l_1)\tdt C(l_k)\tens C(k) \quad
\\
\rTTo^{1^{\tens a}\tens\alpha_{p_{a+1},\dots,p_{a+y}}\tens1^{\tens(c+k+1)}}
\\
\hskip\multlinegap M(p_1)\tdt M(p_a)\tens M(p_{a+1}+\dots+p_{a+y}) \hfill
\\
\tens M(p_{a+y+1})\tdt M(p_n)\tens C(l_1)\tdt C(l_k)\tens C(k) \bigr].
\label{eq-T-523}
\end{multline}
Indexing sets of the two sums are in bijection.
Given \((n^1,\dots,n^k,s,x,y,z)\) define
\begin{gather}
a =n^1+\dots+n^{s-1}+x, \qquad c =z+n^{s+1}+\dots+n^k, \notag
\\
l_s =x+1+z, \qquad l_j =n^j \quad \text{if} \quad j\ne s, \quad 1\le j\le k.
\label{eq-lsx1z-ljnj}
\end{gather}
Given \((a,y,c,l_1,\dots,l_k)\) find $s$ and $x$ from \(a=l_1+\dots+l_{s-1}+x\), \(0\le x<l_s\); define \(z=l_s-x-1\), \(n^s=x+y+z\), \(n^j=l_j\) if $j\ne s$, $1\le j\le k$.
This bijection identifies targets of maps \eqref{eq-L-522} and \eqref{eq-T-523}.
Moreover, these two maps are equal due to equation~\eqref{dia-coassociativity-components} written for parameters~\eqref{eq-lsx1z-ljnj} with \(i_p^j=y\) if $j=s$, $p=x+1$ and \(i_p^j=1\) otherwise.
In fact, postcompose \eqref{dia-coassociativity-components} with the tensor product of identity maps and counits \(\eps:C(1)\to\1\) applied to factors \(C(i_p^j)\) with \((j,p)\ne(s,x+1)\).
This gives an identity which allows to conclude the required equality.

Let us compute $d^2=(S_1+S_2+S_3)^2$ for \(M\odot^\theta C\).
We have
\begin{multline*}
S_1^2 =1^{\tens k}\tens\bigl[\Delta_{T(0,k,0)}\cdot(1\tens\bdelt_0)\bigr] -1^{\tens k}\tens\sum_{x+1+z=k}\Delta_{T(x,1,z)}\cdot(\bdelt_0\tens1):
\\
\hfill M(n_1)\tdt M(n_k)\tens C(k) \to M(n_1)\tdt M(n_k)\tens C(k), \quad
\\
\hskip\multlinegap S_1S_2 +S_2S_1 =0, \hfill
\\
\hskip\multlinegap S_2^2 =-\sum_{i=0}^{k-1}1^{\tens i}\tens(1_M\tens m^\co_0)\alpha\tens1^{\tens(k-i)}: \hfill
\\
M(n_1)\tdt M(n_k)\tens C(k) \to M(n_1)\tdt M(n_k)\tens C(k),
\end{multline*}
\vspace{-1.4em}
\begin{multline*}
S_2S_3 =-\sum_{x+y+z=k} \sum_{0\le i\le k}^{i\ne x+y} \bigl[ M(n_1)\tdt M(n_k)\tens C(k) \rTTo^{1^{\tens k}\tens\Delta_{T(x,y,z)}}
\\
M(n_1)\tdt M(n_k)\tens C(y)\tens C(x+1+z) \rTTo^{1^{\tens k}\tens\theta\tens1} M(n_1)\tdt M(n_k)\tens\co(y)\tens C(x+1+z) \rto\sim
\\
M(n_1)\tdt M(n_x)\tens M(n_{x+1})\tdt M(n_{x+y})\tens\co(y) \tens M(n_{x+y+1})\tdt M(n_k)\tens C(x+1+z)
\\
\rTTo^{1^{\tens i}\tens d\tens1^{\tens(k-i+1)}}
\\
M(n_1)\tdt M(n_x)\tens M(n_{x+1})\tdt M(n_{x+y})\tens\co(y) \tens M(n_{x+y+1})\tdt M(n_k)\tens C(x+1+z)
\\
\rTTo^{1^{\tens x}\tens\alpha_{n_{x+1},\dots,n_{x+y}}\tens1^{\tens(z+1)}}
\\
M(n_1)\tdt M(n_x)\tens M(n_{x+1}+\dots+n_{x+y})\tens M(n_{x+y+1})\tdt M(n_k)\tens C(x+1+z) \bigr]
\end{multline*}
\vspace{-1.4em}
\begin{multline*}
S_3S_2 =\sum_{x+y+z=k} \sum_{j=0}^{x+z} \bigl[ M(n_1)\tdt M(n_k)\tens C(k) \rTTo^{1^{\tens k}\tens\Delta_{T(x,y,z)}}
\\
M(n_1)\tdt M(n_k)\tens C(y)\tens C(x+1+z) \rTTo^{1^{\tens k}\tens\theta\tens1} M(n_1)\tdt M(n_k)\tens\co(y)\tens C(x+1+z) \rto\sim
\\
M(n_1)\tdt M(n_x)\tens M(n_{x+1})\tdt M(n_{x+y})\tens\co(y) \tens M(n_{x+y+1})\tdt M(n_k)\tens C(x+1+z)
\\
\rTTo^{1^{\tens x}\tens\alpha_{n_{x+1},\dots,n_{x+y}}\tens1^{\tens(z+1)}}
\\
M(n_1)\tdt M(n_x)\tens M(n_{x+1}+\dots+n_{x+y})\tens M(n_{x+y+1})\tdt M(n_k)\tens C(x+1+z)
\\
\rTTo^{1^{\tens j}\tens d\tens1^{\tens(x+z-j+1)}}
\\
M(n_1)\tdt M(n_x)\tens M(n_{x+1}+\dots+n_{x+y})\tens M(n_{x+y+1})\tdt M(n_k)\tens C(x+1+z) \bigr]
\end{multline*}
The sum $S_3S_2$ cancels nearly entire against the sum $S_2S_3$.
Still one sort of summands survives from these two and it has the form of
\begin{multline*}
S_1S_3 +S_3S_1 =\sum_{x+y+z=k} \bigl[ M(n_1)\tdt M(n_k)\tens C(k) \rTTo^{1^{\tens k}\tens\Delta_{T(x,y,z)}}
\\
M(n_1)\tdt M(n_k)\tens C(y)\tens C(x+1+z) \rTTo^{1^{\tens k}\tens d\theta\tens1} M(n_1)\tdt M(n_k)\tens\co(y)\tens C(x+1+z) \rto\sim
\\
M(n_1)\tdt M(n_x)\tens M(n_{x+1})\tdt M(n_{x+y})\tens\co(y) \tens M(n_{x+y+1})\tdt M(n_k)\tens C(x+1+z)
\\
\rTTo^{1^{\tens x}\tens\alpha_{n_{x+1},\dots,n_{x+y}}\tens1^{\tens(z+1)}}
\\
M(n_1)\tdt M(n_x)\tens M(n_{x+1}+\dots+n_{x+y})\tens M(n_{x+y+1})\tdt M(n_k)\tens C(x+1+z) \bigr]
\end{multline*}
with the only distinction, namely, $d\theta$ is replaced with $\theta d$.

The square $S_3^2$ of the third summand of $d$ is split into three sums.
Some parameters of the first two of these are described graphically by trees on \eqref{eq-tangles-3-vertices} in that order.
The last one can be represented by tree \eqref{eq-tangle-3-vertices-(4)} with 3 vertices of (maximal) height 3.
\begin{multline}
S_3^2 =\sum_{v+w+u=k}^{x+y+z=u} \bigl[ M(n_1)\tdt M(n_k)\tens C(k) \rTTo^{1^{\tens k}\tens\Delta_{T(v,w,u)}}
\\
M(n_1)\tdt M(n_k)\tens C(w)\tens C(v+1+u)
\\
\rTTo^{1^{\tens(k+1)}\tens\Delta_{T(v+1+x,y,z)}} M(n_1)\tdt M(n_k)\tens C(w)\tens C(y)\tens C(v+1+x+1+z)
\\
\rTTo^{1^{\tens k}\tens\theta\tens\theta\tens1} M(n_1)\tdt M(n_k)\tens\co(w)\tens\co(y)\tens C(v+1+x+1+z) \rto\sim
\\
M(n_1)\tdt M(n_v)\tens M(n_{v+1})\tdt M(n_{v+w})\tens\co(w) \tens M(n_{v+w+1})\tdt M(n_{v+w+x})
\\
\tens M(n_{v+w+x+1})\tdt M(n_{k-z})\tens\co(y)\tens M(n_{k-z+1})\tdt M(n_k)\tens C(v+1+x+1+z)
\\
\rTTo^{1^{\tens v}\tens\alpha_{n_{v+1},\dots,n_{v+w}}\tens1^{\tens x}\tens\alpha_{n_{v+w+x+1},\dots,n_{k-z}}\tens1^{\tens(z+1)}}
\\
\hskip\multlinegap M(n_1)\tdt M(n_v)\tens M(n_{v+1}+\dots+n_{v+w})\tens M(n_{v+w+1})\tdt M(n_{v+w+x}) \hfill
\\
\tens M(n_{v+w+x+1}+\dots+n_{k-z})\tens M(n_{k-z+1})\tdt M(n_k)\tens C(v+1+x+1+z) \bigr]
\label{eq-d20-510}
\end{multline}
\vspace{-1.4em}
\begin{multline}
+\sum_{v+w+x=a}^{a+y+z=k} \bigl[ M(n_1)\tdt M(n_k)\tens C(k) \rTTo^{1^{\tens k}\tens\Delta_{T(a,y,z)}} M(n_1)\tdt M(n_k)\tens C(y)\tens C(a+1+z)
\\
\rTTo^{1^{\tens(k+1)}\tens\Delta_{T(v,w,x+1+z)}} M(n_1)\tdt M(n_k)\tens C(y)\tens C(w)\tens C(v+1+x+1+z)
\\
\rTTo^{1^{\tens k}\tens\theta\tens\theta\tens1} M(n_1)\tdt M(n_k)\tens\co(y)\tens\co(w)\tens C(v+1+x+1+z) \rto\sim
\\
M(n_1)\tdt M(n_v)\tens M(n_{v+1})\tdt M(n_{v+w})\tens\co(w) \tens M(n_{v+w+1})\tdt M(n_a)
\\
\tens M(n_{a+1})\tdt M(n_{a+y})\tens\co(y) \tens M(n_{a+y+1})\tdt M(n_k)\tens C(v+1+x+1+z)
\\
\rTTo^{1^{\tens v}\tens\alpha_{n_{v+1},\dots,n_{v+w}}\tens1^{\tens x}\tens\alpha_{n_{a+1},\dots,n_{a+y}}\tens1^{\tens(z+1)}}
\\
\hskip\multlinegap M(n_1)\tdt M(n_v)\tens M(n_{v+1}+\dots+n_{v+w})\tens M(n_{v+w+1})\tdt M(n_a) \hfill
\\
\tens M(n_{a+1}+\dots+n_{a+y})\tens M(n_{a+y+1})\tdt M(n_k)\tens C(v+1+x+1+z) \bigr]
\label{eq-d20-509}
\end{multline}
\vspace{-1.4em}
\begin{multline}
+\sum_{\substack{v+w=p\\ y+z=q}}^{p+x+q=k} \bigl[ M(n_1)\tdt M(n_k)\tens C(k) \rTTo^{1^{\tens k}\tens\Delta_{T(p,x,q)}} M(n_1)\tdt M(n_k)\tens C(x)\tens C(p+1+q)
\\
\rTTo^{1^{\tens(k+1)}\tens\Delta_{T(v,w+1+y,z)}} M(n_1)\tdt M(n_k)\tens C(x)\tens C(w+1+y)\tens C(v+1+z)
\\
\rTTo^{1^{\tens k}\tens\theta\tens\theta\tens1} M(n_1)\tdt M(n_k)\tens\co(x)\tens\co(w+1+y)\tens C(v+1+z) \rto\sim
\\
M(n_1)\tdt M(n_v)\tens M(n_{v+1})\tdt M(n_p)\tens M(n_{p+1})\tdt M(n_{p+x})\tens\co(x)\tens
\\
M(n_{p+x+1})\tdt M(n_{k-z})\tens\co(w+1+y)\tens M(n_{k-z+1})\tdt M(n_k)\tens C(v+1+z)
\\
\rTTo^{1^{\tens p}\tens\alpha_{n_{p+1},\dots,n_{p+x}}\tens1^{\tens(q+1)}}
\\
\hskip\multlinegap M(n_1)\tdt M(n_v)\tens M(n_{v+1})\tdt M(n_p)\tens M(n_{p+1}+\dots+n_{p+x})\tens M(n_{p+x+1}) \hfill
\\
\hfill \tdt M(n_{k-z})\tens\co(w+1+y)\tens M(n_{k-z+1})\tdt M(n_k)\tens C(v+1+z) \quad
\\
\rTTo^{1^{\tens v}\tens\alpha_{n_{v+1},\dots,n_p,n_{p+1}+\dots+n_{p+x},n_{p+x+1},\dots,n_{k-z}}\tens1^{\tens(z+1)}}
\\
M(n_1)\tdt M(n_v)\tens M(n_{v+1}+\dots+n_{k-z})\tens M(n_{k-z+1})\tdt M(n_k)\tens C(v+1+z) \bigr].
\label{eq-d20-511}
\end{multline}
Sums \eqref{eq-d20-510} and \eqref{eq-d20-509} cancel each other due to equation~\eqref{dia-cooperad-3-OOOOOOO}.
The minus sign here pops out from the relation \(c(\theta\tens\theta)=-(\theta\tens\theta)c\) for the symmetry $c=(12)$.
Sum~\eqref{eq-d20-511} can be transformed thanks to equations \eqref{dia-cooperad-4-OOOOO} and \eqref{eq-MMOO-MMO-M}.
So we obtain
\begin{multline*}
d^2 =S_1^2 +S_2^2 +\sum_{v+s+z=k} \bigl[ M(n_1)\tdt M(n_k)\tens C(k) \rTTo^{1^{\tens k}\tens\Delta_{T(v,s,z)}}
\\
M(n_1)\tdt M(n_k)\tens C(s)\tens C(v+1+z)
\\
\rTTo^{1^{\tens k}\tens(d\theta+\theta d)\tens1} M(n_1)\tdt M(n_k)\tens\co(s)\tens C(v+1+z) \rto\sim
\\
M(n_1)\tdt M(n_v)\tens M(n_{v+1})\tdt M(n_{v+s})\tens\co(s) \tens M(n_{v+s+1})\tdt M(n_k)\tens C(v+1+z)
\\
\rTTo^{1^{\tens v}\tens\alpha_{n_{v+1},\dots,n_{v+s}}\tens1^{\tens(z+1)}}
\\
M(n_1)\tdt M(n_v)\tens M(n_{v+1}+\dots+n_{v+s})\tens M(n_{v+s+1})\tdt M(n_k)\tens C(v+1+z) \bigr]
\\
\hskip\multlinegap +\sum_{\substack{v+w=p\\ y+z=q}}^{p+x+q=k} \bigl[ M(n_1)\tdt M(n_k)\tens C(k) \rTTo^{1^{\tens k}\tens\Delta_{T(v,w+x+y,z)}} \hfill
\\
M(n_1)\tdt M(n_k)\tens C(w+x+y)\tens C(v+1+z) \rTTo^{1^{\tens k}\tens\Delta_{T(w,x,y)}\tens1}
\\
M(n_1)\tdt M(n_k)\tens C(x)\tens C(w+1+y)\tens C(v+1+z)
\rTTo^{1^{\tens k}\tens\theta\tens\theta\tens1}
\\
M(n_1)\tdt M(n_k)\tens\co(x)\tens\co(w+1+y)\tens C(v+1+z) \rTTo^{1^{\tens k}\tens\mu_{T(w,x,y)}\tens1}
\\
M(n_1)\tdt M(n_k)\tens\co(w+x+y)\tens C(v+1+z) \rto\sim
\\
\hskip\multlinegap M(n_1)\tdt M(n_v)\tens M(n_{v+1})\tdt M(n_{k-z})\tens\co(w+x+y) \hfill
\\
\hfill \tens M(n_{k-z+1})\tdt M(n_k)\tens C(v+1+z) \quad
\\
\rTTo^{1^{\tens v}\tens\alpha_{n_{v+1},\dots,n_{k-z}}\tens1^{\tens(z+1)}}
\\
M(n_1)\tdt M(n_v)\tens M(n_{v+1}+\dots+n_{k-z})\tens M(n_{k-z+1})\tdt M(n_k)\tens C(v+1+z) \bigr].
\end{multline*}
Transform this expression applying equation~\eqref{eq-CCO-COO-CCCOOO} and denoting $w+x+y$ by $s$ in the last sum.
Only summands with $s=1$ survive:
\begin{multline*}
d^2 =\bigl[ M(n_1)\tdt M(n_k)\tens C(k) \rTTo^{1^{\tens k}\tens[\Delta_{T(0,k,0)}\cdot(1\tens\bdelt_0) -\sum_{x+1+z=k}\Delta_{T(x,1,z)}\cdot(\bdelt_0\tens1)]}
\\
\hfill M(n_1)\tdt M(n_k)\tens C(k) \bigr] \quad
\\
-\bigl[ M(n_1)\tdt M(n_k)\tens C(k) \rTTo^{\sum_{i=0}^{k-1}1^{\tens i}\tens(1_M\tens m^\co_0)\alpha\tens1^{\tens(k-i)}} M(n_1)\tdt M(n_k)\tens C(k) \bigr]
\\
+\sum_{v+1+z=k} \bigl[ M(n_1)\tdt M(n_k)\tens C(k) \rTTo^{1^{\tens k}\tens\Delta_{T(v,1,z)}} M(n_1)\tdt M(n_k)\tens C(1)\tens C(k)
\\
\rTTo^{1^{\tens k}\tens(\eps m_0+\bdelt_0\eta)\tens1} M(n_1)\tdt M(n_k)\tens\co(1)\tens C(k) \rto\sim
\\
M(n_1)\tdt M(n_v)\tens M(n_{v+1})\tens\co(1)\tens M(n_{v+2})\tdt M(n_k)\tens C(k)
\\
\rTTo^{1^{\tens v}\tens\alpha_{n_{v+1}}\tens1^{\tens(z+1)}} M(n_1)\tdt M(n_v)\tens M(n_{v+1})\tens M(n_{v+2})\tdt M(n_k)\tens C(k) \bigr]
\\
=\bigl[ M(n_1)\tdt M(n_k)\tens C(k) \rTTo^{1^{\tens k}\tens[\Delta_{T(0,k,0)}\cdot(1\tens\bdelt_0)]} M(n_1)\tdt M(n_k)\tens C(k) \bigr].
\end{multline*}
This is precisely the value required in \defref{def-curved-coalgebra-over-curved-augmented-cooperad}.

Clearly, $-\odot^\theta C$ takes a morphism $f$ to the morphism $f\odot1$, therefore, it is a functor.
\end{proof}

Let C be a curved augmented cooperad with \(\bar C\in\cw^\NN_{++}\).
Denote by \(\comodplur C\) the\index{TTsymb}{comod+C@$\comodplur C$} category of curved right $C$\n-comodules \(N\in\cw^\mm\) such that \(N(0)_0=0\), that is, \(N\in\cw^\mm_+\) (connected comodules).
Similarly, denote by \(\modplur\co\) the\index{TTsymb}{mod+O@$\modplur\co$} category of curved right $\co$\n-modules \(M\in\cw^\mm_+\).
Call them\index{TTindex}{connected module over an operad} \emph{connected}.
Due to \remref{rem-number-trees-finite} \(N\odot C\simeq N\bar\odot C\).
There is a functor \(\cw^\mm_+\to\cw^\mm_+\), \(X\mapsto X\odot C\).
Since \(C(0)_0=0=C(1)_0\), it is a comonad and \(\comodplur C\) is equivalent to the category of coalgebras over the comonad \(-\odot C:\cw^\mm_+\to\cw^\mm_+\).
The category of right $\co$\n-modules is equivalent to the category of algebras over the monad \(-\odot\co:\cw^\mm\to\cw^\mm\), which under assumption \(\co(0)_0=0\) restricts to a monad \(\cw^\mm_+\to\cw^\mm_+\).

\begin{proposition}
Let C be a curved augmented cooperad such that \(\bar C\in\cw^\NN_{++}\).
A twisting cochain $\theta:C\to\co$ determines a functor \(\comodplur C\to\modplur\co\), \(N\mapsto N\odot_\theta\co\), which is the free right $\co$\n-module \(N\odot\co\) equipped with the derivation
\begin{multline*}
d_{N\odot_\theta\co} =S_1 +S_2 -S_3,
\\
\hskip\multlinegap S_1 =\bigl[ N(n_1)\tdt N(n_k)\tens\co(k) \rTTo^{1^{\tens k}\tens d_\co} N(n_1)\tdt N(n_k)\tens\co(k) \bigr], \hfill
\\
\hskip\multlinegap S_2 =\bigl[ N(n_1)\tdt N(n_k)\tens\co(k) \rTTo^{\sum_{i=0}^{k-1}1^{\tens i}\tens d_N\tens1^{\tens(k-i)}} N(n_1)\tdt N(n_k)\tens\co(k) \bigr], \hfill
\\
\hskip\multlinegap S_3 =\sum_{x+1+z=k,\,y\in\NN}^{a_1+\dots+a_y=n_{x+1}} \bigl[ N(n_1)\tdt N(n_k)\tens\co(k) \rTTo^{1^{\tens x}\tens\delta_{a_1,\dots,a_y}\tens1^{\tens(z+1)}} \hfill
\\
N(n_1)\tdt N(n_x)\tens N(a_1)\tdt N(a_y)\tens C(y)\tens N(n_{x+2})\tdt N(n_k)\tens\co(k) \rto\sim
\\
N(n_1)\tdt N(n_x)\tens N(a_1)\tdt N(a_y)\tens N(n_{x+2})\tdt N(n_k)\tens C(y)\tens\co(k) \rTTo^{1^{\tens(x+y+z)}\tens\theta\tens1}
\\
N(n_1)\tdt N(n_x)\tens N(a_1)\tdt N(a_y)\tens N(n_{x+2})\tdt N(n_k)\tens\co(y)\tens\co(x+1+z)
\\
\rTTo^{1^{\tens(x+y+z)}\tens\mu_{T(x,y,z)}}
\\
N(n_1)\tdt N(n_x)\tens N(a_1)\tdt N(a_y)\tens N(n_{x+2})\tdt N(n_k)\tens\co(x+y+z) \bigr].
\end{multline*}
\end{proposition}

\begin{proof}
This proposition is ``ideologically dual'' to \propref{pro-twisting-cochain-functor}.
The full duality is lacking, however, at the level of formulas it is exhibited.
The conditions on $N$ and $C$ imply that \(N\bar\odot C=N\odot C\) by \remref{rem-number-trees-finite}.
Therefore the sum $S_3$ makes sense.

Let us prove that \(1\odot m:N\odot\co\odot\co\to N\odot\co\) is a chain map.
That is, the equation
\begin{diagram}
N\odot\co\odot\co &\rTTo^{1\odot1\odot d_\co+\sum1_{N\odot\co}^{\tens-}\tens d_{N\odot\co}\tens1_{N\odot\co}^{\tens-}\tens1_\co} &N\odot\co\odot\co
\\
\dTTo<{1\odot m} &= &\dTTo>{1\odot m}
\\
N\odot\co &\rTTo^{d_{N\odot\co}} &N\odot\co
\end{diagram}
holds.
Clearly, it holds if we replace $d_{N\odot\co}$ with $S_1+S_2$ -- the sum of differentials on all places.
Let us show that it holds (without the summand \(1\odot1\odot d_\co\)) for the third summand $S_3$ in place of $d_{N\odot\co}$, that is,
\begin{equation*}
\sum_{s=1}^k \bigl[ 1_{N\odot\co}^{\tens(s-1)}\tens S_3\tens1_{N\odot\co}^{\tens(k-s)}\tens1 \bigr] \cdot (1\odot m) =(1\odot m) \cdot S_3: N\odot\co\odot\co \to N\odot\co.
\end{equation*}
It is easy to see that this equation reduces to associativity \eqref{dia-operad-4-OOOOOOO} of multiplication in operad $\co$.
Namely, two ways to write down the multiplication
\begin{multline*}
m_t: \co(y)\tens\co(n_1)\tdt\co(n_{s-1})\tens\co(x+1+z)\tens\co(n_{s+1})\tdt\co(n_k)\tens\co(k) \\
\to \co(n_1+\dots+n_{s-1}+x+y+z+n_{s+1}+\dots+n_k)
\end{multline*}
coincide.
Both correspond to the tree
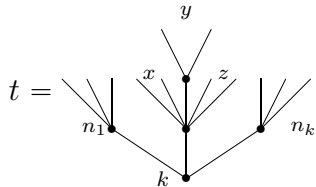
\[ t=
\begin{tangles}{rcl}
&\object{\sss y} & \\
&\object{\sss x}\hstep\nw1\node\ne1\hstep\object{\sss z} & \\
\nw2\nw1\nodel{n_1}\n &\nw2\nw1\n\ne1\ne2 &\n\ne1\noder{n_k}\ne2 \\
& \nw3\step[1.5]\nodel{k}\hstep\n\Step\ne3 &
\end{tangles}
\]
and the identity is obtained from \eqref{dia-operad-4-OOOOOOO} by inserting \(n_1+\dots+n_{s-1}+x+z+n_{s+1}+\dots+n_k\) units \(1_\co\in\co(1)\).

Let us compute $d^2=(S_1+S_2-S_3)^2$ for \(N\odot_\theta\co\).
We have
\begin{multline*}
S_1^2 =1^{\tens k}\tens\sum_{x+1+z=k}(m_0\tens1)m_{T(x,1,z)} -1^{\tens k}\tens
\bigl[(1\tens m_0)m_{T(0,k,0)}\bigr]:
\\
\hfill N(n_1)\tdt N(n_k)\tens\co(k) \to N(n_1)\tdt N(n_k)\tens\co(k), \quad
\\
\hskip\multlinegap S_1S_2 +S_2S_1 =0, \hfill
\\
S_2^2 =\sum_{i=0}^{k-1}1^{\tens i}\tens\delta_1(1\tens\bdelt_0)\tens1^{\tens(k-i)}: N(n_1)\tdt N(n_k)\tens\co(k) \to N(n_1)\tdt N(n_k)\tens\co(k),
\end{multline*}
\vspace{-1.4em}
\begin{multline*}
-S_2S_3 -S_3S_2 =-\sum_{x+1+z=k,\,y\in\NN}^{a_1+\dots+a_y=n_{x+1}} \bigl[ N(n_1)\tdt N(n_k)\tens\co(k) \rTTo^{1^{\tens x}\tens\delta_{a_1,\dots,a_y}\tens1^{\tens(z+1)}}
\\
N(n_1)\tdt N(n_x)\tens N(a_1)\tdt N(a_y)\tens C(y)\tens N(n_{x+2})\tdt N(n_k)\tens\co(k) \rto\sim
\\
N(n_1)\tdt N(n_x)\tens N(a_1)\tdt N(a_y)\tens N(n_{x+2})\tdt N(n_k)\tens C(y)\tens\co(k)
\\
\rTTo^{1^{\tens(x+y+z)}\tens d\theta\tens1}
\\
N(n_1)\tdt N(n_x)\tens N(a_1)\tdt N(a_y)\tens N(n_{x+2})\tdt N(n_k)\tens\co(y)\tens\co(x+1+z)
\\
\rTTo^{1^{\tens(x+y+z)}\tens\mu_{T(x,y,z)}}
\\
N(n_1)\tdt N(n_x)\tens N(a_1)\tdt N(a_y)\tens N(n_{x+2})\tdt N(n_k)\tens\co(x+y+z) \bigr]
\end{multline*}
Contribution $-S_1S_3-S_3S_1$ differs from the above expression only in one map, namely, $d\theta$ is replaced with $\theta d$.

The square $S_3^2$ of the third summand of $d$ is split into three sums.
Some parameters of the first two of these are described graphically by trees on \eqref{eq-tangles-3-vertices}.
The remaining sum could be represented by tree \eqref{eq-tangle-3-vertices-(4)} with 3 vertices of (maximal) height~3.
The first two cancel each other due to identity
\begin{multline*}
\bigl[ C(w)\tens C(y)\tens\co(v+1+x+1+z) \rTTo^{1\tens\theta\tens1} C(w)\tens\co(y)\tens\co(v+1+x+1+z) \rTTo^{1\tens\underset{v+1+x,y,z}\bull\;}
\\
C(w)\tens\co(v+1+x+y+z) \rTTo^{\theta\tens1} \co(w)\tens\co(v+1+x+y+z) \rTTo^{\underset{v,w,x+y+z}\bull\;} \co(v+w+x+y+z) \bigr]
\\
\hskip\multlinegap =-\bigl[ C(w)\tens C(y)\tens\co(v+1+x+1+z) \rTTo^{\theta\tens1\tens1} \co(w)\tens C(y)\tens\co(v+1+x+1+z) \hfill
\\
\rTTo^{(12)}_\sim C(y)\tens\co(w)\tens\co(v+1+x+1+z) \rTTo^{1\tens\underset{v,w,x+1+z}\bull\;} C(y)\tens\co(v+w+x+1+z)
\\
\rTTo^{\theta\tens1} \co(y)\tens\co(v+w+x+1+z) \rTTo^{\underset{v+w+x,y,z}\bull\;} \co(v+w+x+y+z) \bigr],
\end{multline*}
see equation~\eqref{dia-operad-3-OOOOOOOOOO}.
Thus,
\begin{multline*}
S_3^2 =-\sum_{v+1+z=k,\,p\in\NN}^{a_1+\dots+a_p=n_{v+1}} \sum_{w+1+y=p,\,x\in\NN}^{c_1+\dots+c_x=a_{w+1}} \bigl[ N(n_1)\tdt N(n_k)\tens\co(k) \rTTo^{1^{\tens v}\tens\delta_{a_1,\dots,a_p}\tens1^{\tens(z+1)}} \hfill
\\
N(n_1)\tdt N(n_v)\tens N(a_1)\tdt N(a_p)\tens C(p)\tens N(n_{v+2})\tdt N(n_k)\tens\co(k)
\\
\rTTo^{1^{\tens(v+w)}\tens\delta_{c_1,\dots,c_x}\tens1^{\tens(y+1+z+1)}}
\\
\hskip\multlinegap N(n_1)\tdt N(n_v)\tens N(a_1)\tdt N(a_w)\tens N(c_1)\tdt N(c_x)\tens C(x) \hfill
\\
\hfill \tens N(a_{w+2})\tdt N(a_p)\tens C(p)\tens N(n_{v+2})\tdt N(n_k)\tens\co(k) \rto\sim \quad
\\
\hskip\multlinegap N(n_1)\tdt N(n_v)\tens N(a_1)\tdt N(a_w)\tens N(c_1)\tdt N(c_x) \hfill
\\
\hfill \tens N(a_{w+2})\tdt N(a_p)\tens N(n_{v+2})\tdt N(n_k)\tens C(x)\tens C(p)\tens\co(k) \quad
\\
\rTTo^{1^{\tens(v+w+x+y+z)}\tens\theta\tens\theta\tens1}
\\
\hskip\multlinegap N(n_1)\tdt N(n_v)\tens N(a_1)\tdt N(a_w)\tens N(c_1)\tdt N(c_x) \hfill
\\
\hfill \tens N(a_{w+2})\tdt N(a_p)\tens N(n_{v+2})\tdt N(n_k)\tens\co(x)\tens\co(p)\tens\co(k) \quad
\\
\rTTo^{1^{\tens(v+w+x+y+z+1)}\tens\mu_{T(v,p,z)}}
\\
\hskip\multlinegap N(n_1)\tdt N(n_v)\tens N(a_1)\tdt N(a_w)\tens N(c_1)\tdt N(c_x) \hfill
\\
\hfill \tens N(a_{w+2})\tdt N(a_p)\tens N(n_{v+2})\tdt N(n_k)\tens\co(x)\tens\co(v+p+z) \quad
\\
\rTTo^{1^{\tens(v+w+x+y+z)}\tens\mu_{T(v+w,x,y+z)}}
\\
\hskip\multlinegap N(n_1)\tdt N(n_v)\tens N(a_1)\tdt N(a_w)\tens N(c_1)\tdt N(c_x) \hfill
\\
\tens N(a_{w+2})\tdt N(a_p)\tens N(n_{v+2})\tdt N(n_k)\tens\co(v+w+x+y+z) \bigr].
\end{multline*}
This expression is transformed thanks to equations \eqref{eq-NNNNNNNNNNNNNNNNNNNN} and \eqref{dia-operad-4-OOOOOOO}.
Combining all previous formulae and using \eqref{eq-CCO-COO-CCCOOO} we get
\begin{multline*}
d_{N\odot\co}^2 =\Bigl[ 1^{\tens k}\tens\sum_{x+1+z=k}(m_0\tens1)m_{T(x,1,z)} -1^{\tens k}\tens[(1\tens m_0)m_{T(0,k,0)}]
\\
+\sum_{i=0}^{k-1}1^{\tens i}\tens\delta_1(1\tens\bdelt_0)\tens1^{\tens(k-i)}: N(n_1)\tdt N(n_k)\tens\co(k) \to N(n_1)\tdt N(n_k)\tens\co(k) \Bigr]
\\
\hskip\multlinegap -\sum_{x+1+z=k} \bigl[ N(n_1)\tdt N(n_k)\tens\co(k) \rTTo^{1^{\tens x}\tens\delta_1\tens1^{\tens(z+1)}} \hfill
\\
N(n_1)\tdt N(n_x)\tens N(n_{x+1})\tens C(1)\tens N(n_{x+2})\tdt N(n_k)\tens\co(k) \rto\sim
\\
N(n_1)\tdt N(n_x)\tens N(n_{x+1})\tens N(n_{x+2})\tdt N(n_k)\tens C(1)\tens\co(k) \rTTo^{1^{\tens k}\tens(\eps m_0+\bdelt_0\eta)\tens1}
\\
N(n_1)\tdt N(n_x)\tens N(n_{x+1})\tens N(n_{x+2})\tdt N(n_k)\tens\co(1)\tens\co(x+1+z)
\\
\hfill \rTTo^{1^{\tens k}\tens\mu_{T(x,1,z)}} N(n_1)\tdt N(n_k)\tens\co(k) \bigr] \quad
\\
=-\bigl[ 1^{\tens k}\tens[(1\tens m_0)m_{T(0,k,0)}]:
N(n_1)\tdt N(n_k)\tens\co(k) \to N(n_1)\tdt N(n_k)\tens\co(k) \bigr].
\end{multline*}
This is the second equation telling that \(N\odot_\theta\co\) is a curved $\co$\n-module.

Clearly, $-\odot_\theta\co$ takes a morphism $f$ to the morphism $f\odot1$, therefore, it is a functor.
\end{proof}

Ignoring (co)derivations we have by applying twice \lemref{lem-forgetful-functor-has-right-adjoint-T} a natural bijection
\[ \modplurnd\co(N\odot\co,M) \simeq \cw^\mm_+(N,M) \simeq \comodplurnd C(N,M\odot C),
\]
where $\modplurnd\co$ denotes plain $\co$\n-modules $M$ (not equipped with a derivation) such that \(M(0)_0=0\) and $\comodplurnd C$ denotes plain $C$\n-comodules $N$ (not equipped with a coderivation) such that \(N(0)_0=0\).
Thus, functors
\[ -\odot\co: \comodplurnd C \leftrightarrows \modplurnd\co: -\odot C
\]
are adjoint to each other.

\begin{proposition}
Let C be a curved augmented cooperad such that \(\bar C\in\cw^\NN_{++}\).
For any twisting cochain $\theta:C\to\co$ the functors
\[ -\odot_\theta\co: \comodplur C \leftrightarrows \modplur\co: -\odot^\theta C
\]
are adjoint to each other.
\end{proposition}

\begin{proof}
A map \(f:N\to M\in\cw^\mm_+\) corresponds to a morphism of $\co$\n-modules \(g=\bigl(N\odot\co\rTTo^{f\odot1} M\odot\co\rto\alpha M\bigr)\) and to a morphism of $C$\n-comodules \(h=\bigl(N\rto\delta N\odot C\rTTo^{f\odot1} M\odot C\bigr)\).
To be a chain map $g$ has to satisfy equation
\begin{equation}
\bigl( N\odot_\theta\co \rto g M \rto d M \bigr) =\bigl( N\odot_\theta\co \rto d N\odot_\theta\co \rto g M \bigr).
\label{eq-NOMM-NONOM}
\end{equation}
Both sides of this equation are \((g;\id_\co,d_\co)\)\n-derivations and the source is a free $\co$\n-module.
Hence, both sides are determined by their restriction to $N$ and equation~\eqref{eq-NOMM-NONOM} is equivalent to
\[ \bigl( N =N\odot\1 \rMono^{1\odot\eta} N\odot\co \rto g M \rto d M \bigr) =\bigl( N =N\odot\1 \rMono^{1\odot\eta} N\odot_\theta\co \rto d N\odot_\theta\co \rto g M \bigr).
\]
Thus, $g$ is a chain map iff for all $n\ge0$
\begin{multline}
\bigl[ N(n) \rto f M(n) \rto d M(n) \bigr] =\bigl[ N(n) \rto d N(n) \rto f M(n) \bigr]
\\
\hskip\multlinegap -\sum^{y\in\NN}_{a_1+\dots+a_y=n} \bigl[ N(n) \rTTo^{\delta_{a_1,\dots,a_y}} N(a_1)\tdt N(a_y)\tens C(y) \hfill
\\
\rTTo^{f^{\tens y}\tens\theta} M(a_1)\tdt M(a_y)\tens\co(y) \rTTo^{\alpha_{a_1,\dots,a_y}} M(n) \bigr].
\label{eq-fd-df-delta-alpha}
\end{multline}
To be a chain map $h$ has to satisfy equation
\begin{equation}
\bigl( N \rto d N \rto h M\odot^\theta C \bigr) =\bigl( N \rto h M\odot^\theta C  \rto d M\odot^\theta C \bigr).
\label{eq-NNMC-NMCMC}
\end{equation}
Both sides of this equation are \((h;\id_C,d_C)\)\n-coderivations and the target is a cofree $C$\n-comodule.
Hence, both sides are determined by their projections to $M$ and equation~\eqref{eq-NNMC-NMCMC} is equivalent to
\[ \bigl( N \rto d N \rto h M\odot^\theta C \rTTo^{1\odot\eps} M\odot\1 =M \bigr) =\bigl( N \rto h M\odot^\theta C \rto d M\odot^\theta C \rTTo^{1\odot\eps} M\odot\1 =M \bigr).
\]
Thus, $h$ is a chain map iff for all $n\ge0$
\begin{multline*}
\bigl[ N(n) \rto d N(n) \rto f M(n) \bigr] =\bigl[ N(n) \rto f M(n) \rto d M(n) \bigr]
\\
\hskip\multlinegap +\sum^{k\in\NN}_{n_1+\dots+n_k=n} \bigl[ N(n) \rTTo^{\delta_{n_1,\dots,n_k}} N(n_1)\tdt N(n_k)\tens C(k) \hfill
\\
\rTTo^{f^{\tens k}\tens\theta} M(n_1)\tdt M(n_k)\tens\co(k) \rTTo^{\alpha_{n_1,\dots,n_k}} M(n) \bigr].
\end{multline*}
This condition coincides with \eqref{eq-fd-df-delta-alpha}, therefore, $g$ and $h$ are chain maps simultaneously.
\end{proof}

\backmatter
\providecommand{\bysame}{\leavevmode\hbox to3em{\hrulefill}\thinspace}

\Printindex{TTsymb}{Index of notation}
\Printindex{TTindex}{Index of terminology}

\end{document}